\documentclass[12pt,letterpaper]{article}
\usepackage[left=1in,right=1in,top=1in,bottom=1in]{geometry}
\usepackage[backend=biber,citestyle=authoryear,style=apa,natbib=true,maxcitenames=2]{biblatex}
\usepackage{times}
\usepackage{authblk} 
\usepackage{lineno}
\usepackage{xr-hyper}
\usepackage{amsfonts}
\usepackage{amsthm}
\usepackage{authblk}
\usepackage{amsmath}
\usepackage{amssymb}
\usepackage{array}
\usepackage{bm} 
\usepackage{booktabs}
\usepackage{caption}
\usepackage{color}
\usepackage{enumerate}
\usepackage{fancyhdr}
\usepackage{float} 
\usepackage{graphicx}
\usepackage[usenames,dvipsnames]{xcolor}
\usepackage[colorlinks,linkcolor=blue,citecolor=blue,urlcolor=blue]{hyperref}
\usepackage{placeins}
\usepackage{siunitx} 
\graphicspath{ {graphics/} }
\usepackage{verbatim}
\usepackage{multirow}
\usepackage{enumitem}
\usepackage{setspace} 
\usepackage{titlesec}
\DeclareUnicodeCharacter{FB01}{fi}
\usepackage[nottoc,numbib]{tocbibind}
\usepackage[ruled,vlined]{algorithm2e}
\usepackage{algpseudocode}
\usepackage{csquotes}
\usepackage{subcaption}
\usepackage{accents}
\usepackage{lscape}

\newcommand\numberthis{\addtocounter{equation}{1}\tag{\theequation}}

\allowdisplaybreaks 

\makeatletter
\g@addto@macro\normalsize{%
  \setlength\abovedisplayskip{10pt}
  \setlength\belowdisplayskip{10pt}
  \setlength\abovedisplayshortskip{10pt}
  \setlength\belowdisplayshortskip{10pt}
}
\makeatother

\makeatletter
\newcommand*{\addFileDependency}[1]{
  \typeout{(#1)}
  \@addtofilelist{#1}
  \IfFileExists{#1}{}{\typeout{No file #1.}}
}
\makeatother

\newcommand*\linenomathpatch[1]{%
  \cspreto{#1}{\linenomath}%
  \cspreto{#1*}{\linenomath}%
  \csappto{end#1}{\endlinenomath}%
  \csappto{end#1*}{\endlinenomath}%
}

\linenomathpatch{equation}
\linenomathpatch{gather}
\linenomathpatch{multline}
\linenomathpatch{align}
\linenomathpatch{alignat}
\linenomathpatch{flalign}

\newtheorem{theorem}{Theorem} \newtheorem{lemma}{Lemma}   
\newtheorem{corollary}{Corollary} 
\def\bg{\begin{figure}[tbph]\begin{center}}
\def\eg{\end{center}\end{figure}}

\def\diag{\mbox{diag}}
\newcommand{\condition}[2]{#1#2}

\newcommand{\bmbeta}{\bm{\beta}}

\newcommand{\bmtheta}{\bm{\theta}}

\AtEveryBibitem{
\clearfield{number}
}

\newtheorem{innercustomthm}{Corollary}
\newenvironment{customthm}[1]
  {\renewcommand\theinnercustomthm{#1}\innercustomthm}
  {\endinnercustomthm}

\begin{document}

\title{Asymptotic Results for Penalized Quasi-Likelihood Estimation in Generalized Linear Mixed Models}
\author[1]{Xu Ning\thanks{Corresponding author: Email: Xu.Ning@unsw.edu.au. Address: Research School of Finance, Actuarial Studies and Statistics, The Australian National University, Canberra, ACT, 2601, Australia}}
\author[1]{Francis K.C. Hui}
\author[1]{A. H. Welsh}
\affil[1]{Research School of Finance, Actuarial Studies and Statistics,
The Australian National University}
\date{}
\maketitle
\thispagestyle{empty}

\begin{abstract}
Generalized Linear Mixed Models (GLMMs) are widely used for analysing clustered data. 
One well-established method of overcoming the integral in the marginal likelihood function for GLMMs is penalized quasi-likelihood (PQL) estimation, although to date there are few asymptotic distribution results relating to PQL estimation for GLMMs in the literature. In this paper, we establish large sample results for PQL estimators of the parameters and random effects in independent-cluster GLMMs, when both the number of clusters and the cluster sizes go to infinity. This is done under two distinct regimes: conditional on the random effects (essentially treating them as fixed effects) and unconditionally (treating the random effects as random). Under the conditional regime, we show the PQL estimators are asymptotically normal around the true fixed and random effects. Unconditionally, we prove that while the estimator of the fixed effects is asymptotically normally distributed,
the correct asymptotic distribution of the so-called prediction gap of the random effects may in fact be a normal scale-mixture distribution under certain relative rates of growth. A simulation study is used to verify the finite sample performance of our theoretical results. 
\end{abstract}
\providecommand{\keywords}[1]{\textbf{Keywords\ \ \ } #1}


\section{Introduction} \label{sec:intro}

Generalized linear mixed models (GLMMs) are widely used in statistics to model relationships in clustered and correlated data
\citep{mcculloch2004generalized}.
As the marginal likelihood function of GLMMs, except for normally distributed responses with the identity link, contains an intractable integral, many methods have been developed to estimate and perform inference for the parameters in a computationally efficient manner. 
These include the Laplace approximation, Gauss-Hermite quadrature, and variational approximations, among others \citep{mcculloch2004generalized,ormerod2012gaussian,glmmtmb}. 
A connected and well-established approach is penalized Quasi-Likelihood (PQL) estimation \citep{breslow1993approximate}.
As one of the first methods to circumvent the intractable integral, PQL estimation has seen a resurgence in modern statistics as a very computationally efficient method for high-dimensional multivariate GLMMs \citep[e.g.,][]{hui2021,kidzinski2022generalized}. 
However, despite its long history, large sample distributional results for PQL estimation in mixed models are scarce. 

The most often studied asymptotic results for maximum likelihood estimators of GLMMs are based on increasing the number of clusters while keeping the size of each cluster fixed or bounded \citep{mcculloch2004generalized,NIE20071787}. 
Asymptotic results when both the cluster size and number of clusters grow are less developed, although some results for the maximum likelihood estimator as well as the empirical best linear unbiased predictor (EBLUP) for the linear mixed model (LMM) have been developed; see \citet{lyu2021asymptotics, lyu2021increasing} and references therein. 
Recently, \citet{jiang2021usable} proved an asymptotic normality result for a maximum quasi-likelihood estimator of the fixed parameters, which is different from the PQL estimator, for independent-cluster GLMMs. 

This work is distinct from the above results: compared to \citet{lyu2021asymptotics, lyu2021increasing}
we consider a more general random effects structure that permits random slopes in GLMMs. 
Meanwhile, \citet{jiang2021usable} considered GLMMs but not the case when cluster sizes grow faster than the number of clusters; nor did they present results for predictors of random effects, both of which are considered in this article. Furthermore, we establish results for the prediction gap in GLMMs, which are new to the literature and allow unconditional inference to be performed for the random effects \citep[noting unconditional inference for random effects in LMMs has been considered previously in a very different way through the unconditional mean squared error of prediction,][]{kackar1984approximations,prasad1990estimation}.
Note for the PQL estimator specifically, \citet{vonesh2002conditional}, \citet{hui2017joint} and \citet{hui2021} demonstrated estimation consistency under increasing cluster size and number of clusters, but did not develop any large sample distributional results. 



It is important to remark that when cluster sizes do not increase, PQL is known to be asymptotically biased \citep[e.g.,][]{breslow1995bias}. 
As such, increasing both the number of clusters and cluster size is a necessary condition for the PQL estimator to be consistent. Indeed, increasing number of clusters and cluster size is necessary for the consistency of other estimators such as the Laplace estimator \citep{ogden2017asymptotic,hui2021,ogden2021error}.
With this in mind, we develop our large sample distributional results under this setting, with the precise rates of  growth to be formalised later. We note this asymptotic framework is relevant for many applications with large cluster sizes e.g., educational studies with large numbers of students (units) grouped within schools (clusters), and
medical studies with large groups (clusters) of patients (units) treated at different hospitals. 

We derive our asymptotic results for the PQL estimator under two distinct sampling regimes. In the first, we condition on the random effects, i.e. treat them as fixed effects, although we will still refer to them as random effects for consistency. In the second, unconditional regime, we allow the random effects to be random. Conditional inference is appropriate when hypothetical resampling occurs within the same observed clusters, while unconditional inference may be more appropriate when (new) clusters are sampled from some population.  
Importantly, we demonstrate the asymptotic distributional results for the two regimes differ markedly. 
Conditional on the random effects we show the PQL estimator is asymptotically normally distributed around the true parameter values, with a convergence rate of ${N}^{1/2}$ (square root of the total number of observations) for the fixed effects and $n_i^{1/2}$ (square root of the cluster size of the $i$th cluster) for the random effects (which are now also fixed parameters). 
We find that 
when a variable is included as both a fixed and random effect covariate, the PQL estimator is asymptotically normally distributed around a sum-to-zero reparametrized version of the estimand. 
Unconditionally, we demonstrate the asymptotic normality of the PQL estimator for the fixed effects around the true values, but with a slower convergence rate of ${m}^{1/2}$ (square root of the number of clusters). Furthermore, 
we demonstrate that the ``prediction gap" i.e., the difference between the PQL estimator of and the true random effect, is not in general asymptotically normally distributed; 
instead, it follows a normal scale-mixture when $m$ grows faster than $n_i$.


Our results have important implications for inference in GLMMs. There is a choice of whether conditional or unconditional inference is desired, with different asymptotic distributions needing to be applied in each case. 
Also, the potential asymptotic non-normality of the prediction gap has consequences in practice, since normality is often assumed when constructing prediction intervals for the random effects in GLMMs \citep{bates2015fitting,glmmtmb}. 
The theoretical results in this paper offer an important step towards more formal, rigorous asymptotic inference using PQL estimation (and perhaps other similar estimators) for GLMMs.

The structure of the article is as follows. In Section 2, we introduce GLMMs and PQL estimation. In Sections 3 and 4, we present and develop our asymptotic framework and results for the conditional and unconditional regimes. In Section 5, we present results from a simulation study which empirically verify our large sample developments. Finally, in Section 6 we discuss the implications of our results.

\section{Generalized Linear Mixed Models} \label{sec:glmms}

We study the independent-cluster generalized linear mixed model defined as follows. Let $y_{ij}$ denote the $j$th measurement of cluster $i$, $\bm{x}_{ij}$ denote a vector of $p_f$ fixed effect covariates, and $\bm{z}_{ij}$ denote a vector of $p_r$ random effect covariates, for $j = 1,\ldots, n_i$, and $i = 1,\ldots, m$. Let $N = \sum_{i=1}^m n_i$, $n = m^{-1} N$, $n_L = \min_{1 \leq i \leq m} n_i$, and $n_U = \max_{1 \leq i \leq m} n_i$. The $m$ clusters are independent of each other. Conditional on a $p_r$-vector of random effects $\dot{\bm{b}}_i$, where the dot above any quantity is used to denote its true value (or that it is evaluated at the true parameter values), the responses $y_{ij}$ from cluster $i$ are assumed to be independent observations from the exponential family with mean $\dot{\mu}_{ij}$ and dispersion parameter $\dot{\phi}$. That is, $f(y_{ij}|\dot{\bm{\beta}}, \dot{\bm{b}}_{i},\dot{\phi}) = \exp \, [ \dot{\phi}^{-1} \{ y_{ij}  \dot{\vartheta}_{ij} - a(\dot{\vartheta}_{ij}) \} + c(y_{ij},\dot{\phi}) ]$, for known functions $a(\cdot)$,  $c(\cdot)$, and $g(\cdot)$ satisfying $g(\dot{\mu}_{ij}) = g\{a'(\dot{\vartheta}_{ij})\} = \dot{\eta}_{ij} = \bm{x}^\top_{ij} \dot{\bm{\beta}} + \bm{z}^\top_{ij} \dot{\bm{b}}_{i}$, where $\dot{\bm{\beta}}$ denotes a $p_f$-vector of true fixed effect coefficients, and $\dot{\eta}_{ij}$ the corresponding true linear predictor. For ease of development, we assume that the canonical link is used, so that $\dot{\vartheta} = \dot{\eta}$.
 Commonly used distributions within the exponential family include the normal, Poisson, binomial and gamma distributions. The true realised random effects $\dot{\bm{b}}_i$ are independently and identically distributed (i.i.d.), drawn from a multivariate normal distribution with zero mean vector and unstructured $p_r \times p_r$ random effects covariance matrix $\dot{\bm{G}}$. That is, $\dot{\bm{b}}_i \overset{i.i.d.}{\sim} N (\bm{0},\dot{\bm{G}})$. 

Write $\bm{X}_{i} = [\bm{x}_{i1},\ldots,\bm{x}_{in_i}]^\top$, and $\bm{Z}_{i} = [\bm{z}_{i1},\ldots,\bm{z}_{in_i}]^\top$, so we can concatenate the means across the measurements for each cluster to obtain $g(\dot{\bm{\mu}}_{i}) = \bm{X}_{i} \dot{\bm{\beta}} + \bm{Z}_{i} \dot{\bm{b}}_i$ for $\dot{\bm{\mu}}_{i} = (\dot{\mu}_{i1},\ldots, \dot{\mu}_{in_i})^\top$, where $g(\dot{\bm{\mu}}_{i})$ denotes applying the link function $g(\cdot)$ to $\dot{\bm{\mu}}_{i}$ component-wise. We can further concatenate 
across clusters and write $g(\dot{\bm{\mu}}) = \bm{X} \dot{\bm{\beta}} + \bm{Z} \dot{\bm{b}}$, with $\dot{\bm{\mu}} = (\dot{\bm{\mu}}_{1}^\top,\ldots,\dot{\bm{\mu}}_{m}^\top)^\top$, $\bm{X} = [\bm{X}^\top_1,\ldots , \bm{X}^\top_m]^\top$, $\bm{Z} = \text{bdiag}(\bm{Z}_1,\ldots \bm{Z}_m)$, and $\dot{\bm{b}} = (\dot{\bm{b}}_1^\top, \ldots , \dot{\bm{b}}_m^\top)^\top$. Here, $\text{bdiag}()$ is the block-diagonal matrix operator, $\bm{X}$ is of dimension $N \times p_f$, and $\bm{Z}$ is an $N \times m p_r$ sparse block-diagonal matrix, with at most $p_r$ non-zero components per row, and at most $n_U$ non-zero components per column. 

Let $\bm{y}_i = (y_{11}, \ldots, y_{1 n_i})^\top$ and $\bm{y} = (\bm{y}_1^\top,\ldots,\bm{y}_m^\top)^\top$. Then the marginal log-likelihood function for the independent-cluster GLMM is given by
\begin{align}
\ln f(\bm{y} | \bm{\beta}, \phi, \bm{G}) &= \sum_{i=1}^{m} \ln \int \left( \prod_{j=1}^{n_i} f(y_{i j }|\bm{\beta}, \bm{b}_i, \phi) \right) f(\bm{b}_i | \bm{G})  \, d \bm{b}_i.
\end{align}
The above integral is not available analytically except in the special case of a normal response with an identity link function. Let $\bm{\theta} = (\bm{\beta}^\top, \bm{b}^\top)^\top$ denote the full vector of fixed and random effects. Then for a given $\bm{G}$ and $\phi$, the PQL objective function for an independent-cluster GLMM is defined as 
\begin{align}
\label{PQL}
Q(\bm{\theta}) = \sum_{i=1}^{m} \sum_{j=1}^{n_i} \ln f(y_{ij}|\bm{\beta}, \bm{b}_i, \phi) - \frac{1}{2} \sum_{i=1}^{m} \bm{b}_i^\top \bm{G}^{-1} \bm{b}_i,
\end{align}
and the PQL estimator is defined as $\hat{\bm{\theta}} = \underset{\bm{\theta}}{\arg\max} \,  Q(\bm{\theta})$. 
As there are no integrals in \eqref{PQL}, the computational cost of PQL estimation is low relative to standard maximum likelihood estimation \citep{breslow1993approximate}. 
Note for normal linear mixed models, the integral in the likelihood already possesses an analytical solution when an identity link is used, and PQL estimation is equivalent to the mixed model equations of \citet{henderson1973sire} assuming the error variance is known.

The PQL procedure provides explicit estimators of both the fixed and random effects. The latter is practically useful since 
the random effects play an important implicit role in fitting and using the GLMM. For instance, the realised values of the random effects (or functions thereof) are often important in prediction problems such as small-area estimation \citep{jiang2003empirical,pfeffermann2013new}, 
while the empirical distribution of the random effects estimators is often examined in model diagnostics \citep{hui2021random}.
On the other hand, \eqref{PQL} alone does not incorporate estimation of the random effects covariance matrix. From a theoretical standpoint, existing papers on large sample theory for PQL and related objective functions have assumed $\dot{\bm{G}}$ is known for the purposes of asymptotic development \citep[e.g.,][]{vonesh2002conditional,NIE20071787}. 
Practically speaking, several approaches have been suggested to estimate $\dot{\bm{G}}$ when applying PQL, for example by using a working LMM \citep{breslow1993approximate}, the Laplace objective function \citep{hui2017joint}, or simply the sample covariance matrix of the estimated random effects \citep{jiang2001maximum}. Indeed, \citet{jiang2001maximum} and \citet{hui2017joint} demonstrated that the sample covariance of the estimated random effects is a consistent estimator of $\dot{\bm{G}}$ under suitable regularity conditions. 

In this article, we set $\bm{G} = \hat{\bm{G}}$ in \eqref{PQL}, where $\hat{\bm{G}}$ is a symmetric positive definite matrix that is either non-stochastic or its inverse $\hat{\bm{G}}^{-1}$ is stochastically bounded. Importantly, our large sample developments do not require $\hat{\bm{G}}$ to necessarily be a consistent estimator of the true random effects covariance matrix $\dot{\bm{G}}$. For example, while we can use the estimators of $\dot{\bm{G}}$ mentioned above, our theory also permits setting $\hat{\bm{G}}$ to some fixed matrix e.g., the identity matrix, say. 
Intuitively, 
this is because we develop our large sample results for PQL estimation in such a way so as to do not depend on the value of $\hat{\bm{G}}$ itself \citep[in a spirit similar to that of][]{jiang2001maximum,fan2012variable}; 
only the true random effects covariance matrix $\dot{\bm{G}}$ appears in our theorems. 


We also adopt the above approach for the dispersion parameter in the GLMM. That is, we set $\phi = \hat{\phi}$ in \eqref{PQL}, where $\hat{\phi}$ is a known constant or a stochastically bounded term. In the Poisson and binomial distributions, $\hat{\phi}$ is set to its known true value $\dot{\phi} = 1$. In cases where the true dispersion parameter is unknown, we can use either a constant
or one of the suggested estimators of the dispersion parameter in the literature (e.g., a scaled sum of squared conditional Pearson residuals).
For the remainder of this article, and as discussed in Section \ref{sec:intro}, we focus on the fixed and random effects in GLMMs. We do not discuss inferential properties of $\dot{\phi}$ and $\dot{\bm{G}}$.

\section{Conditional on Random Effects} \label{sec:conditionaltheory}

In many applications of independent-cluster GLMMs e.g., for longitudinal data, covariates included as random effects are also included as fixed effects \citep{cheng2010real}. 
With this in mind, we develop our results under the setting where all covariates are partnered i.e. included as both fixed and random effects such that $ \bm{x}_{ij} = \bm{z}_{ij} $ for all $(i,j)$ and $p_f=p_r=:p$. 
Next, let $\bm{A}$ be a $q \times (m+1)p$ matrix with the finite selection property. That is, for any row of $\bm{A}$, there exists an $m_0 \in \mathbb{N}$ such that the $\{(m_0+1)p+1\}$th to $\{(m+1)p\}$th components of the row are zero for all $m>m_0$. All components of $\bm{A}$ must have a component-wise limit, with at least one of these limits being non-zero. We partition $\bm{A}$ into $[\bm{A}_f,\bm{A}_r]$, where $\bm{A}_f$ is of dimension $q \times p$ and $\bm{A}_r$ is of dimension $q \times mp$. Also, for an arbitrary matrix $\bm{C}$, let $\bm{C}_{[i:j,k:l]}$ denote the sub-matrix comprising the $i$th to $j$th row and $k$th to $l$th column of $\bm{C}$ and $\bm{C}_{[i,]}$ and $\bm{C}_{[,j]}$ denote the $i$th row and $j$th column respectively. Similarly, for a vector $\bm{c}$ we let $\bm{c}_{[i:j]}$ denote the sub-vector formed by taking the $i$th to $j$th components; the quantity $\bm{c}_{[i]}$ simply denotes the $i$th component of $\bm{c}$.

Let $\bm{\mu}_i (\bm{\theta}) = \{ a'(\eta_{i1}), \ldots, a'(\eta_{in_i}) \}^\top$, $\bm{\mu} (\bm{\theta}) = \{a'(\eta_{11}), \ldots, a'(\eta_{mn_m}) \}^\top$, \\ $\dot{\bm{W}}_i = \hat{\phi}^{-1} \text{diag}\{a''(\dot{\eta}_{i1}), \ldots, a''(\dot{\eta}_{in_i}) \}$ and $\dot{\bm{W}} = \hat{\phi}^{-1} \text{diag} \{ a''(\dot{\eta}_{11}), \ldots, a''(\dot{\eta}_{mn_m})\}$. Furthermore, write $\dot{\mu}_{ij} = a''(\dot{\eta}_{ij})$, $\dot{\bm{\mu}}_i = \bm{\mu}_i (\dot{\bm{\theta}})$ and $\dot{\bm{\mu}} = \bm{\mu} (\dot{\bm{\theta}})$, and let $\otimes$ denote the Kronecker product operator, $\bm{I}_m$ denote the $m \times m$ identity matrix, and $\bm{1}_m$ denote a matrix or vector of ones, with dimension indicated by the relevant subscripts. Furthermore, let $\bm{D}_r = \text{diag} ({n}_1^{1/2} \bm{1}_{p}, \ldots, {n}_m^{1/2} \bm{1}_{p})$, $\bm{D} = \text{bdiag} ({N}^{1/2} \bm{I}_{p}, \bm{D}_r) $, $\bm{D}^* = \text{bdiag} ({m}^{1/2} \bm{I}_{p}, \bm{D}_r)$, $\bm{D}^+ = {m}^{1/2} \bm{I}_{(m+1)p}$, and define the two limiting quantities
\begin{align*}
\bm{\Omega} &=  \underset{m,n_L \rightarrow \infty}{\lim} \frac{\dot{\phi}}{\hat{\phi}} \bm{A} \, \text{bdiag} \Bigg\{ \frac{1}{m} \sum_{i=1}^m \frac{n}{n_i} \left(\frac{\bm{X}^\top_i \dot{\bm{W}}_i \bm{X}_i}{n_i} \right)^{-1} , \left(\frac{\bm{X}^\top_1 \dot{\bm{W}}_1 \bm{X}_1}{n_1} \right)^{-1}, \ldots , \left(\frac{\bm{X}^\top_m \dot{\bm{W}}_m \bm{X}_m}{n_m} \right)^{-1} \Bigg\} \bm{A}^\top, \\
\bm{\Omega}_r &= \lim_{m,n_L \rightarrow \infty} \frac{\dot{\phi}}{\hat{\phi}} \bm{A}_r \bm{D}_r \left(\bm{Z}^\top \dot{\bm{W}} \bm{Z} \right)^{-1} \bm{D}_r^{\top} \bm{A}_r^{\top}.
\end{align*}
Note $\bm{\Omega}$ and $\bm{\Omega}_r$ are not actually functions of $\hat{\phi}$, since $\hat{\phi} \dot{\phi}^{-1} \dot{\bm{W}}_i = \dot{\phi}^{-1} \text{diag}\{a''(\dot{\eta}_{i1}), \ldots, a''(\dot{\eta}_{in_i}) \}$ and similarly for $\dot{\bm{W}}$. 

We consider the setting where both the minimum cluster size $n_L$ and the number of clusters $m$ grow to infinity, such that $n_i = O(n_L)$ uniformly for $i= 1, \ldots, m$. This implies for any $i = 1 , \ldots , m$, we have $n_i = O(n)$, $n = O(n_i)$, $N = O(m n_i)$, and $m n_i = O(N)$. This restriction on the growth rates of the cluster sizes 
is commonly employed in asymptotic analysis of PQL estimation \citep[e.g.,][]{vonesh2002conditional}. Additionally, we require the following regularity conditions.

\begin{enumerate}[label=(\condition{C}{\arabic*})]
\item The function $a(\eta)$ is at least three times continuously differentiable in its domain, with $0 < c_0 \leq a''(\eta) \leq c_0^{-1} < \infty$ and $| a'''(\eta) | \leq c_0^{-1} < \infty$ for some sufficiently small constant $c_0$. 

\item For every $i=1,\ldots,m$ and $j=1,\ldots,n_i$, there exists a sufficiently large constant $C_1$ such that $\| \bm{x}_{ij} \|_\infty < C_1 $ where $\| \cdot \|_\infty$ is the maximum norm. Furthermore, denote $\dot{\bm{H}}_i = ( n_i^{-1} \hat{\phi} \dot{\phi}^{-1} \bm{X}_i^\top \dot{\bm{W}}_i \bm{X}_i )^{-1}$. Then for all $i = 1,\ldots,m$, the matrices $\lim_{n_i \rightarrow \infty} \dot{\bm{H}}_i = \dot{\bm{K}}_i $ and $\lim_{m,n_L \rightarrow \infty} m^{-1} \sum_{i=1}^m n n_i^{-1} \dot{\bm{H}}_i = \dot{\bm{K}}$ are positive definite with minimum and maximum eigenvalues bounded from above and below by $c_1^{-1}$ and $c_1$ respectively, for a sufficiently small constant $c_1$.

\item 
The vector of true parameters $\dot{\bm{\theta}} = (\dot{\bm{\beta}}^\top, \dot{\bm{b}}^\top)^\top$, where $\dot{\bm{b}} = (\dot{\bm{b}}_1^\top, \ldots, \dot{\bm{b}}_m^\top)^\top$, is an interior point in some compact set $\Theta \subset \mathbb{R}^{(m+1)p}$.

\item The working matrix $\hat{\bm{G}}$ is positive definite, and its inverse is $O_p(1)$. Also, the working quantity $\hat{\phi}$ is strictly positive and $O_p(1)$.

\item For all $i=1,\ldots,m$ and $n_i \in \mathbb{N}$, it holds that $E([n_i^{1/2} (\bm{X}_i^\top \dot{\bm{W}}_i \bm{X}_i + \hat{\bm{G}}^{-1})^{-1} \{ \hat{\phi}^{-1} \bm{X}_i^\top (\bm{y}_{i}-\dot{\bm{\mu}}_i) - \hat{\bm{G}}^{-1} \dot{\bm{b}}_i \}]^4) < \infty$, where the power and expectation are applied component-wise.

\end{enumerate}

Conditions (C1) - (C3) are needed to guarantee the existence and regular behavior of the asymptotic variance for the PQL estimating function, and to establish a Lindeberg condition needed to obtain a central limit theorem.  
Condition (C4) is required to ensure that the shrinkage of the random effects is asymptotically negligible, and formalises our discussion of $\hat{\bm{G}}$ and $\hat{\phi}$ at the end of Section \ref{sec:glmms}. Condition (C5) is needed to bound the order of $\| \hat{\bmtheta} - \dot{\bmtheta} \|_{\infty}$, and is satisfied by many distributions e.g. Poisson and binomial, when the random effects are normally distributed \citep[see also][]{van2012quasi}. 

For the remainder of this section, we consider the regime where we condition on the random effects, so that $\dot{\bm{\theta}}$ is a $(m+1)p$-vector of constants. 
The assumptions and conditions outlined above however will be applied to both the conditional and unconditional regime. 

\subsection{Main Result for the Conditional Regime}

Let $\bm{1}^*_m = (-1, \bm{1}_m^\top)^\top$. Then we have the following:
\begin{theorem} \label{thm:conditionalnormality}
Assume Conditions (C1) - (C5) are satisfied and $m n_L^{-1} \rightarrow 0$. Then as $m,n_L \rightarrow \infty$, and conditional on the true vector of random effects $\dot{\bm{b}}$, it holds that 
\begin{enumerate} 
    \item[(a)] $\| \hat{\bm{\theta}} - ( \dot{\bm{\theta}} - \bm{1}^*_m \otimes m^{-1} \sum_{i=1}^m \dot{\bm{b}}_i ) \|_{\infty} = o_p(1)$.
    \item[(b)] $\bm{A} \bm{D} \{\hat{\bm{\theta}} - (\dot{\bm{\theta}} - \bm{1}^*_m \otimes m^{-1} \sum_{i=1}^m \dot{\bm{b}}_i)\} \overset{D}{\rightarrow} N(\bm{0},\bm{\Omega})$.
\end{enumerate}
\end{theorem}
The first part of the theorem establishes consistency for the PQL estimator around a sum-to-zero reparametrized version of the true parameters (see below for more discussion on the latter aspect). The block diagonal structure of $\bm{\Omega}$ in the second part of the theorem shows that conditional on true random effects vector, the corresponding estimators are asymptotically independent between clusters, and also asymptotically independent of the fixed effects estimators.

We illustrate a few special cases of Theorem \ref{thm:conditionalnormality} using specific selection matrices. First, suppose $\bm{A} = [\bm{I}_{p}, \bm{0}_{p \times mp}]$.
If $\sum_{i=1}^m \dot{\bm{b}}_i = \bm{0}_p$, then we obtain $\bm{A} \bm{D} (\hat{\bm{\theta}} - \dot{\bm{\theta}}) = {N}^{1/2} (\hat{\bm{\beta}} - \dot{\bm{\beta}})  \overset{D}{\rightarrow} N (\bm{0}, \dot{\bm{K}} )$ conditional on the random effects, where $\dot{\bm{K}}$ is the limiting matrix defined in Condition (C2). Also, suppose $\bm{A} = [\bm{0}_{p}, \bm{I}_{p} , \bm{0}_{p \times (m-1) p}]$.
Then from Theorem \ref{thm:conditionalnormality}, we have $\bm{A} \bm{D} (\hat{\bm{\theta}} - \dot{\bm{\theta}}) = {n}_1^{1/2} (\hat{\bm{b}}_1 - \dot{\bm{b}}_1)  \overset{D}{\rightarrow} N (\bm{0}, \dot{\bm{K}}_1 )$, conditional on the random effects. The analogous result holds for choosing any particular cluster. Finally, since the random effects exhibit a slower convergence rate than the fixed effects, and noting the asymptotic independence, then for an arbitrary $p$-dimensional constant $\bm{a}$ we obtain
${n}_i^{1/2} \bm{a}^\top (\hat{\bm{\beta}} + \hat{\bm{b}}_i - \dot{\bm{\beta}} - \dot{\bm{b}}_i)  \overset{D}{\rightarrow} N \left(\bm{0}, \bm{a}^\top \dot{\bm{K}}_i \bm{a} \right); \quad i = 1,\ldots,m$,
conditional on the random effects. As an example, if the linear predictor involves a fixed and random intercept and a fixed and random slope for a single covariate, then we set $\bm{a} = (1,x_{ij})^\top$ and obtain $n_i^{1/2} (\hat{\eta}_{ij} - \dot{\eta}_{ij}) = n_i^{1/2} ( \hat{\beta}_0 + \hat{b}_{i0} + \hat{\beta}_1 x_{ij} + \hat{b}_{i1} x_{ij} - \dot{\beta}_0 - \dot{b}_{i0} - \dot{\beta}_1 x_{ij} - \dot{b}_{i1} x_{ij} ) \overset{d}{\rightarrow} N(0, \bm{a}^\top \dot{\bm{K}}_i \bm{a})$.

For statistical inference, we can appeal to Slutsky's Theorem and replace $\dot{\bm{K}}_i$ with $\hat{\bm{H}}_i$, and $\dot{\bm{K}}$ with $m^{-1} \sum_{i=1}^m n n_i^{-1} \hat{\bm{H}}_i$. Here $\hat{\bm{H}}_i$ is defined as $( n_i^{-1} \bm{X}_i^\top \hat{\bm{W}}_i \bm{X}_i )^{-1}$ where $\hat{\bm{W}}_i = \tilde{\phi}^{-1} \text{diag}\{a''(\hat{\eta}_{i1}), \ldots, a''(\hat{\eta}_{in_i}) \}$ for some consistent estimator of the dispersion parameter $\tilde{\phi}$ e.g., based on the inverse scaled sum of squared conditional Pearson residuals. Theorem \ref{thm:conditionalnormality} then provides a straightforward way to construct confidence intervals, say, for all the parameters and combinations thereof.
In fact, the forms of these intervals are similar to standard results in (penalized) GLMs\citep{mcculloch2004generalized}: 
this is not surprising given we are working conditional on the true vector of random effects. 
say.


Finally, note the PQL estimator is consistent for a sum-to-zero reparametrized version of the true parameters. This occurs because the PQL estimators of the random effects must satisfy a sum-to-zero constraint regardless of the underlying true parameter values, and under a conditional regime, this induces an asymptotic bias $\bm{1}^*_m \otimes (m^{-1} \sum_{i=1}^m \dot{\bm{b}}_i)$ in Theorem \ref{thm:conditionalnormality}, which can be interpreted as shifting the mean of the random effects into the corresponding fixed effects. We offer more discussion around this asymptotic bias in the supplementary material.

\section{Unconditional Regime} \label{sec:unconditionaltheory}

We now turn to establishing results under an unconditional regime i.e., treating $\dot{\bm{b}}_i$'s as random instead of conditioning on them. This has two main implications. First, in the unconditional setting the quantity $m^{-1} \sum_{i=1}^m \dot{\bm{b}}_i$ is no longer deterministic and should not be treated as a bias term. Instead, it is of order $O_p(m^{-1/2})$, and so competes with other leading terms in the relevant Taylor expansion to be the dominating term. This results in a reduction of the rate of convergence for the fixed effects estimator, from $N^{1/2}$ in the conditional regime to $m^{1/2}$ in the unconditional regime.
Second, in contrast to the conditional regime, the observations within the same cluster are no longer independent. This has ramifications when applying the central limit theorem to establish asymptotic multivariate normality. In Section \ref{subsec:poissonpredictiongap}, we provide a simple but insightful example based on a Poisson random intercept model, which demonstrates that the prediction gap is not always asymptotically normally distributed. 

The two approximations below, derived from the Taylor expansion of the PQL objective function, will be central to understanding the large sample developments we make on a more intuitive level. For a given $\hat{\phi}$, we have
\begin{subequations}
\begin{align} 
\hat{\bm{\beta}} - \dot{\bm{\beta}} &= m^{-1} \sum_{i=1}^m \dot{\bm{b}}_i + o_p(1) \label{eq:Mfixed_basic.uncond} \\
\hat{\bm{b}} - \dot{\bm{b}} &= - \bm{1}_m \otimes m^{-1} \sum_{i=1}^m \dot{\bm{b}}_i + (\bm{Z}^\top \dot{\bm{W}} \bm{Z})^{-1} \{ {\hat{\phi}}^{-1} \bm{Z}^\top (\bm{y} - \dot{\bm{\mu}}) \} + o_p(1). \label{eq:Mrand_basic.uncond}
\end{align}
\end{subequations}
We will refer to both equations in the discussion of the theorems to be presented later on.

\subsection{Prediction Gap - Counterexample} \label{subsec:poissonpredictiongap}

We offer a motivating and insightful example to illustrate that the prediction gap is not, in general, asymptotically normally distributed. This example also offers a simple case where $\bm{X}_i \neq \bm{Z}_i$, and offers an interesting comparison to
the theory established under the assumption of $\bm{X}_i = \bm{Z}_i$.

Consider a Poisson random intercept model with canonical log link. That is, the true model is given by $f(y_{ij}|\dot{b}_i) = \exp(y_{ij} \dot{\eta}_{ij} - \dot{\mu}_{ij})/(y_{ij}!)$ with $\ln(\dot{\mu}_{ij}) = \dot{\eta}_{ij} = \dot{b}_i$, and $\dot{b}_i \overset{i.i.d.}{\sim} N(0, \dot{\sigma}_b^2)$. Assume a working $\hat{\sigma}_b^2$, and apply PQL estimation to estimate the random effects $b_i$ for $i = 1,\ldots,n$. For simplicity, we also assume a balanced design, such that $n_i = n$ for all $i=1,\ldots,m$. Then it is possible to show (see 
the supplementary material for the formal derivation) that when $m n^{-2} \rightarrow 0$, the prediction gap of the first cluster $\hat{b}_1 - \dot{b}_1$ satisfies
\begin{align} \label{eq:counterexample}
     {n}^{1/2} (\hat{b}_1 - \dot{b}_1) &= n^{-1/2} \sum_{j=1}^n \{ y_{1j} \exp (-\dot{b}_1) -1\} + o_p(1).
\end{align}
Therefore, we obtain $\hat{b}_1 = \dot{b}_1 + o_p(1)$, and similarly for each cluster $i = 1,\ldots,m$. When conditioned on $\dot{b}_1$, $n^{-1/2} \sum_{j=1}^n \{ y_{1j} \exp (-\dot{b}_1) -1\}$ is a normalised sum of independent random variables. Unconditionally however, the sum consists of an exchangeable collection of uncorrelated but dependent random variables with mean zero and finite non-zero variance. Using the central limit theorem for exchangeable random variables \citep[][]{blum1958central}, it can be subsequently be shown that the quantity $n^{-1/2} \sum_{j=1}^n \{ y_{1j} \exp (-\dot{b}_1) -1\}$, and hence ${n}^{1/2} (\hat{b}_1 - \dot{b}_1)$, is not asymptotically normally distributed. 

With the above example in mind, we now state the main results for the unconditional regime. 

\subsection{Fixed Effects}
We have the following result for the PQL estimator of the fixed effects under an unconditional regime. 

\begin{theorem} \label{thm:unconditional_fixef}
Assume Conditions (C1) - (C5) are satisfied, and $m n_L^{-2} \rightarrow 0 $. Then as $m,n_L \rightarrow \infty$ and unconditional on the random effects $\dot{\bm{b}}$, it holds that ${m}^{1/2} (\hat{\bm{\beta}} - \dot{\bm{\beta}}) \overset{D}{\rightarrow}  N(\bm{0},\dot{\bm{G}})$. 
\end{theorem}

This result should not be too surprising given the form of \eqref{eq:Mfixed_basic.uncond}. Furthermore, the rate of convergence and asymptotic distribution coincides with the result obtained by \citet{jiang2021usable} for the partnered fixed effects for the quasi-maximum likelihood estimator. More importantly, Theorem \ref{thm:unconditional_fixef} allows practitioners to straightforwardly perform statistical inference for the fixed effects, so long as $m n_L^{-2} \rightarrow 0 $. Although $\dot{\bm{G}}$ is not known, we can appeal to Slutsky's theorem and replace it with a consistent estimator (e.g., the sample covariance matrix of the estimated random effects).
Theorem \ref{thm:unconditional_fixef} contrasts with Theorem \ref{thm:conditionalnormality} derived under the conditional regime, where $m n_L^{-1} \rightarrow 0$ is required but the convergence rate is ${N}^{1/2}$. This reduction in the rate of convergence arises because the leading term in the Taylor expansion is different: in the unconditional regime, it is simply the normalised sum of random effects over all the clusters, and thus its variability is dominated by the term $m^{-1/2} \sum_{i=1}^m \dot{\bm{b}}_i$. However, this term is deterministic in the conditional regime, and serves to enforce a sum-to-zero constraint instead as discussed in Section \ref{sec:conditionaltheory}. Generally speaking, the Taylor expansion can be interpreted as comprising terms which either capture the stochasticity in the random effects vector $\dot{\bm{b}}$, or the stochasticity in responses $y_{ij}$ given the random effects. These terms compete with each other, and which one dominates depends on the relative rates of $m$ and $n_i$. This intricacy in the nature of the results will be made apparent in our results for the prediction gap in Section \ref{subsec:predictiongap}.


\subsection{Estimators of the Random Effects} \label{subec:randomeffects_unconditional}

Next, we state a convergence result for the PQL estimators of the random effects under the unconditional regime.

\begin{theorem} \label{thm:unconditional_predictors}
Assume Conditions (C1) - (C5) are satisfied and $m n_L^{-2} \rightarrow 0 $. Then as $m,n_L \rightarrow \infty$ and unconditional on the random effects $\dot{\bm{b}}$, it holds that $\bm{A}_r (\hat{\bm{b}} - \dot{\bm{b}} )\overset{P}{\rightarrow} \bm{0}_q $ . 
\end{theorem}

Practically, Theorem \ref{thm:unconditional_predictors} confirms the asymptotic distribution of a finite subset of the PQL estimators is the distribution of the random effects themselves. This can play a useful role for helping to validate the examination of the empirical distribution of PQL estimators $\hat{\bm{b}}$ as a model diagnostic tool.
For instance, if the random effects are normally distributed and $\bm{A}_r$ only selects the first cluster, then we would expect $\hat{\bm{b}}_1$ to have an approximate $N(\bm{0},\dot{\bm{G}})$ distribution.
On the other hand, Theorem \ref{thm:unconditional_predictors} does not help us when it comes to performing likelihood-based inference for the true random effects $\dot{\bm{b}}$, as this does not appear in the approximation $\hat{\bm{b}}_1 \sim N(\bm{0},\dot{\bm{G}})$ itself. 

As an aside, note the above means we can apply the continuous mapping theorem and show that $q^{-1} \sum_{i=1}^q \hat{\bm{b}}_i \hat{\bm{b}}_i^\top - q^{-1} \sum_{i=1}^q \dot{\bm{b}}_i \dot{\bm{b}}_i^\top \overset{P}{\rightarrow}  0$ for any $q \in \mathbb{N}$. Since $q^{-1} \sum_{i=1}^q \dot{\bm{b}}_i \dot{\bm{b}}_i^\top \overset{P}{\rightarrow} \dot{\bm{G}}$ as $q \rightarrow \infty$, this further reiterates the use of a sample covariance matrix of the estimated random effects as an estimator of $\dot{\bm{G}}$ \citep[consistent with][]{jiang2001maximum,hui2017joint}. 


\subsection{Prediction Gap} \label{subsec:predictiongap}

In this section, we present a result for the large sample distribution of a finite subset of the prediction gap, $\hat{\bm{b}} - \dot{\bm{b}}$, in the unconditional regime. As mentioned above, the asymptotic distribution as well as the convergence rate of the prediction gap depends on the relative rates of growth of $m$ and $n_i$. This contrasts with the conditional regime, where there is no dependence on the relative rate and the PQL estimator of the random effects is always normally distributed with the convergence rate $n_i^{1/2}$.

We first introduce some terminology. Suppose we have two arbitrary continuous cumulative distribution functions (cdfs) $F_1$ and $F_2$ with supports in $\mathbb{R}^p$. 
Then we define the convolution of $F_1$ and $F_2$, denoted $F_1 * F_2$, as $(F_1 * F_2) (\bm{z}) = \int_{\mathbb{R}^p} F_1(\bm{z} - \bm{\tau}) dF_2(\bm{\tau})$. Next, for a random variable $\bm{P}$, we say $\bm{P} \sim \text{mixN}\{\bm{\mu}(\bm{b}), \bm{\Sigma}(\bm{b}), F_{\bm{b}}\}$ if $\bm{P}|\bm{b} \sim N\{\bm{\mu}(\bm{b}), \bm{\Sigma}(\bm{b})\}$ and $F_{\bm{b}}$ is the cdf of $\bm{b}$, where the conditional mean vector $\bm{\mu}(\bm{b})$ and covariance matrix $\bm{\Sigma}(\bm{b})$ may depend on $\bm{b}$. In other words, $\bm{P}$ has cdf $F_{\bm{P}} (\bm{p}) = \int \Psi_{\bm{P} | \bm{b}} (\bm{p}) dF_{\bm{b}}(\bm{b})$, where $\Psi_{\bm{P} | \bm{b}}$ is the cdf of $N\{\bm{\mu}(\bm{b}), \bm{\Sigma}(\bm{b}) \}$. A special case of this normal scale-mixture distribution is when $\bm{\mu}(\bm{b})$ and $\bm{\Sigma}(\bm{b})$ do not depend on $\bm{b}$, so that $F_{\bm{P}} (\bm{p}) = \int \Psi_{\bm{P} | \bm{b}} (\bm{p}) dF_{\bm{b}}(\bm{b})  = \Psi_{\bm{P} | \bm{b}} (\bm{p}) \int dF_{\bm{b}}(\bm{b}) = \Psi_{\bm{P} | \bm{b}} (\bm{p})$; in other words, the normal scale-mixture distribution reduces to a normal distribution. Note estimators with asymptotic normal mixture distributions have arisen in previous literature, for instance, on results relating to local asymptotic normality and non-ergodic models \citep{lecam,basawa2012asymptotic}. 

Using the above definition, we obtain the following results. 

\begin{theorem} \label{thm:unconditonal_ranef}
Assume Conditions (C1)-(C5) are satisfied and $m n_L^{-2} \rightarrow 0$. Then as $m,n_L \rightarrow \infty$ and unconditional on the random effects $\dot{\bm{b}}$, for each $i = 1,\ldots,m$ we have the following:
\begin{enumerate}
    \item[(a)] If $m n_i^{-1} \rightarrow \infty$, then $ {n}_i^{1/2} ( \hat{\bm{b}}_i - \dot{\bm{b}}_i) \overset{D}{\rightarrow} \text{mixN}(\bm{0}, \dot{\bm{K}}_i , F_{\dot{\bm{b}}_i}) $.
    \item[(b)] If $m n_i^{-1} \rightarrow \gamma_i \in (0,\infty)$, then $ {n}_i^{1/2} ( \hat{\bm{b}}_i - \dot{\bm{b}}_i) \overset{D}{\rightarrow} \text{mixN}(\bm{0}, \dot{\bm{K}}_i , F_{\dot{\bm{b}}_i}) * N(\bm{0}, \gamma_i^{-1} \dot{\bm{G}}) $. 
    \item[(c)] If $m n_i^{-1} \rightarrow 0$, then $ {m}^{1/2} ( \hat{\bm{b}}_i - \dot{\bm{b}}_i) \overset{D}{\rightarrow} N(\bm{0}, \dot{\bm{G}})$.
\end{enumerate}
\end{theorem}

\begin{corollary}
\label{corollary:predgap_joint}
Assume Conditions (C1)-(C5) are satisfied, and $m n_L^{-2} \rightarrow 0$. If $m n_L^{-1} \rightarrow \infty$, then as $m,n_L \rightarrow \infty$ and unconditional on the random effects $\dot{\bm{b}}$,
$\bm{A}_r \bm{D}_r ( \hat{\bm{b}} - \dot{\bm{b}}) \overset{D}{\rightarrow} \text{mixN}(\bm{0}, \bm{\Omega}_r , F_{\dot{\bm{b}}}) $.
\end{corollary}

Theorems \ref{thm:unconditional_predictors} and \ref{thm:unconditonal_ranef} bears some similarity to the results of \citet{lyu2021asymptotics}, who show for LMMs that the distribution of the EBLUP can asymptotically be written as the convolution between the distribution of the random effects and the distribution of a smaller order stochastic term. However, the above is the first to establish such results for GLMMs. Theorem \ref{thm:unconditonal_ranef} states that the correct asymptotic distribution to use when performing inference using the PQL estimate of the random effects depends on the relative growth rates of $m$ and $n_i$. As hinted at previously, this is a consequence of there being two competing terms in the corresponding Taylor expansion \eqref{eq:Mrand_basic.uncond}: one term arising from the random effects, and the other term arising from the distribution of the responses given the random effects. 

When $m n_i^{-1} \rightarrow \infty$ i.e., the number of clusters grows faster than the cluster size, the appropriate asymptotic distribution is given by the scale-mixture distribution $ \text{mixN}\{\bm{0},  (\hat{\phi} \dot{\phi}^{-1} \bm{X}_i^\top \dot{\bm{W}}_i \bm{X}_i )^{-1} , F_{\dot{\bm{b}}_i}  \}$, noting again that $\hat{\phi} \dot{\phi}^{-1} \dot{\bm{W}}_i = \dot{\phi}^{-1} \text{diag}\{a''(\dot{\eta}_{i1}), \ldots, a''(\dot{\eta}_{in_i}) \}$. Corollary \ref{corollary:predgap_joint} offers a slightly more general result than that given in Theorem \ref{thm:unconditonal_ranef} for the $m n_L^{-1} \rightarrow \infty$ case. 
Note 
in the linear case, the GLM iterative weights $\dot{\bm{W}}$ do not depend on the random effects $\dot{\bm{b}}$, and so the corresponding normal scale-mixture distribution reduces to a normal distribution, consistent with the asymptotic normality result derived for the EBLUP in \citet{lyu2021asymptotics}.
Practically, numerical techniques or simulation are required to compute the quantiles of the normal scale-mixture distribution for constructing prediction intervals. We use this approach in our simulations in Section \ref{sec:sims}.  

When $m n_i^{-1} \rightarrow 0$ i.e., the cluster sizes grows faster than the number of clusters, 
Theorem \ref{thm:unconditonal_ranef} shows that the appropriate approximation to consider is the normal distribution $N(\bm{0}, n^{-1} \dot{\bm{G}})$. Note this is identical to the fixed effects result of Theorem \ref{thm:unconditional_fixef}, and yields relatively straightforward prediction intervals for $\dot{\bm{b}}_i$ as long as we have a consistent estimator for $\dot{\bm{G}}$. Intuitively, the asymptotic distribution here is identical to that derived in Theorem \ref{thm:unconditional_fixef} because the dominating terms in the Taylor expansions in both cases are effectively the same. 
Finally, when $m n_i^{-1} \rightarrow \gamma \in (0,\infty)$, 
Theorem \ref{thm:unconditonal_ranef}b states that the asymptotic distribution of the PQL estimates is given by the convolution of the two cases above, noting that these two leading terms in the Taylor expansion are asymptotically independent. 
Again, numerical techniques/simulations are needed to compute prediction intervals.


In summary, Theorem \ref{thm:unconditonal_ranef} offers an asymptotically valid way of computing prediction intervals for the realised random effects in the unconditional regime, when the random effects have a corresponding partnered fixed effect in the model. It implies that estimating the variance of the prediction gap, and then naively assuming normality in order to construct prediction intervals for the random effects, will fail to yield asymptotically correct inference under the unconditional regime for PQL estimation. 


\subsection{Linear Predictor} \label{subsec:linearpredictor_unconditional}

Neither Theorems \ref{thm:unconditional_fixef} nor \ref{thm:unconditonal_ranef} above derive the joint distribution of the fixed effects estimator and prediction gap, of which the linear predictor is a function. Below, to address this, we establish a separate result specifically for the sum of a random effect and its partnered fixed effect, given an arbitrary $p$-dimensional constant vector $\bm{a}$.


\begin{theorem} \label{thm:unconditional_linpred}
Assume Conditions (C1)-(C5) are satisfied, $m n_L^{-2} \rightarrow 0$, and $m n_U^{-1/2} \rightarrow \infty$. Then as $m,n_L \rightarrow \infty$ and unconditional on the random effects $\dot{\bm{b}}$, it holds for each $i = 1,\ldots,m$ that $ {n}_i^{1/2} \bm{a}^\top (\hat{\bm{\beta}} + \hat{\bm{b}}_i - \dot{\bm{\beta}} - \dot{\bm{b}}_i) \overset{D}{\rightarrow} \text{mixN}(\bm{0}, \bm{a}^\top \dot{\bm{K}}_i \bm{a} , F_{\dot{\bm{b}}_i}) $. 
\end{theorem}
As an example, consider again a linear predictor involving a fixed and random intercept and a fixed and random slope for a single covariate. Then we set $\bm{a} = (1,x_{ij})^\top$ and obtain $n_i^{1/2} (\hat{\eta}_{ij} - \dot{\eta}_{ij}) = n_i^{1/2} ( \hat{\beta}_0 + \hat{b}_{i0} + \hat{\beta}_1 x_{ij} + \hat{b}_{i1} x_{ij} - \dot{\beta}_0 - \dot{b}_{i0} - \dot{\beta}_1 x_{ij} - \dot{b}_{i1} x_{ij} ) \overset{d}{\rightarrow} mixN(0, \bm{a}^\top \dot{\bm{K}}_i \bm{a}, F_{\dot{\bm{b}}_i})$.
For performing inference on the linear predictor in a GLMM or functions thereof, Theorem \ref{thm:unconditional_linpred} states that we again need to employ the normal scale-mixture distribution. 
This result differs from the asymptotic normality of the linear predictor derived under the conditional regime in Section \ref{sec:conditionaltheory}.
Note we can also develop a similar result for the difference between the prediction gaps of two clusters, and we refer to the supplementary material for details of this result.

To conclude the section, we remark that Theorems \ref{thm:unconditional_fixef}-\ref{thm:unconditional_linpred} 
do not offer results on the joint distribution of the prediction gap and fixed effects. However, we know from the associated proof that the prediction gaps for each cluster are asymptotically independent from each other as well as from the fixed effects estimator when $m n_U^{-1} \rightarrow \infty$ and $m n_L^{-2} \rightarrow 0$, and so a joint distribution result can be derived from this. 

\section{Simulation Study} \label{sec:sims}

We performed a numerical study to assess the usefulness of our asymptotic results in finite samples. We simulated data from an independent-cluster GLMM with five fixed and random effect covariates, considering Poisson and Bernoulli responses, as follows. First, we set the first component of $\bm{x}_{ij} = \bm{z}_{ij}$ equal to one to represent a fixed/random intercept. The second and third components are simulated from a bivariate normal distribution with mean zero and standard deviation one, with correlation equal to 0.5. The fourth component is generated independently from a standard normal distribution, and the last component is simulated from a Bernoulli distribution with a probability of success equal to 0.5. Next, we set the 5-vector of true fixed effect coefficients to either $\dot{\bm{\beta}} = (2,0.1,-0.1,0.1,0.1)^\top$ for Poisson responses, or $\dot{\bm{\beta}} = (-0.1,0.1,-0.1,0.1,0.1)^\top$ for Bernoulli responses and the $5 \times 5$ random effects covariance matrix in both cases to $\dot{\bm{G}} = \bm{I}_5$. Based on these true parameter values, we simulated the random effect coefficients $\dot{\bm{b}}_i \overset{i.i.d.}{\sim} \mathcal{N}(0, \dot{\bm{G}})$. Finally, conditional on $\dot{\bm{b}}_i$ the responses $y_{ij}$ were generated from either a Poisson distribution with log link,
or a Bernoulli distribution with logit link. 
We varied the number of clusters as $m = \{25,50,100,200,400 \}$ and the cluster sizes $n_i = n = \{ 25,50,100,200,400 \}$, noting we assumed equal cluster sizes in the simulation design for simplicity For each combination of $(m,n)$, we simulated 1000 datasets. For the conditional regime, we simulated $\dot{\bm{b}}$ only once and conditioned on this for all simulated datasets; for unconditional regime, we simulated a new $\dot{\bm{b}}$ for each simulated dataset.

For each simulated dataset, we fitted the corresponding GLMM using PQL estimation, where we use the sample covariance matrix of the estimated random effects as our update for $\hat{\bm{G}}$. That is, we iteratively maximize equation \eqref{PQL} with respect to $\bm{\beta}$ and $\bm{b}$ for a given $\hat{\bm{G}}$ (noting $\hat{\phi} = 1$ is known for both these distributions), and update $\hat{\bm{G}}$ as $m^{-1} \sum_{i=1}^m \hat{\bm{b}}_i \hat{\bm{b}}^\top_i$, until convergence. 

We assessed performance separately under the conditional and unconditional regimes. In the former, we examined the empirical coverage probability of 95\% coverage intervals constructed for $\bm{\beta}$ and for $\bm{b}_1$ (the choice of the first cluster is arbitrary). The intervals were constructed based on Theorem \ref{thm:conditionalnormality}, with the asymptotic covariance matrix $\bm{\Omega}$ computed using the true parameter values. We refer to such intervals as coverage intervals as opposed to confidence intervals. 
We also performed Shapiro-Wilk tests on the components of the (1000) realised PQL estimates of $\bm{\beta}$ and $\bm{b}_1$, in order to assess the asymptotic normality of their respective sampling distributions. For the unconditional regime, we examined the empirical coverage probability of 95\% coverage intervals constructed from Theorems \ref{thm:unconditional_fixef} and \ref{thm:unconditonal_ranef} respectively. Again, this was done for the fixed effect coefficients $\bm{\beta}$ and the random effects for the the first cluster $\bm{b}_1$. To construct all intervals, we used the true parameter values to compute the relevant asymptotic variance (this was done solely to reduce the computational burden of the numerical study), 
and, when required, obtained quantiles of relevant normal scale-mixture distributions by directly simulating 10,000 samples from them. We also performed Shapiro-Wilk tests on the components of the (1000) realised values of $\hat{\bmbeta} - \dot{\bmbeta}$ and $\hat{\bm{b}}_1 - \dot{\bm{b}}_1$. Finally, we examined histograms for the third components of $\hat{\bmbeta} - \dot{\bmbeta}$ and $\hat{\bm{b}}_1 - \dot{\bm{b}}_1$, which are representative of the histograms of the other components, as an additional method of assessing asymptotic normality of the corresponding sampling distributions.


\subsection{Simulation Results} \label{subsec:simresults}
For reasons of brevity, below we focus on results for the unconditional regime. Results for the conditional regime are presented in the supplementary material and largely support the use of Theorem \ref{thm:conditionalnormality} for inference.

For the unconditional regime, Figures \ref{fig: uncond_coverage_pois} and \ref{fig: uncond_shapiro_pois} display the empirical coverage probabilities and results from applying the Shapiro-Wilk test, respectively. For the fixed effect coefficients, the coverage probabilities for the intervals obtained based on Theorem \ref{thm:unconditional_fixef} were relatively accurate across most combinations of $(m ,n)$, with the exception of when $(m,n)=(25,25)$. For the random effect coefficients, the coverage probabilities for intervals calculated based on Theorem \ref{thm:unconditonal_ranef} approached the nominal coverage rapidly as $(m,n)$ increased for the Poisson response case, while for the Bernoulli case convergence was slightly slower due to the reduced amount of information per response. 

\begin{figure}
\centering
\includegraphics[width=0.7\linewidth]{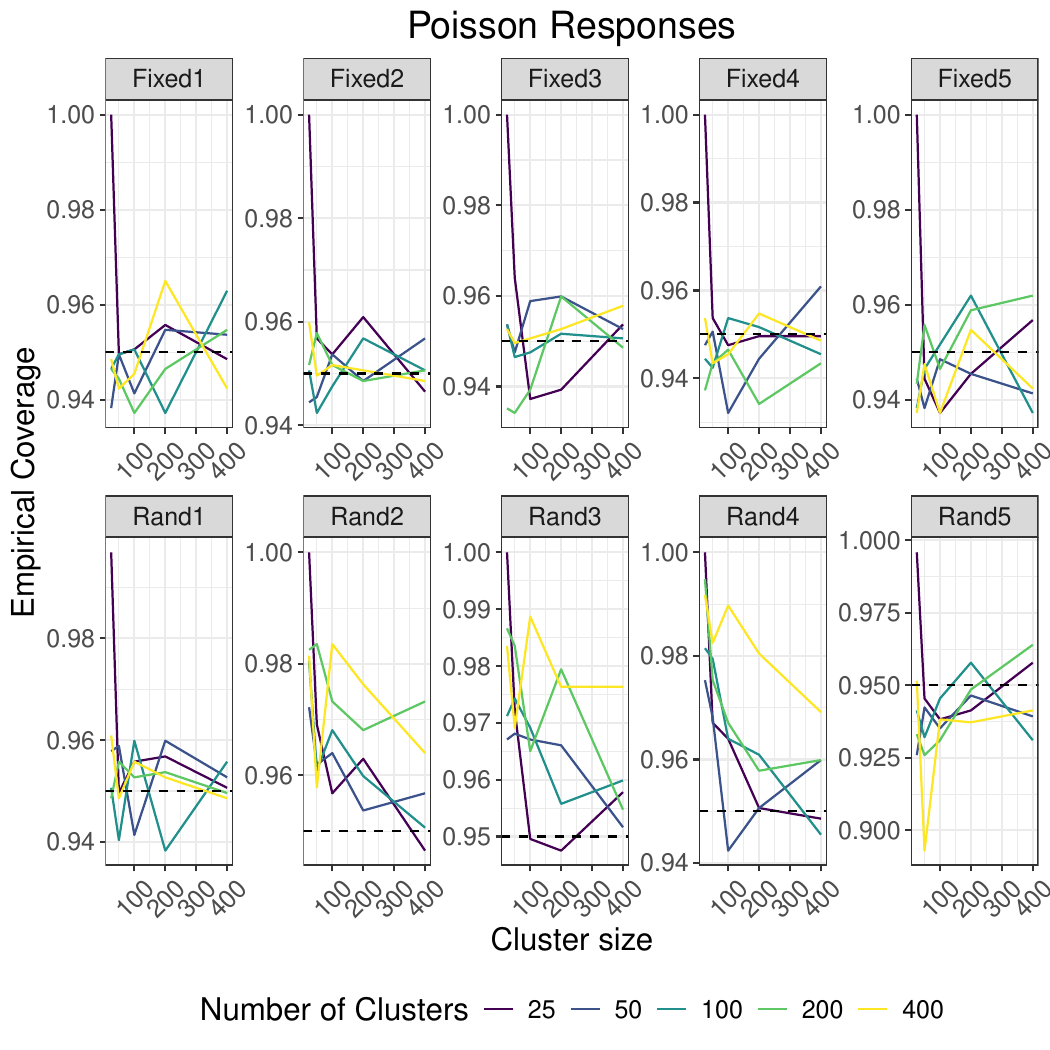}
\includegraphics[width=0.7\linewidth]{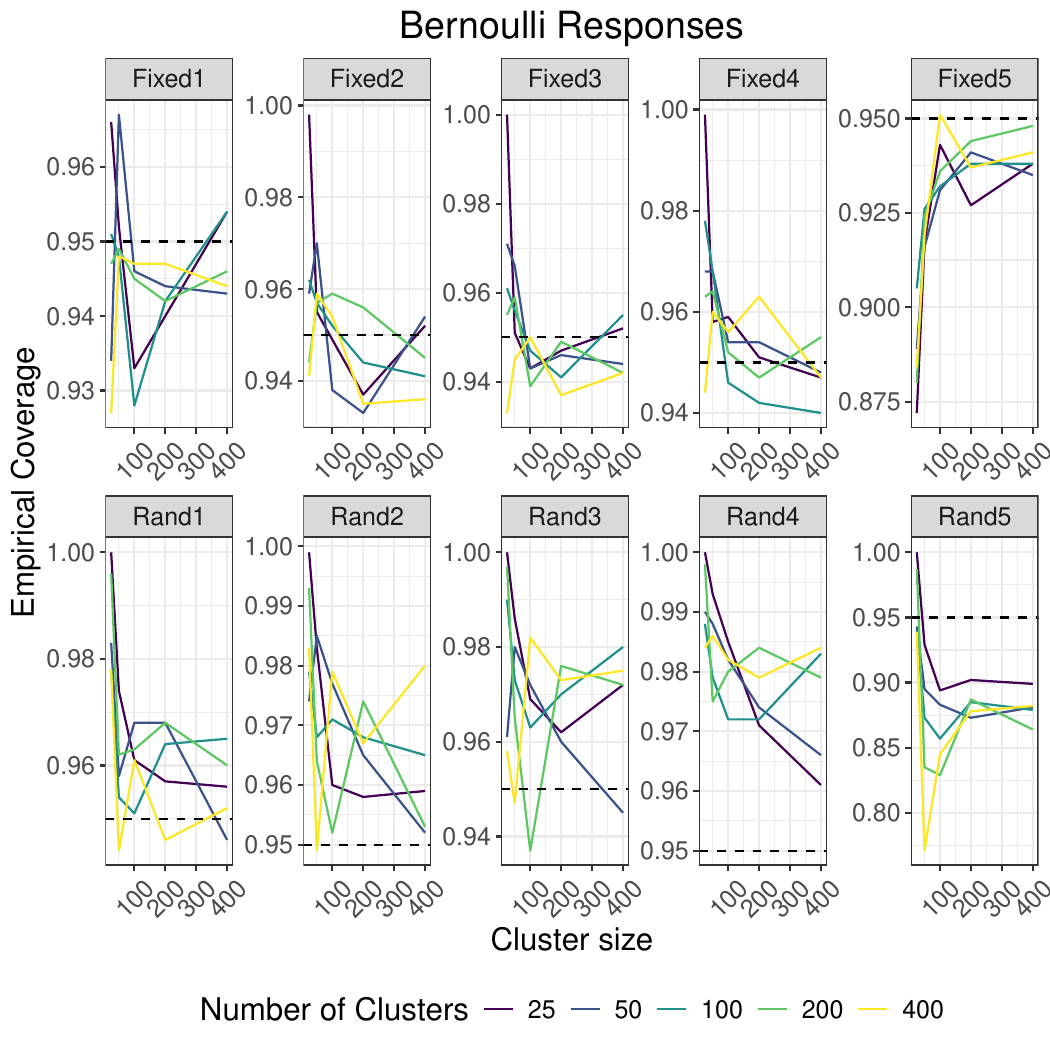}
\caption{Empirical coverage probability of 95\% coverage intervals for the fixed and random effects, obtained under the unconditional regime.} 
\label{fig: uncond_coverage_pois}
\end{figure}

\begin{figure}
\includegraphics[width=0.49\linewidth]{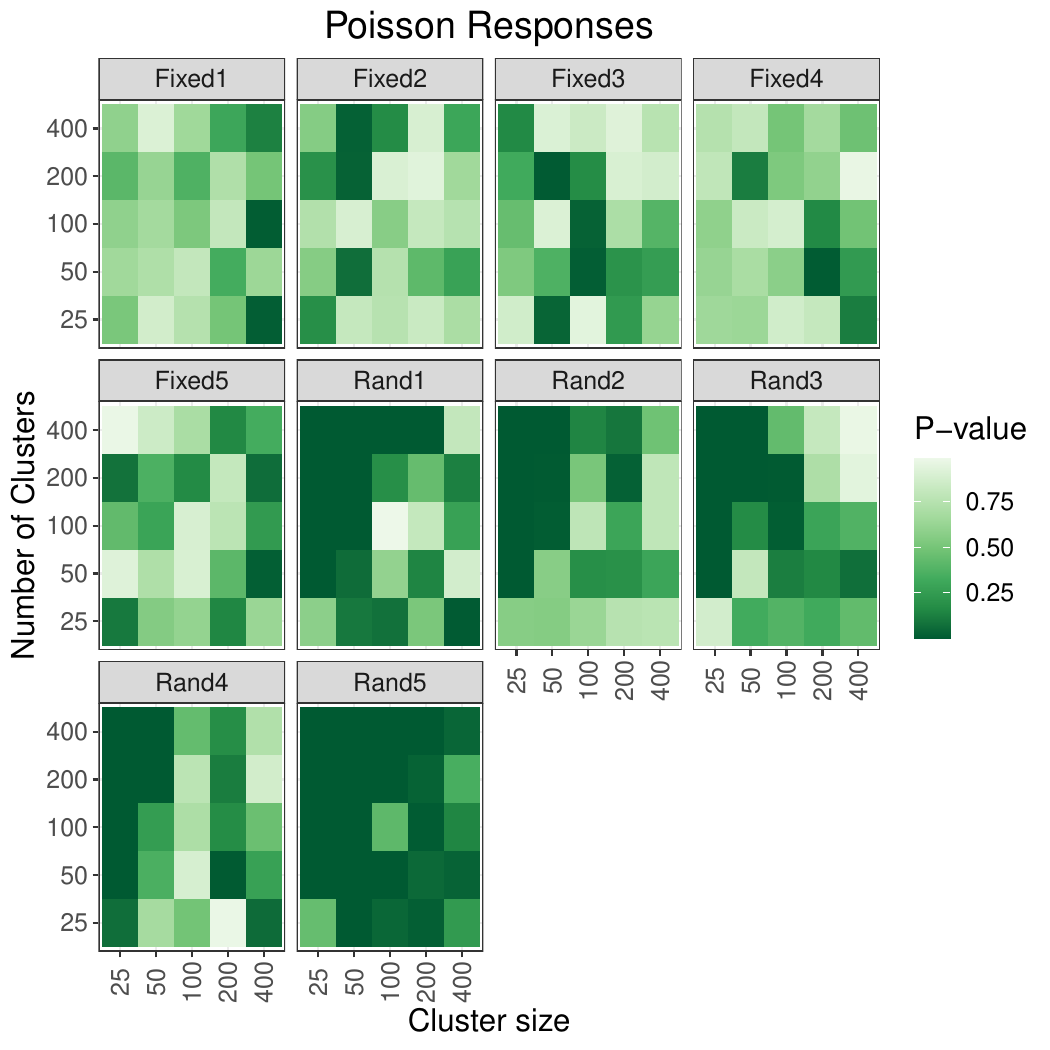}
\includegraphics[width=0.49\linewidth]{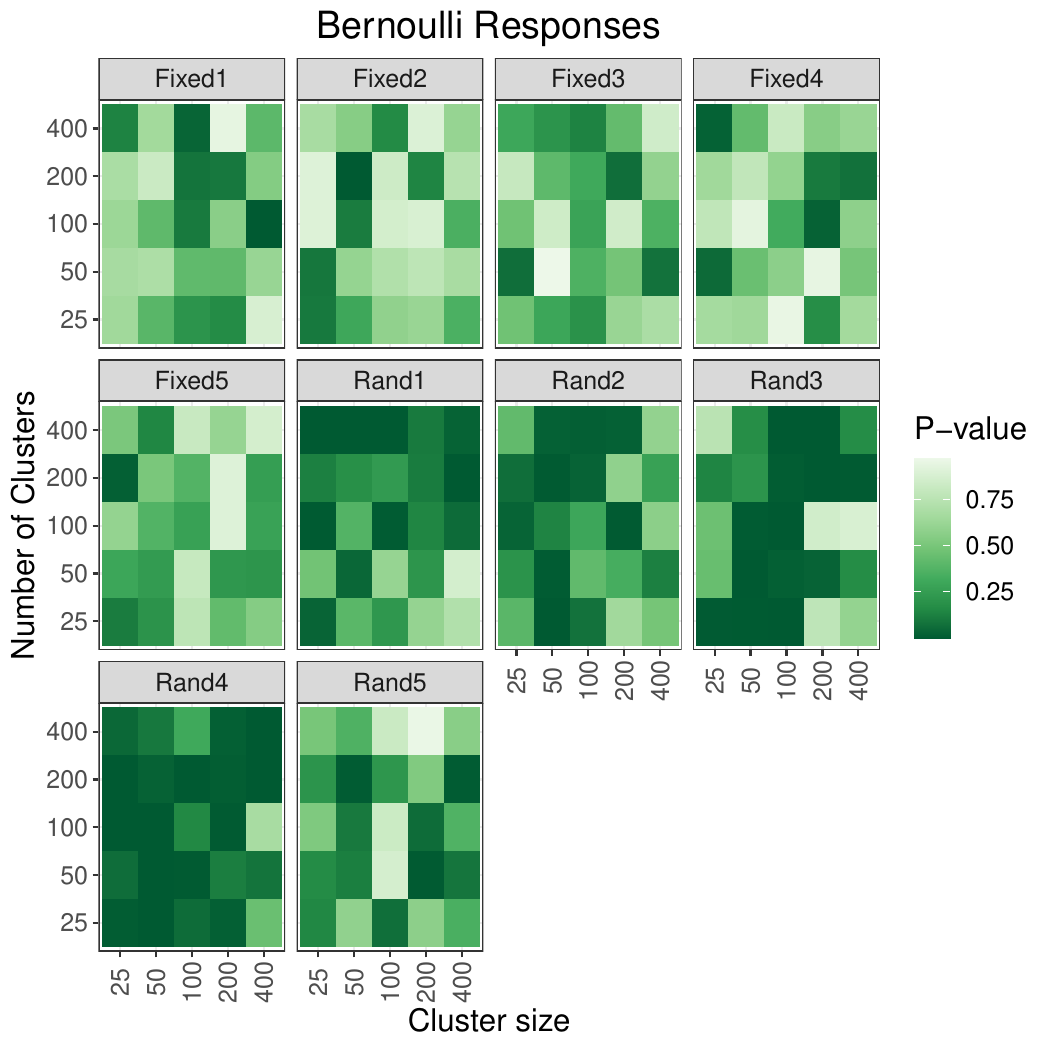}
\caption{$p$-values from Shapiro-Wilk tests applied to the fixed and random effects estimates obtained using maximum PQL estimation, under the unconditional regime.} 
\label{fig: uncond_shapiro_pois}
\end{figure}

\begin{figure}
\includegraphics[width=0.49\linewidth]{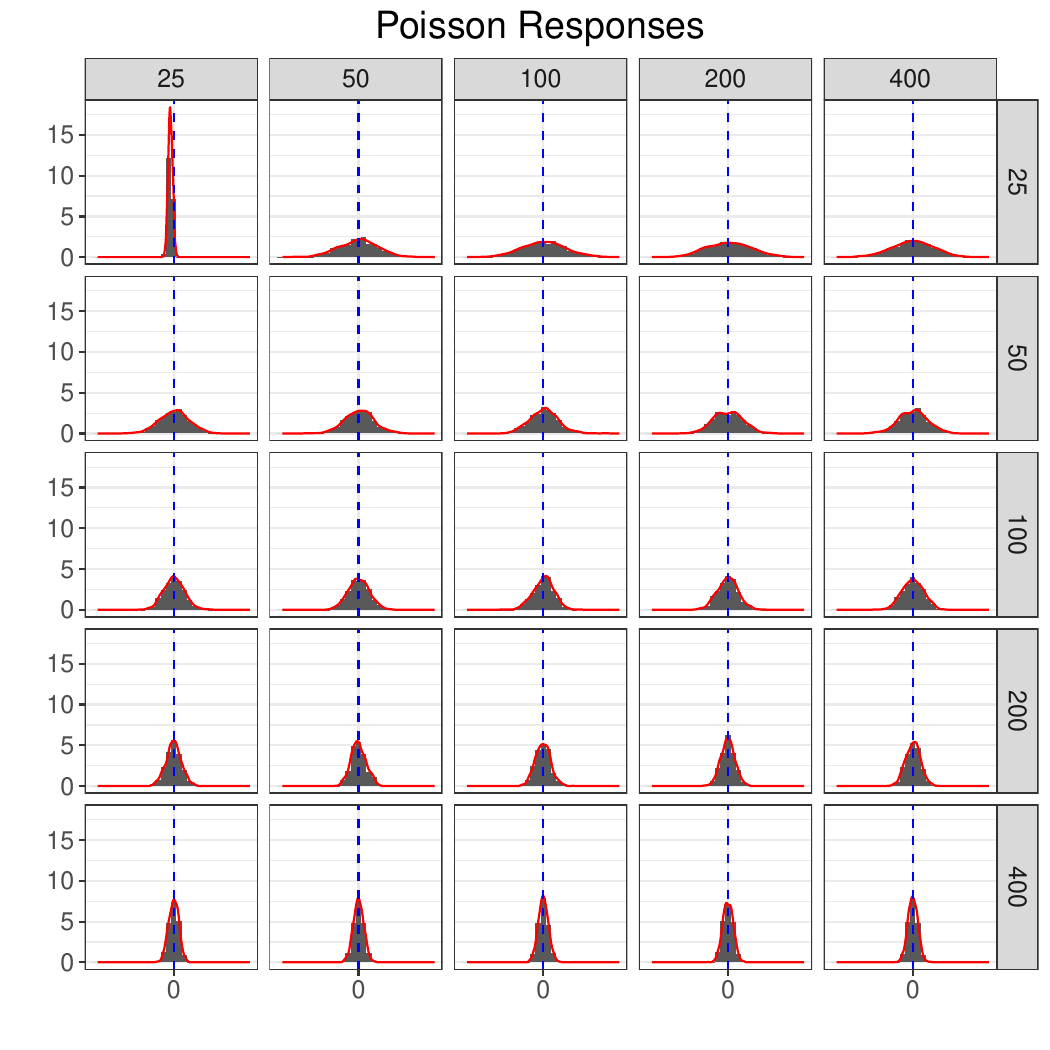}
\includegraphics[width=0.49\linewidth]{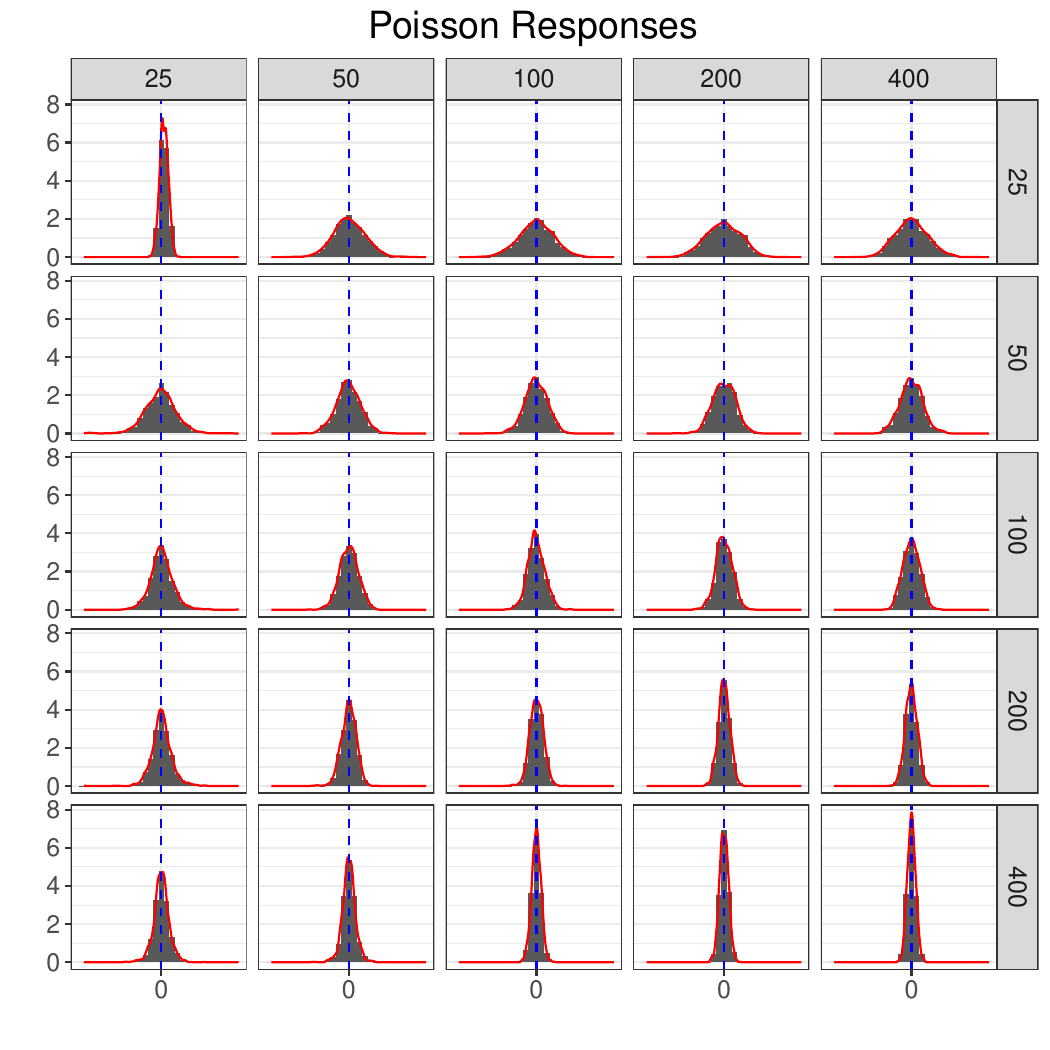}
\includegraphics[width=0.49\linewidth]{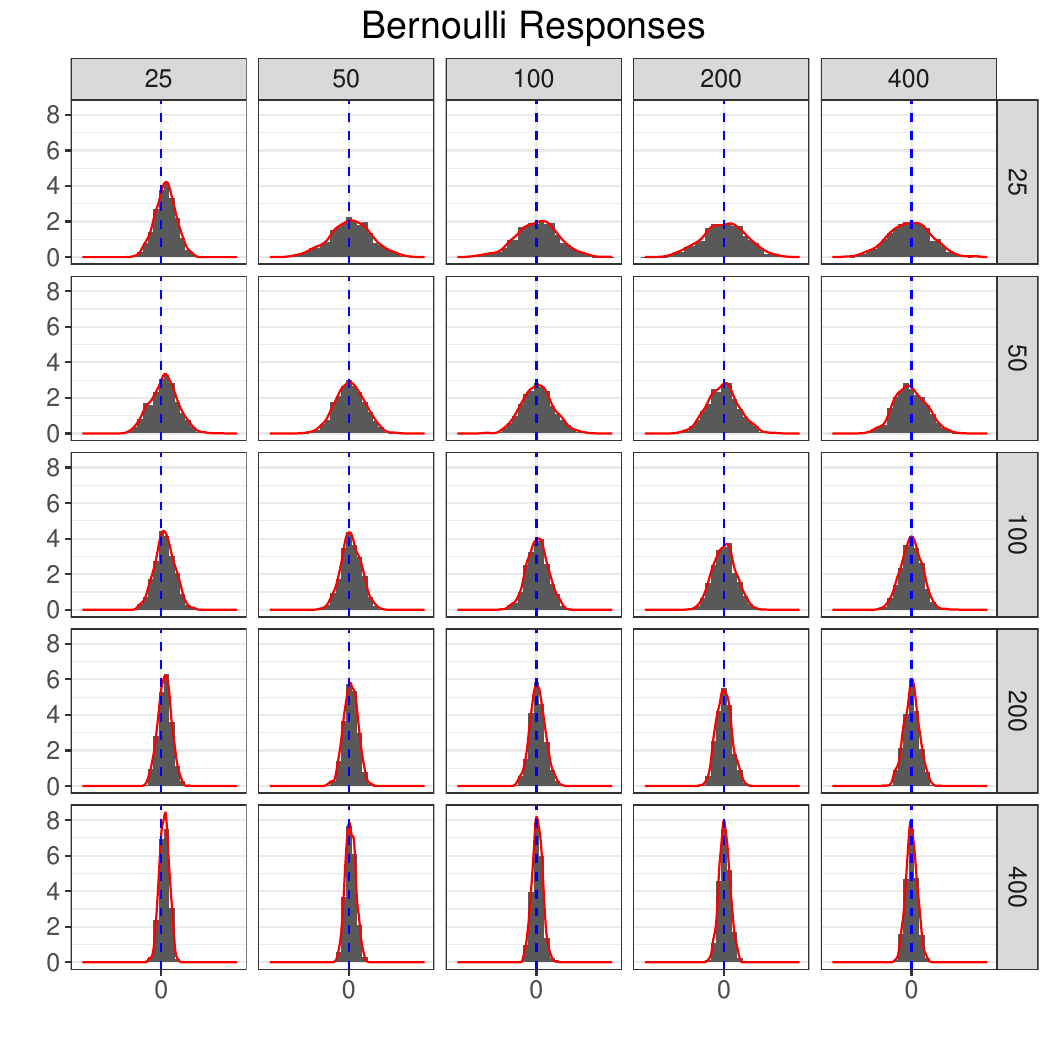}
\includegraphics[width=0.49\linewidth]{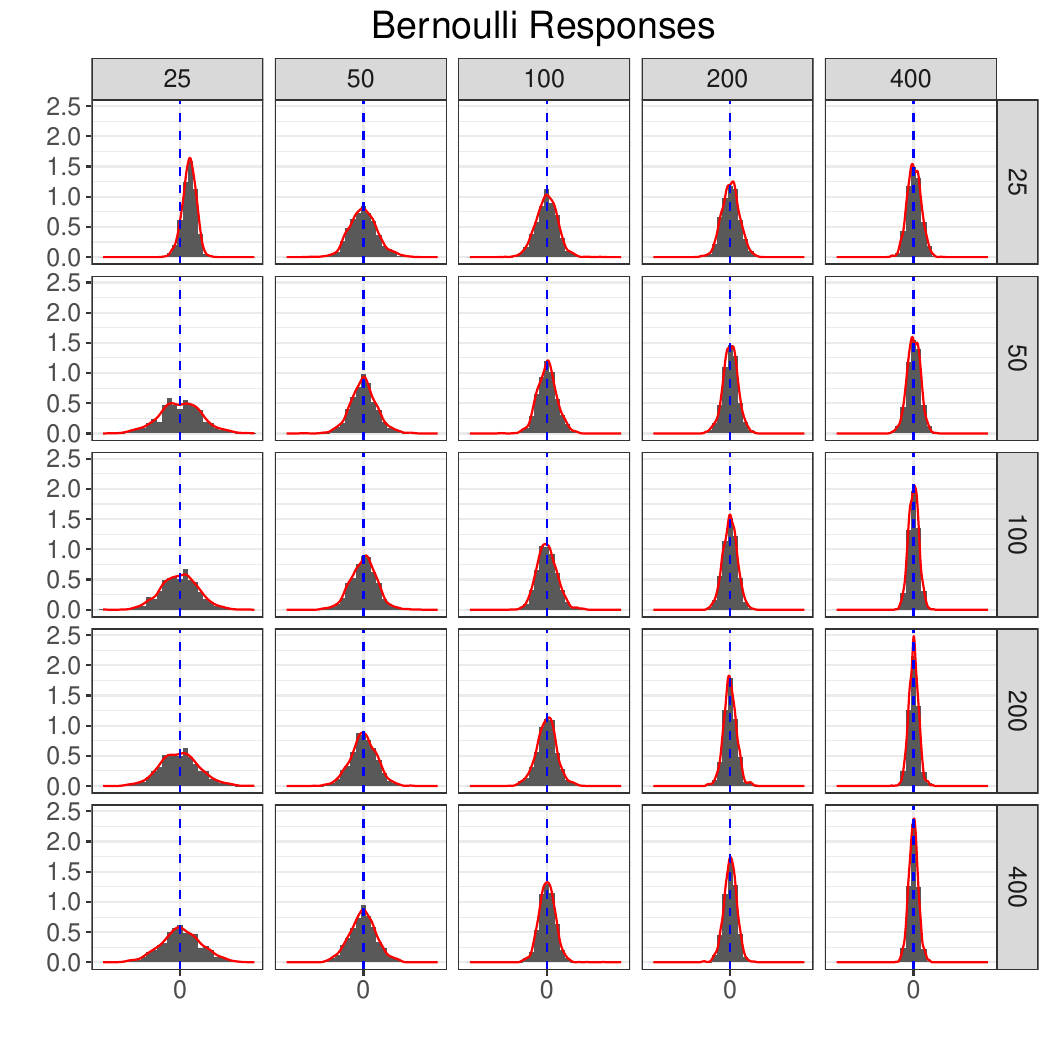}
\caption{Histograms for the third components of $\hat{\bmbeta} - \dot{\bmbeta}$ (left panels) and $\hat{\bm{b}}_1 - \dot{\bm{b}}_1$ (right panels), under the unconditional regime. Vertical facets represent the cluster sizes, while horizontal facets represent the number of clusters. The dotted blue line indicates zero, and the red curve is a kernel density smoother.} 
\label{fig: uncond_hist_pois}
\end{figure}

The Shapiro-Wilk tests run were consistent with the conclusions reached in Theorems \ref{thm:unconditional_fixef}--\ref{thm:unconditonal_ranef}. Specifically, PQL estimates of the fixed effect coefficients generally did not exhibit signs of non-normality, but the \emph{difference} between the estimators and true random effects displayed evidence of non-normality except when $n$ grew faster than $m$. This is also supported by the histograms in Figure \ref{fig: uncond_hist_pois} which show some evidence of higher kurtosis in the cases corresponding to small p-values in the Shapiro Wilk test. The histograms also suggest that both $m$ and $n$ need to grow for the estimators to be consistent for the true fixed and random effects, and in particular $n$ needs to grow for the estimators to be unbiased. This is true especially for the Bernoulli responses, for which convergence was much slower and very large cluster sizes were needed for the estimators to be relatively unbiased.


In the supplementary material, we present additional results which showed that the sample covariance matrix of the estimated random effects became a better estimator of the true random effects covariance matrix $\dot{\bm{G}}$ as both $m$ and $n$ grew. 
Also, recall from our discussion in Section \ref{sec:glmms} that our asymptotic developments only require a working $\hat{\bm{G}}$, which need not be a consistent estimator of the true random effects covariance matrix. 
As a demonstration of this, we performed several additional simulations where in the PQL estimation procedure, we simply fix $\hat{\bm{G}}$ to a constant matrix and considered choices e.g., some constant multiplied by the identity matrix. 
Results in the supplementary material demonstrate that coverage probabilities for our proposed intervals still tended to the nominal level as $(m,n)$ increased, 
while corresponding Shapiro-Wilk tests and histograms
were also consistent with our theory in large sample sizes and the empirical results presented above. 

\section{Discussion} \label{sec:discussion}
In this article, we established new asymptotic distributional results for fixed effects, random effects, and the prediction gap, for an independent-cluster GLMM fitted using penalized quasi-likelihood estimation. Our results 
have important implications when it comes to inference and prediction for mixed-effects models. For the conditional regime, we
establish asymptotic normality for any finite subset of the parameters. For random effects predictions and inference in the unconditional regime, we
validate examining the empirical distribution of the estimated random effects as a diagnostic tool for assessing deviations away from the assumed random effects distribution \citep[as is already commonly done in practice for GLMMs e.g.,][]{hui2021random}. On the other hand, while the random effects estimators obtained using PQL are asymptotically normally distributed when the true random effects are normally distributed, 
we demonstrate that the difference between these two i.e., the prediction gap, need not be normally distributed.
Our large sample results thus suggest the use of a normal approximation when performing unconditional inference for the random effects, as is commonly done in practice \citep{bates2015fitting,glmmtmb}, can be potentially misleading. 


An important avenue of future research is to establish rates of convergence, especially in the unconditional regime, when $\bm{x}_{ij}$ contains both $\bm{z}_{ij}$ plus additional components which are only included as purely fixed effects in the model. 
In the supplementary material, we develop some further results for such unpartnered fixed effects in the special cases of LMMs and GLMs. In both these cases, we see the convergence rate improves from $O_p(m^{1/2})$ to $O_p(N^{1/2})$, compared to the partnered fixed effects.
On the other hand, for random effects without a partnered fixed effect, it is likely that the correct asymptotic distribution for the prediction gap will be the normal scale-mixture irrespective of the relative rates of $m$ and $n_i$, as we saw in the motivating counterexample. Also, relaxing the canonical link assumption is an interesting and important extension to our work; we conjecture that non-canonical links could be accounted for by generalising the form of the GLM iterative weights, as is done in GLMs.

\FloatBarrier

\printbibliography

\newpage

\setcounter{section}{0}
\setcounter{equation}{0}
\def\theequation{S\arabic{section}.\arabic{equation}}
\def\thesection{S\arabic{section}}

{\Huge Supplementary Material} \\~\\

In the developments, we prove all results below assuming the working dispersion parameter $\hat{\phi}$ is equal to the true dispersion parameter $\dot{\phi}$. Then for the general result using any $O_p(1)$ working $\hat{\phi}$, we note that solving
\begin{align*}
    \nabla Q(\hat{\bm{\theta}}) &= \left[ \begin{matrix} {\hat{\phi}}^{-1} \bm{X}^\top \{ \bm{y} - \bm{\mu} (\hat{\bm{\theta}}) \} \\ {\hat{\phi}}^{-1} \bm{Z}^\top \{ \bm{y} - \bm{\mu} (\hat{\bm{\theta}}) \} - (\bm{I}_m \otimes \hat{\bm{G}}^{-1}) \hat{\bm{b}} \end{matrix} \right] = \bm{0}_{(m+1)p}
\end{align*}
for $\hat{\bm{\theta}}$ is equivalent to solving 
\begin{align*}
    \nabla Q(\hat{\bm{\theta}}) &= \left[ \begin{matrix} \dot{\phi}^{-1} \bm{X}^\top \{ \bm{y} - \bm{\mu} (\hat{\bm{\theta}}) \} \\ \dot{\phi}^{-1} \bm{Z}^\top \{ \bm{y} - \bm{\mu} (\hat{\bm{\theta}}) \} - (\bm{I}_m \otimes \hat{\bm{G}}_s^{-1}) \hat{\bm{b}} \end{matrix} \right] = \bm{0}_{(m+1)p},
\end{align*}
where $\hat{\bm{G}}_s =  \dot{\phi} \hat{\phi}^{-1} \hat{\bm{G}}$, whose inverse is still $O_p(1)$ and positive definite. This is equivalent to setting $\hat{\phi}$ to $\dot{\phi}$ and scaling $\hat{\bm{G}}$ by $\dot{\phi} \hat{\phi}^{-1}$. The general result then follows since the results proved under $\hat{\phi} = \dot{\phi}$ hold for any $\hat{\bm{G}}$ that has an $O_p(1)$, positive definite inverse. 

\subsection{Bias and Identifiability in the Conditional Regime} \label{subsec:biasidentifiability_conditional}

By differentiating \eqref{PQL}, we see that the PQL estimators satisfy $\sum_{i=1}^m {\hat{\phi}}^{-1} \bm{X}_i^\top \{\bm{y}_i - \bm{\mu}_i (\hat{\bm{\theta}})\} = \bm{0}$ and ${\hat{\phi}}^{-1} \bm{Z}_i^\top \{\bm{y}_i - \bm{\mu}_i (\hat{\bm{\theta}})\} - \bm{G}^{-1} \hat{\bm{b}}_i = \bm{0}$, $i=1,\ldots,m$. Summing both sides of the second equation across all $i$, since $\bm{X}_i = \bm{Z}_i$, it follows that $\sum_{i=1}^m \hat{\bm{b}}_{i} = \bm{0}_p$. That is, the PQL estimators of the random effects must satisfy a sum-to-zero constraint regardless of the underlying true parameter values. Under a conditional regime, this induces an asymptotic bias as captured by the term $\bm{1}^*_m \otimes (m^{-1} \sum_{i=1}^m \dot{\bm{b}}_i)$ in Theorem \ref{thm:conditionalnormality}, which can be interpreted as shifting the mean of the random effects into the corresponding fixed effects. We can deal with the bias by reparametrising the model \emph{a priori} to satisfy a sum-to-zero constraint. That is, we can define a reparametrized vector of true values $\dot{\bm{\theta}}^*$ which satisfy $\bm{1}^*_m \otimes (m^{-1} \sum_{i=1}^m \dot{\bm{b}}_i^*) = \bm{0}_{(m+1)p}$, and the PQL estimator will then be asymptotically normally distributed centered around $\dot{\bm{\theta}}^* $. 
Furthermore, Theorem \ref{thm:conditionalnormality} remains practically useful as, for any given sample size, we can always reparameterise the GLMM to satisfy this identifiability constraint. 


The asymptotic bias discussed above is analogous to that seen in a over-parametrized one-way analysis of variance (ANOVA) model. That is, in the ANOVA model one can always reparametrise to satisfy a sum-to-zero constraint, and the corresponding estimator is consistent for this vector of the reparametrized true values. 
Note however that when we work unconditionally (Section \ref{sec:unconditionaltheory}), reparametrising in this way will lead to a different model to the original, since the clusters are no longer independent.

\section{Proofs for Consistency} \label{sec:consistency}

To establish our large sample distributional results, we first require the following consistency result.
\begin{lemma} \label{lem:infinitynormconsistency}
Suppose Conditions (C1)-(C5) hold and $m n_L^{-2} \rightarrow 0$. Then, as $m,n_L \rightarrow \infty$ and unconditional on the random effects $\dot{\bm{b}}$, $\| \hat{\bm{\theta}} -  \dot{\bm{\theta}} \|_{\infty} = o_p(1)$.
\end{lemma}
These results are required to control the remainder term in the Taylor expansions we use to derive the distributional results in Section \ref{sec:distributional}. To prove the result, we wish to show that for any given $\epsilon>0$, there exists a large enough constant $C > 0$ such that, for large $m,n_L$, we have 
\begin{align*}
P \left\{ \underset{\|\bm{u} \|_{\infty} = C}{\sup} Q(\dot{\bm{\theta}} + \delta_{m,n_L}^{-1}  \bm{u}) < Q(\dot{\bm{\theta}}) \right\} \geq 1-\epsilon,
\end{align*}
for some positive, unbounded, monotonically increasing sequence $\delta_{m,n_L} $. The above result implies that with probability tending to one, there exists a local maximum $\hat{\bm{\theta}}$ in the ball $\{ \dot{\bm{\theta}} + \delta_{m,n_L}^{-1}  \bm{u} : \|\bm{u}\|_{\infty} \leq C \}$ so that $\| \delta_{m,n_L}  (\hat{\bm{\theta}} -  \dot{\bm{\theta}}) \|_{\infty} = O_p(1)$, and thus $\|  \hat{\bm{\theta}} -  \dot{\bm{\theta}} \|_{\infty} = o_p(1)$. 

Consider the difference $Q(\dot{\bm{\theta}} + \bm{u}) - Q(\dot{\bm{\theta}})$. By a Taylor expansion, we obtain
\begin{align} \label{consistency_diff}
Q(\dot{\bm{\theta}} +  \bm{u}) - Q(\dot{\bm{\theta}}) &= \bm{u}^\top  \{ \nabla Q(\dot{\bm{\theta}}) \} - 0.5 \bm{u}^\top  \{- \nabla^2 Q(\bar{\bm{\theta}}) \}  \bm{u}.
\end{align}
where $\bar{\bm{\theta}}$ lies on the line segment joining $\dot{\bm{\theta}}$ and $\dot{\bm{\theta}} + \bm{u}$. If we can prove that \eqref{consistency_diff} is negative as $m,n_L \rightarrow \infty$ for any choice of $C$, then there must exist some $\delta_{m,n_L}$ such that $Q(\dot{\bm{\theta}} + \delta_{m,n_L}^{-1} \bm{u}) - Q(\dot{\bm{\theta}})$ is negative for large enough $C$, and the required result follows. We have
\begin{align*}
\nabla Q(\dot{\bm{\theta}}) &= \left[ \begin{matrix} \dot{\phi}^{-1} \bm{X}^\top (\bm{y} - \dot{\bm{\mu}}) \\ \dot{\phi}^{-1} \bm{Z}^\top (\bm{y} - \dot{\bm{\mu}}) - (\bm{I}_m \otimes \hat{\bm{G}}^{-1})\dot{\bm{b}} \end{matrix} \right] = \left[ \begin{matrix} \dot{\phi}^{-1} \bm{X}^\top (\bm{y} - \dot{\bm{\mu}}) \\ \dot{\phi}^{-1} \bm{Z}^\top (\bm{y} - \dot{\bm{\mu}}) \end{matrix} \right] + \left[ \begin{matrix} \bm{0}_p \\ - (\bm{I}_m \otimes \hat{\bm{G}}^{-1})\dot{\bm{b}} \end{matrix} \right] \\
&\triangleq \bm{\lambda}_1 + \bm{\lambda}_2,
\end{align*}
and
\begin{align*} 
- \nabla^2 Q(\bar{\bm{\theta}}) &= \left[ \begin{matrix} 
\bm{X}^\top  \bar{\bm{W}} \bm{X} & \bm{X}_1^\top  \bar{\bm{W}}_1 \bm{X}_1 & \cdots & \bm{X}_m^\top  \bar{\bm{W}}_m \bm{X}_m \\
\bm{X}_1^\top  \bar{\bm{W}}_{1} \bm{X}_1 & \bm{X}_1^\top  \bar{\bm{W}}_{1} \bm{X}_1 + \hat{\bm{G}}^{-1} & & \bm{0} \\
\vdots & \vdots & \ddots & \vdots \\
\bm{X}_m^\top  \bar{\bm{W}}_{m} \bm{X}_m & \bm{0} & & \bm{X}_m^\top  \bar{\bm{W}}_{m} \bm{X}_m + \hat{\bm{G}}^{-1}
\end{matrix} \right] \\
&= \left[ \begin{matrix} 
\bm{X}^\top  \bar{\bm{W}} \bm{X} & \bm{X}_1^\top  \bar{\bm{W}}_1 \bm{X}_1 & \cdots & \bm{X}_m^\top  \bar{\bm{W}}_m \bm{X}_m \\
\bm{X}_1^\top  \bar{\bm{W}}_{1} \bm{X}_1 & \bm{X}_1^\top  \bar{\bm{W}}_{1} \bm{X}_1 & & \bm{0} \\
\vdots & \vdots & \ddots & \vdots \\
\bm{X}_m^\top  \bar{\bm{W}}_{m} \bm{X}_m & \bm{0} & & \bm{X}_m^\top  \bar{\bm{W}}_{m} \bm{X}_m
\end{matrix} \right] + \text{blockdiag}(\bm{0}_p, \bm{I}_m \otimes \hat{\bm{G}}^{-1}) \\
& \triangleq \bm{\Gamma}_1(\bar{\bmtheta}) + \bm{\Gamma}_2,
\end{align*}
where $\bar{\bm{W}}_{i} = \dot{\phi}^{-1} \diag{\{a''(\bar{\eta}_{i1}), \ldots, a''(\bar{\eta}_{in_i})\}}$ and $\bar{\bm{W}} = \dot{\phi}^{-1} \diag{\{a''(\bar{\eta}_{11}), \ldots, a''(\bar{\eta}_{mn_m})\}}$. Also, let $\bm{\Gamma}_1(\dot{\bmtheta}) + \bm{\Gamma}_2$ denote the analogous decomposition of $- \nabla^2 Q(\dot{\bm{\theta}})$. For both the conditional and unconditional regimes, we will prove that the second term is positive and dominates the first. However, the treatment of the terms differs between the two cases, and as such the proofs will need to be dealt with separately. In the following three sections, we will first treat the Poisson pure random intercept example, followed by the more general conditional and unconditional regimes.

Before proceeding, we demonstrate an inequality that is used in the proofs below. Write $\bm{u} = (\bm{u}_1^\top , \bm{u}_{2}^\top)^\top$, $\bm{u}_{2} = (\bm{u}_{21}^\top, \ldots, \bm{u}_{2m}^\top)^\top$. First, for any $\bmtheta$ we have 
\begin{align*}
    \bm{u}^\top \bm{\Gamma}_1(\bmtheta) \bm{u} = \bm{u}_1^\top \bm{X}^\top  \bm{W} \bm{X} \bm{u}_1 + 2 \bm{u}_1^\top \bm{X}^\top  \bm{W} \bm{Z} \bm{u}_2 + \bm{u}_2^\top \bm{Z}^\top  \bm{W} \bm{Z} \bm{u}_2^\top \geq 0.
\end{align*}
Next, we have 
\begin{align*}
    &\bm{u}^\top \bm{\Gamma}_1(\bar{\bmtheta}) \bm{u} - c_0^2 \bm{u}^\top \bm{\Gamma}_1(\bmtheta) \bm{u} \\
    &= \bm{u}_1^\top \bm{X}^\top  (\bar{\bm{W}} - c_0^2 \bm{W}) \bm{X} \bm{u}_1 + 2 \bm{u}_1^\top \bm{X}^\top  (\bar{\bm{W}} - c_0^2 \bm{W}) \bm{Z} \bm{u}_2 + \bm{u}_2^\top \bm{Z}^\top  (\bar{\bm{W}} - c_0^2 \bm{W}) \bm{Z} \bm{u}_2^\top.
\end{align*}
If we denote $\bm{W}^* = \bar{\bm{W}} - c_0^2 \bm{W}$, then by Condition (C1) $\bm{W}^*$ is a diagonal matrix with non-negative entries as the entries of $c_0^2 \bm{W}$ are upper bounded by the smallest component in $\bar{\bm{W}}$. Therefore 
\begin{align*}
    \bm{u}^\top \bm{\Gamma}_1(\bar{\bmtheta}) \bm{u} - c_0^2 \bm{u}^\top \bm{\Gamma}_1(\bmtheta) \bm{u} = \bm{u}_1^\top \bm{X}^\top  \bm{W}^* \bm{X} \bm{u}_1 + 2 \bm{u}_1^\top \bm{X}^\top  \bm{W}^* \bm{Z} \bm{u}_2 + \bm{u}_2^\top \bm{Z}^\top  \bm{W}^* \bm{Z} \bm{u}_2^\top \geq 0,
\end{align*}
so that $\bm{u}^\top \bm{\Gamma}_1(\bar{\bmtheta}) \bm{u} \geq c_0^2 \bm{u}^\top \bm{\Gamma}_1(\bmtheta) \bm{u}$. Finally, note that we can choose $\bmtheta = \dot{\bmtheta}$ or $\bmtheta = E(\dot{\bmtheta})$ without altering the above argument.

\subsection{Poisson pure random intercept example}

We begin with the Poisson pure random intercept example, which gives insight and covers a case where $\bm{X}_i \neq \bm{Z}_i$. The following result is unconditional on the random effects $\dot{\bm{b}}$.

\begin{lemma}
Assume Conditions (C1)-(C5) hold, and let $m n^{-2} \rightarrow 0$. Then for the Poisson pure random intercept model, as $m,n \rightarrow \infty$ and unconditional on the random effects $\dot{\bm{b}}$, it holds that $\| \hat{\bm{\theta}} -  \dot{\bm{\theta}} \|_{\infty} = o_p(1)$.
\end{lemma}

\begin{proof}
Let $\bm{u}= \bm{u}_2 = ({u}_{21}, \ldots, {u}_{2m})^\top$, $\bmtheta = \bm{b} = (b_1, \ldots, b_m)^\top$, $\hat{\bm{G}} = \hat{\sigma}_b^2$ (a scalar), \\ $- \nabla^2 Q(\bar{\bm{\theta}}) = \diag{(n e^{\bar{b}_1} + \hat{\sigma}_b^{-2}, \ldots, n e^{\bar{b}_m} + \hat{\sigma}_b^{-2})} \equiv \bm{\Gamma}_1(\bar{\bmtheta}) + \bm{\Gamma}_2$, and
\begin{align*}
    \nabla Q(\dot{\bm{\theta}}) = \left[ \begin{matrix} \sum_{j=1}^n ( y_{1j} - e^{\dot{b}_1}) - \hat{\sigma}_b^{-2} \dot{b}_1 \\ 
    \vdots \\ \sum_{j=1}^n ( y_{mj} - e^{\dot{b}_m}) - \hat{\sigma}_b^{-2} \dot{b}_m \end{matrix} \right] \equiv \bm{\lambda}_1 + \bm{\lambda}_2.
\end{align*}
Let $\bm{M} = E\{ \diag{(ne^{\dot{b}_1}, \ldots, ne^{\dot{b}_m})} \}$. Then $\bm{M} = \text{Var} \{ \dot{\phi}^{-1} \bm{Z}^\top (\bm{y} - \dot{\bm{\mu}}) \}$. By Condition (C1), $c_0^2 \bm{u}^\top \bm{M} \bm{u} \leq \bm{u}^\top \bm{\Gamma}_1(\bar{\bmtheta}) \bm{u}$. Next, let $\lambda = \dot{\hat{\sigma}}_b^2 \hat{\sigma}_b^{-2}$. Then $\text{Var}( \bm{\lambda}_2) = \dot{\hat{\sigma}}_b^2 \hat{\sigma}_b^{-4} \bm{I}_m$ and 
\begin{align*}
    \lambda^{-1} \bm{u}_2^\top ( \dot{\hat{\sigma}}_b^2 \hat{\sigma}_b^{-4} \bm{I}_m) \bm{u}_2 =  \bm{u}_2^\top ( \hat{\sigma}_b^{-2} \bm{I}_m) \bm{u}_2 = \bm{u}^\top \bm{\Gamma}_2 \bm{u}.
\end{align*}
Finally, by the laws of iterated expectation and variance we have 
\begin{align*}
    E\{ \nabla Q(\dot{\bmtheta}) \nabla Q(\dot{\bmtheta})^\top \} &= \text{Var} \{ \nabla Q(\dot{\bmtheta}) \} \\
    &= E [ \text{Var} \{\nabla Q(\dot{\bmtheta}) | \dot{\bm{b}}\} ] + \text{Var} [ E \{\nabla Q(\dot{\bmtheta}) | \dot{\bm{b}} \} ] \\
    &= E \{ \text{Var} (\bm{\lambda}_1 | \dot{\bm{b}} ) \} + \text{Var} (\bm{\lambda}_2) \\ 
    &= \text{Var} (\bm{\lambda}_1) + \text{Var} (\bm{\lambda}_2).
\end{align*}
Therefore, we have that 
\begin{align*}
    \bm{u}^\top \{ - \nabla^2 Q(\bar{\bm{\theta}}) \} \bm{u} &\geq \min(\lambda^{-1}, c_0^2) \{ \bm{u}^\top \bm{M} \bm{u} + \bm{u}_2^\top ( \dot{\hat{\sigma}}_b^2 \hat{\sigma}_b^{-4} \bm{I}_m) \bm{u}_2 \} \\
    &= \min(\lambda^{-1}, c_0^2) E \{ \nabla Q(\dot{\bmtheta}) \nabla Q(\dot{\bmtheta})^\top \},
\end{align*}
where the latter and hence former term grows at the same rate as $\{ \bm{u}^\top  \nabla Q(\dot{\bmtheta}) \}^2$. Since at least one component of $\bm{u}$ equals $\pm C$, for any given $\bm{u}$, $\bm{u}^\top \{ - \nabla^2 Q(\bar{\bm{\theta}}) \} \bm{u}$ is at least of order $O_p(m)$ in probability and hence always dominates.

Since the choice of which $|u_{2i}| = C$ is arbitrary however, we also need to make sure that the $m$th order statistic $\underset{i \in \{ 1,\ldots, m \}}{\max} [ \{ \sum_{j=1}^n ( y_{ij} - e^{\dot{b}_i}) - \hat{\sigma}_b^{-2} \dot{b}_i \}/( n e^{\dot{b}_i} + \hat{\sigma}_b^{-2})]$, which grows with the dimension, is of order $o_p(1)$. We know that the leading term in \eqref{eq:Mrand_basic.uncond} is $(\bm{Z}^\top \dot{\bm{W}} \bm{Z})^{-1} \{ \dot{\phi}^{-1} \bm{Z}^\top (\bm{y} - \dot{\bm{\mu}}) \}$ when $m n^{-1} \rightarrow \infty$; for this Poisson random intercept example, up to some smaller order terms, this simplifies to the ratio $\{ \sum_{j=1}^n ( y_{ij} - e^{\dot{b}_i}) - \hat{\sigma}_b^{-2} \dot{b}_i \}/( n e^{\dot{b}_i} + \hat{\sigma}_b^{-2})$. Intuitively then, proving a result for $\| \hat{\bmtheta} - \dot{\bmtheta} \|_{\infty}$ should involve studying $\underset{i \in \{ 1,\ldots, m \}}{\max} [ \{ \sum_{j=1}^n ( y_{ij} - e^{\dot{b}_i}) - \hat{\sigma}_b^{-2} \dot{b}_i \}/( n e^{\dot{b}_i} + \hat{\sigma}_b^{-2})]$.

Put another way, consider the set of $\bm{u}$ such that one component of $\bm{u}$ equals $\pm C$ and zero elsewhere. When $C$ is the $i$th component of $\bm{u}$, this corresponds to deviating away from $\dot{\bmtheta}$ in the $i$th direction. In this case, we need $C \{ \sum_{j=1}^n ( y_{ij} - e^{\dot{b}_i}) - \hat{\sigma}_b^{-2} \dot{b}_i \}$ to be dominated by $C^2 ne^{\dot{b}_i}$ for any $C$ and all $m,n$ large enough, i.e., $\{ \sum_{j=1}^n ( y_{ij} - e^{\dot{b}_i}) - \hat{\sigma}_b^{-2} \dot{b}_i \}/ ne^{\dot{b}_i} = o_p(1)$. This is indeed true as this ratio is $O_p(n^{-1/2})$, since $\sum_{j=1}^n ( y_{ij} - e^{\dot{b}_i}) - \hat{\sigma}_b^{-2} \dot{b}_i = O_p(n^{1/2})$ due to conditional independence and Chebyshev's inequality, and $e^{\dot{b}_i} = O_p(1)$. However, although the ratio is of order $O_p(n^{-1/2})$, for any given $m,n$ there is still a positive probability that the ratio (a random variable) is greater than one in magnitude. On the other hand, for the consistency argument to hold we need to make sure the ratio is smaller than one in magnitude for all $m$ directions with probability tending to one, as $m,n \rightarrow \infty$. In particular, it is sufficient for the maximum of $m$ of these ratios to be $o_p(1)$: this maximum grows with $m$, corresponding to the number of directions we need to bound. Intuitively, this should hold if $m$ does not grow too fast relative to $n$. 

Now, \citet{downey1990distribution} proves that the maximum over $m$ realisations of independently and identically distributed random variables with a finite $q$th moment is $o_p(m^{1/q})$. By Condition (C5), the ratio $n^{1/2}  \{ \sum_{j=1}^n ( y_{ij} - e^{\dot{b}_i}) - \hat{\sigma}_b^{-2} \dot{b}_i \}/(n e^{\dot{b}_i} + \hat{\sigma}_b^{-2})$ has finite fourth moments for all $i$ and $n$. Thus, the maximum of these (normalised) ratios over $m$ clusters is of order $o_p(m^{1/4})$. As a result, the maximum ratio of interest is $o_p(m^{1/4} n^{-1/2})$. Therefore, when $m n^{-2} \rightarrow 0$, there exists $\delta_{m,n}$ such that we can always choose a large enough $C$ for $\delta_{m,n}^{-1} \bm{u}^\top  \nabla Q(\dot{\bm{\theta}})$ to be dominated by $\delta_{m,n}^{-2} \bm{u}^\top \{ - \nabla^2 Q(\bar{\bm{\theta}}) \} \bm{u}$, and hence $\| \delta_{m,n} (\hat{\bm{\theta}} -  \dot{\bm{\theta}}) \|_{\infty} = O_p(1)$ as required.

\end{proof}
To conclude this section, we remark that although $m n^{-2} \rightarrow 0 $ is needed for the consistency and thus distributional result, this is a sufficient condition. Intuitively, in the Poisson pure random effects model there are no fixed parameters to estimate, and the estimate of the random effects for each cluster only depends on observations in that cluster. Thus, the relative rates of $m$ and $n$ should not matter for a distributional result concerning a finite subset of the random effects.

\subsection{Conditional on the Random Effects}

In this section, we prove the consistency result under the conditional regime. In the conditional regime, we assume without loss of generality that $\sum_{i=1}^m \dot{\bm{b}}_i = \bm{0}_p$, recalling that we can always reparametrise the random effect coefficients so this holds. 

Let $\bm{M} = \bm{\Gamma}_1(\dot{\bmtheta})$. Then $\bm{M} = \text{Var} ( \bm{\lambda}_1 | \dot{\bm{b}} ) = E( \bm{\lambda}_1 \bm{\lambda}_1^\top| \dot{\bm{b}} )$ since $E( \bm{\lambda}_1 | \dot{\bm{b}}) = \bm{0}_{(m+1)p}$. By Condition (C1), we have $c_0^2 \bm{u}^\top \bm{M} \bm{u} \leq \bm{u}^\top \bm{\Gamma}_1(\bar{\bmtheta}) \bm{u}$.

We now consider two cases: the special case when $\bm{u}_1 = - \bm{u}_{2i}$ for all $i$, and when this is not the case. For the former, we have $\bm{u}^\top \bm{\lambda}_1 = \bm{u}^\top \bm{M} \bm{u} = 0$. Then we must examine $\bm{u}^\top \bm{\lambda}_2$ and $\bm{u}^\top \bm{\Gamma}_2 \bm{u}$. In this case, we have $\bm{u}^\top \bm{\lambda}_2 = \sum_{i=1}^m \bm{u}_{2i}^\top \hat{\bm{G}}^{-1} \dot{\bm{b}}_i = - \bm{u}_{1}^\top \hat{\bm{G}}^{-1} \sum_{i=1}^m  \dot{\bm{b}}_i = 0$, and $\bm{u}^\top \bm{\Gamma}_2 \bm{u} = m \bm{u}_1^\top \hat{\bm{G}}^{-1} \bm{u}_1 > 0$ since $\hat{\bm{G}}$ is a positive definite matrix. Thus the difference \eqref{consistency_diff} is negative for large enough $m,n_L$ and any choice of constant $C$.

Next, consider the case when $\bm{u}_1 = - \bm{u}_{2i}$ for all $i$ does not hold. Under this setting, as $\bm{\Gamma}_2$ is a positive semi-definite matrix, we still have $\bm{u}^\top \{ - \nabla^2 Q(\bar{\bm{\theta}}) \} \bm{u} \geq \bm{u}^\top \bm{\Gamma}_1(\bar{\bmtheta}) \bm{u} \geq c_0^2 \bm{u}^\top \bm{M} \bm{u}$, where the last and hence former terms grow at the same rate as $( \bm{u}^\top  \bm{\lambda}_1 )^2$. Since at least one component of $\bm{u}$ equals $\pm C$, by Conditions (C1)-(C3) we have that $\bm{u}^\top \{ - \nabla^2 Q(\bar{\bm{\theta}}) \} \bm{u}$ is at least of order $O_p(n_L)$, and always dominates since $\bm{u}^\top \bm{\lambda}_2 = O_p(m)$ at most.

Since the choice of $\bm{u}$ is arbitrary, we must take into account the growth rate of the $m$th order statistic. That is, for any $1 \leq k \leq p$, we need $\underset{i \in \{ 1,\ldots, m \}}{\max} [ (\bm{X}_i^\top \dot{\bm{W}}_i \bm{X}_i + \hat{\bm{G}}^{-1})^{-1} \{ \dot{\phi}^{-1} \bm{X}_i^\top (\bm{y}_{i}-\dot{\bm{\mu}}_i) - \hat{\bm{G}}^{-1} \dot{\bm{b}}_i \} ]_{[k]} = o_p(1)$, as per the argument for the Poisson pure random intercept model. Since the responses $y_{ij}$ are from the exponential family and thus the moment generating function always exists, the maximum is of order $o_p( m^{1/r} n_L^{-1/2} )$ for any $r \in \mathbb{N}$ \citep[][]{downey1990distribution}, and hence $o_p(1)$ since $m n_L^{-1} \rightarrow 0$ by taking $r=2$, for example. Note that the first $p$ components of $\nabla Q(\dot{\bmtheta})$, which are associated with the fixed effects, do not need to be bounded in this way because the dimension is fixed.

\subsection{Unconditional on the Random Effects}

In this section, we prove the consistency result under the unconditional regime. The main differences to the derivation under the conditional regime arise from the treatment of $\bm{\lambda}_2$, and the distribution of $\bm{y}$. In the unconditional regime it holds that $\sum_{i=1}^m \dot{\bm{b}}_i = O_p(m^{1/2})$, while in the conditional regime we impose a sum to zero constraint. Furthermore, in the unconditional regime we bound $\bm{u}^\top \bm{\lambda}_2$ using its variance, while in the conditional regime this is not possible because $\bm{\lambda}_2$ is not a random variable. Finally, in the unconditional regime we cannot use the properties of the exponential family to bound the $m$th order statistic, instead requiring Condition (C5).

Let $\bm{M} = E \{ \bm{\Gamma}_1(\dot{\bmtheta}) \}$. Then $\bm{M} = \text{Var} ( \bm{\lambda}_1 ) = E( \bm{\lambda}_1 \bm{\lambda}_1^\top )$ since $E( \bm{\lambda}_1 ) = \bm{0}_{(m+1)p}$. By Condition (C1), $c_0^2 \bm{u}^\top \bm{M} \bm{u} \leq \bm{u}^\top \bm{\Gamma}_1(\bar{\bmtheta}) \bm{u}$.

We consider two cases: the special case when $\bm{u}_1 = - \bm{u}_{2i}$ for all $i$, and when this is not the case. In the former, we have $\bm{u}^\top \bm{\lambda}_1 = \bm{u}^\top \bm{M} \bm{u} = 0$. Thus we must examine $\bm{u}^\top \bm{\lambda}_2$ and $\bm{u}^\top \bm{\Gamma}_2 \bm{u}$. In this case, we have $\bm{u}^\top \bm{\lambda}_2 = \sum_{i=1}^m \bm{u}_{2i}^\top \hat{\bm{G}}^{-1} \dot{\bm{b}}_i = - \bm{u}_{1}^\top \hat{\bm{G}}^{-1} \sum_{i=1}^m  \dot{\bm{b}}_i = O_p(m^{1/2})$, and $\bm{u}^\top \bm{\Gamma}_2 \bm{u} = m \bm{u}_1^\top \hat{\bm{G}}^{-1} \bm{u}_1 > 0$ since $\hat{\bm{G}}$ is a positive definite matrix. Hence the difference \eqref{consistency_diff} is negative for large enough $m,n_L$, and any choice of constant $C$.

Next, consider the case when $\bm{u}_1 = - \bm{u}_{2i}$ for all $i$ does not hold. Then we still have $\bm{u}^\top \{ - \nabla^2 Q(\bar{\bm{\theta}}) \} \bm{u} \geq c_0^2 \bm{u}^\top \bm{M} \bm{u}$. Letting $\lambda = \lambda_{\max}(\hat{\bm{G}}^{-1} \dot{\bm{G}} \hat{\bm{G}}^{-1}) / \lambda_{\min} (\hat{\bm{G}}^{-1})$, we have
\begin{align*}
    \text{Var}( \bm{\lambda}_2) = \bm{I}_m \otimes \hat{\bm{G}}^{-1} \dot{\bm{G}} \hat{\bm{G}}^{-1}
\end{align*}
and
\begin{align*}
    \lambda^{-1} \bm{u}_2^\top ( \bm{I}_m \otimes \hat{\bm{G}}^{-1} \dot{\bm{G}} \hat{\bm{G}}^{-1}) \bm{u}_2 \leq  \bm{u}_2^\top ( \bm{I}_m \otimes \hat{\bm{G}}^{-1}) \bm{u}_2 = \bm{u}^\top \bm{\Gamma}_2 \bm{u}.
\end{align*}
Now, by the laws of iterated expectation and variance,
\begin{align*}
    E\{ \nabla Q(\dot{\bmtheta}) \nabla Q(\dot{\bmtheta})^\top \} &= \text{Var} \{ \nabla Q(\dot{\bmtheta}) \} \\
    &= E [ \text{Var} \{\nabla Q(\dot{\bmtheta}) | \dot{\bm{b}}\} ] + \text{Var} [ E \{\nabla Q(\dot{\bmtheta}) | \dot{\bm{b}} \} ] \\
    &= E \{ \text{Var} (\bm{\lambda}_1 | \dot{\bm{b}} ) \} + \text{Var} (\bm{\lambda}_2) \\
    &= \text{Var} (\bm{\lambda}_1) + \text{Var} (\bm{\lambda}_2).
\end{align*}
Thus we have that 
\begin{align*}
    \bm{u}^\top \{ - \nabla^2 Q(\bar{\bm{\theta}}) \} \bm{u} &\geq \min(\lambda^{-1}, c_0^2) \{ \bm{u}^\top \bm{M} \bm{u} + \bm{u}_2^\top ( \bm{I}_m \otimes \hat{\bm{G}}^{-1} \dot{\bm{G}} \hat{\bm{G}}^{-1}) \bm{u}_2 \} \\
    &= \min(\lambda^{-1}, c_0^2) \bm{u}^\top E \{ \nabla Q(\dot{\bmtheta}) \nabla Q(\dot{\bmtheta})^\top  \} \bm{u},
\end{align*}
where the latter and hence former term grows at the same rate as $ \{ \bm{u}^\top  \nabla Q(\dot{\bmtheta}) \}^2$. Since at least one component of $\bm{u}$ equals $\pm C$, for any given $\bm{u}$ we have that $\bm{u}^\top \{ - \nabla^2 Q(\bar{\bm{\theta}}) \} \bm{u}$ is at least of order $O_p(n_L)$ and always dominates.

Since the choice of $\bm{u}$ is arbitrary, we must take into account the growth rate of the $n$th order statistic. That is, for any $1 \leq k \leq p$, we require $\underset{i \in \{ 1,\ldots, m \}}{\max} [ (\bm{X}_i^\top \dot{\bm{W}}_i \bm{X}_i + \hat{\bm{G}}^{-1})^{-1} \{ \dot{\phi}^{-1} \bm{X}_i^\top (\bm{y}_{i}-\dot{\bm{\mu}}_i) - \hat{\bm{G}}^{-1} \dot{\bm{b}}_i \} ]_{[k]} = o_p(1)$, as per the argument for the Poisson pure random intercept model. By Condition (C5), this term is of order $o_p(m^{1/4} n_L^{-1/2})$, and hence the result follows. Note that the first $p$ components of $\nabla Q(\dot{\bmtheta})$, which are associated with the fixed effects, do not need to be bounded in this way because the dimension is fixed.

\section{Proofs of Distributional Results} \label{sec:distributional}
\setcounter{equation}{0}

For both the conditional and unconditional regimes, our proof relies on examining the behaviour of the leading term in the Taylor expansion of the estimating function. Under Conditions (C1) and (C3), we take the Taylor expansion of $\nabla Q(\hat{\bm{\theta}})$ around $\dot{\bm{\theta}}$ and obtain, as $m,n_L \rightarrow \infty$,
\begin{align} \label{eqn:taylor1}
\nabla Q(\hat{\bm{\theta}}) = \bm{0}_{(m+1)p} =\nabla Q(\dot{\bm{\theta}}) + \nabla^2 Q(\dot{\bm{\theta}}) (\hat{\bm{\theta}} - \dot{\bm{\theta}}) + \frac{1}{2} \bm{R}(\tilde{\bm{\theta}}),
\end{align} 
where $\tilde{\bm{\theta}}$ is a $(m+1)p \times (m+1)p$ matrix with each row lying on the line segment between $\dot{\bm{\theta}}$ and $\hat{\bm{\theta}}$ and $\bm{R}(\tilde{\bm{\theta}})$ is the remainder term. Rearranging, we have
\begin{align}
\label{taylor_fundamental}
\hat{\bm{\theta}} - \dot{\bm{\theta}} &= - \{\nabla^2 Q(\dot{\bm{\theta}})\}^{-1} \nabla Q(\dot{\bm{\theta}}) - \frac{1}{2} \{\nabla^2 Q(\dot{\bm{\theta}})\}^{-1} \bm{R}(\tilde{\bm{\theta}}).
\end{align}
We show in Section \ref{sec:remainder} that the remainder term is of smaller order than the leading term and thus negligible in the limit, in both the conditional and unconditional regimes.

From \eqref{taylor_fundamental}, to study the asymptotic behaviour of the PQL estimator we will first apply the blockwise matrix inversion formula to obtain an expression for $- \{\nabla^2 Q(\dot{\bm{\theta}})\}^{-1}$. Using this result, we will then obtain an expression for $- \{\nabla^2 Q(\dot{\bm{\theta}})\}^{-1} \nabla Q(\dot{\bm{\theta}})$, and subsequently study the asymptotic behaviour of each constituent term. Note that since $\nabla Q(\dot{\bm{\theta}})$ is a $(m+1)p$-vector and $- \{\nabla^2 Q(\dot{\bm{\theta}})\}^{-1}$ is a $(m+1)p \times (m+1)p$ matrix, we cannot simply take their limits as per standard fixed dimension asymptotics. Instead, we must evaluate $ - \{\nabla^2 Q(\dot{\bm{\theta}})\}^{-1} \nabla Q(\dot{\bm{\theta}})$ as a whole.

We can write
\begin{align*}
\nabla Q(\dot{\bm{\theta}}) &= \left[ \begin{matrix} \dot{\phi}^{-1} \bm{X}^\top (\bm{y} - \dot{\bm{\mu}}) \\ \dot{\phi}^{-1} \bm{Z}^\top (\bm{y} - \dot{\bm{\mu}}) - (\bm{I}_m \otimes \hat{\bm{G}}^{-1})\dot{\bm{b}} \end{matrix} \right] = \left[ \begin{matrix} \dot{\phi}^{-1} \sum_{i=1}^m \sum_{j=1}^{n_i} \bm{x}_{ij} (y_{ij} - \dot{\mu}_{ij}) \\ \dot{\phi}^{-1} \sum_{j=1}^{n_1} \bm{x}_{1j} (y_{1j} - \dot{\mu}_{1j}) -  \hat{\bm{G}}^{-1} \dot{\bm{b}}_1 \\ \vdots \\ \dot{\phi}^{-1} \sum_{j=1}^{n_m} \bm{x}_{mj} (y_{mj} - \dot{\mu}_{mj}) -  \hat{\bm{G}}^{-1} \dot{\bm{b}}_m \end{matrix} \right] \\
&\triangleq \left[ \begin{matrix} \bm{S}_1 \\ \bm{S}_{21} + \bm{S}_{31} \\ \vdots \\ \bm{S}_{2m} + \bm{S}_{3m} \end{matrix} \right]  \triangleq \left[ \begin{matrix} \bm{S}_1 \\ \bm{S}_4 + \bm{S}_5 \end{matrix} \right] \triangleq \left[ \begin{matrix} \bm{S}_1 \\ \bm{S}_6 \end{matrix} \right] , \\
\bm{B}(\dot{\bm{\theta}}) &= - \nabla^2 Q(\dot{\bm{\theta}}) = \left[ \begin{matrix}
\bm{X}^\top \dot{\bm{W}} \bm{X} & \bm{X}^\top \dot{\bm{W}} \bm{Z}\\
\bm{Z}^\top \dot{\bm{W}} \bm{X} & \bm{Z}^\top \dot{\bm{W}} \bm{Z} + \bm{I}_m \otimes \hat{\bm{G}}^{-1}
\end{matrix} \right] \triangleq \left[ \begin{matrix}
\underset{p \times p}{\bm{B}_1}  & \underset{p \times mp}{\bm{B}_2}\\
\underset{mp \times p}{\bm{B}_2^\top} & \underset{mp \times mp}{\bm{B}_3 + \bm{B}_4} 
\end{matrix} \right].
\end{align*}
Letting $\bm{C} = \bm{B}_1 - \bm{B}_2 (\bm{B}_3+\bm{B}_4)^{-1} \bm{B}_2^\top$, by the matrix block inversion formula we have
\begin{align}
\label{eq:B_inv}
\bm{B}^{-1} = \left[ \begin{matrix}
\bm{C}^{-1} & -\bm{C}^{-1} \bm{B}_2 (\bm{B}_3+\bm{B}_4)^{-1}\\
-(\bm{B}_3+\bm{B}_4)^{-1} \bm{B}_2^\top \bm{C}^{-1} & (\bm{B}_3+\bm{B}_4)^{-1} + (\bm{B}_3+\bm{B}_4)^{-1} \bm{B}_2^\top \bm{C}^{-1} \bm{B}_2 (\bm{B}_3+\bm{B}_4)^{-1}
\end{matrix} \right].
\end{align}
Next, based on the forms of $\bm{B}_2$ and $(\bm{B}_3 + \bm{B}_4)$, we obtain
\begin{align*}
\bm{B}_2 (\bm{B}_3+\bm{B}_4)^{-1} = [\bm{I}_p - \hat{\bm{G}}^{-1} (\bm{X}^\top_1 \dot{\bm{W}}_1 \bm{X}_1 + \hat{\bm{G}}^{-1})^{-1},\ldots, \bm{I}_p - \hat{\bm{G}}^{-1} (\bm{X}^\top_m \dot{\bm{W}}_m \bm{X}_m + \hat{\bm{G}}^{-1})^{-1} ].
\end{align*}
Then since $\bm{Z}_i = \bm{X}_i$ for all $i$, we can show that
\begin{align*}
\bm{B}_2 (\bm{B}_3+\bm{B}_4)^{-1} \bm{B}_2^\top &= \sum_{i=1}^m \bm{X}^\top_i \dot{\bm{W}}_i \bm{X}_i (\bm{X}^\top_i \dot{\bm{W}}_i \bm{X}_i + \hat{\bm{G}}^{-1})^{-1} \bm{X}^\top_i \dot{\bm{W}}_i \bm{X}_i  \\
&= \sum_{i=1}^m (\bm{X}^\top_i \dot{\bm{W}}_i \bm{X}_i + \hat{\bm{G}}^{-1} - \hat{\bm{G}}^{-1} ) (\bm{X}^\top_i \dot{\bm{W}}_i \bm{X}_i + \hat{\bm{G}}^{-1})^{-1} \bm{X}^\top_i \dot{\bm{W}}_i \bm{X}_i  \\
&= \sum_{i=1}^m \bm{X}^\top_i \dot{\bm{W}}_i \bm{X}_i - \hat{\bm{G}}^{-1} (\bm{X}^\top_i \dot{\bm{W}}_i \bm{X}_i + \hat{\bm{G}}^{-1})^{-1} \bm{X}^\top_i \dot{\bm{W}}_i \bm{X}_i \\
&= \bm{B}_1 - \sum_{i=1}^m  \hat{\bm{G}}^{-1} (\bm{X}^\top_i \dot{\bm{W}}_i \bm{X}_i + \hat{\bm{G}}^{-1})^{-1} \bm{X}^\top_i \dot{\bm{W}}_i \bm{X}_i.
\end{align*}
It follows that
\begin{align} \label{eq:C}
\bm{C} &= \sum_{i=1}^m  \hat{\bm{G}}^{-1} (\bm{X}^\top_i \dot{\bm{W}}_i \bm{X}_i + \hat{\bm{G}}^{-1})^{-1} \bm{X}^\top_i \dot{\bm{W}}_i \bm{X}_i = \sum_{i=1}^m  \bm{X}^\top_i \dot{\bm{W}}_i \bm{X}_i (\bm{X}^\top_i \dot{\bm{W}}_i \bm{X}_i + \hat{\bm{G}}^{-1})^{-1} \hat{\bm{G}}^{-1},
\end{align}
where the second equality arises from the fact that as a covariance matrix, $\bm{C}$ must be symmetric. We may also write $\bm{C}$ as
\begin{align*}
\sum_{i=1}^m  \hat{\bm{G}}^{-1} (\bm{X}^\top_i \dot{\bm{W}}_i \bm{X}_i + \hat{\bm{G}}^{-1})^{-1} \bm{X}^\top_i \dot{\bm{W}}_i \bm{X}_i &= \sum_{i=1}^m \{ \bm{I}_p - \hat{\bm{G}}^{-1} (\bm{X}^\top_i \dot{\bm{W}}_i \bm{X}_i + \hat{\bm{G}}^{-1})^{-1} \} \hat{\bm{G}}^{-1} \\
&= \hat{\bm{G}}^{-1} \sum_{i=1}^m \{ \bm{I}_p - (\bm{X}^\top_i \dot{\bm{W}}_i \bm{X}_i + \hat{\bm{G}}^{-1})^{-1} \hat{\bm{G}}^{-1}\}. \numberthis \label{eq:C_alt}
\end{align*}
Note that $\bm{C}$ is of order $O_p(m)$ component-wise in probability in both the conditional and unconditional regimes. Using the fact that $\bm{C}^{-1}$ must also be symmetric, we obtain
\begin{align} \label{eq:Cinv}
\bm{C}^{-1} &= \left\{ \sum_{i=1}^m (\bm{X}^\top_i \dot{\bm{W}}_i \bm{X}_i + \hat{\bm{G}}^{-1})^{-1} \bm{X}^\top_i \dot{\bm{W}}_i \bm{X}_i \right\}^{-1}  \hat{\bm{G}} = \hat{\bm{G}} \left\{ \sum_{i=1}^m  \bm{X}^\top_i \dot{\bm{W}}_i \bm{X}_i (\bm{X}^\top_i \dot{\bm{W}}_i \bm{X}_i + \hat{\bm{G}}^{-1})^{-1} \right\}^{-1}
\end{align}
or equivalently
\begin{align*}
\bm{C}^{-1} &= \bm{C}^{-1 \top} = \left[ \sum_{i=1}^m \{ \bm{I}_p - (\bm{X}^\top_i \dot{\bm{W}}_i \bm{X}_i + \hat{\bm{G}}^{-1})^{-1} \hat{\bm{G}}^{-1}\} \right]^{-1} \hat{\bm{G}} \\
&= \left\{ m^{-1} \bm{I}_p  + m^{-1} \sum_{i=1}^m (\bm{X}^\top_i \dot{\bm{W}}_i \bm{X}_i + \hat{\bm{G}}^{-1})^{-1} \hat{\bm{G}}^{-1} \bm{C}^{-1} \right\} \hat{\bm{G}}, \numberthis \label{eq:Cinv_woodbury}
\end{align*}
where the last line is derived from (a special case of) the Woodbury identity, given by $(\bm{Q} - \bm{R})^{-1} = \bm{Q}^{-1} + \bm{Q}^{-1} \bm{R} (\bm{Q} - \bm{R})^{-1}$ for arbitrary matrices $\bm{Q}$ and $\bm{R}$ such that $\bm{Q}$ and $(\bm{Q} - \bm{R})$ are invertible. The first term in \eqref{eq:Cinv_woodbury} is the dominating term, being of order $O(m^{-1})$, while the second term is $O_p(m^{-1} n_L^{-1})$ in both the conditional and unconditional regimes. We will use all the above forms of $\bm{C}$ and $\bm{C}^{-1}$ in subsequent developments. Similarly, we can apply the Woodbury identity to $(\bm{B}_3+\bm{B}_4)^{-1}$ and $(\bm{X}^\top_i \dot{\bm{W}}_i \bm{X}_i + \hat{\bm{G}}^{-1})^{-1}$ to obtain $n_L (\bm{B}_3+ \bm{B}_4)^{-1} = n_L \bm{B}_3^{-1} - n_L \bm{B}_3^{-1} \bm{B}_4 (\bm{B}_3+ \bm{B}_4)^{-1} = O_p(1) + O_p(n_L^{-1})$ and $n_i (\bm{X}^\top_i \dot{\bm{W}}_i \bm{X}_i + \hat{\bm{G}}^{-1})^{-1} = n_i (\bm{X}^\top_i \dot{\bm{W}}_i \bm{X}_i)^{-1} - n_i (\bm{X}^\top_i \dot{\bm{W}}_i \bm{X}_i)^{-1} \hat{\bm{G}}^{-1} (\bm{X}^\top_i \dot{\bm{W}}_i \bm{X}_i + \hat{\bm{G}}^{-1})^{-1} = O_p(1) + O_p(n_i^{-1})$, where the order results hold component-wise. These hold irrespective of whether we are conditioning on the random effects. 

To further simplify expressions, for the rest of this article we will only use order results when representing quantities associated with these smaller order terms. Furthermore, as we want the derivations for the remainder of this section to be applicable to both the conditional and unconditional regime, we will not distinguish between $O(\cdot)$ and $O_p(\cdot)$ in the following developments, and simply use $O_p()$ to represent both as appropriate. The terms we use ``big-O notation" for will have the same order under both the conditional and unconditional regime. To simplify expressions, we will also drop the dependence on $\bmtheta$, unless stated otherwise.

Finally, it is worth emphasising that
\begin{align} \label{eq:key_identity}
    [-\bm{I}_p,\bm{I}_p,\ldots,\bm{I}_p] \left[ \begin{matrix} \dot{\phi}^{-1} \bm{X}^\top (\bm{y} - \dot{\bm{\mu}}) \\ \dot{\phi}^{-1} \bm{Z}^\top (\bm{y} - \dot{\bm{\mu}}) \end{matrix} \right] =  - \bm{S}_1 + \sum_{i=1}^m \bm{S}_{2i} = \bm{S}_1 - \sum_{i=1}^m \bm{S}_{2i} = \bm{0}_p,
\end{align}
due to the $\bm{X}_i = \bm{Z}_i$ assumption. This is a key identity that is critical to the proofs throughout this article.

We now use the expressions above to multiply out $- \{\nabla^2 Q(\dot{\bm{\theta}})\}^{-1} \nabla Q(\dot{\bm{\theta}})$ and obtain expressions for $\hat{\bmbeta} - \dot{\bmbeta}$ and $\hat{\bm{b}} - \dot{\bm{b}}$. From equation (\ref{taylor_fundamental}), the first $p$ components of $\hat{\bm{\theta}} - \dot{\bm{\theta}}$ are
\begin{align*}
\hat{\bm{\beta}} - \dot{\bm{\beta}} &= \left[ \begin{matrix}
\bm{C}^{-1} & -\bm{C}^{-1} \bm{B}_2 (\bm{B}_3+\bm{B}_4)^{-1} 
\end{matrix} \right] \nabla Q + \frac{1}{2} \{ \bm{B}^{-1} \bm{R}(\tilde{\bmtheta})\}_{[1:p]} \\
&= \bm{C}^{-1} \left[ \begin{matrix}
\bm{I}_{p} &  - [\bm{I}_p - \hat{\bm{G}}^{-1} (\bm{X}^\top_1 \dot{\bm{W}}_1 \bm{X}_1 + \hat{\bm{G}}^{-1})^{-1},\ldots, \bm{I}_p - \hat{\bm{G}}^{-1} (\bm{X}^\top_m \dot{\bm{W}}_m \bm{X}_m + \hat{\bm{G}}^{-1})^{-1} ]
\end{matrix} \right] \nabla Q \\ 
&+ \frac{1}{2} \{ \bm{B}^{-1} \bm{R}(\tilde{\bmtheta})\}_{[1:p]} \\
&= \bm{C}^{-1} \left(\bm{S}_1 - \sum_{i=1}^m \bm{S}_{2i} - \sum_{i=1}^m \bm{S}_{3i} \right) + \bm{C}^{-1} \hat{\bm{G}}^{-1} \Bigg\{ \sum_{i=1}^m (\bm{X}_i^\top \dot{\bm{W}}_i \bm{X}_i + \hat{\bm{G}}^{-1})^{-1}  \bm{S}_{2i}  \\
&  +  \sum_{i=1}^m (\bm{X}_i^\top \dot{\bm{W}}_i \bm{X}_i + \hat{\bm{G}}^{-1})^{-1} \bm{S}_{3i} \Bigg\}+ \frac{1}{2} \{ \bm{B}^{-1} \bm{R}(\tilde{\bmtheta})\}_{[1:p]} \\
&= \bm{C}^{-1} \hat{\bm{G}}^{-1} \Bigg\{ \sum_{i=1}^m (\bm{X}_i^\top \dot{\bm{W}}_i \bm{X}_i + \hat{\bm{G}}^{-1})^{-1}  \bm{S}_{2i} - \hat{\bm{G}} \sum_{i=1}^m \bm{S}_{3i}  \\
&  +  \sum_{i=1}^m (\bm{X}_i^\top \dot{\bm{W}}_i \bm{X}_i + \hat{\bm{G}}^{-1})^{-1} \bm{S}_{3i} \Bigg\}+ \frac{1}{2} \{ \bm{B}^{-1} \bm{R}(\tilde{\bmtheta})\}_{[1:p]},
\end{align*}
where the final equality uses equation \eqref{eq:key_identity}. Thus, letting $\bm{V}_1 = \sum_{i=1}^m (\bm{X}_i^\top \dot{\bm{W}}_i \bm{X}_i + \hat{\bm{G}}^{-1})^{-1}  \bm{S}_{2i} - \hat{\bm{G}} \sum_{i=1}^m \bm{S}_{3i} +  \sum_{i=1}^m (\bm{X}_i^\top \dot{\bm{W}}_i \bm{X}_i + \hat{\bm{G}}^{-1})^{-1} \bm{S}_{3i}$ and applying equation \eqref{eq:Cinv_woodbury}, we obtain
\begin{align*}
\hat{\bm{\beta}} - \dot{\bm{\beta}} &= m^{-1} \bm{V}_1 + \frac{1}{2} \{ \bm{B}^{-1} \bm{R}(\tilde{\bmtheta})\}_{[1:p]} + O_p(n_L^{-1}) \times m^{-1} \bm{V}_1.
\end{align*}
Finally, using the Woodbury identity for $(\bm{X}^\top_i \dot{\bm{W}}_i \bm{X}_i + \hat{\bm{G}}^{-1})^{-1}$, we have that $\sum_{i=1}^m (\bm{X}_i^\top \dot{\bm{W}}_i \bm{X}_i + \hat{\bm{G}}^{-1})^{-1}  \bm{S}_{2i} = \sum_{i=1}^m (\bm{X}_i^\top \dot{\bm{W}}_i \bm{X}_i )^{-1}  \bm{S}_{2i} + \sum_{i=1}^m O_p(n_L^{-2})  \bm{S}_{2i}$. Letting $\bm{V}_2 = \sum_{i=1}^m (\bm{X}_i^\top \dot{\bm{W}}_i \bm{X}_i)^{-1}  \bm{S}_{2i} - \hat{\bm{G}} \sum_{i=1}^m \bm{S}_{3i} +  \sum_{i=1}^m (\bm{X}_i^\top \dot{\bm{W}}_i \bm{X}_i + \hat{\bm{G}}^{-1})^{-1} \bm{S}_{3i}$, we obtain
\begin{align*}
\hat{\bm{\beta}} - \dot{\bm{\beta}} &= m^{-1} \bm{V}_2 + \frac{1}{2} \{ \bm{B}^{-1} \bm{R}(\tilde{\bmtheta})\}_{[1:p]} + O_p(n_L^{-1}) \times m^{-1} \bm{V}_1 + m^{-1} \sum_{i=1}^m O_p(n_L^{-2})  \bm{S}_{2i}.
\end{align*}
Next, the last $m p$ components of $\hat{\bm{\theta}} - \dot{\bm{\theta}}$ are
\begin{align*}
\hat{\bm{b}} - \dot{\bm{b}} &= [-(\bm{B}_3+\bm{B}_4)^{-1} \bm{B}_2^\top \bm{C}^{-1} \quad (\bm{B}_3+\bm{B}_4)^{-1} + (\bm{B}_3+\bm{B}_4)^{-1} \bm{B}_2^\top \bm{C}^{-1} \bm{B}_2 (\bm{B}_3+\bm{B}_4)^{-1} ] \nabla Q \\
&+ \frac{1}{2} \{ \bm{B}^{-1} \bm{R}(\tilde{\bmtheta})\}_{[p +1:(m+1)p]} \\
&= [-(\bm{B}_3+\bm{B}_4)^{-1} \bm{B}_2^\top \bm{C}^{-1} \quad (\bm{B}_3+\bm{B}_4)^{-1} \bm{B}_2^\top \bm{C}^{-1} \bm{B}_2 (\bm{B}_3+\bm{B}_4)^{-1} ] \nabla Q \\
&+ [ \bm{0}_{mp \times p} \quad (\bm{B}_3+\bm{B}_4)^{-1}]  \nabla Q + \frac{1}{2} \{ \bm{B}^{-1} \bm{R}(\tilde{\bmtheta})\}_{[p +1:(m+1)p]}\\
&= -(\bm{B}_3+\bm{B}_4)^{-1} \bm{B}_2^\top [ \bm{C}^{-1} \quad - \bm{C}^{-1} \bm{B}_2 (\bm{B}_3+\bm{B}_4)^{-1} ] \nabla Q \\
&+ [ \bm{0}_{mp \times p} \quad (\bm{B}_3+\bm{B}_4)^{-1}]  \nabla Q + \frac{1}{2} \{ \bm{B}^{-1} \bm{R}(\tilde{\bmtheta})\}_{[p +1:(m+1)p]}.
\end{align*}
Notice that we already have an expression for $[ \bm{C}^{-1} \quad - \bm{C}^{-1} \bm{B}_2 (\bm{B}_3+\bm{B}_4)^{-1} ] \nabla Q$ from the fixed effects above. Namely, it is $m^{-1} \bm{V}_1+ O_p(n_L^{-1}) \times m^{-1} \bm{V}_1$. Thus we have
\begin{align*}
\hat{\bm{b}} - \dot{\bm{b}} &= -(\bm{B}_3+\bm{B}_4)^{-1} \bm{B}_2^\top (m^{-1} \bm{V}_1+ O_p(n_L^{-1}) \times m^{-1} \bm{V}_1) \\
&+ (\bm{B}_3+\bm{B}_4)^{-1} \bm{S}_6 + \frac{1}{2} \{ \bm{B}^{-1} \bm{R}(\tilde{\bmtheta})\}_{[p +1:(m+1)p]}.
\end{align*}
Applying the Woodbury identity for $(\bm{B}_3+\bm{B}_4)^{-1}$, we obtain
\begin{align*}
\hat{\bm{b}} - \dot{\bm{b}} &= - \bm{1}_m \otimes (m^{-1} \bm{V}_1+ O_p(n_L^{-1}) \times m^{-1} \bm{V}_1) + O_p(n_L^{-1}) (m^{-1} \bm{V}_1+ O_p(n_L^{-1}) \times m^{-1} \bm{V}_1) \\
&+ \bm{B}_3^{-1} \bm{S}_6 + O_p(n_L^{-2}) \bm{S}_6 + \frac{1}{2} \{ \bm{B}^{-1} \bm{R}(\tilde{\bmtheta})\}_{[p +1:(m+1)p]} \\
&= - \bm{1}_m \otimes m^{-1} \bm{V}_1 + O_p(n_L^{-1}) \times m^{-1} \bm{V}_1+ O_p(n_L^{-2}) \times m^{-1} \bm{V}_1 \\
&+ \bm{B}_3^{-1} \bm{S}_4 + \bm{B}_3^{-1} \bm{S}_5 + O_p(n_L^{-2}) \bm{S}_6 + \frac{1}{2} \{ \bm{B}^{-1} \bm{R}(\tilde{\bmtheta})\}_{[p +1:(m+1)p]}.
\end{align*}
Replacing all the $\bm{V_{\cdot}}$ and $\bm{S_{\cdot}}$ terms in the above with their definitions, we finally obtain
\begin{align*}
\hat{\bm{\beta}} - \dot{\bm{\beta}} &= m^{-1} \sum_{i=1}^m (\bm{X}_i^\top \dot{\bm{W}}_i \bm{X}_i)^{-1}  \dot{\phi}^{-1} \bm{X}_i^\top (\bm{y}_i - \dot{\bm{\mu}}_i )  + m^{-1} \sum_{i=1}^m \dot{\bm{b}}_i \\
&  -  m^{-1} \sum_{i=1}^m (\bm{X}_i^\top \dot{\bm{W}}_i \bm{X}_i + \hat{\bm{G}}^{-1})^{-1} \hat{\bm{G}}^{-1} \dot{\bm{b}}_i + \frac{1}{2} \{ \bm{B}^{-1} \bm{R}(\tilde{\bmtheta})\}_{[1:p]}  \\
& + O_p(n_L^{-1}) \Bigg\{ m^{-1} \sum_{i=1}^m (\bm{X}_i^\top \dot{\bm{W}}_i \bm{X}_i + \hat{\bm{G}}^{-1})^{-1}  \dot{\phi}^{-1} \bm{X}_i^\top (\bm{y}_i - \dot{\bm{\mu}}_i )  + m^{-1} \sum_{i=1}^m \dot{\bm{b}}_i \\
&  -  m^{-1} \sum_{i=1}^m (\bm{X}_i^\top \dot{\bm{W}}_i \bm{X}_i + \hat{\bm{G}}^{-1})^{-1} \hat{\bm{G}}^{-1} \dot{\bm{b}}_i \Bigg\} + m^{-1} \sum_{i=1}^m O_p(n_L^{-2}) \dot{\phi}^{-1} \bm{X}_i^\top (\bm{y}_i - \dot{\bm{\mu}}_i ), \numberthis \label{eq:fixed}
\end{align*}
and
\begin{align*}
\hat{\bm{b}} - \dot{\bm{b}} &= - \bm{1}_m \otimes \Bigg\{  m^{-1} \sum_{i=1}^m (\bm{X}_i^\top \dot{\bm{W}}_i \bm{X}_i + \hat{\bm{G}}^{-1})^{-1}  \dot{\phi}^{-1} \bm{X}_i^\top (\bm{y}_i - \dot{\bm{\mu}}_i )   \\
&  + m^{-1} \sum_{i=1}^m \dot{\bm{b}}_i -  m^{-1} \sum_{i=1}^m (\bm{X}_i^\top \dot{\bm{W}}_i \bm{X}_i + \hat{\bm{G}}^{-1})^{-1} \hat{\bm{G}}^{-1} \dot{\bm{b}}_i \Bigg\} \\
& + \bm{B}_3^{-1} \{\dot{\phi}^{-1} \bm{Z}^\top (\bm{y} - \dot{\bm{\mu}}) \} - \bm{B}_3^{-1} \{ (\bm{I}_m \otimes \hat{\bm{G}}^{-1})\dot{\bm{b}}\} + \frac{1}{2} \{ \bm{B}^{-1} \bm{R}(\tilde{\bmtheta})\}_{[p +1:(m+1)p]} \\
& + O_p(n_L^{-1}) \Bigg\{ m^{-1} \sum_{i=1}^m (\bm{X}_i^\top \dot{\bm{W}}_i \bm{X}_i + \hat{\bm{G}}^{-1})^{-1}  \dot{\phi}^{-1} \bm{X}_i^\top (\bm{y}_i - \dot{\bm{\mu}}_i )   \\
& + m^{-1} \sum_{i=1}^m \dot{\bm{b}}_i -  m^{-1} \sum_{i=1}^m (\bm{X}_i^\top \dot{\bm{W}}_i \bm{X}_i + \hat{\bm{G}}^{-1})^{-1} \hat{\bm{G}}^{-1} \dot{\bm{b}}_i \Bigg\} \\
& + O_p(n_L^{-2}) \Bigg\{  m^{-1} \sum_{i=1}^m (\bm{X}_i^\top \dot{\bm{W}}_i \bm{X}_i + \hat{\bm{G}}^{-1})^{-1}  \dot{\phi}^{-1} \bm{X}_i^\top (\bm{y}_i - \dot{\bm{\mu}}_i )   \\
& +  m^{-1} \sum_{i=1}^m \dot{\bm{b}}_i -  m^{-1} \sum_{i=1}^m (\bm{X}_i^\top \dot{\bm{W}}_i \bm{X}_i + \hat{\bm{G}}^{-1})^{-1} \hat{\bm{G}}^{-1} \dot{\bm{b}}_i \Bigg\} \\
& + O_p(n_L^{-2}) \{\dot{\phi}^{-1} \bm{Z}^\top (\bm{y} - \dot{\bm{\mu}}) - (\bm{I}_m \otimes \hat{\bm{G}}^{-1})\dot{\bm{b}}\}. \numberthis \label{eq:random}
\end{align*}
The expressions for $\hat{\bm{\beta}} - \dot{\bm{\beta}}$ and $\hat{\bm{b}} - \dot{\bm{b}}$ above underlie our proofs. We use these same expressions in both the conditional and unconditional regimes, but the asymptotic behaviours of the terms on the right hand side, and the way we treat them, will differ greatly between the two cases.

As we will show later, the key leading terms for the fixed effects are \\ $m^{-1} \sum_{i=1}^m (\bm{X}_i^\top \dot{\bm{W}}_i \bm{X}_i)^{-1}  \dot{\phi}^{-1} \bm{X}_i^\top (\bm{y}_i - \dot{\bm{\mu}}_i )$ and $m^{-1} \sum_{i=1}^m \dot{\bm{b}}_i$. The key leading terms for the random effects are $- \bm{1}_m \otimes m^{-1} \sum_{i=1}^m \dot{\bm{b}}_i$ and $\bm{B}_3^{-1} \{\dot{\phi}^{-1} \bm{Z}^\top (\bm{y} - \dot{\bm{\mu}}) \}$. When conditioning on the random effects $\dot{\bm{b}}$, we have $m^{-1} \sum_{i=1}^m \dot{\bm{b}}_i = O(1)$, while in the unconditional regime the same quantity is of order $O_p(m^{-1/2})$ in probability. In both the conditional and unconditional regimes, we have that $m^{-1} \sum_{i=1}^m (\bm{X}_i^\top \dot{\bm{W}}_i \bm{X}_i)^{-1}  \dot{\phi}^{-1} \bm{X}_i^\top (\bm{y}_i - \dot{\bm{\mu}}_i )$ is of order $O_p(N^{-1/2})$ component-wise, while the quantity $\bm{B}_3^{-1} \{\dot{\phi}^{-1} \bm{Z}^\top (\bm{y} - \dot{\bm{\mu}}) \}$ is of order $O_p(n_L^{-1/2})$ component-wise.

\subsection{Proof of Theorem 1}
The dominating terms on the right hand sides of equations \eqref{eq:fixed} and \eqref{eq:random} are $m^{-1} \sum_{i=1}^m \dot{\bm{b}}_i$ and $\bm{1}_m \otimes m^{-1} \sum_{i=1}^m \dot{\bm{b}}_i$ for the fixed and random effects, respectively. Conditional on the random effects $\dot{\bm{b}}_i$, these dominating terms are deterministic and of order $O(1)$. Thus we treat them as bias terms and move them to the left hand side. Next, by Conditions (C1)-(C2), $\bm{B}_3^{-1}$ is a component-wise $O(n_L^{-1})$ block-diagonal matrix, while we also have $\bm{B}_2^\top = O(n_U)$, $\bm{X}_i^\top \dot{\bm{W}}_i \bm{X}_i^\top = O(n_i)$, and $\bm{C}^{-1} = O(m^{-1})$ component-wise. Since $ E \{ \bm{Z}^\top (\bm{y} - \dot{\bm{\mu}} ) |\dot{\bm{b}} \} = \bm{0}_{mp}$ and $\text{Var} \{ \bm{Z}^\top (\bm{y} - \dot{\bm{\mu}}) |\dot{\bm{b}} \} = \bm{Z}^\top \dot{\bm{W}} \bm{Z}$, we obtain $\dot{\phi}^{-1} \bm{D}_r^{-1} \bm{Z}^\top (\bm{y} - \dot{\bm{\mu}}) = O_p(1)$ using Chebyshev's inequality and the conditional independence.

Multiplying both sides of \eqref{eq:fixed} and \eqref{eq:random} by $N^{1/2}$ and $\bm{D}_r$ respectively, and applying the order results for the remainder term in Section \ref{sec:cond_remainder}, we obtain
\begin{align*}
N^{1/2} \left(\hat{\bm{\beta}} - \dot{\bm{\beta}} -  m^{-1} \sum_{i=1}^m \dot{\bm{b}}_i \right) &=  m^{-1/2} \sum_{i=1}^m n^{1/2} n_i^{-1/2} (n_i^{-1} \bm{X}_i^\top \dot{\bm{W}}_i \bm{X}_i)^{-1}  n_i^{-1/2} \dot{\phi}^{-1} \bm{X}_i^\top (\bm{y}_i - \dot{\bm{\mu}}_i ) \\
& + O_p(m^{1/2} n_L^{-1/2}),
\end{align*}
and
\begin{align*}
\bm{D}_r \left( \hat{\bm{b}} - \dot{\bm{b}} + \bm{1}_m \otimes m^{-1} \sum_{i=1}^m \dot{\bm{b}}_i \right) &= \bm{D}_r \bm{B}_3^{-1} \bm{D}_r \bm{D}_r^{-1} \{  \dot{\phi}^{-1} \bm{Z}^\top (\bm{y} - \dot{\bm{\mu}}) \} + O_p(n_L^{-1/2}).
\end{align*}

Recalling that $\bm{X}_i = \bm{Z}_i$, to prove Theorem \ref{thm:conditionalnormality} we will show a Lindeberg condition for 
\begin{align*}
   \bm{A} \left[ \begin{matrix}
m^{-1/2} \sum_{i=1}^m n^{1/2} n_i^{-1/2} (n_i^{-1} \bm{X}_i^\top \dot{\bm{W}}_i \bm{X}_i)^{-1}  n_i^{-1/2} \dot{\phi}^{-1} \bm{X}_i^\top (\bm{y}_i - \dot{\bm{\mu}}_i ) \\
    (n_1^{-1} \bm{X}_1^\top \dot{\bm{W}}_1 \bm{X}_1)^{-1} \{n_1^{-1/2} \dot{\phi}^{-1} \bm{X}_1^\top (\bm{y}_1 - \dot{\bm{\mu}}_1 ) \} \\
    \vdots \\
    (n_m^{-1} \bm{X}_m^\top \dot{\bm{W}}_m \bm{X}_m)^{-1} \{n_m^{-1/2} \dot{\phi}^{-1} \bm{X}_m^\top (\bm{y}_m - \dot{\bm{\mu}}_m ) \}
    \end{matrix} \right] =: S,
\end{align*}
and thus apply the Lindeberg-Feller central limit theorem, from which the result follows from Slutsky's theorem. 


To prove the condition, first define $\bm{U} = [\bm{Z} \bm{B}_3^{-1} (\bm{1}_m \otimes \bm{I}_p) , \bm{Z} \bm{B}_3^{-1} ]$, and $\bm{U}_k$ as the $k$th row of $\bm{U}$, noting it only has $2p$ non-zero components. Then we can write $S = \sum_{k=1}^N \bm{A} \bm{D} \bm{U}_k \dot{\phi}^{-1} \{ y_k - \mu_k(\dot{\bm{\theta}}) \} \overset{\Delta}{=} \sum_{k=1}^N \bm{\xi}_k$, where $y_k$ is the $k$th component in $(y_{11},y_{12},\ldots,y_{1n_1},y_{21},\ldots,y_{mn_m})^\top$, and similarly for $\mu_k(\dot{\bm{\theta}})$.

Conditional on $\dot{\bm{b}}$, the quantities $\{\bm{\xi}_k\}_{k=1}^N$ are independent $q$-vectors with expectation zero and covariance $\text{Var} (\bm{\xi}_k | \dot{\bm{b}} ) = \bm{A} \bm{D} \bm{U}_k W_k \bm{U}_k^\top \bm{D} \bm{A}^\top$, where $W_k$ is the $k$th diagonal component in $\dot{\bm{W}}$. Therefore, we have that
\begin{align*}
& \sum_{k=1}^N \text{Var} (\bm{\xi}_k | \dot{\bm{b}} ) = \sum_{k=1}^N \bm{A}  \bm{D} \bm{U}_k W_k \bm{U}_k^\top \bm{D} \bm{A}^\top \\
&= \bm{A} \left[ \begin{matrix}
\frac{1}{m} \sum_{i=1}^m \frac{n}{n_i} \left(\frac{\bm{X}^\top_i \dot{\bm{W}}_i \bm{X}_i}{n_i} \right)^{-1} & \frac{1}{\sqrt{m}} \sqrt{\frac{n}{n_1}} \left(\frac{\bm{X}^\top_1 \dot{\bm{W}}_1 \bm{X}_1}{n_1} \right)^{-1} & \cdots & \frac{1}{\sqrt{m}} \sqrt{\frac{n}{n_m}} \left(\frac{\bm{X}^\top_m \dot{\bm{W}}_m \bm{X}_m}{n_m} \right)^{-1} \\
\frac{1}{\sqrt{m}} \sqrt{\frac{n}{n_1}} \left(\frac{\bm{X}^\top_1 \dot{\bm{W}}_1 \bm{X}_1}{n_1} \right)^{-1} & \left(\frac{\bm{X}^\top_1 \dot{\bm{W}}_1 \bm{X}_1}{n_1} \right)^{-1} & \bm{0} & \bm{0} \\
\vdots & \bm{0} & \ddots & \bm{0} \\
\frac{1}{\sqrt{m}} \sqrt{\frac{n}{n_m}} \left(\frac{\bm{X}^\top_m \dot{\bm{W}}_m \bm{X}_m}{n_m} \right)^{-1} & \bm{0} & \bm{0} & \left(\frac{\bm{X}^\top_m \dot{\bm{W}}_m \bm{X}_m}{n_m} \right)^{-1}
\end{matrix} \right] \bm{A}^\top.
\end{align*}
Hence using the finite selection property of $\bm{A}$, and the fact that ${m}^{-1/2} n^{1/2} n_i^{-1/2} \left(n_i^{-1} \bm{X}^\top_i \dot{\bm{W}}_i \bm{X}_i \right)^{-1} = o(1)$ component-wise, we obtain
\begin{align*}
& \underset{m,n_L \rightarrow \infty}{\lim} \sum_{k=1}^N \text{Cov} (\bm{\xi}_k | \dot{\bm{b}} ) \\
&= \underset{m,n_L \rightarrow \infty}{\lim} \bm{A} \, \text{bdiag} \Bigg\{ \frac{1}{m}  \sum_{i=1}^m \frac{n}{n_i} \left(\frac{\bm{X}^\top_i \dot{\bm{W}}_i \bm{X}_i}{n_i} \right)^{-1} , \left(\frac{\bm{X}^\top_1 \dot{\bm{W}}_1 \bm{X}_1}{n_1} \right)^{-1}, \ldots , \left(\frac{\bm{X}^\top_m \dot{\bm{W}}_m \bm{X}_m}{n_m} \right)^{-1} \Bigg\} \bm{A}^\top \\
&= \bm{\Omega}.
\end{align*}

Next, by the Cauchy-Schwarz inequality, we have
\begin{align*}
E\{\| \bm{\xi}_k \|^2 I(\| \bm{\xi}_k \| > \epsilon) |\dot{\bm{b}}\} \leq E(\| \bm{\xi}_k \|^4 |\dot{\bm{b}})^{1/2} P(\| \bm{\xi}_k \| > \epsilon |\dot{\bm{b}})^{1/2}.
\end{align*}
Finally, we make a note about the form of Cov$[\bm{D} \bm{U}_k \{y_k - \mu_k(\dot{\bm{\theta}})\}]$. Without loss of generality, suppose $k = 1$. Then
\begin{align*}
    & \text{Cov}[\bm{D}\bm{U}_1 \{y_1 - \mu_1(\dot{\bm{\theta}})\}] = \\
    & \left[ \begin{matrix}
n (mn_1^2)^{-1} \bm{H}_1 \bm{x}_{11} W_1 \bm{x}_{11}^\top \bm{H}_1^\top & n_1^{-1} m^{-1/2} (n n_1^{-1})^{1/2} \bm{H}_1 \bm{x}_{11} W_1 \bm{x}_{11}^\top \bm{H}_1^\top & \bm{0} \\
n_1^{-1} m^{-1/2} (n n_1^{-1})^{1/2} \bm{H}_1 \bm{x}_{11} W_1 \bm{x}_{11}^\top \bm{H}_1^\top & n_1^{-1} \bm{H}_1 \bm{x}_{11} W_1 \bm{x}_{11}^\top \bm{H}_1^\top & \bm{0} \\
\bm{0} & \bm{0} & \bm{0} \\
\end{matrix} \right]. \numberthis \label{eq:lindeburg}
\end{align*}
Again without loss of generality, consider the case $\bm{A} = [\bm{I}_{2p} , \bm{0}_{(p+p) \times (m-1)p} ]$. Then by equation \eqref{eq:lindeburg} and Chebyshev's inequality, when $k \in \{ 1, 2, \ldots, n_1 \}$ we have that $P(\| \bm{\xi}_k \| > \epsilon | \dot{\bm{b}}) \leq \text{tr} \{\text{Cov} (\bm{\xi}_k | \dot{\bm{b}})\}/\epsilon^2 = O(n_1^{-1})$. Thus, given $\dot{\bm{b}}$, we obtain $\| \bm{\xi}_k \| = O_p(n_1^{-1/2})$ and $E(\| \bm{\xi}_k \|^4 | \dot{\bm{b}}) = O(n_1^{-2})$ by Conditions (C1)-(C3) and the properties of the exponential family. However when $k > n_1$, by equation \eqref{eq:lindeburg} and Chebyshev's inequality, we have that $P(\| \bm{\xi}_k \| > \epsilon | \dot{\bm{b}}) \leq \text{tr} \{\text{Cov} (\bm{\xi}_k | \dot{\bm{b}})\}/\epsilon^2 = O(N^{-1})$ since $n (m n_1^2)^{-1} = O(N^{-1})$. Thus given $\dot{\bm{b}}$, it holds that $\| \bm{\xi}_k \| = O_p(N^{-1/2})$ and $E(\| \bm{\xi}_k \|^4 | \dot{\bm{b}}) = O(N^{-2})$. Therefore
\begin{align*}
\sum_{k=1}^N E\{\| \bm{\xi}_k \|^2 I(\| \bm{\xi}_k \| > \epsilon ) | \dot{\bm{b}} \} &\leq \sum_{k=1}^N E(\| \bm{\xi}_k \|^4 |\dot{\bm{b}})^{1/2} P(\| \bm{\xi}_k \| > \epsilon |\dot{\bm{b}})^{1/2} \\
&= \sum_{k=1}^{n_1} E(\| \bm{\xi}_k \|^4 |\dot{\bm{b}})^{1/2} P(\| \bm{\xi}_k \| > \epsilon |\dot{\bm{b}})^{1/2} \\
&+ \sum_{k=n_1+1}^N E(\| \bm{\xi}_k \|^4 |\dot{\bm{b}})^{1/2} P(\| \bm{\xi}_k \| > \epsilon |\dot{\bm{b}})^{1/2} \\
&\leq n_1 \, \underset{1 \leq k \leq n_1}{\max} \{ E(\| \bm{\xi}_k \|^4 |\dot{\bm{b}})^{1/2} P(\| \bm{\xi}_k \| > \epsilon |\dot{\bm{b}})^{1/2} \} \\
&+ (N-n_1) \, \underset{k>n_1}{\sup} \{ E(\| \bm{\xi}_k \|^4 |\dot{\bm{b}})^{1/2} P(\| \bm{\xi}_k \| > \epsilon |\dot{\bm{b}})^{1/2} \} \\
&= n_1 \times O(n_1^{-3/2}) + (N-n_1) \times O(N^{-3/2}) \\
&= O(n_1^{-1/2}) + O(N^{-1/2}) \\
&= o(1).
\end{align*}

The required result follows by Conditions (C1)-(C2) and the Lindeberg-Feller Central Limit Theorem. Furthermore, the general result holds straightforwardly by replacing $n_1$ with $O(n_L)$ in the above argument, noting that any row of $\bm{A}$ can only select a fixed number of clusters.

\subsection{Proof of Equation (4)} \label{sec:poisson_pure}

For the Poisson pure random intercept model, we have $\bm{B} = \text{diag}(ne^{\dot{b}_1} + \hat{\sigma}_b^{-2},\ldots, ne^{\dot{b}_m} + \hat{\sigma}_b^{-2})$ and $\bm{R}(\tilde{\bm{\theta}}) = \{ne^{\tilde{b}_1} (\hat{b}_1 - \dot{b}_1)^2,\ldots, ne^{\tilde{b}_m} (\hat{b}_m - \dot{b}_m)^2\}^\top$. Next, suppose that $\bm{A}$ picks out the first random intercept, i.e., $\bm{A} = [1, \bm{0}_{m-1}^\top]$. Then we have
\begin{align*}
     {n}^{1/2} (\hat{b}_1 - \dot{b}_1) &= {n}^{1/2} \bm{A B}^{-1}  \nabla Q(\dot{\bm{\theta}}) + \frac{1}{2} {n}^{1/2}  \bm{A} \bm{B}^{-1}  \bm{R}(\tilde{\bm{\theta}})  \\
    &= n^{-1/2} \left\{ \left( \sum_{j=1}^n y_{1j}  - e^{\dot{b}_1} \right) - \dot{b}_1/\hat{\sigma}_b^2 \right\} \Bigg/ \Big\{ e^{\dot{b}_1} + 1/(\hat{\sigma}_b^2 n) \Big\} \\
    &- \frac{1}{2} \left\{n^{1/2} e^{\tilde{b}_1} (\hat{b}_1 - \dot{b}_1)^2 \right\} \Bigg/ \Big\{ e^{\dot{b}_1} + 1/(\hat{\sigma}_b^2 n) \Big\} \\
    &= \Bigg[ n^{-1/2} \left\{ \left( \sum_{j=1}^n y_{1j}  - e^{\dot{b}_1} \right) - \dot{b}_1/\hat{\sigma}_b^2 \right\} \Bigg/ \Big\{ e^{\dot{b}_1} + 1/(\hat{\sigma}_b^2 n) \Big\} \Bigg] \Bigg/\\
    & \left[ 1+ \left\{ \frac{1}{2} e^{\tilde{b}_1} (\hat{b}_1 - \dot{b}_1) \right\} \Bigg/ \Big\{ e^{\dot{b}_1} + 1/(\hat{\sigma}_b^2 n) \Big\} \right] \\
    &= n^{-1/2} \sum_{j=1}^n (y_{1j} e^{-\dot{b}_1} -1) + o_p(1),
\end{align*}
where $\tilde{b}_1$ lies between $\hat{b}_1$ and $\dot{b}_1$, and for the last line we have used the fact that $\hat{b}_1 - \dot{b}_1 = o_p(1)$.

Now, $\{ y_{1j} e^{-\dot{b}_1} -1 \}_{j=1}^m$ is an exchangeable collection of uncorrelated random variables with mean zero and finite non-zero variance. Furthermore, we have for $k \neq l$ 
\begin{align*}
\text{Cov}\{(y_{1k} e^{-\dot{b}_1} - 1)^2,(y_{1l} e^{-\dot{b}_1} -1)^2\} &= E [\text{Cov}\{(y_{1k} e^{-\dot{b}_1} - 1)^2,(y_{1l} e^{-\dot{b}_1} -1)^2 | \dot{b}_1\}] \\
&\quad + \text{Cov}[ E \{ (y_{1k} e^{-\dot{b}_1} - 1)^2 | \dot{b}_1\} , E \{ (y_{1l} e^{-\dot{b}_1} -1)^2 | \dot{b}_1\}]  \\
&= 0 + \text{Cov}(e^{-\dot{b}_1},e^{-\dot{b}_1}) \\
&= e^{\dot{\sigma}_b^2} (e^{\dot{\sigma}_b^2} -1) \neq 0.
\end{align*}
Thus by the Central Limit Theorem for exchangeable random variables \citep{blum1958central}, it holds that $n^{-1/2} \sum_{j=1}^n (y_{1j} e^{-\dot{b}_1} -1) \overset{D}{\not\rightarrow} N(0, e^{\dot{\sigma}_b^2}) $. Since we know $\text{Var}\{ n^{-1/2} \sum_{j=1}^n (y_{1j} e^{-\dot{b}_1} -1)\} = e^{\dot{\sigma}_b^2/2}$ and also that $n^{-1/2} \sum_{j=1}^n (y_{1j} e^{-\dot{b}_1} -1) = O_p(1) $ by Chebyshev's inequality, there is no other normalization possible for an asymptotic normality result to hold. 

Finally, we also have
\begin{align*}
     {n}^{1/2} (\hat{b}_1 - \dot{b}_1) &= n^{-1/2} \sum_{j=1}^n (y_{1j} e^{-\dot{b}_1} -1) + O_p(n^{-1/2}) \\
     \implies \hat{b}_1 &= \dot{b}_1 + n^{-1} \sum_{j=1}^n (y_{1j} e^{-\dot{b}_1} -1) + O_p(n^{-1}) \\
     &= \dot{b}_1 + o_p(1), \quad \text{by the Weak Law of Large Numbers}.
\end{align*}

\subsection{Proof of Theorem 2}

We begin by developing two key equations, \eqref{eq:fixed_basic.uncond} and \eqref{eq:rand_basic.uncond}, that will be used throughout the unconditional regime. These are derived from equations \eqref{eq:fixed} and \eqref{eq:random} and are used in the proofs of Theorems \ref{thm:unconditional_fixef}-\ref{thm:unconditional_linpred} as well as Corollary 1. Under Conditions (C1)-(C2), the following order results are used: $\bm{B}_3^{-1}$ is a component-wise $O_p(n_L^{-1})$ block-diagonal matrix, $\bm{B}_2 = O_p(n_U)$ component-wise, $\bm{X}_i^\top \dot{\bm{W}}_i \bm{X}_i^\top = O_p(n_i)$ component-wise, and $\bm{C}^{-1} = O_p(m^{-1})$ component-wise. Also, by the conditional independence, we have
\begin{align*}
E\{ \bm{Z}^\top (\bm{y} - \dot{\bm{\mu}}) \} = E[ E \{ \bm{Z}^\top (\bm{y} - \dot{\bm{\mu}}) |\dot{\bm{b}} \} ]= \bm{0}_{mp},
\end{align*}
\begin{align*}
\text{Var}\{\bm{Z}^\top (\bm{y} - \dot{\bm{\mu}})\} = E[ \text{Var}\{\bm{Z}^\top (\bm{y} - \dot{\bm{\mu}}) |\dot{\bm{b}} \} ]+ \text{Var}[ E\{\bm{Z}^\top (\bm{y} - \dot{\bm{\mu}}) |\dot{\bm{b}} \}] = E(\bm{Z}^\top \dot{\bm{W}} \bm{Z}),
\end{align*}
so that $\dot{\phi}^{-1} \bm{D}_r^{-1} \bm{Z}^\top (\bm{y} - \dot{\bm{\mu}}) = O_p(1)$ using Chebyshev's inequality. Therefore we have the key equations
\begin{align} \label{eq:fixed_basic.uncond}
\hat{\bm{\beta}} - \dot{\bm{\beta}} &= m^{-1} \sum_{i=1}^m \dot{\bm{b}}_i + O_p(N^{-1/2}) + O_p(n_L^{-1}) + \frac{1}{2} \{ \bm{B}^{-1} \bm{R}(\tilde{\bmtheta})\}_{[1:p]}
\end{align}
and
\begin{align*}
\hat{\bm{b}} - \dot{\bm{b}} &= - \bm{1}_m \otimes m^{-1} \sum_{i=1}^m \dot{\bm{b}}_i + \bm{B}_3^{-1} \{ \dot{\phi}^{-1} \bm{Z}^\top (\bm{y} - \dot{\bm{\mu}}) \} \\ 
&+ O_p(N^{-1/2}) + O_p(n_L^{-1}) + \frac{1}{2} \{ \bm{B}^{-1} \bm{R}(\tilde{\bmtheta})\}_{[p +1:(m+1)p]}. \numberthis \label{eq:rand_basic.uncond}
\end{align*}

By equation (\ref{eq:fixed_basic.uncond}), we have
\begin{align*}
m^{1/2} (\hat{\bm{\beta}} - \dot{\bm{\beta}}) &= m^{-1/2} \sum_{i=1}^m \dot{\bm{b}}_i + O_p(n_L^{-1/2}) + O_p(m^{1/2} n_L^{-1}) + \frac{1}{2} m^{1/2} \{ \bm{B}^{-1} \bm{R}(\tilde{\bmtheta})\}_{[1:p]}.
\end{align*}
Next, we consider two separate scenarios. First, suppose that $m n_U^{-1} \rightarrow \infty$. Then by the order results for the remainder term in Section \ref{sec:uncond_remainder}, the first $p$ components of $\bm{D}^* \bm{B}^{-1} \bm{R}(\tilde{\bm{\theta}})$ are of order $O_p(m^{1/2} n_L^{-1})$, and so the first $p$ components of $\bm{D}^* (\hat{\bm{\theta}} - \dot{\bm{\theta}})$ can be shown to be
\begin{align*}
{m}^{1/2} (\hat{\bm{\beta}} - \dot{\bm{\beta}}) &= m^{-1/2} \sum_{i=1}^m \dot{\bm{b}}_i + o_p(1).
\end{align*}
The required result then follows from the independence of the random effects and the normal assumption on the $\dot{\bm{b}}_i$; note the $m n_L^{-2} \rightarrow 0$ assumption is required for the remainder term to be smaller order than the linear term. 

On the other hand, when $m n_L^{-1} \rightarrow 0$, the only difference from the $m n_L^{-1} \rightarrow \infty$ case is that the first $p$ components of $\bm{D}^+ \bm{B}^{-1}  \bm{R}(\tilde{\bm{\theta}})$ are now of order $O_p(m^{-1/2})$ due to the different convergence rate of the prediction gap. The result however follows along similar lines as above.

\subsection{Proof of Theorem 3}

Again we consider two different scenarios. First, suppose $m n_L^{-1} \rightarrow \infty$. Then from equation (\ref{eq:rand_basic.uncond}) and the order results for the remainder term in Section \ref{sec:uncond_remainder}, we have that
\begin{align*}
\bm{D}_r (\hat{\bm{b}} - \dot{\bm{b}}) &= O_p(n_U^{1/2} m^{-1/2}) + O_p(1) + O_p(n_L^{-1/2}).
\end{align*}
Based on the above, we obtain 
$\bm{D}_r \hat{\bm{b}}= \bm{D}_r \dot{\bm{b}} + O_p(1) $, and thus $\hat{\bm{b}}= \dot{\bm{b}} + O_p(n_L^{-1/2})$. The required result follows by multiplying both sides by $\bm{A}_r$.

On the other hand, suppose now $m n_L^{-1} \rightarrow 0$. Then a normalization by ${m}^{1/2}$ is needed instead, and the third derivative term is consequently of order $O_p(m^{-1/2})$ in probability. We thus obtain
\begin{align*}
{m}^{1/2} (\hat{\bm{b}} - \dot{\bm{b}}) &= O_p(1) + O_p(m^{1/2} n_L^{-1/2}) + O_p(m^{-1/2}),
\end{align*}
and the result follows.

As an side remark, note from the above proof when $m n_L^{-1} \rightarrow 0$, it holds that $\| \hat{\bm{b}} - \dot{\bm{b}} \|_2 = O_p(1)$, where $\|\cdot\|_2$ denotes the $l_2$-norm. But if $m n_U^{-1} \rightarrow \infty$ then we instead obtain $\| \hat{\bm{b}} - \dot{\bm{b}} \|_2 = O_p(m^{1/2} n_U^{-1/2})$. This implies that, under the unconditional regime, a consistency result based on the $l_2$-norm cannot hold for the entire vector of random effects when there is a partnered fixed effect. If there is no partnered fixed effect though, consistency of the entire vector is sometimes possible. For example, in the Poisson counterexample, we demonstrate in 
the Appendix that $\| \hat{\bm{b}} - \dot{\bm{b}} \|_2 = O_p(m^{1/2} n^{-1/2}) = o_p(1)$ when $m n^{-1} \rightarrow 0$.

\subsection{Proof of Theorem 4 and Corollary 1}

We will prove each of the three parts of the theorem separately. The proof of part (a) also proves Corollary 1.

\underline{Part (a):} When $m n_U^{-1} \rightarrow \infty$, we have from equation (\ref{eq:rand_basic.uncond}) and the order results for the remainder term in Section \ref{sec:uncond_remainder} that
\begin{align*}
\bm{D}_r (\hat{\bm{b}} - \dot{\bm{b}}) &= \bm{D}_r \bm{B}_3^{-1} \bm{D}_r \bm{D}_r^{-1} \dot{\phi}^{-1} \bm{Z}^\top (\bm{y} - \dot{\bm{\mu}})  + o_p(1).
\end{align*}
This is identical to the proof of Theorem \ref{thm:unconditional_predictors}. Next, without loss of generality, suppose $\bm{A}_r$ selects the first cluster only. Then we have 
\begin{align*}
{n}_1^{1/2} (\hat{\bm{b}}_1 - \dot{\bm{b}}_1) &=   (n_1^{-1} \bm{X}_1^\top \dot{\bm{W}}_1 \bm{X}_1)^{-1} n_1^{-1/2} \{\dot{\phi}^{-1} \bm{X}_1^\top (\bm{y}_{1} - \dot{\bm{\mu}}_1 ) \} + o_p(1) \\
&\overset{\Delta}{=} \bm{P}_{n_1} + o_p(1).
\end{align*}

We wish to study the distribution of $\bm{P}_{n_1}$ as $m,n_L \rightarrow \infty$. By definition,
\begin{align*}
    \underset{{m,n_L \rightarrow \infty}}{\lim} F_{\bm{P}_{n_1}} (\bm{x}) = \underset{{m,n_L \rightarrow \infty}}{\lim} \int  F_{\bm{P}_{n_1} |\dot{\bm{b}}_1} (\bm{x}) f(\dot{\bm{b}}_1) d\dot{\bm{b}}_1.
\end{align*}
Since $F_{\bm{P}_{n_1} |\dot{\bm{b}}_1} (\bm{x})$ is a cdf, then $F_{\bm{P}_{n_1} |\dot{\bm{b}}_1} (\bm{x}) f(\dot{\bm{b}}_1)$ is bounded by $f(\dot{\bm{b}}_1)$. Hence applying $\int f(\dot{\bm{b}}_1) d\dot{\bm{b}}_1 = 1$ and the dominated convergence theorem, we obtain
\begin{align*}
    \underset{{m,n_L \rightarrow \infty}}{\lim} F_{\bm{P}_{n_1}} (\bm{x}) &= \int \underset{{m,n_L \rightarrow \infty}}{\lim} F_{\bm{P}_{n_1} |\dot{\bm{b}}_1} (\bm{x}) f(\dot{\bm{b}}_1) d\dot{\bm{b}}_1 = \int \Psi_{\bm{P}_{n_1} | \dot{\bm{b}}_1} (\bm{x}) f(\dot{\bm{b}}_1) d\dot{\bm{b}}_1,
\end{align*}
where $ \Psi_{\bm{P}_{n_1} | \dot{\bm{b}}_1}(\cdot) $ is the cdf associated with $N(\bm{0}, \bm{K}_1 )$, a result which follows from conditional independence and the Lindeberg-Feller Central Limit Theorem used in Theorem \ref{thm:conditionalnormality}. The general result follows by noting that the same argument can be applied to any finite subset of the random effects. Note also that the result holds regardless of the true distribution of $\dot{\bm{b}}_i$.

\underline{Part (b):} When $m n_i^{-1} \rightarrow \gamma_i \in (0,\infty)$, we have from \eqref{eq:rand_basic.uncond} and the order results for the remainder term in Section \ref{sec:uncond_remainder} that
\begin{align*}
{n}_i^{1/2} (\hat{\bm{b}}_i - \dot{\bm{b}}_i) &=  (n_i^{-1} \bm{X}_i^\top \dot{\bm{W}}_i \bm{X}_i)^{-1} n_i^{-1/2} \{\dot{\phi}^{-1} \bm{X}_i^\top (\bm{y}_{i} - \dot{\bm{\mu}}_i ) \} - (\gamma_i m)^{-1/2} \sum_{i=1}^m \dot{\bm{b}}_i + O_p(n_L^{-1/2}),
\end{align*}
from the same development as in the proof of Part (a). Letting \\ $\bm{E}_1 = (n_i^{-1} \bm{X}_i^\top \dot{\bm{W}}_i \bm{X}_i )^{-1} n_i^{-1/2} \dot{\phi}^{-1} \bm{X}_i^\top (\bm{y}_i - \dot{\bm{\mu}}_i)$ and $\bm{E}_2 = m^{-1/2} \sum_{i=1}^m \dot{\bm{b}}_i$, then since $\bm{E}_1$ and $\bm{E}_2$ are independent given $\dot{\bm{b}}_i$, we obtain for any $i$, 
\begin{align*}
    \lim_{m,n_L \rightarrow \infty} F_{\bm{E}_1,\bm{E}_2}(\bm{x},\bm{y}) &= \lim_{m,n_L \rightarrow \infty} \int F_{\bm{E}_1,\bm{E}_2|\dot{\bm{b}}_i}(\bm{x},\bm{y}) f(\dot{\bm{b}}_i) d\dot{\bm{b}}_i \\
    &= \lim_{m,n_L \rightarrow \infty} \int F_{\bm{E}_1|\dot{\bm{b}}_i}(\bm{x}) F_{\bm{E}_2|\dot{\bm{b}}_i}(\bm{y}) f(\dot{\bm{b}}_i) d\dot{\bm{b}}_i \\
    &= \int \lim_{m,n_L \rightarrow \infty} F_{\bm{E}_1|\dot{\bm{b}}_i}(\bm{x}) F_{\bm{E}_2|\dot{\bm{b}}_i}(\bm{y}) f(\dot{\bm{b}}_i) d\dot{\bm{b}}_i \\
    &= \int \lim_{n_L \rightarrow \infty} F_{\bm{E}_1|\dot{\bm{b}}_i}(\bm{x}) \lim_{m \rightarrow \infty} F_{\bm{E}_2|\dot{\bm{b}}_i}(\bm{y}) f(\dot{\bm{b}}_i) d\dot{\bm{b}}_i \\
    &= \Psi_{\bm{E}_2}(\bm{y}) \int \lim_{n_L \rightarrow \infty} F_{\bm{E}_1|\dot{\bm{b}}_i}(\bm{x}) f(\dot{\bm{b}}_i) d\dot{\bm{b}}_i,
\end{align*}
where $\Psi_{\bm{E}_2}(\cdot)$ is the cdf of $N(\bm{0},\dot{\bm{G}})$. The third line follows from the Dominated Convergence Theorem since $F_{\bm{E}_1|\dot{\bm{b}}_i}(\bm{x})$ and $F_{\bm{E}_2|\dot{\bm{b}}_i}(\bm{y})$ are cdfs and $\int f(\dot{\bm{b}}_i) d\dot{\bm{b}}_i = 1$. Thus $\bm{E}_1$ and $\bm{E}_2$ are asymptotically independent. The result follows from this asymptotic independence.

\underline{Part (c):} When $m n_L^{-1} \rightarrow 0$, we have from \eqref{eq:rand_basic.uncond} and the order results for the remainder term in Section \ref{sec:uncond_remainder} that
\begin{align*}
{m}^{1/2} (\hat{\bm{b}} - \dot{\bm{b}}) &= - \bm{1}_m \otimes \bm{I}_{p} m^{-1/2} \sum_{i=1}^m \dot{\bm{b}}_i + o_p(1).
\end{align*}
The result then follows immediately from the normality assumption on $\dot{\bm{b}}_i$. 


\subsection{Proof of Theorem 5}

Given $m n_L^{-2} \rightarrow 0$ and $m n_U^{-1/2} \rightarrow \infty$, by summing equations \eqref{eq:fixed_basic.uncond} and \eqref{eq:rand_basic.uncond} we see that the $ m^{-1} \sum_{i=1}^m \dot{\bm{b}}_i$ terms cancel. Therefore, we are left with 
\begin{align*}
{n}_i^{1/2} (\hat{\bm{\beta}} + \hat{\bm{b}}_i - \dot{\bm{\beta}} - \dot{\bm{b}}_i) &= n_i (\bm{X}_i^\top \dot{\bm{W}}_i \bm{X}_i)^{-1} n_i^{-1/2} \{\dot{\phi}^{-1} \bm{X}_i^\top (\bm{y}_i - \dot{\bm{\mu}}_i) \} \\
&+ O_p(m^{-1/2}) + O_p(n_L^{-1/2}) + O_p(m^{-1} n_U^{1/2} ) \\
&= n_i (\bm{X}_i^\top \dot{\bm{W}}_i \bm{X}_i)^{-1} n_i^{-1/2} \dot{\phi}^{-1} \bm{X}_i^\top (\bm{y}_i - \dot{\bm{\mu}}_i) + o_p(1).  
\end{align*}
The required result follows from the Dominated Convergence Theorem.

\subsection{Result for Difference Between the Prediction Gaps of Two Clusters}
\textit{Assume Conditions (C1)-(C5) are satisfied, $m n_L^{-2} \rightarrow 0$, $m n_U^{-1/2} \rightarrow \infty$, and $n_i n_{i'}^{-1} \rightarrow \gamma \in (0,\infty)$. Then as $m,n_L \rightarrow \infty$ and unconditional on the random effects $\dot{\bm{b}}$, for each $i \neq i'\in \{ 1,\ldots,m \}$ we have}
\begin{align*}
    {n}_i^{1/2} \{(\hat{\bm{b}}_i - \dot{\bm{b}}_i) - (\hat{\bm{b}}_{i'} - \dot{\bm{b}}_{i'})\} & \overset{D}{\rightarrow} \text{mixN}(\bm{0}, \dot{\bm{K}}_i , F_{\dot{\bm{b}}_i}) * \, \text{mixN}(\bm{0}, \gamma \dot{\bm{K}}_{i'} , F_{\dot{\bm{b}}_{i'}}).
\end{align*}

\underline{Proof:} Theorem \ref{thm:unconditonal_ranef} implies that, given $m n_L^{-2} \rightarrow 0$, $m n_U^{-1/2} \rightarrow \infty$, and $n_i n_{i'}^{-1} \rightarrow \gamma \in (0,\infty)$, we have
\begin{align*}
{n}_i^{1/2} (\hat{\bm{b}}_i - \dot{\bm{b}}_i - \hat{\bm{b}}_{i'} + \dot{\bm{b}}_{i'}) &=  (n_i^{-1} \bm{X}_i^\top \dot{\bm{W}}_i \bm{X}_i )^{-1} n_i^{-1/2} \bm{X}_i^\top (\bm{y}_i - \dot{\bm{\mu}}_i )  \\
&+ \gamma^{1/2} (n_{i'}^{-1} \bm{X}_{i'}^\top \dot{\bm{W}}_{i'} \bm{X}_{i'})^{-1} n_{i'}^{-1/2} \bm{X}_{i'}^\top (\bm{y}_{i'} - \dot{\bm{\mu}}_{i'} )  \\
&+ O_p(m^{-1/2}) + O_p(n_L^{-1/2}) + O_p(m^{-1} n_U^{1/2}),
\end{align*}
and the result follows by the independence of $\dot{\bm{b}}_i$ and $\dot{\bm{b}}_{i'}$.

\section{Remainder Term in the Taylor Expansion} \label{sec:remainder}
\setcounter{equation}{0}

In this section, we show that in the Taylor expansion \eqref{taylor_fundamental}, the remainder term $- \frac{1}{2} \{\nabla^2 Q(\dot{\bm{\theta}})\}^{-1} \bm{R}(\tilde{\bm{\theta}})$ is of smaller order component-wise than $- \{\nabla^2 Q(\dot{\bm{\theta}})\}^{-1} \nabla Q(\dot{\bm{\theta}})$. To deal with this remainder term, we have the following from equation \eqref{taylor_fundamental}
\begin{align*} 
&\hat{\bm{\theta}} - \dot{\bm{\theta}} = \bm{B}^{-1} \nabla Q(\dot{\bm{\theta}}) + \frac{1}{2} \bm{B}^{-1} \bm{R}(\tilde{\bm{\theta}}) \\
\Rightarrow \, &\hat{\bm{\theta}} - \dot{\bm{\theta}} - \frac{1}{2} \bm{B}^{-1} \bm{R}(\tilde{\bm{\theta}}) = \bm{B}^{-1} \nabla Q(\dot{\bm{\theta}}) \\
\Rightarrow \, &(\bm{I}_{(m+1)p } - \bm{\Lambda} ) ( \hat{\bm{\theta}} - \dot{\bm{\theta}} ) = \bm{B}^{-1} \nabla Q(\dot{\bm{\theta}}) \\
\Rightarrow \, &\hat{\bm{\theta}} - \dot{\bm{\theta}} = (\bm{I}_{(m+1)p } - \bm{\Lambda} )^{-1} \bm{B}^{-1} \nabla Q(\dot{\bm{\theta}}) \\
&= \bm{B}^{-1} \nabla Q(\dot{\bm{\theta}}) + \left( \sum_{s=1}^\infty \bm{\Lambda}^s \right) \bm{B}^{-1} \nabla Q(\dot{\bm{\theta}}),
\end{align*}
where the last line is derived from repeated application of the Woodbury identity, and $\bm{\Lambda}$ is the appropriate $(m+1)p \times (m+1)p$ matrix defined in detail later on. The convergence of the geometric sum and thus invertibility of $(\bm{I}_{(m+1)p } - \bm{\Lambda} )$ is shown in Lemma \ref{lem:invertibility}. We will show, using the consistency result $\| \hat{\bmtheta} - \dot{\bmtheta} \|_{\infty} = o_p(1)$, that $ \sum_{s=1}^\infty \bm{\Lambda}^s  \bm{B}^{-1} \nabla Q(\dot{\bm{\theta}})$ is of smaller order component-wise than $\bm{B}^{-1} \nabla Q(\dot{\bm{\theta}})$. This is equivalent to $0.5 \bm{B}^{-1} \bm{R}(\tilde{\bm{\theta}})$ being smaller order component-wise than $\bm{B}^{-1} \nabla Q(\dot{\bm{\theta}})$ in \eqref{taylor_fundamental}.

Let $\bm{T}_1$ denote the first $p$ components of $\bm{R}(\tilde{\bm{\theta}})$, $\bm{T}_2$ its remaining $m p$ components, and $\bm{T}_{2i}$ denote the $\{(i-1)p +1\}$-th to $(ip)$-th components of $\bm{T}_2$. We first prove a result needed for later developments.

\begin{lemma} \label{lem:remainderterms}
Assume Conditions (C1) and (C3) are satisfied. Then irrespective of whether $\dot{\bm{b}}$ is conditioned on, it holds that $ \bm{R}(\tilde{\bm{\theta}})_{[1:p]}  =  \sum_{i=1}^m \bm{R}(\tilde{\bm{\theta}})_{[i p + 1: (i+1) p]}. $
\end{lemma}

\begin{proof}
Recall the Taylor expansion $\nabla Q(\hat{\bm{\theta}}) = \bm{0} =\nabla Q(\dot{\bm{\theta}}) + \nabla^2 Q(\dot{\bm{\theta}}) (\hat{\bm{\theta}} - \dot{\bm{\theta}}) + \bm{R}(\tilde{\bm{\theta}})$. Then
\begin{align*}
\bm{0}_{p \times 1}  &= \nabla Q(\hat{\bm{\theta}})_{[1:p]} \\
&=\{ \nabla Q(\dot{\bm{\theta}}) + \nabla^2 Q(\dot{\bm{\theta}}) (\hat{\bm{\theta}} - \dot{\bm{\theta}}) + \bm{R}(\tilde{\bm{\theta}})\}_{[1:p]} \\
&= \sum_{i=1}^m \nabla Q(\hat{\bm{\theta}})_{[i p + 1: (i+1) p]} \\
&= \sum_{i=1}^m \{ \nabla Q(\dot{\bm{\theta}}) + \nabla^2 Q(\dot{\bm{\theta}}) (\hat{\bm{\theta}} - \dot{\bm{\theta}}) + \bm{R}(\tilde{\bm{\theta}})\}_{[i p + 1: (i+1) p]}. 
\end{align*}
Since $\bm{Z}_i = \bm{X}_i$ for all $i = 1,\ldots,m$ under our simplifying assumption, and $\sum_{i=1}^m \hat{\bm{b}}_i = \bm{0}$, then we obtain
\begin{align*}
 \{ \nabla Q(\dot{\bm{\theta}}) + \nabla^2 Q(\dot{\bm{\theta}}) (\hat{\bm{\theta}} - \dot{\bm{\theta}}) + \bm{R}(\tilde{\bm{\theta}})\}_{[1:p]} &= \sum_{i=1}^m \{ \nabla Q(\dot{\bm{\theta}}) + \nabla^2 Q(\dot{\bm{\theta}}) (\hat{\bm{\theta}} - \dot{\bm{\theta}}) + \bm{R}(\tilde{\bm{\theta}})\}_{[i p + 1: (i+1) p]}. 
\end{align*}
Therefore, we have $\bm{T}_1 = \bm{R}(\tilde{\bm{\theta}})_{[1:p]} =  \sum_{i=1}^m  \bm{R}(\tilde{\bm{\theta}})_{[i p + 1: (i+1) p]} = \sum_{i=1}^m \bm{T}_{2i}$, which follows from the fact that $\sum_{i=1}^m  \nabla Q(\dot{\bm{\theta}})_{[i p + 1: (i+1) p]} =  \nabla Q(\dot{\bm{\theta}})_{[1:p]} - \sum_{i=1}^m  \hat{\bm{G}}^{-1} \dot{\bm{b}}_i$ and $\sum_{i=1}^m \{\nabla^2 Q(\dot{\bm{\theta}}) (\hat{\bm{\theta}} - \dot{\bm{\theta}})\}_{[i p + 1: (i+1) p]}  = \{\nabla^2 Q(\dot{\bm{\theta}}) (\hat{\bm{\theta}} - \dot{\bm{\theta}})\}_{[1:p]} + \sum_{i=1}^m  \hat{\bm{G}}^{-1} \dot{\bm{b}}_i - \sum_{i=1}^m  \hat{\bm{G}}^{-1} \hat{\bm{b}}_i $.
\end{proof}

Next, let $\bm{S}(\bm{\theta}) = \nabla Q(\bmtheta)$, $\tilde{\bm{W}}' = \dot{\phi}^{-1} \text{diag} \{a'''(\tilde{\eta}_{11}), \ldots, a'''(\tilde{\eta}_{1n_1}), \ldots, a'''(\tilde{\eta}_{mn_m})\}$. Then the remainder term can be written as 
\begin{align*}
\bm{R}(\tilde{\bm{\theta}}) = \left[ \begin{matrix}
(\hat{\bm{\theta}} - \dot{\bm{\theta}})^\top  \frac{\partial^2 \bm{S}_{[1]} (\tilde{\bm{\theta}})}{\partial \bm{\theta} \partial \bm{\theta}^\top }  (\hat{\bm{\theta}} - \dot{\bm{\theta}}) \\
\vdots \\
(\hat{\bm{\theta}} - \dot{\bm{\theta}})^\top  \frac{\partial^2 \bm{S}_{[(m+1)p]} (\tilde{\bm{\theta}})}{\partial \bm{\theta} \partial \bm{\theta}^\top }  (\hat{\bm{\theta}} - \dot{\bm{\theta}})     \\
\end{matrix} \right].
\end{align*}
Now, for $1 \leq j \leq p$, we have $\bm{S}_{[j]}(\bmtheta) = \dot{\phi}^{-1} \bm{X}_{[,j]}^\top \{\bm{y} - \bm{\mu}(\bmtheta)\} = \dot{\phi}^{-1} \sum_{i=1}^m \sum_{l=1}^{n_i} x_{il[j]} \{ y_{il} - a'(\eta_{il}) \}$, noting this is a scalar. Thus
\begin{align*}
\frac{\partial}{\partial \bmtheta} \bm{S}_{[j]}(\bmtheta) = - \dot{\phi}^{-1} \sum_{i=1}^m \sum_{l=1}^{n_i} \left[ \begin{matrix} \bm{x}_{il} \\ \frac{\partial}{\partial \bm{b}} \eta_{il} \end{matrix} \right]  a''(\eta_{il}) x_{il[j]} = - \left[ \begin{matrix} \bm{X}^\top \bm{W} \bm{X}_{[,j]} \\ \bm{Z}^\top \bm{W} \bm{X}_{[,j]} \end{matrix} \right],
\end{align*}
which is an $(m+1)p$-vector. Hence the $(m+1)p \times (m+1)p$ matrix can be written as
\begin{align*}
\frac{\partial^2 \bm{S}_{[j]} (\tilde{\bm{\theta}})}{\partial \bm{\theta} \partial \bm{\theta}^\top }  &= - \dot{\phi}^{-1} \sum_{i=1}^m \sum_{l=1}^{n_i} \left[ \begin{matrix} \bm{x}_{il} \\ \frac{\partial}{\partial \bm{b}} \eta_{il} \end{matrix} \right]  a'''(\tilde{\eta}_{il}) x_{il[j]} \left[ \begin{matrix} \bm{x}_{il} \\ \frac{\partial}{\partial \bm{b}} \eta_{il} \end{matrix} \right]^\top \\
&= - \left[ \begin{matrix}
\bm{X}^\top \text{diag} (\bm{X}_{[,j]}) \tilde{\bm{W}}' \bm{X} & \bm{X}^\top \text{diag} (\bm{X}_{[,j]}) \tilde{\bm{W}}' \bm{Z} \\
\bm{Z}^\top \text{diag} (\bm{X}_{[,j]}) \tilde{\bm{W}}' \bm{X} & \bm{Z}^\top \text{diag} (\bm{X}_{[,j]}) \tilde{\bm{W}}' \bm{Z}
\end{matrix} \right], \quad  1 \leq j \leq p.
\end{align*}
Similarly, for $1 \leq k \leq m p$, $\bm{S}_{[p + k]} (\bmtheta) = \dot{\phi}^{-1} \bm{Z}_{[,k]}^\top \{\bm{y} - \bm{\mu}(\bmtheta)\} - \{(\bm{I}_{m} \otimes \hat{\bm{G}}) \bm{b}\}_{[k]}$, such that
\begin{align*}
\frac{\partial}{\partial \bmtheta} \bm{S}_{[p + k]}(\bmtheta) = - \left[ \begin{matrix} \bm{X}^\top \bm{W} \bm{Z}_{[,k]} \\ \bm{Z}^\top \bm{W} \bm{Z}_{[,k]} + \frac{\partial}{\partial \bm{b}} \{(\bm{I}_{m} \otimes \hat{\bm{G}}) \bm{b}\}_{[k]} \end{matrix} \right],
\end{align*}
where $\partial/\partial \bm{b} \{(\bm{I}_{m} \otimes \hat{\bm{G}}) \bm{b}\}_{[k]}$ is not a function of $\bmtheta$. Thus
\begin{align*}
\frac{\partial^2 \bm{S}_{[p + k]} (\tilde{\bm{\theta}})}{\partial \bm{\theta} \partial \bm{\theta}^\top }  = - \left[ \begin{matrix}
\bm{X}^\top \text{diag} (\bm{Z}_{[,k]}) \tilde{\bm{W}}' \bm{X} & \bm{X}^\top \text{diag} (\bm{Z}_{[,k]}) \tilde{\bm{W}}' \bm{Z} \\
\bm{Z}^\top \text{diag} (\bm{Z}_{[,k]}) \tilde{\bm{W}}' \bm{X} & \bm{Z}^\top \text{diag} (\bm{Z}_{[,k]}) \tilde{\bm{W}}' \bm{Z}
\end{matrix} \right], \quad  1 \leq k \leq m p.
\end{align*}
Next, recall that $\bm{B}_2 (\bm{B}_3+\bm{B}_4)^{-1} = [\bm{I}_p - \hat{\bm{G}}^{-1} (\bm{X}^\top_1 \dot{\bm{W}}_1 \bm{X}_1 + \hat{\bm{G}}^{-1})^{-1},\ldots, \bm{I}_p - \hat{\bm{G}}^{-1} (\bm{X}^\top_m \dot{\bm{W}}_m \bm{X}_m + \hat{\bm{G}}^{-1})^{-1} ]$. By Lemma 1 and the blockwise inversion formula for $\bm{B}^{-1}$, the first $p$ components of $\bm{B}^{-1}  \bm{R}(\tilde{\bm{\theta}})$ are given by
\begin{align*}
& \left[ \begin{matrix}
\bm{C}^{-1} & -\bm{C}^{-1} \bm{B}_2 (\bm{B}_3+\bm{B}_4)^{-1} 
\end{matrix} \right]  \bm{R}(\tilde{\bm{\theta}}) \\
&= \bm{C}^{-1} \left[ \begin{matrix}
\bm{I}_{p} &  - [\bm{I}_p - \hat{\bm{G}}^{-1} (\bm{X}^\top_1 \dot{\bm{W}}_1 \bm{X}_1 + \hat{\bm{G}}^{-1})^{-1},\ldots, \bm{I}_p - \hat{\bm{G}}^{-1} (\bm{X}^\top_m \dot{\bm{W}}_m \bm{X}_m + \hat{\bm{G}}^{-1})^{-1} ]
\end{matrix} \right] \bm{R}(\tilde{\bm{\theta}}) \\
&= \bm{C}^{-1} \left\{\bm{T}_1 - \sum_{i=1}^m \bm{T}_{2i} + \sum_{i=1}^m \hat{\bm{G}}^{-1} (\bm{X}^\top_i \dot{\bm{W}}_i \bm{X}_i + \hat{\bm{G}}^{-1})^{-1} \bm{T}_{2i} \right\}.  \numberthis \label{cond:fixed3rd_order}
\end{align*}
Similarly, the last $mp$ components of $ \bm{B}^{-1}  \bm{R}(\tilde{\bm{\theta}})$ are
\begin{align*}
& \left[ \begin{matrix}
-(\bm{B}_3+\bm{B}_4)^{-1} \bm{B}_2^\top \bm{C}^{-1} & (\bm{B}_3+\bm{B}_4)^{-1} + (\bm{B}_3+\bm{B}_4)^{-1} \bm{B}_2^\top \bm{C}^{-1} \bm{B}_2 (\bm{B}_3+\bm{B}_4)^{-1}
\end{matrix} \right] \bm{R}(\tilde{\bm{\theta}}) \\
&=  - (\bm{B}_3+\bm{B}_4)^{-1} \bm{B}_2^\top \left[ \begin{matrix}
 \bm{C}^{-1} & - \bm{C}^{-1} \bm{B}_2 (\bm{B}_3+\bm{B}_4)^{-1}
\end{matrix} \right] \bm{R}(\tilde{\bm{\theta}}) +  (\bm{B}_3+\bm{B}_4)^{-1} \bm{T}_2. \numberthis \label{lastnp_3rd} 
\end{align*}
Hence the first $p$ components of $\bm{B}^{-1} \bm{R}(\tilde{\bm{\theta}})$ are given by
\begin{align*}
    \bm{F}_1 = \bm{C}^{-1} \sum_{i=1}^m \hat{\bm{G}}^{-1} (\bm{X}^\top_i \dot{\bm{W}}_i \bm{X}_i + \hat{\bm{G}}^{-1})^{-1} \bm{T}_{2i},
\end{align*}
and the last $mp$ components of $\bm{B}^{-1} \bm{R}(\tilde{\bm{\theta}})$ are given by
\begin{align*}
    \bm{F}_2 = - (\bm{B}_3+\bm{B}_4)^{-1} \bm{B}_2^\top \bm{F}_1 + (\bm{B}_3+\bm{B}_4)^{-1} \bm{T}_2.
\end{align*}
Next, we have
\begin{align*}
\bm{T}_2 &= \left[ \begin{matrix}
(\hat{\bm{\theta}} - \dot{\bm{\theta}})^\top  \frac{\partial^2 \bm{S}_{[p+1]} (\tilde{\bm{\theta}})}{\partial \bm{\theta} \partial \bm{\theta}^\top }  (\hat{\bm{\theta}} - \dot{\bm{\theta}}) \\
\vdots \\
(\hat{\bm{\theta}} - \dot{\bm{\theta}})^\top  \frac{\partial^2 \bm{S}_{[(m+1)p]} (\tilde{\bm{\theta}})}{\partial \bm{\theta} \partial \bm{\theta}^\top }  (\hat{\bm{\theta}} - \dot{\bm{\theta}})
\end{matrix} \right] = \left[ \begin{matrix}
(\hat{\bm{\theta}} - \dot{\bm{\theta}})^\top  \frac{\partial^2 \bm{S}_{[p+1]} (\tilde{\bm{\theta}})}{\partial \bm{\theta} \partial \bm{\theta}^\top } \\
\vdots \\
(\hat{\bm{\theta}} - \dot{\bm{\theta}})^\top  \frac{\partial^2 \bm{S}_{[(m+1)p]} (\tilde{\bm{\theta}})}{\partial \bm{\theta} \partial \bm{\theta}^\top }
\end{matrix} \right] (\hat{\bm{\theta}} - \dot{\bm{\theta}}) \triangleq \bm{F}_3 (\hat{\bm{\theta}} - \dot{\bm{\theta}})
\end{align*}
and
\begin{align*}
\bm{T}_{2i} &= \left[ \begin{matrix}
(\hat{\bm{\theta}} - \dot{\bm{\theta}})^\top  \frac{\partial^2 \bm{S}_{[(i-1)p + 1]} (\tilde{\bm{\theta}})}{\partial \bm{\theta} \partial \bm{\theta}^\top }  (\hat{\bm{\theta}} - \dot{\bm{\theta}}) \\
\vdots \\
(\hat{\bm{\theta}} - \dot{\bm{\theta}})^\top  \frac{\partial^2 \bm{S}_{[ip]} (\tilde{\bm{\theta}})}{\partial \bm{\theta} \partial \bm{\theta}^\top }  (\hat{\bm{\theta}} - \dot{\bm{\theta}})
\end{matrix} \right] = \left[ \begin{matrix}
(\hat{\bm{\theta}} - \dot{\bm{\theta}})^\top  \frac{\partial^2 \bm{S}_{[(i-1)p + 1]} (\tilde{\bm{\theta}})}{\partial \bm{\theta} \partial \bm{\theta}^\top } \\
\vdots \\
(\hat{\bm{\theta}} - \dot{\bm{\theta}})^\top  \frac{\partial^2 \bm{S}_{[ip]} (\tilde{\bm{\theta}})}{\partial \bm{\theta} \partial \bm{\theta}^\top }
\end{matrix} \right] (\hat{\bm{\theta}} - \dot{\bm{\theta}}) \triangleq \bm{F}_{3i} (\hat{\bm{\theta}} - \dot{\bm{\theta}}).
\end{align*}
Here, $\bm{F}_3$ is a $mp \times (m+1)p$ matrix and $\bm{F}_{3i}$ is $p \times (m+1)p$. Notice that $\bm{F}_3 = [\bm{F}_{31}^\top,\ldots, \bm{F}_{3n}^\top]^\top$. Furthermore,
\begin{align*}
& \bm{B}^{-1} \bm{R}(\tilde{\bm{\theta}}) = \left[ \begin{matrix}
\bm{F}_1 \\
\bm{F}_2
\end{matrix} \right] = \left[ \begin{matrix}
\sum_{i=1}^m \bm{C}^{-1}  \hat{\bm{G}}^{-1} (\bm{X}^\top_i \dot{\bm{W}}_i \bm{X}_i + \hat{\bm{G}}^{-1})^{-1} \bm{T}_{2i} \\
- (\bm{B}_3+\bm{B}_4)^{-1} \bm{B}_2^\top \bm{F}_1 + (\bm{B}_3+\bm{B}_4)^{-1} \bm{T}_2
\end{matrix} \right] \\
&= \left[ \begin{matrix}
\sum_{i=1}^m \bm{C}^{-1}  \hat{\bm{G}}^{-1} (\bm{X}^\top_i \dot{\bm{W}}_i \bm{X}_i + \hat{\bm{G}}^{-1})^{-1} \bm{F}_{3i} (\hat{\bm{\theta}} - \dot{\bm{\theta}}) \\
- (\bm{B}_3+\bm{B}_4)^{-1} \bm{B}_2^\top \sum_{i=1}^m \bm{C}^{-1}  \hat{\bm{G}}^{-1} (\bm{X}^\top_i \dot{\bm{W}}_i \bm{X}_i + \hat{\bm{G}}^{-1})^{-1} \bm{F}_{3i} (\hat{\bm{\theta}} - \dot{\bm{\theta}}) + (\bm{B}_3+\bm{B}_4)^{-1} \bm{F}_{3} (\hat{\bm{\theta}} - \dot{\bm{\theta}})
\end{matrix} \right] \\
&= \left[ \begin{matrix}
\sum_{i=1}^m \bm{C}^{-1}  \hat{\bm{G}}^{-1} (\bm{X}^\top_i \dot{\bm{W}}_i \bm{X}_i + \hat{\bm{G}}^{-1})^{-1} \bm{F}_{3i}  \\
- (\bm{B}_3+\bm{B}_4)^{-1} \bm{B}_2^\top \sum_{i=1}^m \bm{C}^{-1}  \hat{\bm{G}}^{-1} (\bm{X}^\top_i \dot{\bm{W}}_i \bm{X}_i + \hat{\bm{G}}^{-1})^{-1} \bm{F}_{3i} + (\bm{B}_3+\bm{B}_4)^{-1} \bm{F}_{3}
\end{matrix} \right] (\hat{\bm{\theta}} - \dot{\bm{\theta}}) \\
&= 2 \bm{\Lambda} (\hat{\bm{\theta}} - \dot{\bm{\theta}}).
\end{align*}
The $k$th row of $\bm{F}_{3i}$ for $1 \leq k \leq p$ is given by
\begin{align*}
& - (\hat{\bm{\theta}} - \dot{\bm{\theta}})^\top \left[ \begin{matrix}
\bm{X}^\top \text{diag} (\bm{Z}_{[,(i-1)p + k]}) \tilde{\bm{W}}' \bm{X} & \bm{X}^\top \text{diag} (\bm{Z}_{[,(i-1)p + k]}) \tilde{\bm{W}}' \bm{Z} \\
\bm{Z}^\top \text{diag} (\bm{Z}_{[,(i-1)p + k]}) \tilde{\bm{W}}' \bm{X} & \bm{Z}^\top \text{diag} (\bm{Z}_{[,(i-1)p + k]}) \tilde{\bm{W}}' \bm{Z}
\end{matrix} \right] \\
&= - \delta_{m,n_L}^{-1} \big[ \delta_{m,n_L} (\hat{\bmbeta} - \dot{\bmbeta})^\top \bm{X}^\top \text{diag} (\bm{Z}_{[,(i-1)p + k]}) \tilde{\bm{W}}' \bm{X}  + \delta_{m,n_L} (\hat{\bm{b}} - \dot{\bm{b}})^\top \bm{Z}^\top \text{diag} (\bm{Z}_{[,(i-1)p + k]}) \tilde{\bm{W}}' \bm{X} , \\
& \delta_{m,n_L} (\hat{\bmbeta} - \dot{\bmbeta})^\top \bm{X}^\top \text{diag} (\bm{Z}_{[,(i-1)p + k]}) \tilde{\bm{W}}' \bm{Z}  + \delta_{m,n_L} (\hat{\bm{b}} - \dot{\bm{b}})^\top \bm{Z}^\top \text{diag} (\bm{Z}_{[,(i-1)p + k]}) \tilde{\bm{W}}' \bm{Z} \big], \numberthis \label{eq:F_3i}
\end{align*}
where $\delta_{m,n_L}$ is a positive unbounded monotonically increasing sequence such that $\delta_{m,n_L} \| \hat{\bmtheta} - \dot{\bmtheta} \|_{\infty} = O_p(1)$. The consistency results proved in Section \ref{sec:consistency} ensure that such a $\delta_{m,n_L}$ must exist; this is true for both the conditional and unconditional regimes. 

Observe that only the $\{ (\sum_{l = 0}^{i-1} n_l) + 1 \}$th to $(\sum_{l = 0}^{i} n_l)$th components of $\bm{Z}_{[,(i-1)p + k]}$ are non-zero, where we define $n_0 := 0$. This means that, for any $1 \leq k \leq p$, only the $\{(i-1)p+1\}$th to $(ip)$th columns of both \\ $\bm{X}^\top \text{diag} (\bm{Z}_{[,(i-1)p + k]}) \tilde{\bm{W}}' \bm{Z}$ and $\bm{Z}^\top \text{diag} (\bm{Z}_{[,(i-1)p + k]}) \tilde{\bm{W}}' \bm{Z}$ will be non-zero. In other words, other than the first $p$ columns, only the $(ip+1)$th to $\{(i+1)p\}$th columns of $\bm{F}_{3i}$ are non-zero. Thus $\bm{F}_3$, disregarding its first $p$ columns, is an $mp \times mp$ block-diagonal matrix.

The non-zero components of $\delta_{m,n_L} \bm{F}_{3}$ and $\delta_{m,n_L} \bm{F}_{3i}$ are all $O_p(n_U)$ component-wise, again because at most $n_i$ components of $\bm{Z}_{[,(i-1)p + k]}$ are non-zero. For ease of notation and understanding, we now represent all terms using their orders only. Since $\bm{C}^{-1} = O_p(m^{-1})$ and $\hat{\bm{G}}^{-1} (\bm{X}^\top_i \dot{\bm{W}}_i \bm{X}_i + \hat{\bm{G}}^{-1})^{-1} = O_p(n_i^{-1})$, from the above discussion we have that \\ $\sum_{i=1}^m \bm{C}^{-1}  \hat{\bm{G}}^{-1} (\bm{X}^\top_i \dot{\bm{W}}_i \bm{X}_i + \hat{\bm{G}}^{-1})^{-1} \bm{F}_{3i}$ is a $p \times (m+1)p$ matrix of the form \\ $\delta_{m,n_L}^{-1} [O_p(1), O_p(m^{-1}), \ldots, O_p(m^{-1}) ]$. Next, $ (\bm{B}_3+\bm{B}_4)^{-1} \bm{B}_2^\top = [\bm{I}_p + O_p(n_1^{-1}) , \ldots, \bm{I}_p + O_p(n_m^{-1})]^\top$ and $(\bm{B}_3+\bm{B}_4)^{-1}$ is a block-diagonal $O_p(n_L^{-1})$ matrix component-wise. Therefore, we find that $\bm{\Lambda}$ is of the form
\begin{align*}
& 0.5 \left[ \begin{matrix}
\sum_{i=1}^m \bm{C}^{-1}  \hat{\bm{G}}^{-1} (\bm{X}^\top_i \dot{\bm{W}}_i \bm{X}_i + \hat{\bm{G}}^{-1})^{-1} \bm{F}_{3i}  \\
- (\bm{B}_3+\bm{B}_4)^{-1} \bm{B}_2^\top \sum_{i=1}^m \bm{C}^{-1}  \hat{\bm{G}}^{-1} (\bm{X}^\top_i \dot{\bm{W}}_i \bm{X}_i + \hat{\bm{G}}^{-1})^{-1} \bm{F}_{3i} 
\end{matrix} \right] + 0.5 \left[ \begin{matrix} \bm{0}_{p \times (m+1)p} \\ (\bm{B}_3+\bm{B}_4)^{-1} \bm{F}_{3} \end{matrix} \right] \\
& \triangleq \bm{\Lambda}_1 + \bm{\Lambda}_2 \\
&=  \frac{1}{\delta_{m,n_L}}  \left[ \begin{matrix} 
    O_p(1) & O_p(m^{-1}) & \cdots & O_p(m^{-1}) \\
    \vdots & \vdots & \vdots & \vdots \\
    O_p(1) & O_p(m^{-1}) & \cdots & O_p(m^{-1})
    \end{matrix} \right] + 
    \frac{1}{\delta_{m,n_L}} \left[ \begin{matrix} 
    & & \bm{0}_{p \times (m+1)p} \\
    O_p(1) & O_p(1) & \bm{0} & \cdots & \bm{0}  \\
    O_p(1) & \bm{0} & O_p(1) & \cdots & \bm{0} \\
    \vdots & \vdots & \vdots & \ddots & \vdots\\
    O_p(1) & \bm{0} & \cdots & \bm{0} & O_p(1)
    \end{matrix} \right] \\
&=  \frac{1}{\delta_{m,n_L}} \left[ \begin{matrix} 
     O_p(1) & O_p(m^{-1}) & O_p(m^{-1}) & \cdots & O_p(m^{-1}) \\
    O_p(1) & O_p(1) & O_p(m^{-1}) & \cdots & O_p(m^{-1})  \\
    O_p(1) & O_p(m^{-1}) & O_p(1) & \cdots & O_p(m^{-1}) \\
    \vdots & \vdots & \vdots & \ddots & \vdots\\
    O_p(1) & O_p(m^{-1}) & \cdots & O_p(m^{-1}) & O_p(1)
    \end{matrix} \right]. \numberthis \label{eq:lambda}
\end{align*}
Writing $\bm{\Lambda} = \delta_{m,n_L}^{-1} \bm{\Lambda}_{\delta}$, we see that the component-wise order of $\bm{\Lambda}_{\delta}$ remains the same no matter how many times it is multiplied by itself. Furthermore, each row of $\bm{\Lambda}_\delta^s$ is $O_p(1)$ for only a finite number of components, and $O_p(m^{-1})$ for the others. We will use these facts to examine the behaviour of $ \sum_{s=1}^\infty \bm{\Lambda}^s  \bm{B}^{-1} \nabla Q(\dot{\bm{\theta}}) = \sum_{s=1}^\infty \delta_{m,n_L}^{-s} \bm{\Lambda}_\delta^s  \bm{B}^{-1} \nabla Q(\dot{\bm{\theta}})$, and we will do so separately for the conditional and unconditional regimes. Before proceeding, we first confirm the convergence of $\sum_{s=1}^\infty \bm{\Lambda}^s$.
\begin{lemma} \label{lem:invertibility}
Assume Conditions (C1)-(C5) are satisfied. Then with probability tending to one as $m,n_L \rightarrow \infty$, the geometric sum $\sum_{s=1}^\infty \bm{\Lambda}^s$ converges.
\end{lemma}
\begin{proof}
To prove the result we will show that, with probability tending to one as $m, n_L \rightarrow \infty$, $\|\bm{\Lambda}\| < 1$ for some sub-multiplicative matrix norm $\| \cdot \|$. In particular, we will consider the maximum absolute row sum of $\bm{\Lambda}$, denoted by $\| \cdot \|_{\infty}$ i.e., the operator norm induced by the vector infinity norm.

From \eqref{eq:lambda}, we have $\| \bm{\Lambda} \|_{\infty} \leq \| \bm{\Lambda}_1 \|_{\infty} + \| \bm{\Lambda}_2 \|_{\infty}$. We first examine $\| \bm{\Lambda}_1 \|_{\infty}$. We may break up $\bm{\Lambda}_1$ into 
\begin{align*}
& 0.5 \left[ \begin{matrix}
\sum_{i=1}^m \bm{C}^{-1}  \hat{\bm{G}}^{-1} (\bm{X}^\top_i \dot{\bm{W}}_i \bm{X}_i + \hat{\bm{G}}^{-1})^{-1} \bm{F}_{3i}  \\
- \sum_{i=1}^m \bm{C}^{-1}  \hat{\bm{G}}^{-1} (\bm{X}^\top_i \dot{\bm{W}}_i \bm{X}_i + \hat{\bm{G}}^{-1})^{-1} \bm{F}_{3i} \\
- \sum_{i=1}^m \bm{C}^{-1}  \hat{\bm{G}}^{-1} (\bm{X}^\top_i \dot{\bm{W}}_i \bm{X}_i + \hat{\bm{G}}^{-1})^{-1} \bm{F}_{3i} \\
\vdots
\end{matrix} \right] + \\ & 0.5 \left[ \begin{matrix}
\bm{0}_p  \\
 (\bm{X}^\top_1 \dot{\bm{W}}_1 \bm{X}_1 + \hat{\bm{G}}^{-1})^{-1} \hat{\bm{G}}^{-1} \sum_{i=1}^m \bm{C}^{-1}  \hat{\bm{G}}^{-1} (\bm{X}^\top_i \dot{\bm{W}}_i \bm{X}_i + \hat{\bm{G}}^{-1})^{-1} \bm{F}_{3i} \\
 (\bm{X}^\top_2 \dot{\bm{W}}_2 \bm{X}_2 + \hat{\bm{G}}^{-1})^{-1} \hat{\bm{G}}^{-1} \sum_{i=1}^m \bm{C}^{-1}  \hat{\bm{G}}^{-1} (\bm{X}^\top_i \dot{\bm{W}}_i \bm{X}_i + \hat{\bm{G}}^{-1})^{-1} \bm{F}_{3i} \\
\vdots \\
(\bm{X}^\top_m \dot{\bm{W}}_m \bm{X}_m + \hat{\bm{G}}^{-1})^{-1} \hat{\bm{G}}^{-1} \sum_{i=1}^m \bm{C}^{-1}  \hat{\bm{G}}^{-1} (\bm{X}^\top_i \dot{\bm{W}}_i \bm{X}_i + \hat{\bm{G}}^{-1})^{-1} \bm{F}_{3i} 
\end{matrix} \right] \\
&\triangleq \bm{\Lambda}_3 + \bm{\Lambda}_4 \bm{\Lambda}_5,
\end{align*}
where $\bm{\Lambda}_4 = \text{bdiag}(\bm{0}_{p \times p}, (\bm{X}^\top_1 \dot{\bm{W}}_1 \bm{X}_1 + \hat{\bm{G}}^{-1})^{-1} \hat{\bm{G}}^{-1}, \ldots, (\bm{X}^\top_m \dot{\bm{W}}_m \bm{X}_m + \hat{\bm{G}}^{-1})^{-1} \hat{\bm{G}}^{-1})$, and $\bm{\Lambda}_5 = (0 , \bm{1}_m^\top)^\top \otimes \sum_{i=1}^m \bm{C}^{-1}  \hat{\bm{G}}^{-1} (\bm{X}^\top_i \dot{\bm{W}}_i \bm{X}_i + \hat{\bm{G}}^{-1})^{-1} \bm{F}_{3i}$. We can also write \\
$\bm{\Lambda}_3 = - 0.5 ( \bm{1}_m^* \otimes \bm{I}_p ) \sum_{i=1}^m \bm{C}^{-1}  \hat{\bm{G}}^{-1} (\bm{X}^\top_i \dot{\bm{W}}_i \bm{X}_i + \hat{\bm{G}}^{-1})^{-1} \bm{F}_{3i} $ and use the (component-wise) order results as used in \eqref{eq:lambda} to see that $\| \bm{\Lambda}_3 \|_{\infty} \leq \| - 0.5 ( \bm{1}_m^* \otimes \bm{I}_p ) \|_{\infty} \| \sum_{i=1}^m \bm{C}^{-1}  \hat{\bm{G}}^{-1} (\bm{X}^\top_i \dot{\bm{W}}_i \bm{X}_i + \hat{\bm{G}}^{-1})^{-1} \bm{F}_{3i} \|_{\infty} = o_p(1)$. Next, we have $\| \bm{\Lambda}_4 \bm{\Lambda}_5 \|_{\infty} \leq \| \bm{\Lambda}_4 \|_{\infty} \| \bm{\Lambda}_5 \|_{\infty}$. We know $\| \bm{\Lambda}_5 \|_{\infty} = o_p(1)$, and under conditions (C1)-(C2), we have $\| \bm{\Lambda}_4 \|_{\infty} = O_p(1)$. Thus we obtain $\| \bm{\Lambda}_1 \|_{\infty} = o_p(1)$.

Turning to $\bm{\Lambda}_2$, we examine each row of $(\bm{B}_3+\bm{B}_4)^{-1} \bm{F}_{3}$. First, $\| (\bm{B}_3+\bm{B}_4)^{-1} \bm{F}_{3} \|_{\infty} \leq \| (\bm{B}_3+\bm{B}_4)^{-1} \|_{\infty} \| \bm{F}_{3} \|_{\infty}$ and by conditions (C1)-(C2), we have $\| (\bm{B}_3+\bm{B}_4)^{-1} \|_{\infty} = O_p(n_L^{-1})$. Now, without loss of generality consider the first row of $\bm{F}_{3}$. This is given by
\begin{align*}
& -  \Bigg[  (\hat{\bmbeta} - \dot{\bmbeta})^\top \bm{X}^\top \text{diag} (\bm{Z}_{[,1]}) \tilde{\bm{W}}' \bm{X}  + (\hat{\bm{b}} - \dot{\bm{b}})^\top \bm{Z}^\top \text{diag} (\bm{Z}_{[,1]}) \tilde{\bm{W}}' \bm{X} , \\ & (\hat{\bmbeta} - \dot{\bmbeta})^\top \bm{X}^\top \text{diag} (\bm{Z}_{[,1]}) \tilde{\bm{W}}' \bm{Z}  + (\hat{\bm{b}} - \dot{\bm{b}})^\top \bm{Z}^\top \text{diag} (\bm{Z}_{[,1]}) \tilde{\bm{W}}' \bm{Z} \Bigg] \\
&= -  \Bigg[  (\hat{\bmbeta} - \dot{\bmbeta})^\top \bm{X}^\top \text{diag} (\bm{Z}_{[,1]}) \tilde{\bm{W}}' \bm{X}  + (\hat{\bm{b}} - \dot{\bm{b}})^\top \bm{Z}^\top \text{diag} (\bm{Z}_{[,1]}) \tilde{\bm{W}}' \bm{X} , \\  & (\hat{\bmbeta} - \dot{\bmbeta})^\top \bm{X}^\top \text{diag} (\bm{Z}_{[,1]}) \tilde{\bm{W}}' \bm{X}  + (\hat{\bm{b}} - \dot{\bm{b}})^\top \bm{Z}^\top \text{diag} (\bm{Z}_{[,1]}) \tilde{\bm{W}}' \bm{X} , \bm{0}_{(m-1)p}^\top \Bigg],
\end{align*}
since $\text{diag} (\bm{Z}_{[,1]})$ selects for the first cluster. Let $\bar{\bm{1}}_p$ be a $p$-vector whose entries consist of the (component-wise) signs of $(\hat{\bmbeta} - \dot{\bmbeta})^\top \bm{X}^\top \text{diag} (\bm{Z}_{[,1]}) \tilde{\bm{W}}' \bm{X}  + (\hat{\bm{b}} - \dot{\bm{b}})^\top \bm{Z}^\top \text{diag} (\bm{Z}_{[,1]}) \tilde{\bm{W}}' \bm{X}$. Then the absolute row sum of the first row of $\bm{F}_3$ is given by
\begin{align*}
    & 2 | \{ (\hat{\bmbeta} - \dot{\bmbeta})^\top \bm{X}^\top \text{diag} (\bm{Z}_{[,1]}) \tilde{\bm{W}}' \bm{X}  + (\hat{\bm{b}} - \dot{\bm{b}})^\top \bm{Z}^\top \text{diag} (\bm{Z}_{[,1]}) \tilde{\bm{W}}' \bm{X} \} \bar{\bm{1}}_p| \\
    =& 2 | \{ (\hat{\bmbeta} - \dot{\bmbeta})^\top \bm{X}^\top \text{diag} (\bm{Z}_{[,1]}) \tilde{\bm{W}}' \bm{X}  + (\hat{\bm{b}}_1 - \dot{\bm{b}}_1)^\top \bm{X}^\top \text{diag} (\bm{Z}_{[,1]}) \tilde{\bm{W}}' \bm{X} \} \bar{\bm{1}}_p | \\
    =& 2 | \{ (\hat{\bmbeta} - \dot{\bmbeta} + \hat{\bm{b}}_1 - \dot{\bm{b}}_1)^\top \bm{X}^\top \text{diag} (\bm{Z}_{[,1]}) \tilde{\bm{W}}' \bm{X} \} \bar{\bm{1}}_p | \\
    \leq & 2 p \| \{ \bm{X}^\top \text{diag} (\bm{Z}_{[,1]}) \tilde{\bm{W}}' \bm{X} (\hat{\bmbeta} - \dot{\bmbeta} + \hat{\bm{b}}_1 - \dot{\bm{b}}_1) \} \|_{\infty} \\
    \leq & 2 p \| \{ \bm{X}^\top \text{diag} (\bm{Z}_{[,1]}) \tilde{\bm{W}}' \bm{X} \} \|_{\infty}  \| \hat{\bmbeta} - \dot{\bmbeta} + \hat{\bm{b}}_1 - \dot{\bm{b}}_1  \|_{\infty} \\
    \leq & 2 p \| \{ \bm{X}^\top \text{diag} (\bm{Z}_{[,1]}) \tilde{\bm{W}}' \bm{X} \} \|_{\infty} ( \| \hat{\bmbeta} - \dot{\bmbeta}  \|_{\infty} + \| \hat{\bm{b}}_1 - \dot{\bm{b}}_1  \|_{\infty} ) \\
    \leq & 2 p \underset{k \in \{ 1, \ldots, mp \} }{\max} \| \{ \bm{X}^\top \text{diag} (\bm{Z}_{[,k]}) \tilde{\bm{W}}' \bm{X} \} \|_{\infty} ( \| \hat{\bmbeta} - \dot{\bmbeta}  \|_{\infty} + \| \hat{\bm{b}}_1 - \dot{\bm{b}}_1  \|_{\infty} ) \\
    =& 2 p \underset{k \in \{ 1, \ldots, mp \} }{\max} \| \{ \bm{X}^\top \text{diag} (\bm{Z}_{[,k]}) \tilde{\bm{W}}' \bm{X} \} \|_{\infty} \| \hat{\bmbeta} - \dot{\bmbeta}  \|_{\infty} \\ 
    &+ 2 p \underset{k \in \{ 1, \ldots, mp \} }{\max} \| \{ \bm{X}^\top \text{diag} (\bm{Z}_{[,k]}) \tilde{\bm{W}}' \bm{X} \} \|_{\infty} \| \hat{\bm{b}}_1 - \dot{\bm{b}}_1  \|_{\infty} \\
    & \triangleq \alpha + \alpha_{1} \triangleq \omega_1,
\end{align*}
where the second equality follows from $\text{diag} (\bm{Z}_{[,1]})$ selecting for only the first cluster, and $\bm{X}_i = \bm{Z}_i$. The first inequality is due to H\"{o}lder's inequality. Again, using Conditions (C1)-(C2) we have $\underset{k \in \{ 1, \ldots, mp \} }{\max} \| \{ \bm{X}^\top \text{diag} (\bm{Z}_{[,k]}) \tilde{\bm{W}}' \bm{X} \} \|_{\infty} = O_p(n_U)$. 

Now, $p$ is a constant and the absolute row sum of any row of $\bm{F}_3$ can be bounded analogously in the above way, the only difference being that for the $k$th row, then the quantity $(\hat{\bm{b}}_1 - \dot{\bm{b}}_1)$ changes to the prediction gap for the cluster that $\text{diag} (\bm{Z}_{[,k]})$ selects for. This means that the absolute row sums for the first $p$ rows of $\bm{F}_3$ are bounded by $\omega_1$, the next $p$ rows by $\omega_2$, and so on. Hence, to ensure $\| \bm{\Lambda}_2 \|_{\infty} = o_p(1)$ it suffices to ensure that $\| \bm{\omega} \otimes \bm{1}_p \|_{\infty} = \| \bm{\omega} \|_{\infty} = o_p(n_L)$, where $\bm{\omega} = (\omega_1, \ldots, \omega_m)^\top$. 

To show this, define $\bm{\alpha} = (\alpha_1, \ldots, \alpha_m)^\top$. Then $\| \bm{\omega} \|_{\infty} \leq \| \alpha \bm{1}_m \|_{\infty} + \| \bm{\alpha} \|_{\infty}$. By Conditions (C1)-(C2), we have $\| \alpha \bm{1}_m \|_{\infty} = \alpha = O_p(n_U) \times o_p(1) = o_p(n_U) = o_p(n_L)$. We also have 
\begin{align*}
    \| \bm{\alpha} \|_{\infty} &= 2 p \underset{k \in \{ 1, \ldots, mp \} }{\max} \| \{ \bm{X}^\top \text{diag} (\bm{Z}_{[,k]}) \tilde{\bm{W}}' \bm{X} \} \|_{\infty} \underset{i \in \{ 1, \ldots, m \} }{\max} \| \hat{\bm{b}}_i - \dot{\bm{b}}_i  \|_{\infty} \\
    &= 2 p \underset{k \in \{ 1, \ldots, mp \} }{\max} \| \{ \bm{X}^\top \text{diag} (\bm{Z}_{[,k]}) \tilde{\bm{W}}' \bm{X} \} \|_{\infty}  \| \hat{\bm{b}} - \dot{\bm{b}}  \|_{\infty} \\
    &= O_p(n_U) \times o_p(1) = o_p(n_U) = o_p(n_L),
\end{align*}
where the last line follows from conditions (C1)-(C2), and the fact that $\| \hat{\bmtheta} - \dot{\bmtheta} \|_{\infty} = o_p(1)$. The result follows since $\| \bm{\Lambda} \|_{\infty}$ is therefore of order $o_p(1)$, and for any $\epsilon > 0$ we have $ \| \bm{\Lambda} \|_{\infty} < \epsilon$ with probability tending to one as $m,n_L \rightarrow \infty$. The argument above holds for both the conditional and unconditional regime, and the required result follows.
\end{proof}

\subsection{Conditional Regime} \label{sec:cond_remainder}
In the conditional regime, we assume without loss of generality that $\sum_{i=1}^m \dot{\bm{b}}_i = \bm{0}_p$, recalling that we can always reparametrise the random effects to satisfy this. From previous derivations, we know that when $m n_L^{-1} \rightarrow 0$, the quantity $\bm{B}^{-1} \nabla Q(\dot{\bm{\theta}})$ is of order $O_p(N^{-1/2})$ for the first $p$ components and $O_p(n_L^{-1/2})$ for the last $mp$ components. By the two properties of $\bm{\Lambda}_\delta^s$ noted above, we therefore know that $\bm{\Lambda}_\delta^s  \bm{B}^{-1} \nabla Q(\dot{\bm{\theta}})$ is at most $O_p(n_L^{-1/2})$ component-wise for any $s$. Hence $\sum_{s=1}^\infty \delta_{m,n_L}^{-s} \bm{\Lambda}_\delta^s  \bm{B}^{-1} \nabla Q(\dot{\bm{\theta}}) = \delta_{m,n_L}^{-1} O_p(n_L^{-1/2}) = o_p(n_L^{-1/2})$ for sufficiently large $m,n_L$ by the properties of a geometric sum. This is sufficient to show that the last $mp$ components of $ \sum_{s=1}^\infty \bm{\Lambda}^s  \bm{B}^{-1} \nabla Q(\dot{\bm{\theta}})$ are of smaller order component-wise than $\bm{B}^{-1} \nabla Q(\dot{\bm{\theta}})$, so that the result for the prediction gap holds. In particular, we thus know that $\hat{\bm{b}} - \dot{\bm{b}} = O_p(n_L^{-1/2})$. Furthermore, we also know that the convergence rate of $\hat{\bmbeta} - \dot{\bmbeta}$ is at least of order $O_p(n_L^{-1/2})$. As a result, we can choose $\delta_{m,n_L} = n_L^{1/2}$ without affecting the component-wise order properties of $\bm{\Lambda}_\delta$. Applying $\delta_{m,n_L} = n_L^{1/2}$, we thus have that $\sum_{s=1}^\infty \delta_{m,n_L}^{-s} \bm{\Lambda}_\delta^s  \bm{B}^{-1} \nabla Q(\dot{\bm{\theta}})$ is at most of order $O_p(n_L^{-1})$ component-wise. This is smaller than $O_p(N^{-1/2})$ when $m n_L^{-1} \rightarrow 0$, and the required result follows.

\subsection{Unconditional Regime} \label{sec:uncond_remainder}
For the unconditional regime, we consider two cases: when $m n_L^{-1} \rightarrow 0$, and when $m n_U^{-1} \rightarrow \infty$ but $m n_L^{-2} \rightarrow 0$. 

First, consider the case when $m n_L^{-1} \rightarrow 0$. From previous derivations, we know that when $m n_L^{-1} \rightarrow 0$, the quantity $\bm{B}^{-1} \nabla Q(\dot{\bm{\theta}})$ is of order $O_p(m^{-1/2})$ for the first $p$ components and $O_p(m^{-1/2})$ for the last $mp$ components. By the two properties of $\bm{\Lambda}_\delta^s$ noted above, we therefore know that $\bm{\Lambda}_\delta^s  \bm{B}^{-1} \nabla Q(\dot{\bm{\theta}})$ is at most $O_p(m^{-1/2})$ component-wise for any $s$. Hence \\ $\sum_{s=1}^\infty \delta_{m,n_L}^{-s} \bm{\Lambda}_\delta^s  \bm{B}^{-1} \nabla Q(\dot{\bm{\theta}}) = \delta_{m,n_L}^{-1} O_p(m^{-1/2}) = o_p(m^{-1/2})$ for sufficiently large $m,n_L$, by the properties of a geometric sum. The required result follows from this. Furthermore, this implies we may set $\delta_{m,n_L} = m^{1/2}$ without affecting the component-wise order properties of $\bm{\Lambda}_\delta$. Applying $\delta_{m,n_L} = m^{1/2}$, we thus have that $\sum_{s=1}^\infty \delta_{m,n_L}^{-s} \bm{\Lambda}_\delta^s  \bm{B}^{-1} \nabla Q(\dot{\bm{\theta}})$ is at most of order $O_p(m^{-1})$ component-wise.

Next, consider the case when $m n_U^{-1} \rightarrow \infty$ and $m n_L^{-2} \rightarrow 0$. From previous derivations, we know in this setting it holds that $\bm{B}^{-1} \nabla Q(\dot{\bm{\theta}})$ is of order $O_p(m^{-1/2})$ for the first $p$ components and $O_p(n_L^{-1/2})$ for the last $mp$ components. By the two properties of $\bm{\Lambda}_\delta^s$ noted above, we therefore obtain that $\bm{\Lambda}_\delta^s  \bm{B}^{-1} \nabla Q(\dot{\bm{\theta}})$ is at most $O_p(n_L^{-1/2})$ component-wise for any $s$. Hence $\sum_{s=1}^\infty \delta_{m,n_L}^{-s} \bm{\Lambda}_\delta^s  \bm{B}^{-1} \nabla Q(\dot{\bm{\theta}}) = \delta_{m,n_L}^{-1} O_p(n_L^{-1/2}) = o_p(n_L^{-1/2})$ for sufficiently large $m,n_L$, by the properties of a geometric sum. This is sufficient to show that the last $mp$ components of $ \sum_{s=1}^\infty \bm{\Lambda}^s  \bm{B}^{-1} \nabla Q(\dot{\bm{\theta}})$ are of smaller order component-wise than $\bm{B}^{-1} \nabla Q(\dot{\bm{\theta}})$, so that the result for the prediction gap holds. In particular, we thus know that $\hat{\bm{b}} - \dot{\bm{b}} = O_p(n_L^{-1/2})$. Furthermore, we also know that the convergence rate of $\hat{\bmbeta} - \dot{\bmbeta}$ is at least $O_p(n_L^{-1/2})$. As a result, we can set $\delta_{m,n_L} = n_L^{1/2}$ without affecting the component-wise order properties of $\bm{\Lambda}_\delta$. Applying $\delta_{m,n_L} = n_L^{1/2}$, we thus have that $\sum_{s=1}^\infty \delta_{m,n_L}^{-s} \bm{\Lambda}_\delta^s  \bm{B}^{-1} \nabla Q(\dot{\bm{\theta}})$ is at most of order $O_p(n_L^{-1})$ component-wise. This is smaller than $O_p(m^{-1/2})$ when $m n_U^{-1} \rightarrow \infty$, $m n_L^{-2} \rightarrow 0$ and the result follows.

\section{Unpartnered Fixed Effects}
\setcounter{equation}{0}

\subsection{Generalised Linear Models}

In the special case when $\dot{\bm{G}} = \bm{0}_{p \times p}$, i.e., all fixed effects are unpartnered in the true data generating process, the GLMM reduces to a GLM. We may then obtain a result based on a special case of our results in the conditional case, when all the true random effects are equal to zero. The result is as follows.

\begin{customthm}{A1}
    Assume Conditions (C1) - (C5) are satisfied and $m n_L^{-1} \rightarrow 0$. Then as $m,n_L \rightarrow \infty$ and when the true vector of random effects $\dot{\bm{b}} = \bm{0}_{mp}$, it holds that $\bm{A} \bm{D} (\hat{\bm{\theta}} - \dot{\bm{\theta}}) \overset{D}{\rightarrow} N(\bm{0},\bm{\Omega})$.
\end{customthm}

\subsection{Linear Mixed Models}

Suppose for $i=1,\ldots,m$ and $j=1,\ldots,n_i$ we observe data from the model $y_{ij} = \bm{x}_{ij}^\top \bm{\beta} + \bm{z}_{ij}^\top \bm{b}_i + \bm{x}_{ij}^{(O) \top} \bm{\beta}^{(O)} + \epsilon_{ij}$, where $\bm{x}_{ij} = \bm{z}_{ij}$ for all $(i,j)$, $\bm{b}_i \overset{i.i.d.}{\sim} N(\bm{0}, \hat{\bm{G}})$ and $\epsilon_{ij} \overset{i.i.d.}{\sim} N(0,\phi)$. Note that this is part of the exponential family. Partition $\bm{\beta} = (\bm{\beta}^{(P) \top}, \bm{\beta}^{(U) \top})^\top $, corresponding to the $p_P$ partnered and $p_U$ unpartnered fixed effects ($p_P + p_U = p$). That is, if we partition $\bm{b}_i = (\bm{b}_i^{(P) \top}, \bm{b}^{(U) \top}_i)^\top$, then $\bm{b}^{(U)}_i = \bm{0}_{p_U}$ for all $i$, and the corresponding elements in $\dot{\bm{G}}$ are zero. Let $\bmtheta^{\times} = (\bm{\beta}^\top, \bm{b}_1^{(P) \top}, \ldots, \bm{b}_m^{(P) \top}, \bm{b}^{(U) \top}_1, \ldots, \bm{b}^{(U) \top}_m , \bm{\beta}^{(O) \top})^\top$, $\bmtheta^{-} = (\bm{\beta}^\top, \bm{b}^\top, \bm{\beta}^{(O) \top})^\top$, \\ $\bm{D}^{\times} = \text{diag}(m^{1/2} \bm{1}_{p_P},N^{1/2} \bm{1}_{p_U},m^{1/2} \bm{1}_{m p_P}, n_1^{1/2} \bm{1}_{p_U}, \ldots, n_m^{1/2} \bm{1}_{p_U}, N^{1/2} \bm{1}_{p_O})$, and \\ $\bm{D}^{-} = \text{diag}(m^{1/2} \bm{1}_{p_P},N^{1/2} \bm{1}_{p_U},n_1^{1/2} \bm{1}_{p}, \ldots, n_m^{1/2} \bm{1}_{p}, N^{1/2} \bm{1}_{p_O})$. Also let $\bm{X}_i^{(O)} = [\bm{x}_{i1}^{(O)}, \ldots, \bm{x}_{in_i}^{(O)}]^\top$ and $\bm{X}^{(O)} = [\bm{X}_1^{(O) \top}, \ldots , \bm{X}_m^{(O) \top}]^\top$. The $p_O$ orthogonal fixed effects $\bm{x}_{ij}^{(O)}$ satisfy $\bm{X}^{(O) \top} \bm{Z} = \bm{0}_{p_O \times mp}$, for example orthogonal polynomials of $\bm{x}_{ij}$. This implies $\bm{X}_i^{(O) \top} \bm{X}_i = \bm{0}_{p_O \times p} $ for all $i$. For a $q \times \{(m+1)p + p_O\}$ matrix $\bm{A}^*$ with the finite selection property, we have the following. \\

\begin{customthm}{A2}
    Assume Conditions (C1) - (C4) are satisfied. Then as $m,n_L \rightarrow \infty$ and unconditional on the random effects $\dot{\bm{b}}$, it holds that
\begin{enumerate}
    \item  $\bm{A}^* \bm{D}^{\times} (\hat{\bm{\theta}}^{\times} - \dot{\bm{\theta}}^{\times}) \overset{D}{\rightarrow} N(\bm{0},\bm{\Omega}_a)$ if $mn_L^{-1} \rightarrow 0$, and
    \item  $\bm{A}^* \bm{D}^{-} (\hat{\bm{\theta}}^{-} - \dot{\bm{\theta}}^{-}) \overset{D}{\rightarrow} N(\bm{0},\bm{\Omega}_b)$ if $mn_U^{-1} \rightarrow \infty$,
\end{enumerate}
where
\begin{align*}
    \bm{\Omega}_a &= \underset{m,n_L \rightarrow \infty}{\lim} \bm{A}^* \left[ \begin{matrix}
        \dot{\bm{G}}_{[1:p_P,1:p_P]} & \bm{0}_{p_P \times p_U} & \bm{1}_{m}^\top \otimes \dot{\bm{G}}_{[1:p_P,1:p_P]} & \bm{0}_{p_P \times mp_U} & \bm{0}_{p_P \times p_O} \\
        \bm{0}_{p_U \times p_P} & \bm{\Omega}_1 & \bm{0}_{p_U \times mp_P} & \bm{0}_{p_U \times mp_U} & \bm{0}_{p_U \times p_O} \\
        \bm{1}_{m} \otimes \dot{\bm{G}}_{[1:p_P,1:p_P]} & \bm{0}_{mp_P \times p_U} & \bm{1}_{m \times m} \otimes \dot{\bm{G}}_{[1:p_P,1:p_P]} & \bm{0}_{mp_P \times mp_U} & \bm{0}_{mp_P \times p_O} \\
        \bm{0}_{mp_U \times p_P} & \bm{0}_{mp_U \times p_U} & \bm{0}_{mp_U \times mp_P} & \bm{\Omega}_2 & \bm{0}_{mp_U \times p_O} \\
        \bm{0}_{p_O \times p_P} & \bm{0}_{p_O \times p_U} & \bm{0}_{p_O \times mp_P} & \bm{0}_{p_O \times mp_U} & \bm{\Omega}_3 \\
    \end{matrix} \right] \bm{A}^{* \top} \\
    \bm{\Omega}_b &= \underset{m,n_L \rightarrow \infty}{\lim} \bm{A}^* \left[ \begin{matrix}
        \dot{\bm{G}}_{[1:p_P,1:p_P]} & \bm{0}_{p_P \times p_U} & \bm{0}_{p_P \times mp} & \bm{0}_{p_P \times p_O} \\
        \bm{0}_{p_U \times p_P} & \bm{\Omega}_1 & \bm{0}_{p_U \times mp} & \bm{0}_{p_U \times p_O} \\
        \bm{0}_{mp \times p_P} & \bm{0}_{mp \times p_U}  & \bm{\Omega}_4 & \bm{0}_{mp \times p_O} \\
        \bm{0}_{p_O \times p_P} & \bm{0}_{p_O \times p_U} & \bm{0}_{p_O \times mp} & \bm{\Omega}_3 \\
    \end{matrix} \right] \bm{A}^{* \top} \\
    \bm{\Omega}_1 &= \left\{ \frac{\dot{\phi}}{m} \sum_{i=1}^m \frac{n}{n_i} \left(\frac{\bm{X}^\top_i \bm{X}_i}{n_i} \right)^{-1} \right\}_{[(p-p_U+1):p,(p-p_U+1):p]} \\
    \bm{\Omega}_2 &= \text{bdiag} \left[ \left\{ \dot{\phi} \left(\frac{\bm{X}^\top_1 \bm{X}_1}{n_1} \right)^{-1} \right\}_{[(p-p_U+1):p,(p-p_U+1):p]} , \ldots, \left\{ \dot{\phi} \left(\frac{\bm{X}^\top_m \bm{X}_m}{n_m} \right)^{-1} \right\}_{[(p-p_U+1):p,(p-p_U+1):p]} \right] \\
    \bm{\Omega}_3 &= \dot{\phi} \left(\frac{\bm{X}^{(O) \top} \bm{X}^{(O)}}{N} \right)^{-1} \\
    \bm{\Omega}_4 &= \text{bdiag} \left[ \left\{ \dot{\phi} \left(\frac{\bm{X}^\top_1 \bm{X}_1}{n_1} \right)^{-1} \right\} , \ldots, \left\{ \dot{\phi} \left(\frac{\bm{X}^\top_m \bm{X}_m}{n_m} \right)^{-1} \right\} \right] .
\end{align*}
\end{customthm}


\begin{proof}

We use the same approach as previous proofs and examine the Taylor expansion \eqref{eqn:taylor1}. In this case, we have the expressions

\begin{align*}
\nabla Q(\dot{\bm{\theta}}) &= \left[ \begin{matrix} \dot{\phi}^{-1} \bm{X}^\top (\bm{y} - \dot{\bm{\mu}}) \\ \dot{\phi}^{-1} \bm{Z}^\top (\bm{y} - \dot{\bm{\mu}}) - (\bm{I}_m \otimes \hat{\bm{G}}^{-1})\dot{\bm{b}} \\ \dot{\phi}^{-1} \bm{X}^{(O) \top} (\bm{y} - \dot{\bm{\mu}}) \end{matrix} \right], \\
\bm{B}(\dot{\bm{\theta}}) &= - \nabla^2 Q(\dot{\bm{\theta}}) = \left[ \begin{matrix}
\bm{X}^\top \dot{\bm{W}} \bm{X} & \bm{X}^\top \dot{\bm{W}} \bm{Z} & \bm{X}^\top \dot{\bm{W}} \bm{X}^{(O)} \\
\bm{Z}^\top \dot{\bm{W}} \bm{X} & \bm{Z}^\top \dot{\bm{W}} \bm{Z} + \bm{I}_m \otimes \hat{\bm{G}}^{-1} & \bm{Z}^\top \dot{\bm{W}} \bm{X}^{(O)} \\
\bm{X}^{(O) \top} \dot{\bm{W}} \bm{X} & \bm{X}^{(O) \top} \dot{\bm{W}} \bm{Z} &  \bm{X}^{(O) \top} \dot{\bm{W}} \bm{X}^{(O)}
\end{matrix} \right] \\
&= \dot{\phi}^{-1} \left[ \begin{matrix}
\bm{X}^\top \bm{X} & \bm{X}^\top \bm{Z} & \bm{0}_{p \times p_O} \\
\bm{Z}^\top  \bm{X} & \bm{Z}^\top \bm{Z} + \bm{I}_m \otimes \hat{\bm{G}}^{-1} & \bm{0}_{mp \times p_O} \\
\bm{0}_{p_O \times p} & \bm{0}_{p_O \times mp} &  \bm{X}^{(O) \top} \bm{X}^{(O)}
\end{matrix} \right],
\end{align*}
where the last equality follows from the fact that $\dot{\bm{W}} = \dot{\phi}^{-1} \bm{I}_{N}$ and $\bm{X}^{(O) \top} \bm{Z} = \bm{0}_{p_O \times mp}$. Since $\bm{B}(\dot{\bm{\theta}})$ is block diagonal, we thus know that expressions \eqref{eq:fixed} and \eqref{eq:random} still hold. Recall
\begin{align*}
\hat{\bm{\beta}} - \dot{\bm{\beta}} &= m^{-1} \sum_{i=1}^m (\bm{X}_i^\top \dot{\bm{W}}_i \bm{X}_i)^{-1}  \dot{\phi}^{-1} \bm{X}_i^\top (\bm{y}_i - \dot{\bm{\mu}}_i )  + m^{-1} \sum_{i=1}^m \dot{\bm{b}}_i \\
&  -  m^{-1} \sum_{i=1}^m (\bm{X}_i^\top \dot{\bm{W}}_i \bm{X}_i + \hat{\bm{G}}^{-1})^{-1} \hat{\bm{G}}^{-1} \dot{\bm{b}}_i + \frac{1}{2} \{ \bm{B}^{-1} \bm{R}(\tilde{\bmtheta})\}_{[1:p]}  \\
& + O_p(n_L^{-1}) \Bigg\{ m^{-1} \sum_{i=1}^m (\bm{X}_i^\top \dot{\bm{W}}_i \bm{X}_i + \hat{\bm{G}}^{-1})^{-1}  \dot{\phi}^{-1} \bm{X}_i^\top (\bm{y}_i - \dot{\bm{\mu}}_i )  + m^{-1} \sum_{i=1}^m \dot{\bm{b}}_i \\
&  -  m^{-1} \sum_{i=1}^m (\bm{X}_i^\top \dot{\bm{W}}_i \bm{X}_i + \hat{\bm{G}}^{-1})^{-1} \hat{\bm{G}}^{-1} \dot{\bm{b}}_i \Bigg\} + m^{-1} \sum_{i=1}^m O_p(n_L^{-2}) \dot{\phi}^{-1} \bm{X}_i^\top (\bm{y}_i - \dot{\bm{\mu}}_i )
\end{align*}
and
\begin{align*}
\hat{\bm{b}} - \dot{\bm{b}} &= - \bm{1}_m \otimes m^{-1} \sum_{i=1}^m \dot{\bm{b}}_i + \bm{B}_3^{-1} \{ \dot{\phi}^{-1} \bm{Z}^\top (\bm{y} - \dot{\bm{\mu}}) \} \\ 
&+ O_p(N^{-1/2}) + O_p(n_L^{-1}) + \frac{1}{2} \{ \bm{B}^{-1} \bm{R}(\tilde{\bmtheta})\}_{[p +1:(m+1)p]}.
\end{align*}
In the LMM case, the remainder term in the Taylor expansion is zero. Thus the dominating term on the right hand side for $\hat{\bm{\beta}}^{(U)} - \dot{\bm{\beta}}^{(U)}$ are the last $p_U$ components of $m^{-1} \sum_{i=1}^m (\bm{X}_i^\top \dot{\bm{W}}_i \bm{X}_i)^{-1}  \dot{\phi}^{-1} \bm{X}_i^\top (\bm{y}_i - \dot{\bm{\mu}}_i )$, since the last $p_U$ components of $m^{-1} \sum_{i=1}^m \dot{\bm{b}}_i$ are zero. Noting that $\bm{y}_i - \dot{\bm{\mu}}_i = (\epsilon_{i1}, \ldots, \epsilon_{in_i})^\top =: \bm{\epsilon}_i$, the result for the unpartnered fixed effects follows after normalising by $N^{1/2}$.

Next, again from the Taylor expansion we have from the block-diagonal structure of $\bm{B}(\dot{\bm{\theta}})$ that $\hat{\bm{\beta}}^{(O)} - \dot{\bm{\beta}}^{(O)} = (\bm{X}^{(O) \top} \bm{X}^{(O)})^{-1} \bm{X}^{(O) \top} (\bm{y} - \dot{\bm{\mu}})$ and the result follows after normalising by $N^{1/2}$ since $\bm{y} - \dot{\bm{\mu}} = (\epsilon_{11}, \ldots, \epsilon_{m n_m})^\top =: \bm{\epsilon}$.

Finally, the result for the unpartnered random effects follows from the fact that the last $p_U$ components of $m^{-1} \sum_{i=1}^m \dot{\bm{b}}_i$ are zero so that the dominating term on the right hand side is $(\bm{X}_i^\top \bm{X}_i)^{-1} \bm{X}_i^\top (\bm{y}_i - \dot{\bm{\mu}}_i)$, and normalising by $n_i$.

The proofs for the partnered fixed and random effects are analogous to the proofs of Theorems 2 and 4, based on examining the leading term in the Taylor expansion. 

For the joint behaviour of the estimator, we examine the joint behaviour of the leading terms in the Taylor Expansion. Note that $\bm{\epsilon}$ is multivariate normal with covariance matrix $\dot{\phi} \bm{I}_N$, $\dot{\bm{b}}$ is multivariate normal with covariance matrix $\bm{I}_m \otimes \dot{\bm{G}}$, $\bm{\epsilon}$ and $\dot{\bm{b}}$ are independent, and all the leading terms in the Taylor expansion are linear functions of $\bm{\epsilon}$ and $\dot{\bm{b}}$. To determine the joint behaviour of the estimator it is thus sufficient to derive the limiting covariance between the normalised leading terms, as we see (from the leading terms) that the estimator itself is also (asymptotically) multivariate normal. For example,
\begin{align*}
    &\text{Cov} \left\{N^{1/2} m^{-1} \sum_{i=1}^m (\bm{X}_i^\top \bm{X}_i)^{-1} \bm{X}_i^\top (\bm{y}_i - \dot{\bm{\mu}}_i ), N^{1/2}(\bm{X}^{(O) \top} \bm{X}^{(O)})^{-1} \bm{X}^{(O) \top} (\bm{y} - \dot{\bm{\mu}}) \right\} \\
    &= n \dot{\phi} \sum_{i=1}^m (\bm{X}_i^\top \bm{X}_i)^{-1} \bm{X}_i^\top \bm{X}^{(O)}_i (\bm{X}^{(O) \top} \bm{X}^{(O)})^{-1}  \\
    &= \bm{0}_{p \times p_O}
\end{align*}
due to the mutual independence of the $\epsilon_{ij}$ and orthogonality condition of $\bm{X}^{(O)}$. The pairwise limiting covariances between the leading terms can all be derived in a similar way and the result follows. Notice here that quantities with different convergence rates are always asymptotically uncorrelated and independent in this case.

\end{proof}

Note that the results hold by the Lindeberg-Feller Central Limit Theorem even if the true distribution of $\epsilon_{ij}$ is not normal, as long as it is mean zero with finite variance. Also note that condition (C5) is no longer required, and that there is no restriction on the relative rates of $m$ and $n_L$, since there is no remainder term to deal with. Our result is consistent with the results derived in \citet{lyu2021asymptotics,lyu2021increasing} who also derive a $N^{1/2}$ convergence rate for unpartnered fixed effects that are time-varying.

In practice, we do not know if a fixed effect is truly partnered with a random effect or not, and therefore the correct asymptotic distribution and convergence rate is also unknown. In this case, an appropriate finite sample approximation, given consistent estimators $\tilde{\bm{G}}$ and $\tilde{\phi}$ of $\dot{\bm{G}}$ and $\dot{\phi}$ respectively, is 
\begin{align*}
    \hat{\bm{\beta}} - \dot{\bm{\beta}} \sim N \left\{ \bm{0}, m^{-1} \tilde{\bm{G}} + N^{-1} \frac{\tilde{\phi}}{m} \sum_{i=1}^m \frac{n}{n_i} \left(\frac{\bm{X}^\top_i \bm{X}_i}{n_i} \right)^{-1} \right\},
\end{align*}
which is based on the distribution of $m^{-1} \sum_{i=1}^m (\bm{X}_i^\top \dot{\bm{W}}_i \bm{X}_i)^{-1}  \dot{\phi}^{-1} \bm{X}_i^\top (\bm{y}_i - \dot{\bm{\mu}}_i )  + m^{-1} \sum_{i=1}^m \dot{\bm{b}}_i$, noting that the two terms are independent.


\newpage

\section{Additional Simulation Results}
\setcounter{equation}{0}

\subsection{Main Results for the Conditional Regime}

Figures \ref{fig: cond_coverage_pois}, \ref{fig: cond_coverage_bin} and \ref{fig: cond_shapiro} display the empirical coverage probabilities and results from applying the Shapiro-Wilk test, respectively, under the conditional regime and for the 25 combinations of $(m,n)$. Although our coverage intervals often undercovered or overcovered for small cluster sizes e.g., $n=25$, especially for the Bernoulli case, they all moved toward nominal coverage as $n$ becomes larger than $m$. This is consistent with Theorem \ref{thm:conditionalnormality}. The fact the empirical coverage probabilities were slow in tending towards the nominal 95\% level was also not overly surprising, as the third derivative term in the corresponding Taylor expansion is $O_p(m^{1/2} n_L^{-1/2})$. The Shapiro-Wilk tests overall did not indicate any evidence of deviations away from normality when $m<n$, although there were occasionally a few $p$-values less than 0.05. Overall, these results strongly support the use of Theorem \ref{thm:conditionalnormality} for inference under the conditional regime.

\begin{figure}[H]
\centering
\includegraphics[width=0.95\linewidth]{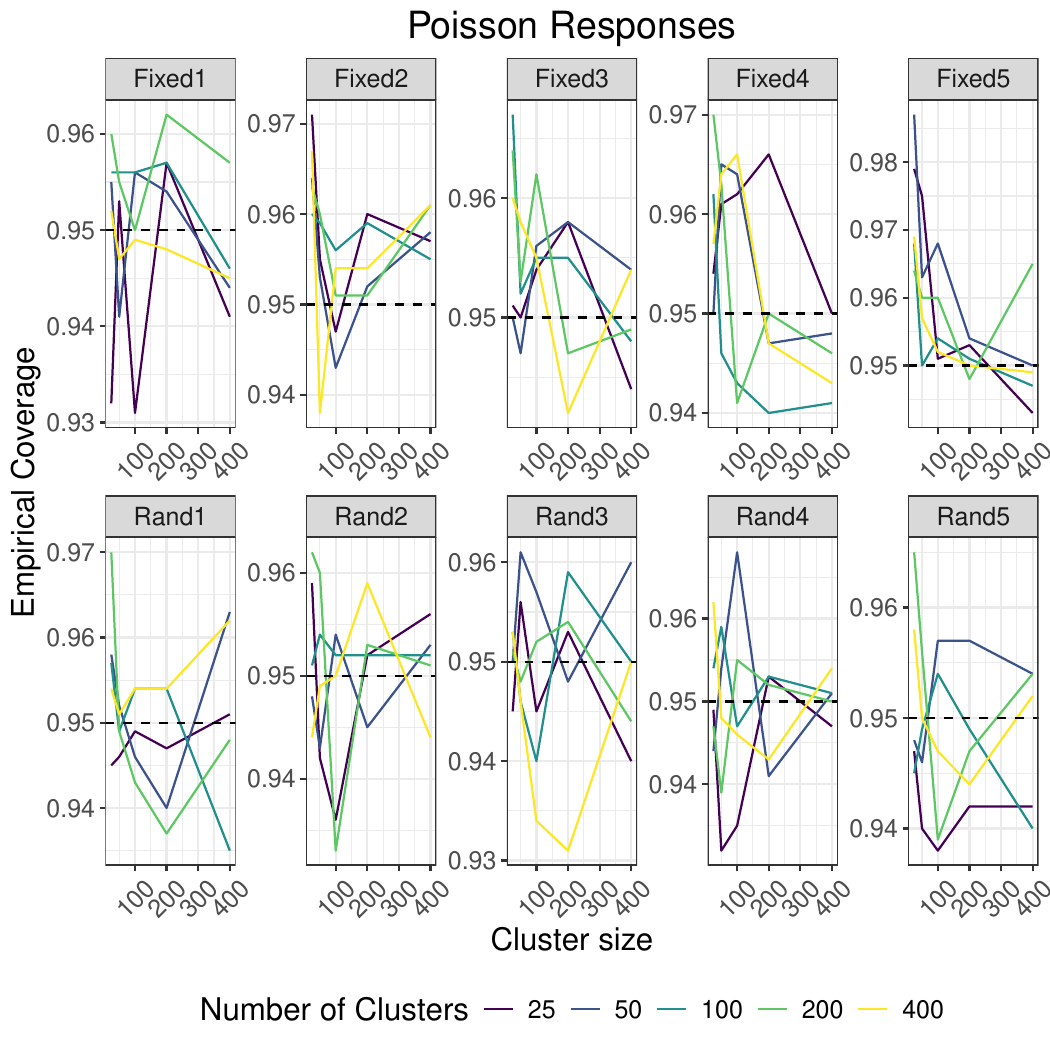}
\caption{Empirical coverage probability of 95\% coverage intervals for the five fixed and random effects estimates, obtained under the conditional regime with Poisson responses.} 
\label{fig: cond_coverage_pois}
\end{figure}

\begin{figure}[H]
\centering
\includegraphics[width=0.95\linewidth]{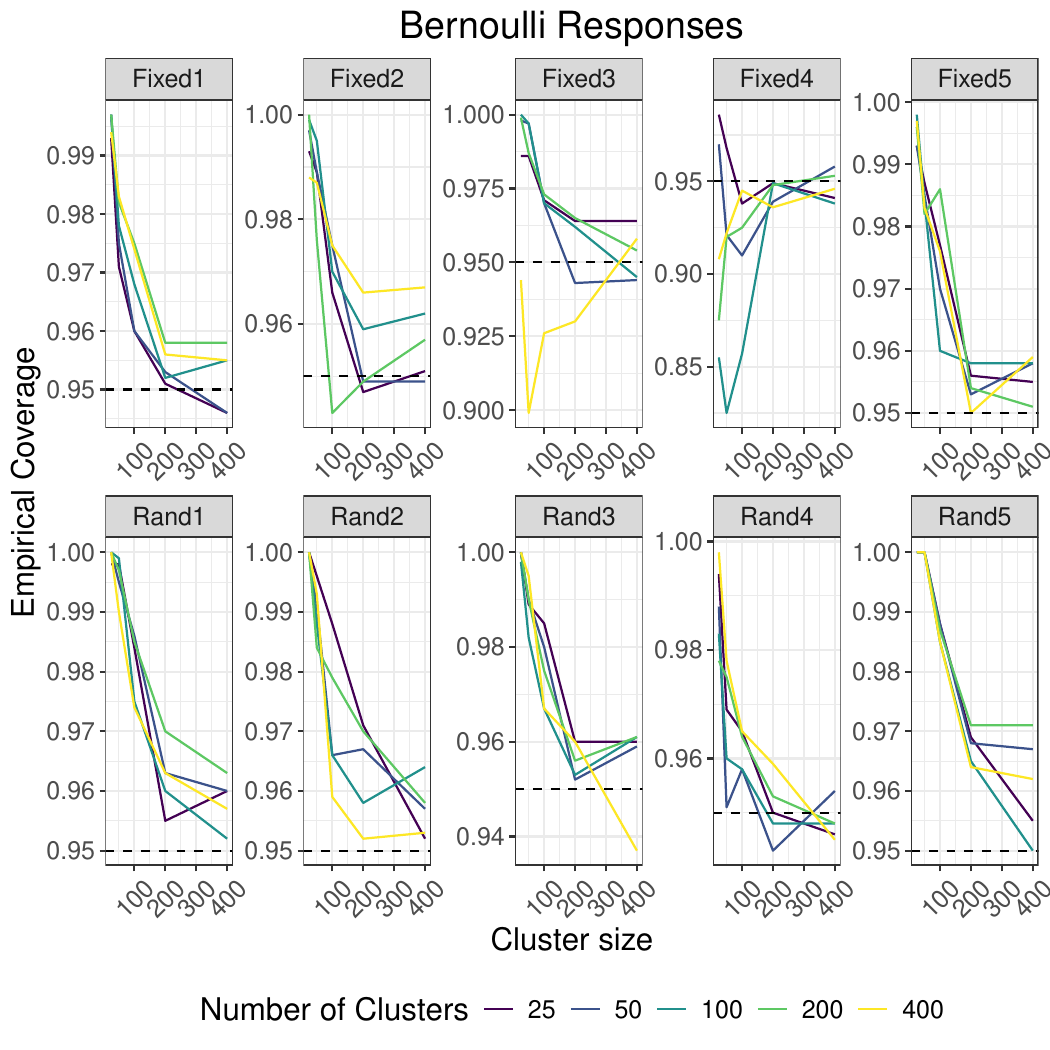}
\caption{Empirical coverage probability of 95\% coverage intervals for the five fixed and random effects estimates, obtained under the conditional regime with Bernoulli responses.} 
\label{fig: cond_coverage_bin}
\end{figure}

\begin{figure}[H]
\includegraphics[width=0.7\linewidth]{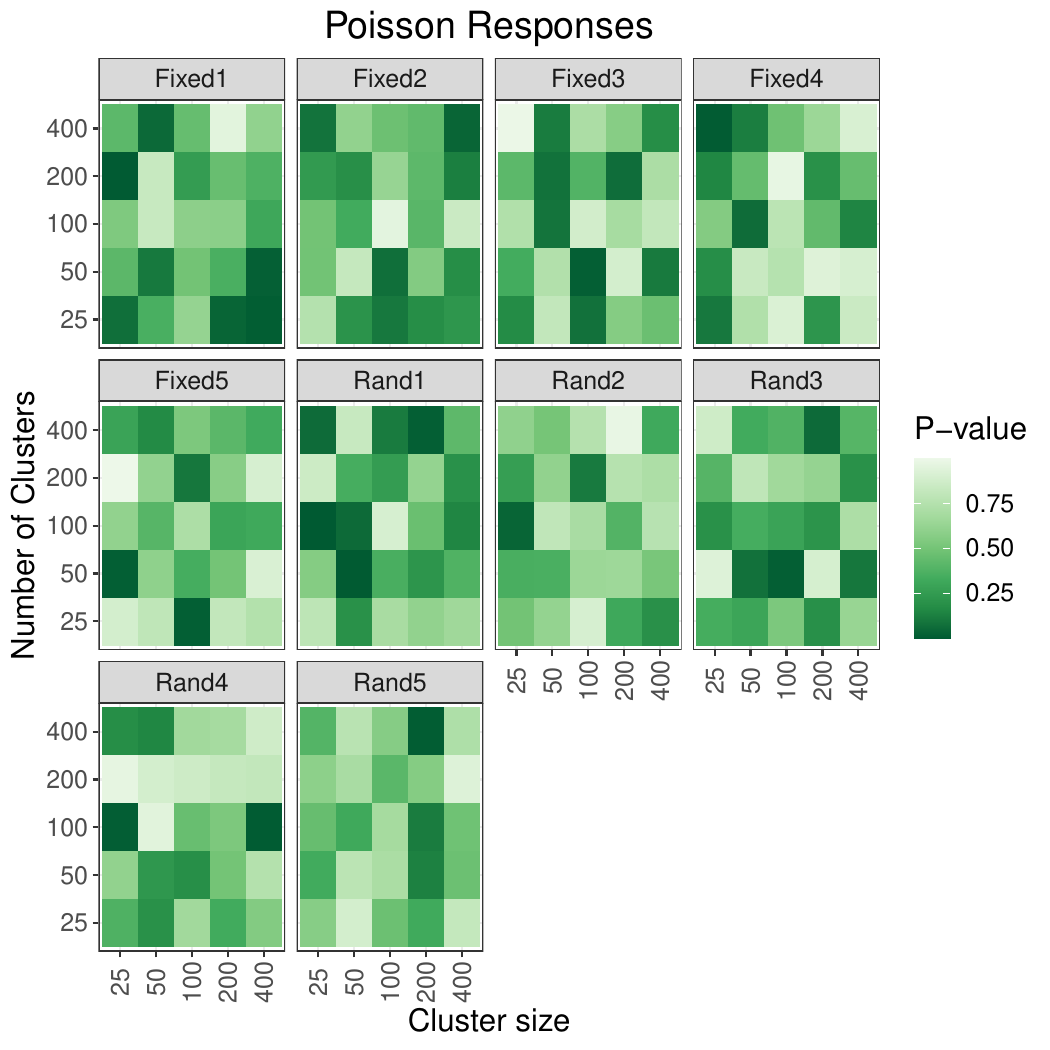}
\includegraphics[width=0.7\linewidth]{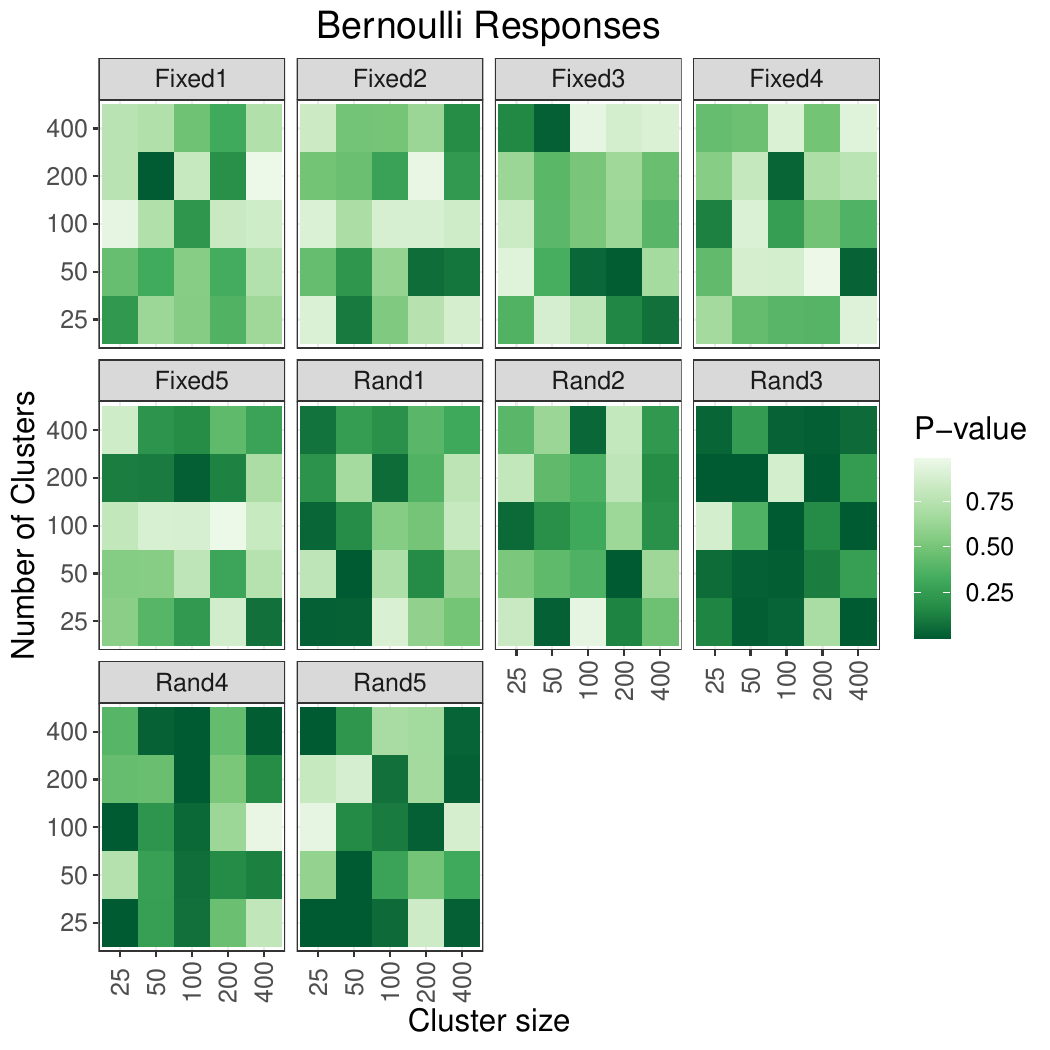}
\caption{$p$-values from Shapiro-Wilk tests applied to the fixed and random effects estimates obtained using maximum PQL estimation, under the conditional regime.} 
\label{fig: cond_shapiro}
\end{figure}

\begin{figure}[H]
\includegraphics[width=0.49\linewidth]{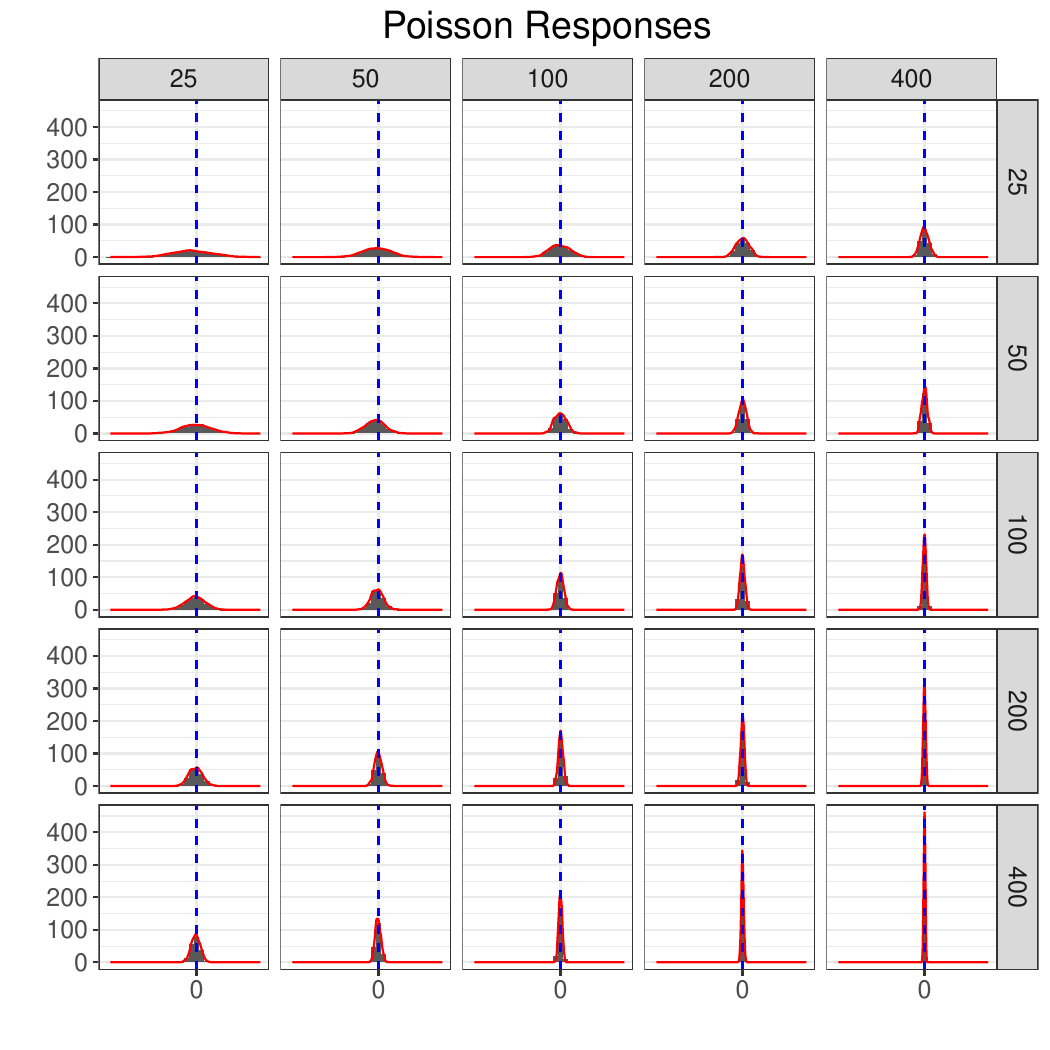}
\includegraphics[width=0.49\linewidth]{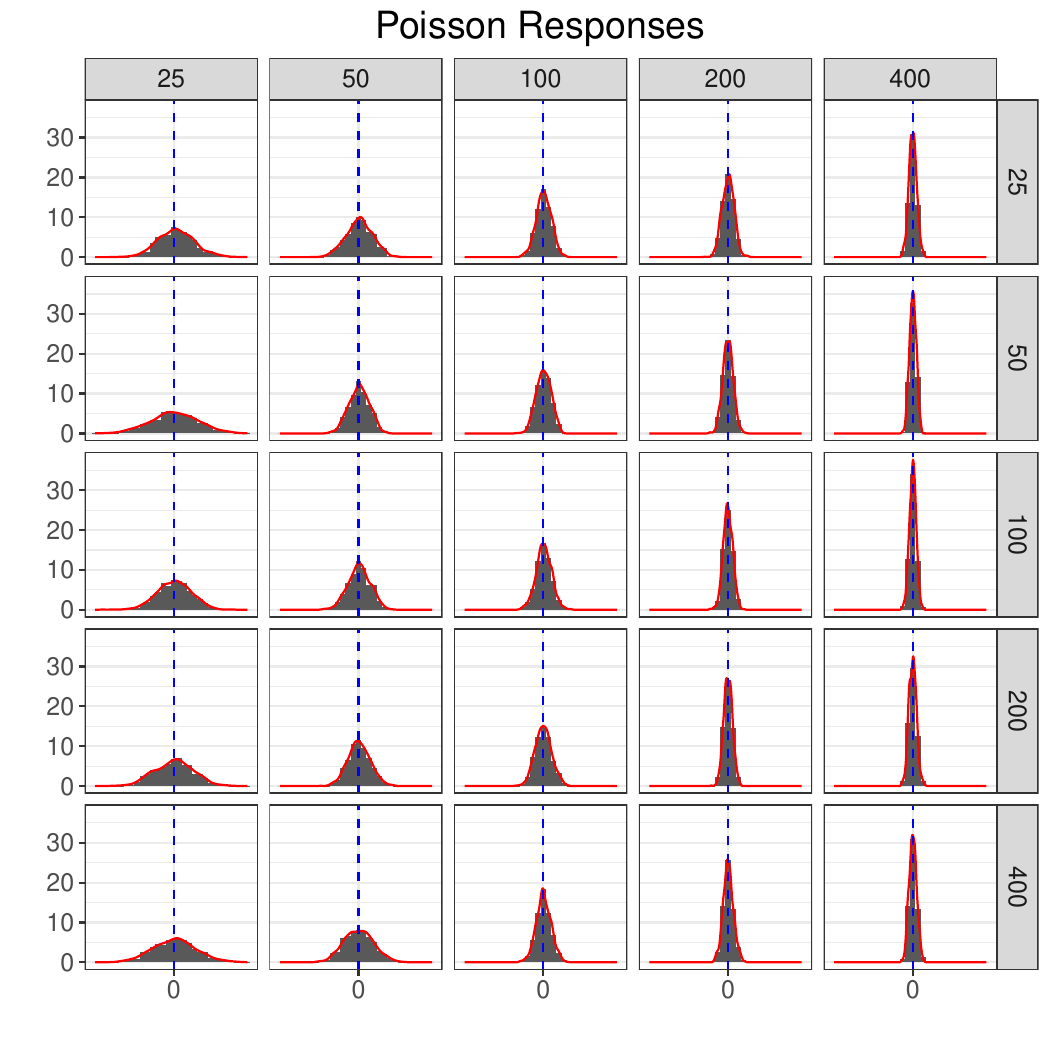}
\includegraphics[width=0.49\linewidth]{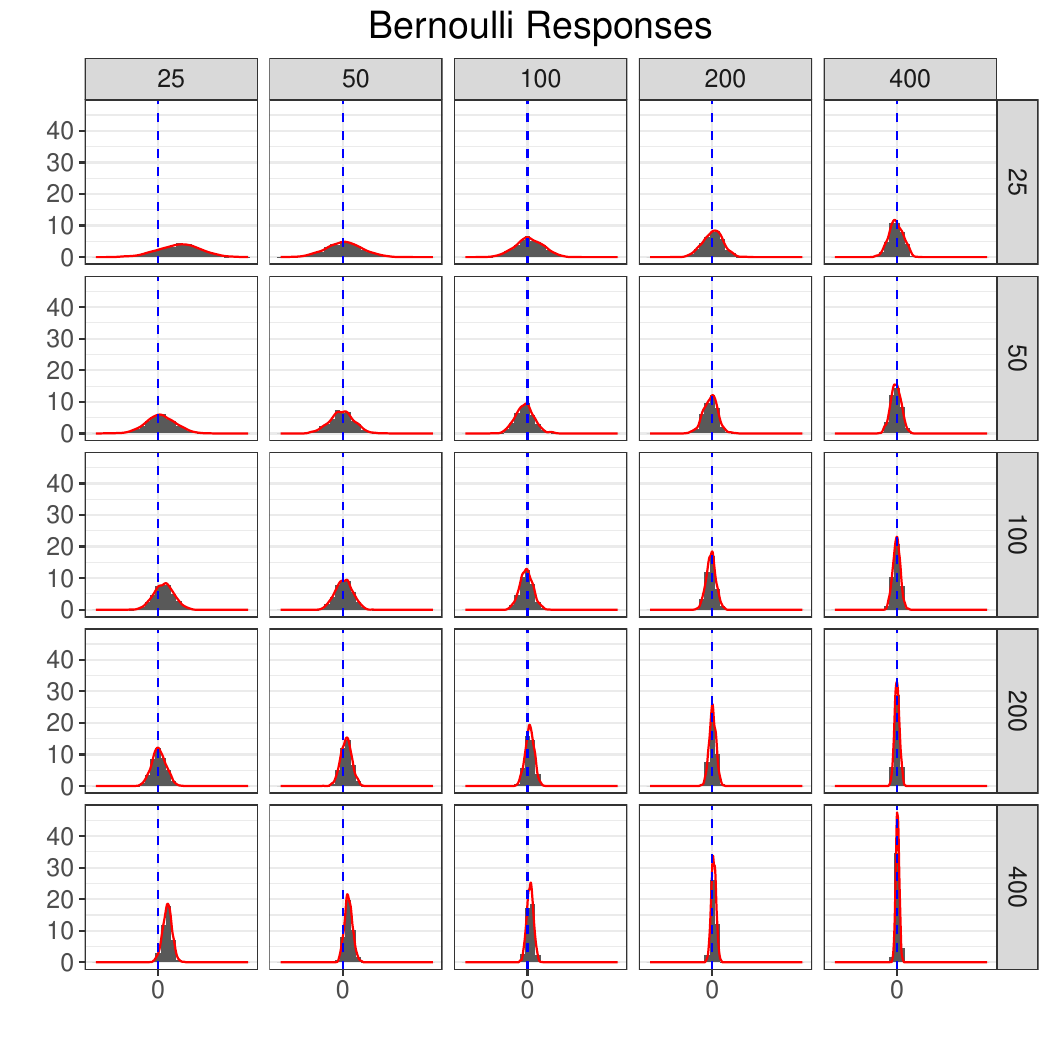}
\includegraphics[width=0.49\linewidth]{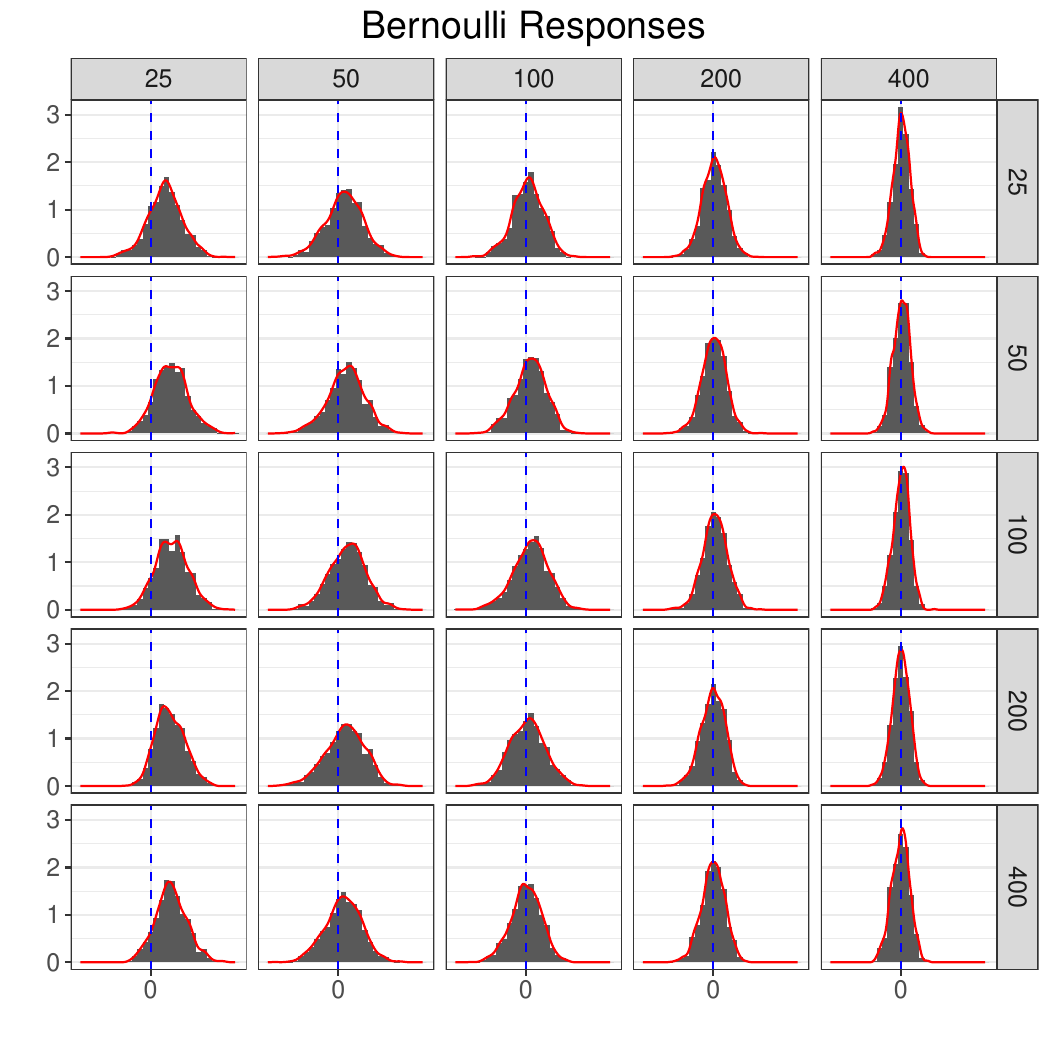}
\caption{Histograms for the third components of $\hat{\bmbeta} - \dot{\bmbeta}$ (left panels) and $\hat{\bm{b}}_1 - \dot{\bm{b}}_1$ (right panels), under the conditional regime. Vertical facets represent the cluster sizes, while horizontal facets represent the number of clusters. The dotted blue line indicates zero, and the red curve is a kernel density smoother.} 
\end{figure}

\newpage
\subsection{Frobenius Norm}

\begin{table}[H]
\caption{Empirical mean Frobenius norm of the difference between estimated and true random effects covariance matrix.}
\centering
\scalebox{0.75}{
\begin{tabular}{lr rrrrr c rrrrr}
\toprule[1.5pt]
\multicolumn{2}{r}{} & \multicolumn{5}{c}{Poisson}                         & \multicolumn{5}{c}{Bernoulli}                        \\ 
& $m$ & $ n = 25 $ & $ n = 50 $ & $ n = 100 $ & $ n = 200 $ & $n = 400$ & & $ n = 25 $ & $ n = 50 $ & $ n = 100 $ & $ n = 200 $ & $n = 400$ \\ 
\cmidrule{2-13}
\multirow{5}{*}{$ \hat{\bm{G}} = m^{-1} \sum_{i=1}^m \hat{\bm{b}}_i$} & $25$ & 1.06 & 1.06 & 1.06 & 1.05 & 1.06 & & 1.76 & 1.47 & 1.24 & 1.12 & 1.09       \\
& $50$  & 0.77 & 0.75 & 0.75 & 0.76 & 0.75 & & 1.79 & 1.40 & 1.03 & 0.83 & 0.77        \\
& $100$ & 0.54 & 0.54 & 0.54 & 0.54 & 0.54 & & 1.80 & 1.38 & 0.89 & 0.63 & 0.56        \\
& $200$ & 0.39 & 0.38 & 0.38 & 0.38 & 0.38 & & 1.80 & 1.35 & 0.82 & 0.51 & 0.41        \\
& $400$ & 0.27 & 0.27 & 0.27 & 0.27 & 0.27 & & 1.81 & 1.35 & 0.78 & 0.43 & 0.32        \\\\ 
\multirow{5}{*}{$ \hat{\bm{G}} = 0.25 \bm{I}_2 $} & $25$ & 1.02 & 1.01 & 1.03 & 1.05 & 1.04 & & 1.90 & 1.71 & 1.47 & 1.23 & 1.04       \\
& $50$  & 0.73 & 0.73 & 0.74 & 0.74 & 0.76 & & 1.90 & 1.70 & 1.44 & 1.15 & 0.91        \\
& $100$ & 0.56 & 0.53 & 0.52 & 0.53 & 0.54 & & 1.90 & 1.69 & 1.42 & 1.11 & 0.84        \\
& $200$ & 0.44 & 0.39 & 0.37 & 0.38 & 0.38 & & 1.90 & 1.68 & 1.41 & 1.09 & 0.79        \\
& $400$ & 0.38 & 0.29 & 0.27 & 0.27 & 0.27 & & 1.89 & 1.68 & 1.40 & 1.08 & 0.77        \\\\ 
\multirow{5}{*}{$ \hat{\bm{G}} = 0.5 \bm{I}_2 $} & $25$ & 1.02 & 1.03 & 1.04 & 1.05 & 1.05 & & 1.61 & 1.39 & 1.16 & 1.01 & 0.96       \\
& $50$  & 0.74 & 0.75 & 0.75 & 0.75 & 0.74 & & 1.61 & 1.34 & 1.07 & 0.86 & 0.75        \\
& $100$ & 0.53 & 0.52 & 0.54 & 0.54 & 0.54 & & 1.60 & 1.32 & 1.03 & 0.77 & 0.61        \\
& $200$ & 0.39 & 0.38 & 0.38 & 0.38 & 0.38 & & 1.59 & 1.31 & 1.01 & 0.73 & 0.52        \\
& $400$ & 0.30 & 0.27 & 0.27 & 0.27 & 0.27 & & 1.59 & 1.30 & 0.99 & 0.70 & 0.47        \\\\ 
\multirow{5}{*}{$ \hat{\bm{G}} = \bm{I}_2 $}  & $25$ & 1.06 & 1.05 & 1.04 & 1.04 & 1.06 & & 1.21 & 1.06 & 0.98 & 0.97 & 1.00       \\
& $50$  & 0.74 & 0.75 & 0.75 & 0.75 & 0.76 & & 1.17 & 0.93 & 0.78 & 0.73 & 0.71        \\
& $100$ & 0.53 & 0.53 & 0.54 & 0.53 & 0.54 & & 1.13 & 0.86 & 0.65 & 0.56 & 0.53        \\
& $200$ & 0.38 & 0.38 & 0.38 & 0.38 & 0.38 & & 1.12 & 0.82 & 0.58 & 0.44 & 0.39        \\
& $400$ & 0.27 & 0.27 & 0.27 & 0.27 & 0.27 & & 1.10 & 0.80 & 0.55 & 0.38 & 0.30        \\\\ 
\multirow{5}{*}{$ \hat{\bm{G}} = 2 \bm{I}_2 $}  & $25$ & 1.06 & 1.06 & 1.06 & 1.05 & 1.06 & & 0.84 & 0.98 & 1.03 & 1.04 & 1.05       \\
& $50$  & 0.75 & 0.74 & 0.75 & 0.76 & 0.75 & & 0.71 & 0.71 & 0.74 & 0.74 & 0.75        \\
& $100$ & 0.54 & 0.54 & 0.54 & 0.53 & 0.53 & & 0.56 & 0.51 & 0.53 & 0.53 & 0.54        \\
& $200$ & 0.38 & 0.38 & 0.38 & 0.38 & 0.38 & & 0.47 & 0.38 & 0.37 & 0.38 & 0.38        \\
& $400$ & 0.27 & 0.27 & 0.27 & 0.27 & 0.27 & & 0.42 & 0.29 & 0.27 & 0.27 & 0.27        \\\\ 
\multirow{5}{*}{$ \hat{\bm{G}} = 4 \bm{I}_2 $}  & $25$ & 1.06 & 1.06 & 1.06 & 1.06 & 1.06 & & 1.33 & 1.42 & 1.25 & 1.16 & 1.11       \\
& $50$  & 0.77 & 0.76 & 0.76 & 0.75 & 0.76 & & 1.18 & 1.07 & 0.93 & 0.83 & 0.80        \\
& $100$ & 0.55 & 0.54 & 0.54 & 0.54 & 0.54 & & 0.97 & 0.86 & 0.70 & 0.61 & 0.57        \\
& $200$ & 0.39 & 0.38 & 0.38 & 0.38 & 0.38 & & 0.86 & 0.72 & 0.55 & 0.45 & 0.41        \\
& $400$ & 0.27 & 0.27 & 0.27 & 0.27 & 0.27 & & 0.80 & 0.65 & 0.46 & 0.34 & 0.30        \\\\ 
\bottomrule[1.5pt]
\end{tabular}
}
\end{table}

\clearpage

\subsection{{\boldmath ${G}$} = 0.25  {\boldmath ${I}_2$}}

Using a large $\hat{\bm{G}}$ of $4 \bm{I}_2$ had the least impact on the results, while a small $\hat{\bm{G}}$, e.g., $0.25 \bm{I}_2$ had more of a noticeable impact at small sample sizes. This is not surprising since the latter corresponds to more shrinkage, such that larger sample sizes are needed before asymptotic results apply.

\begin{figure}[H]
\centering
\includegraphics[width=0.95\linewidth]{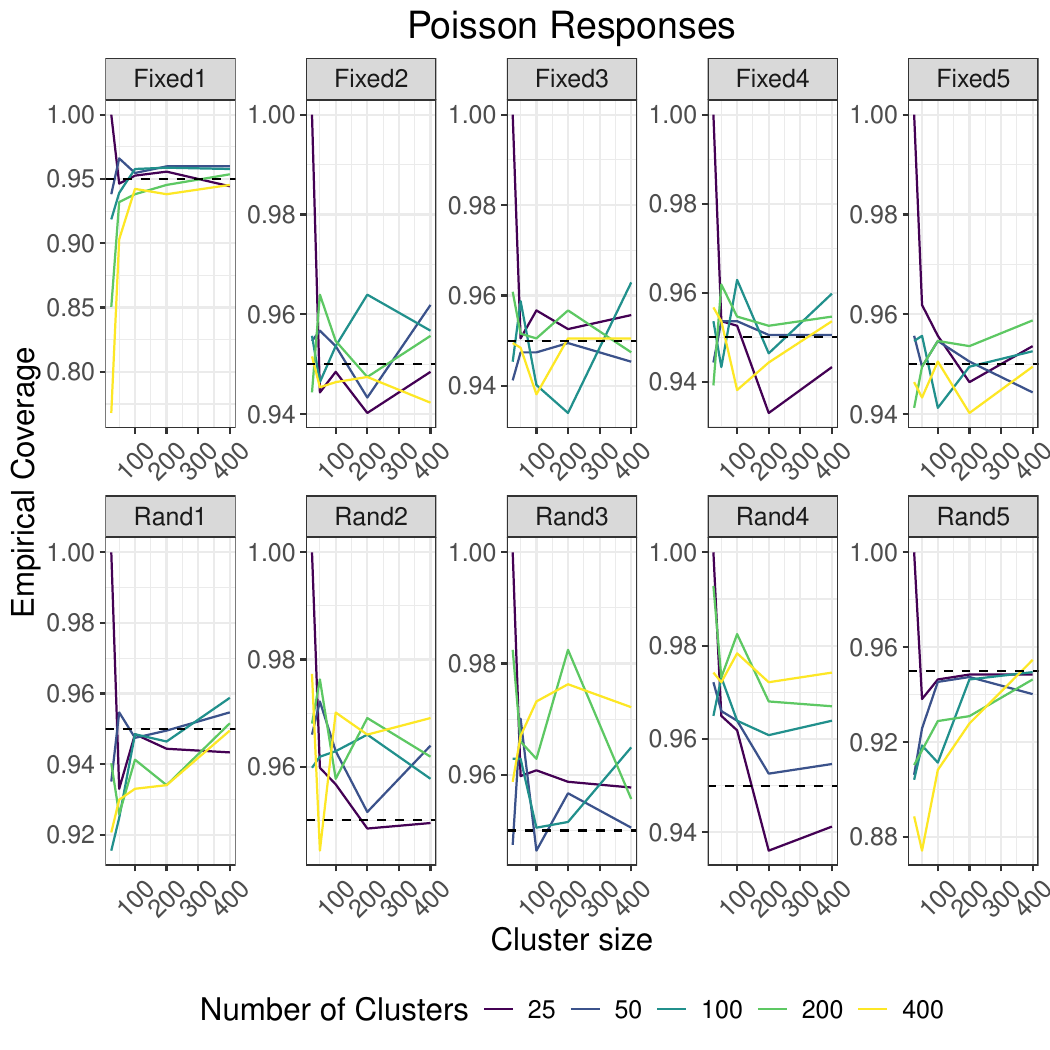}
\caption{Empirical coverage probability of 95\% coverage intervals for the five fixed and random effects estimates, obtained under the unconditional regime with Poisson responses.} 
\end{figure}

\begin{figure}[H]
\centering
\includegraphics[width=0.95\linewidth]{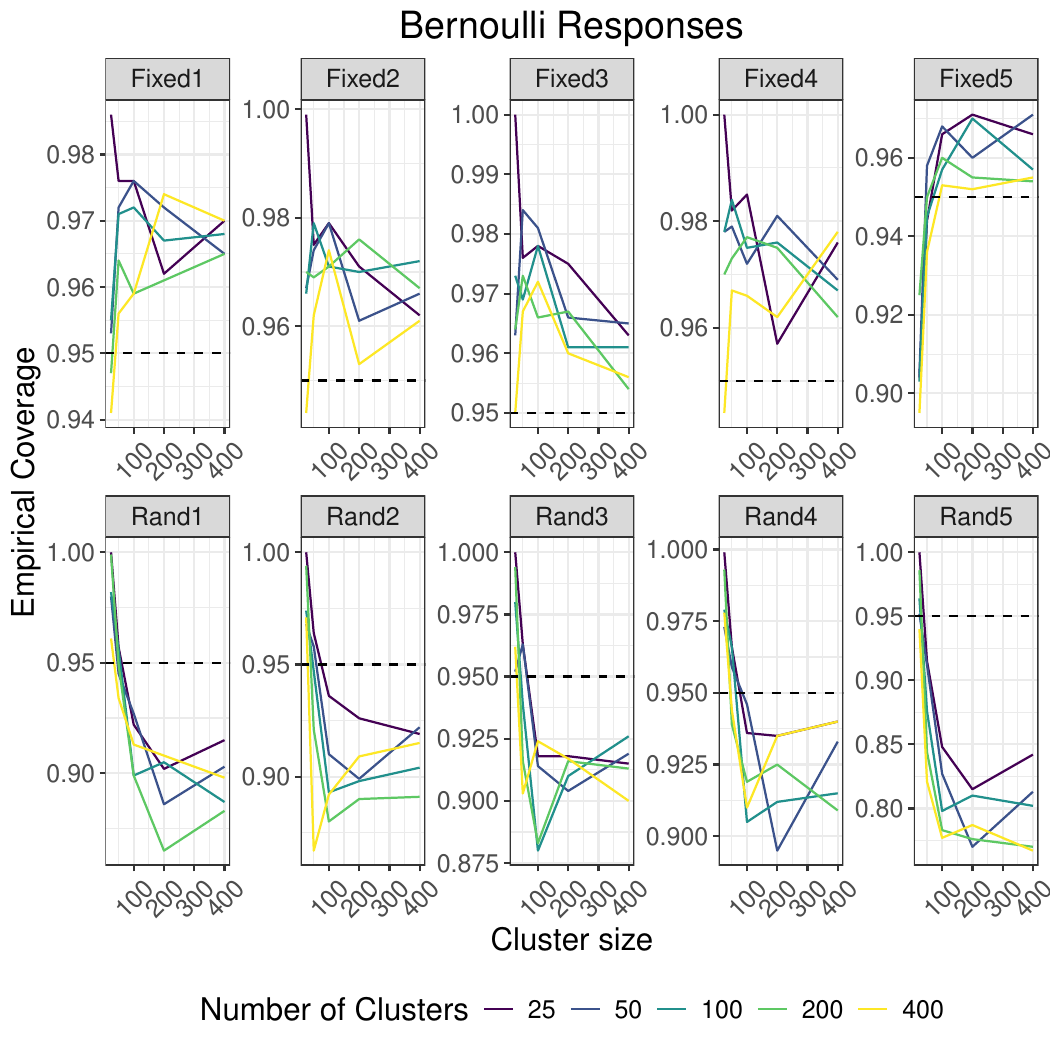}
\caption{Empirical coverage probability of 95\% coverage intervals for the five fixed and random effects estimates, obtained under the unconditional regime with Bernoulli responses.} 
\end{figure}

\begin{figure}[H]
\includegraphics[width=0.7\linewidth]{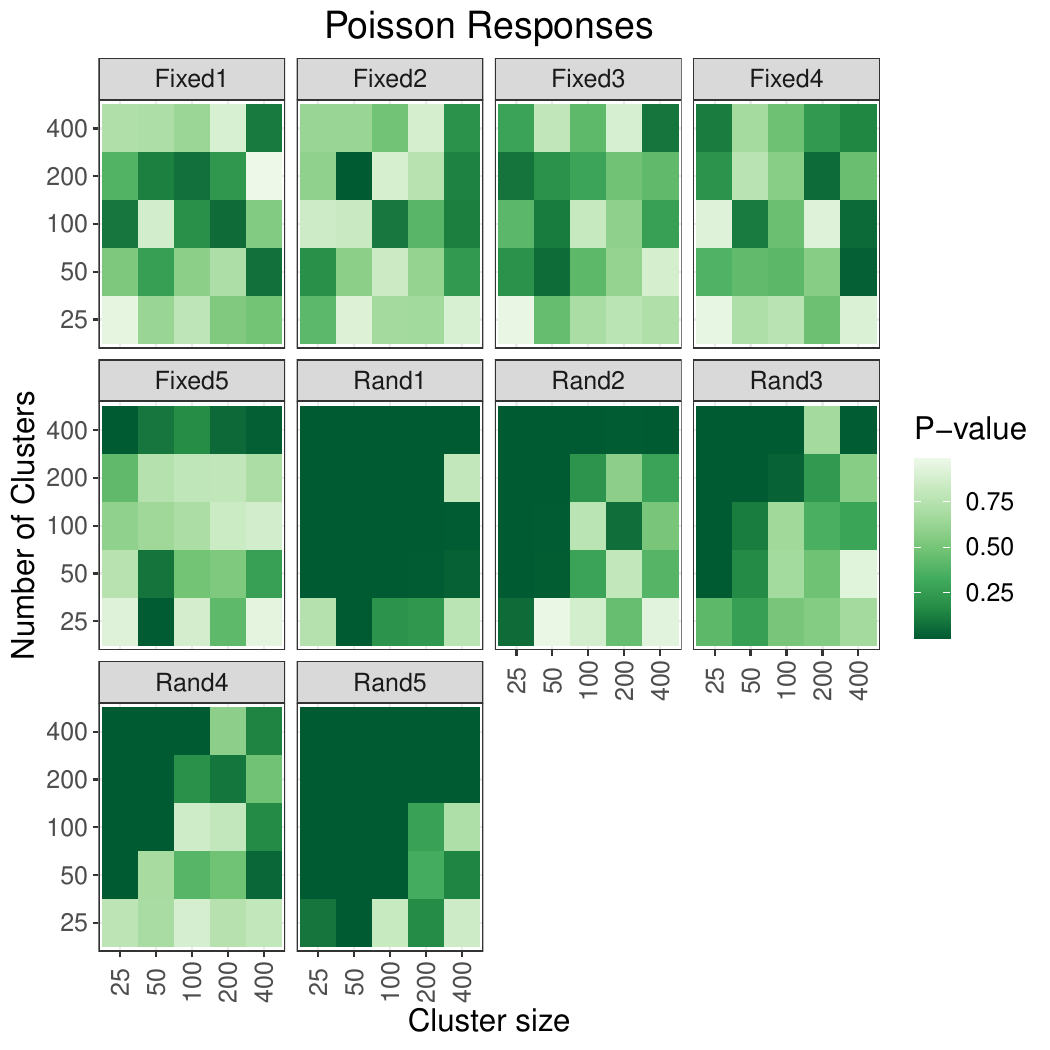}
\includegraphics[width=0.7\linewidth]{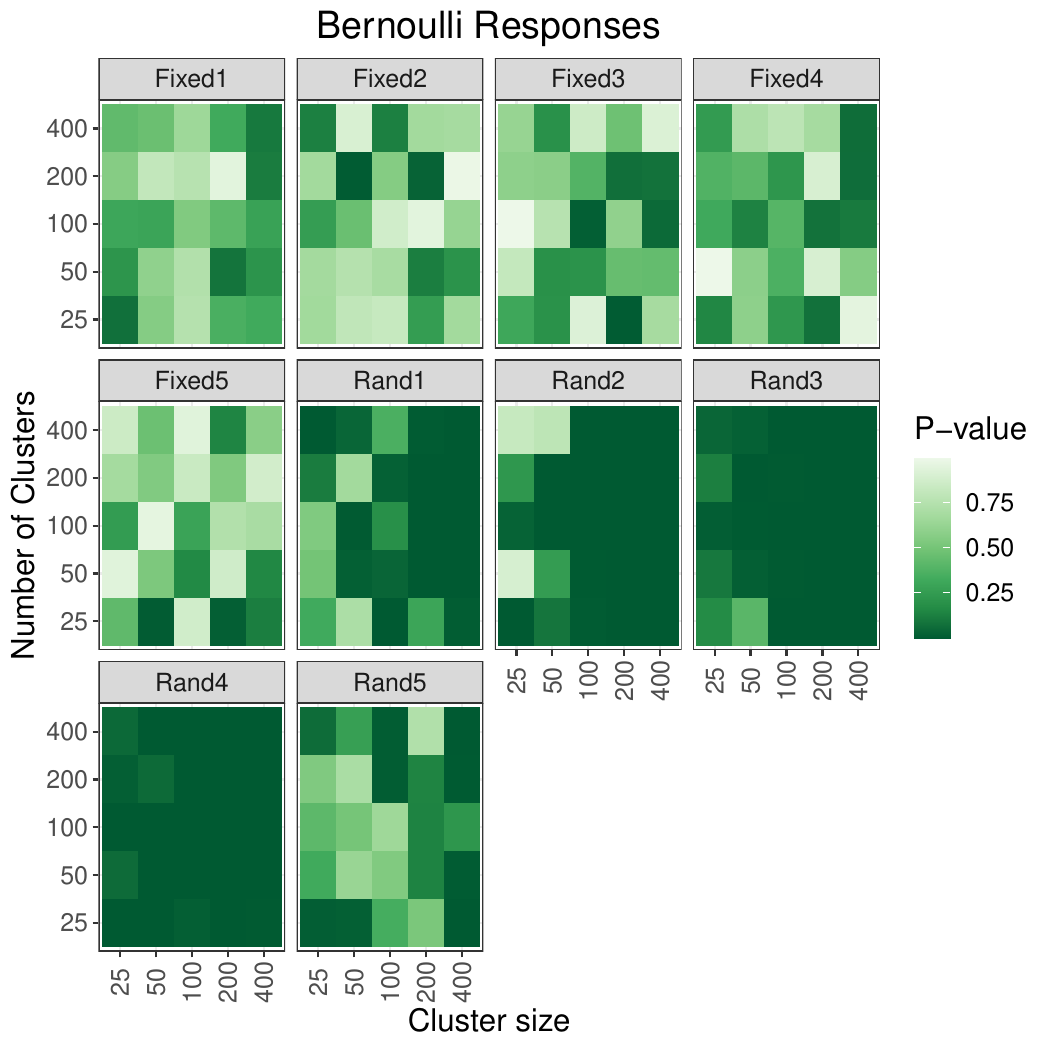}
\caption{$p$-values from Shapiro-Wilk tests applied to the fixed and random effects estimates obtained using maximum PQL estimation, under the unconditional regime.} 
\end{figure}

\begin{figure}[H]
\includegraphics[width=0.49\linewidth]{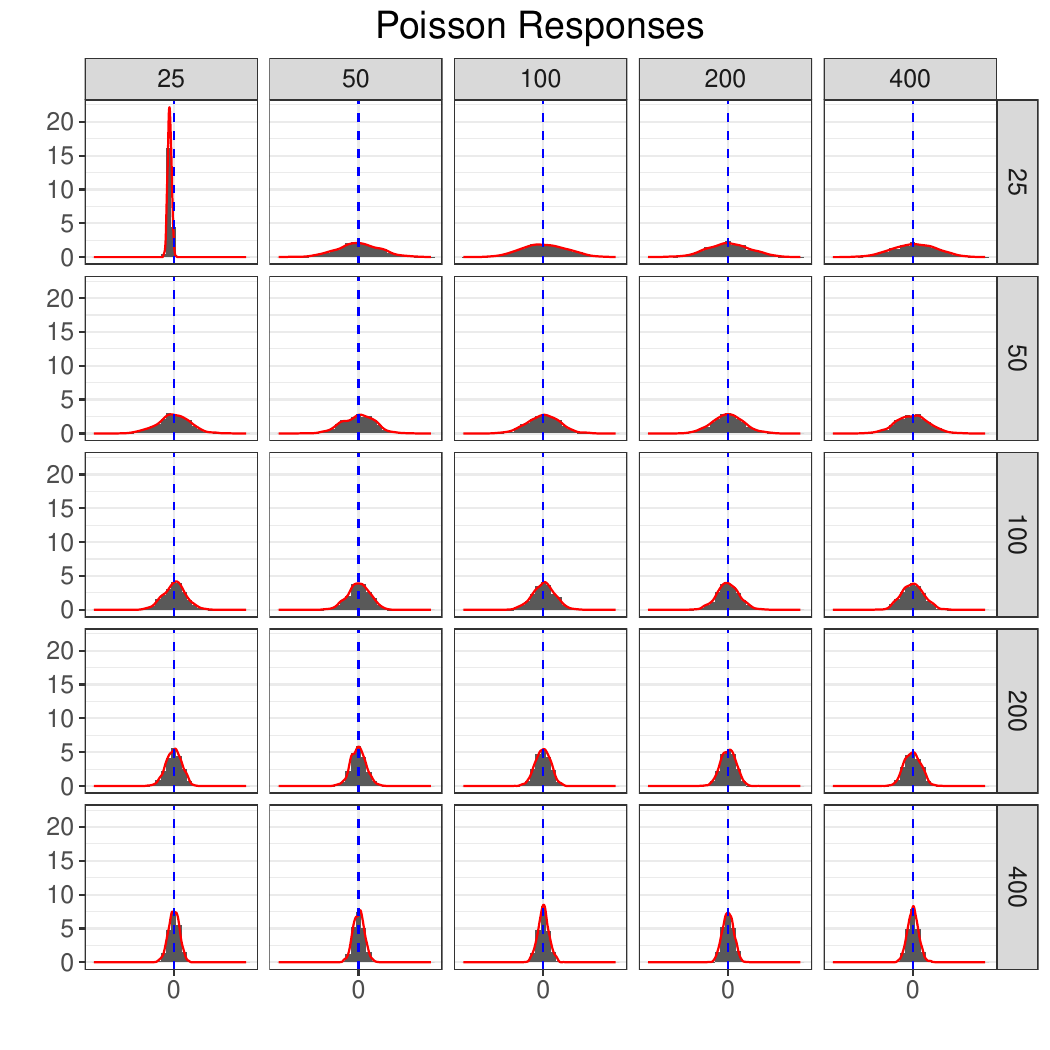}
\includegraphics[width=0.49\linewidth]{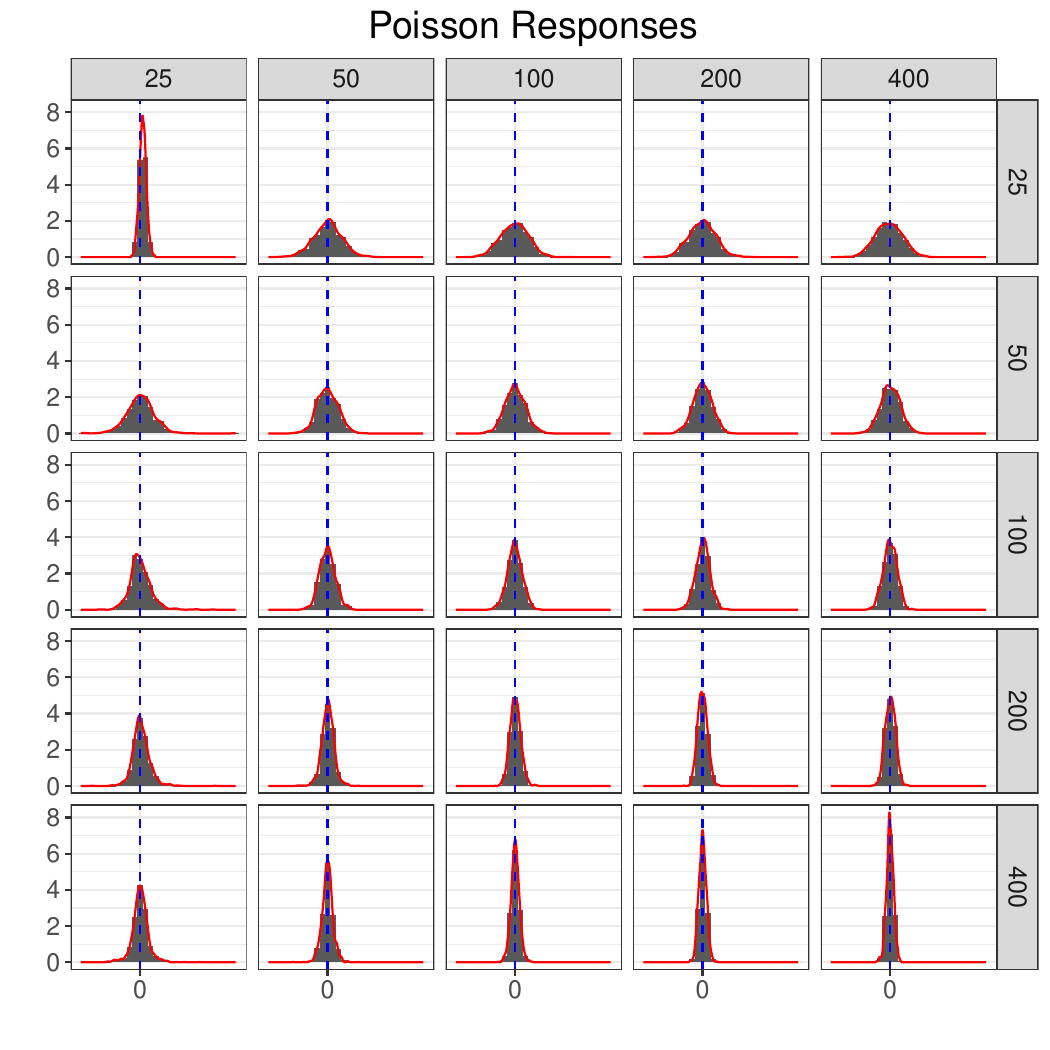}
\includegraphics[width=0.49\linewidth]{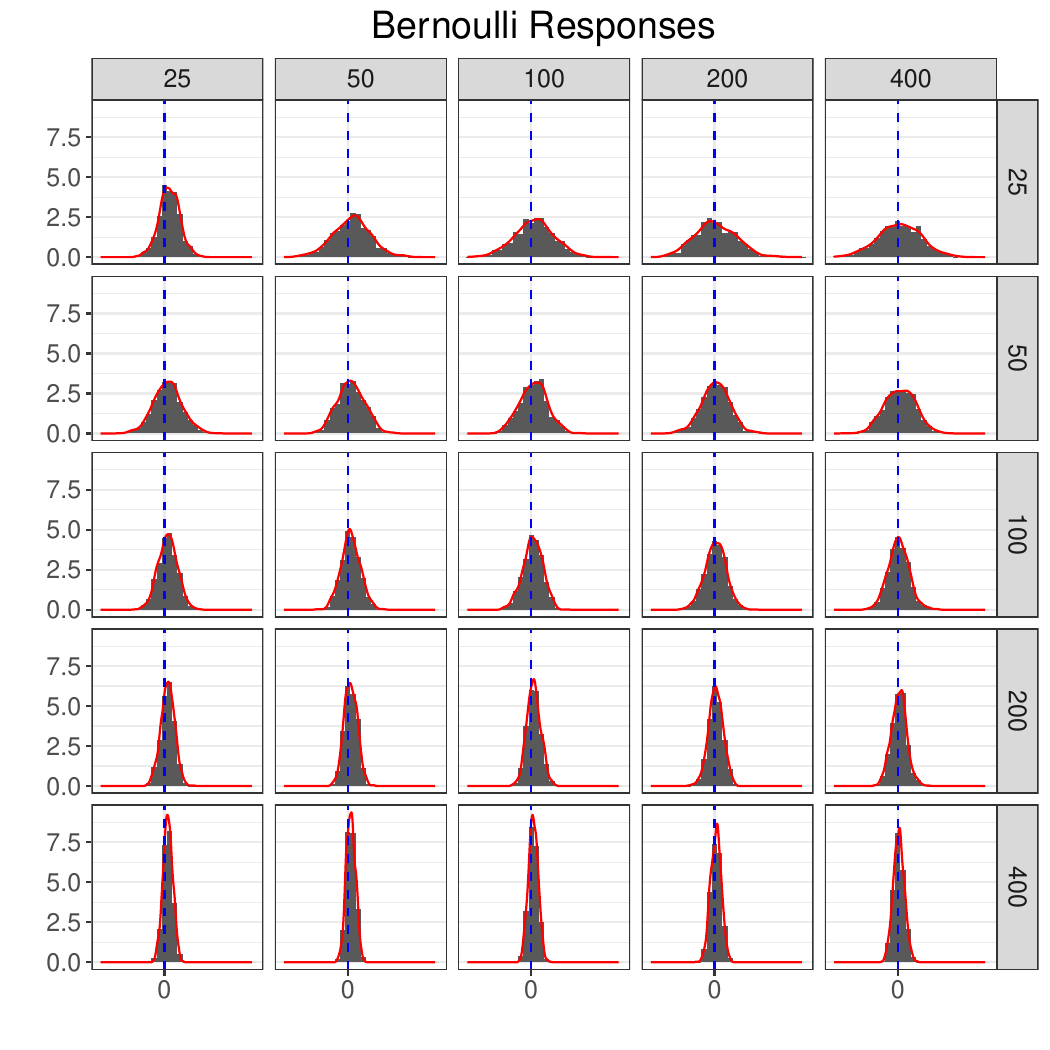}
\includegraphics[width=0.49\linewidth]{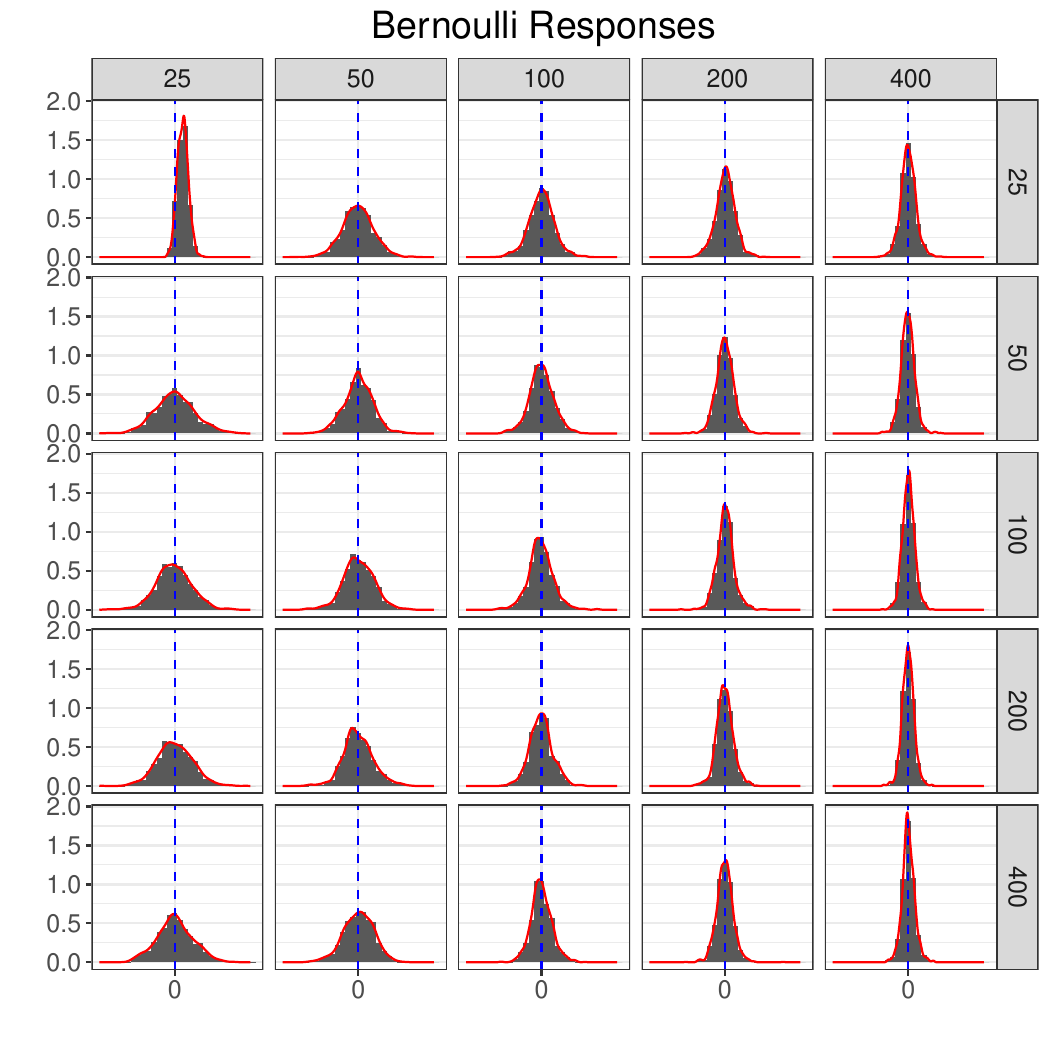}
\caption{Histograms for the third components of $\hat{\bmbeta} - \dot{\bmbeta}$ (left panels) and $\hat{\bm{b}}_1 - \dot{\bm{b}}_1$ (right panels), under the unconditional regime. Vertical facets represent the cluster sizes, while horizontal facets represent the number of clusters. The dotted blue line indicates zero, and the red curve is a kernel density smoother.} 
\end{figure}

\begin{figure}[H]
\centering
\includegraphics[width=0.95\linewidth]{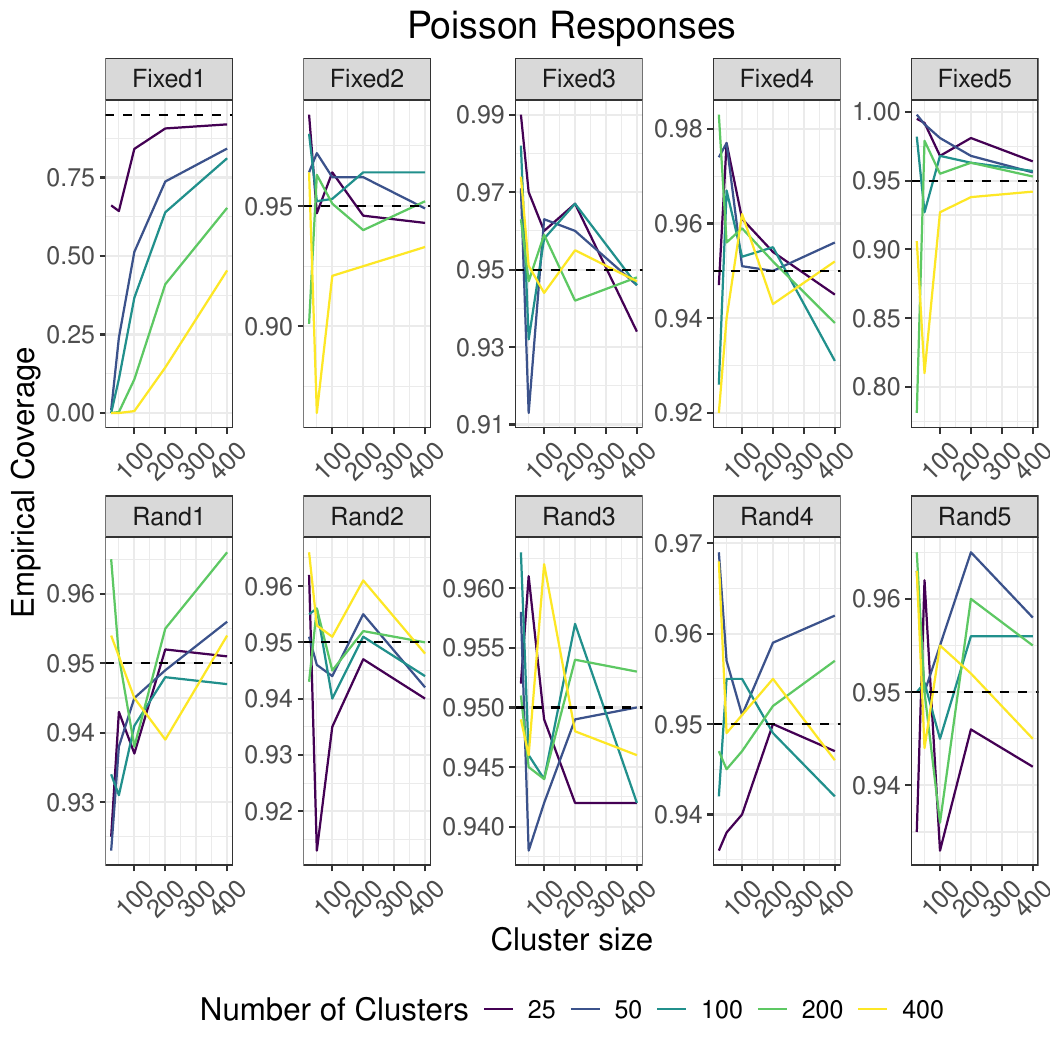}
\caption{Empirical coverage probability of 95\% coverage intervals for the five fixed and random effects estimates, obtained under the conditional regime with Poisson responses.} 
\end{figure}

\begin{figure}[H]
\centering
\includegraphics[width=0.95\linewidth]{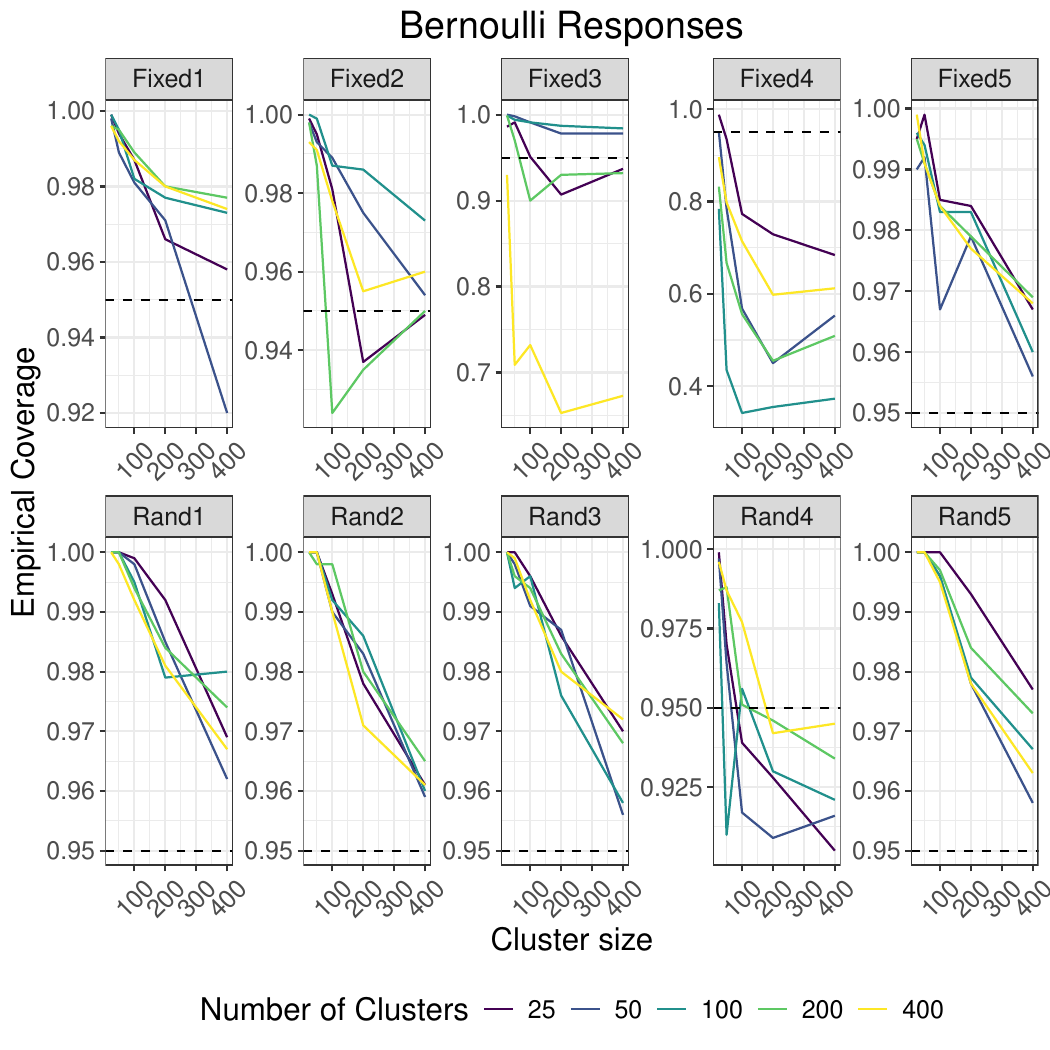}
\caption{Empirical coverage probability of 95\% coverage intervals for the five fixed and random effects estimates, obtained under the conditional regime with Bernoulli responses.} 
\end{figure}

\begin{figure}[H]
\includegraphics[width=0.7\linewidth]{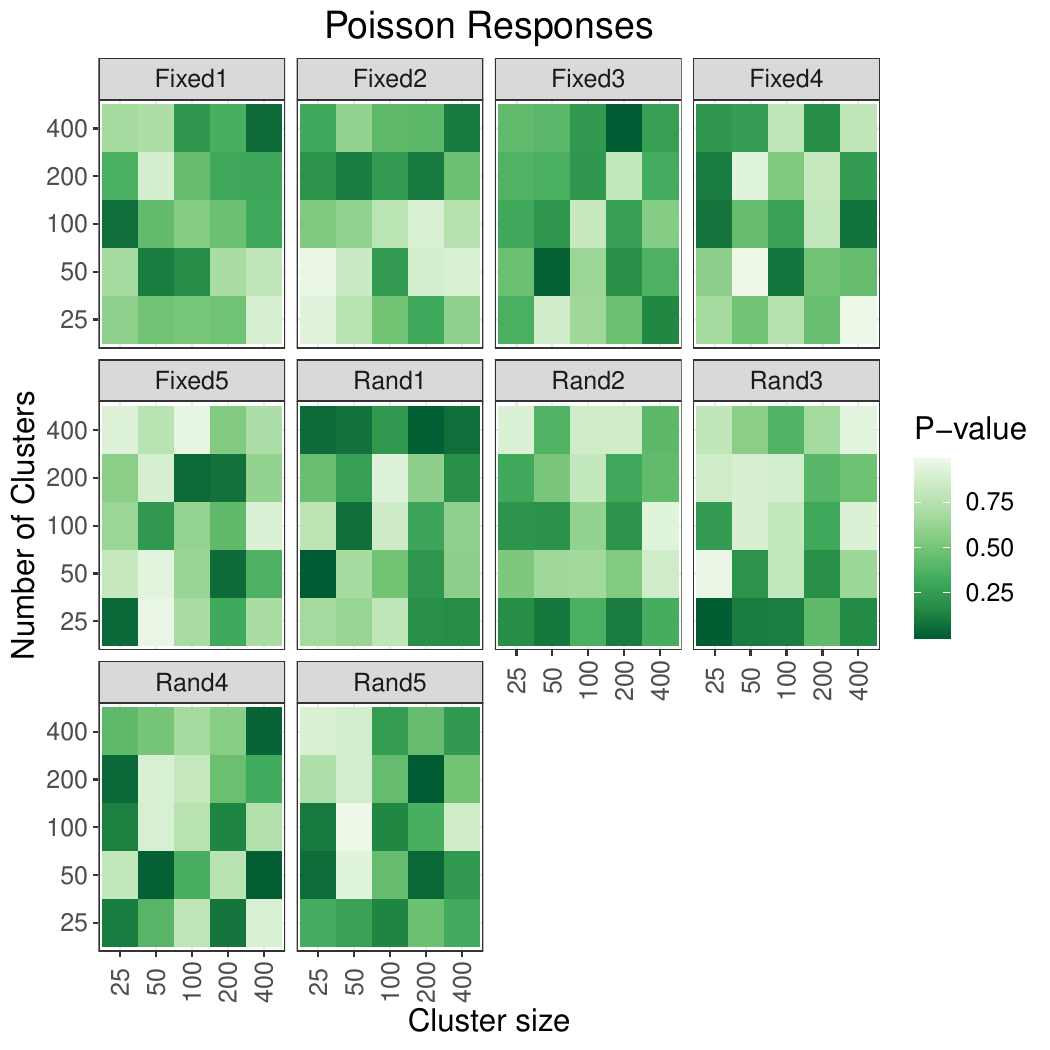}
\includegraphics[width=0.7\linewidth]{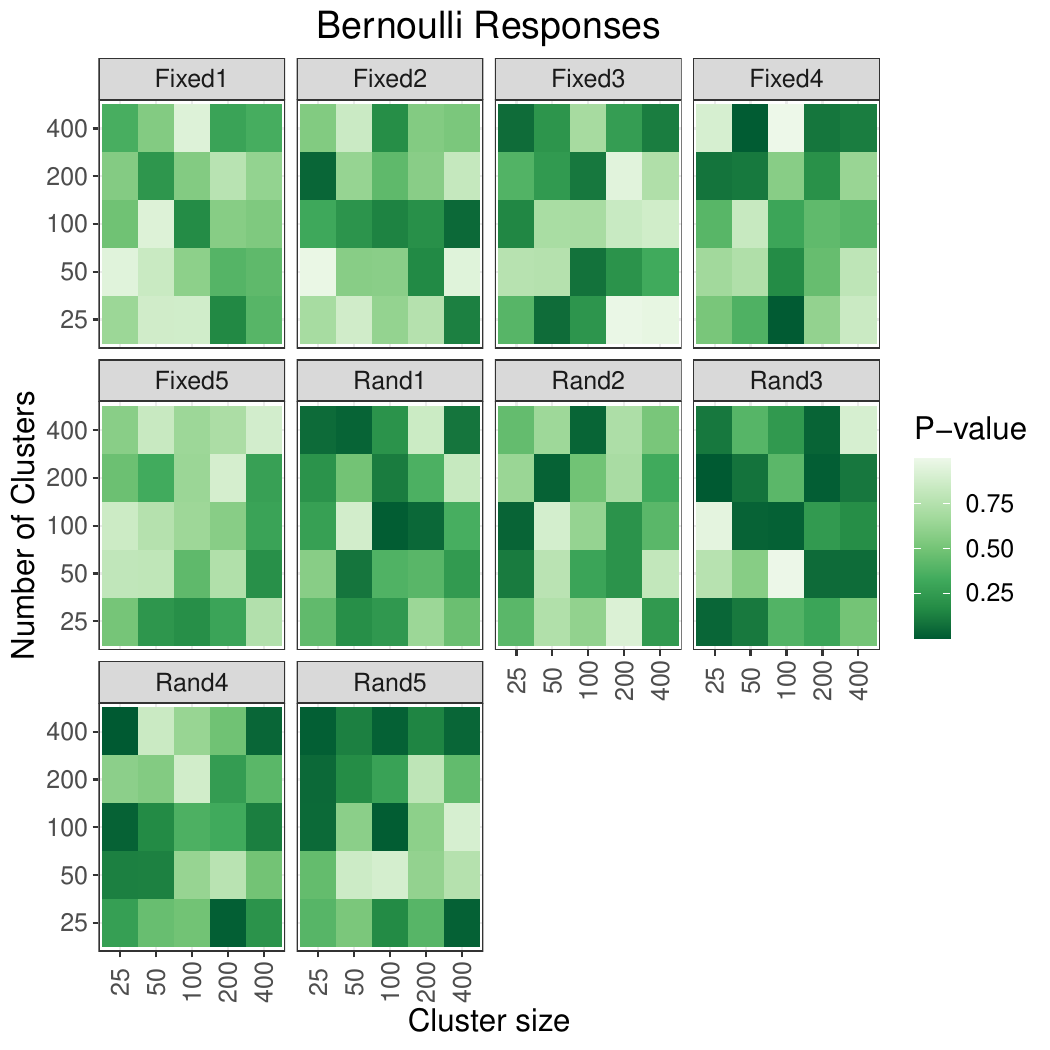}
\caption{$p$-values from Shapiro-Wilk tests applied to the fixed and random effects estimates obtained using maximum PQL estimation, under the conditional regime.} 
\end{figure}

\begin{figure}[H]
\includegraphics[width=0.49\linewidth]{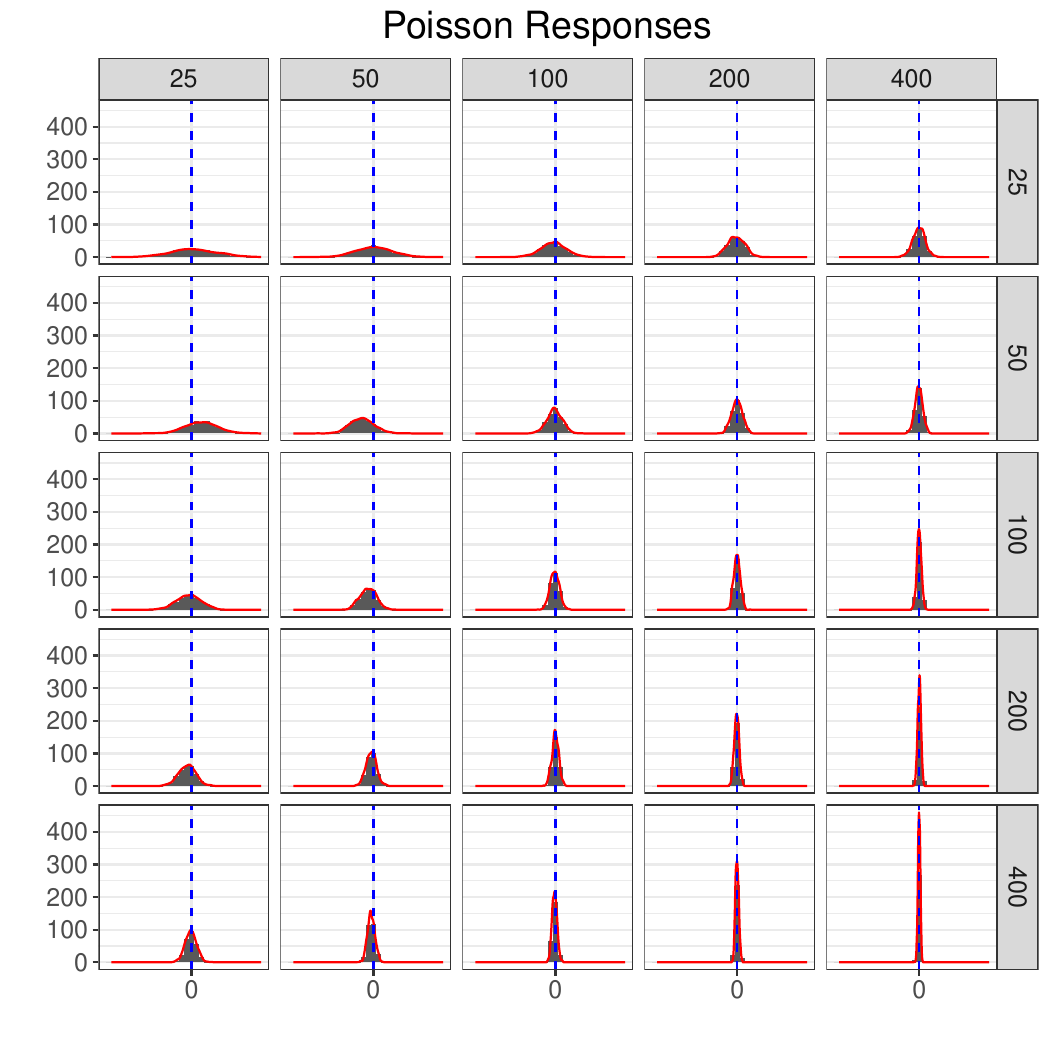}
\includegraphics[width=0.49\linewidth]{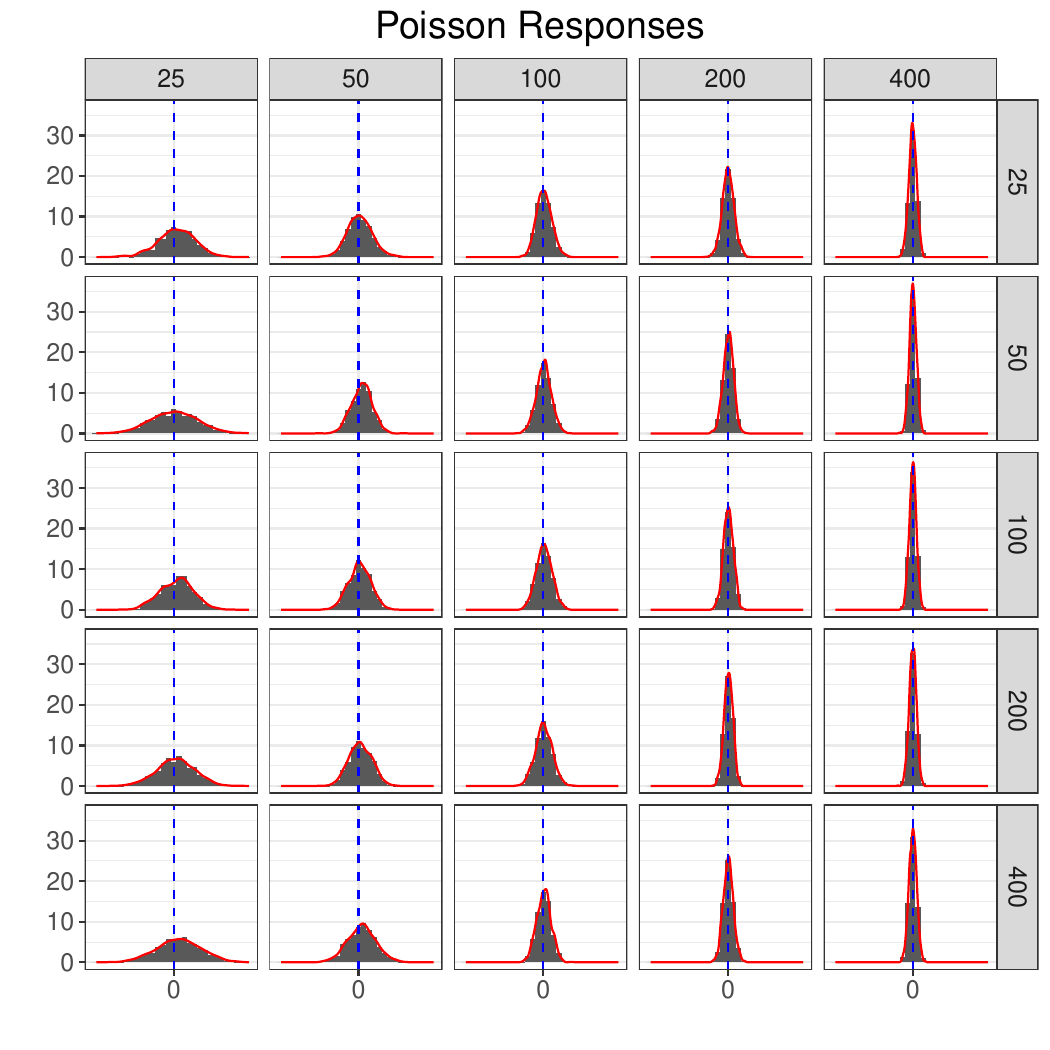}
\includegraphics[width=0.49\linewidth]{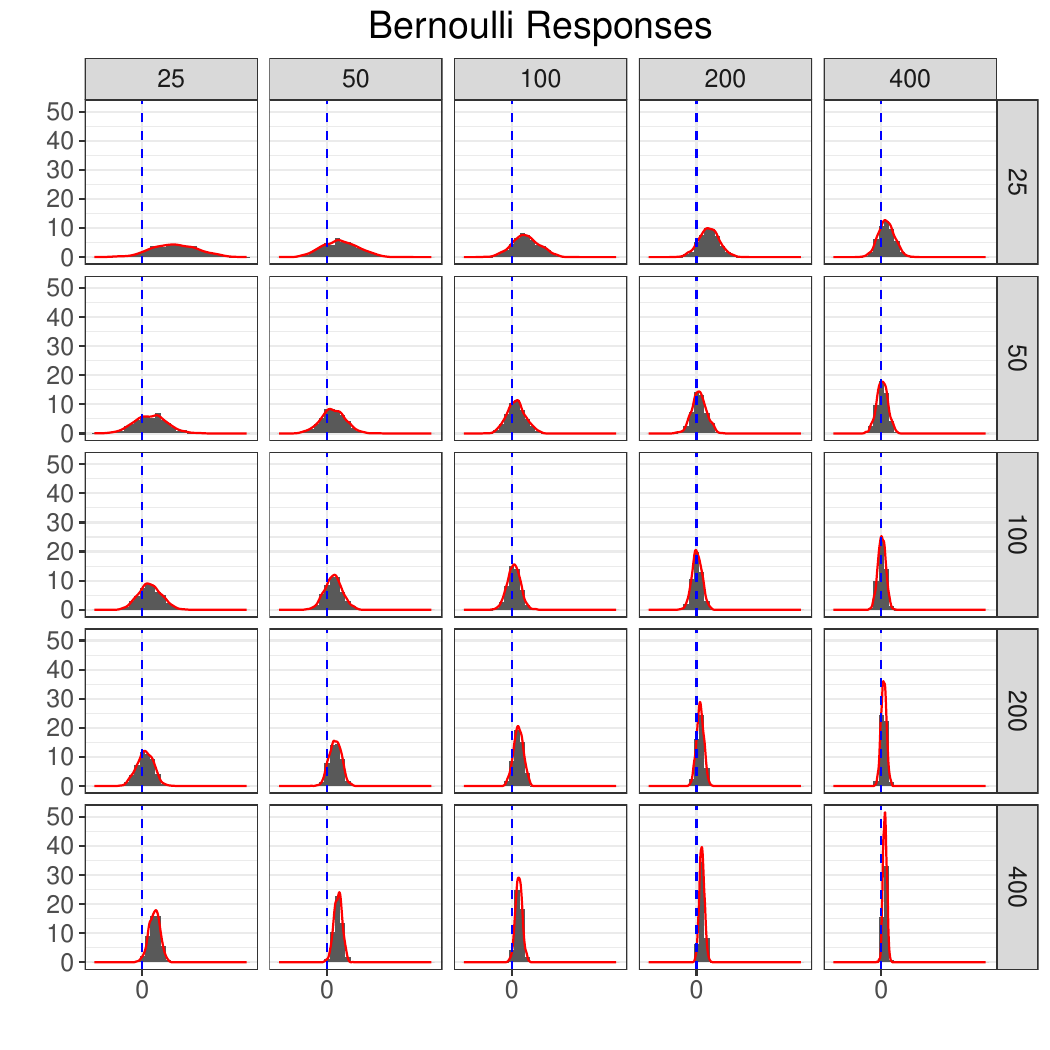}
\includegraphics[width=0.49\linewidth]{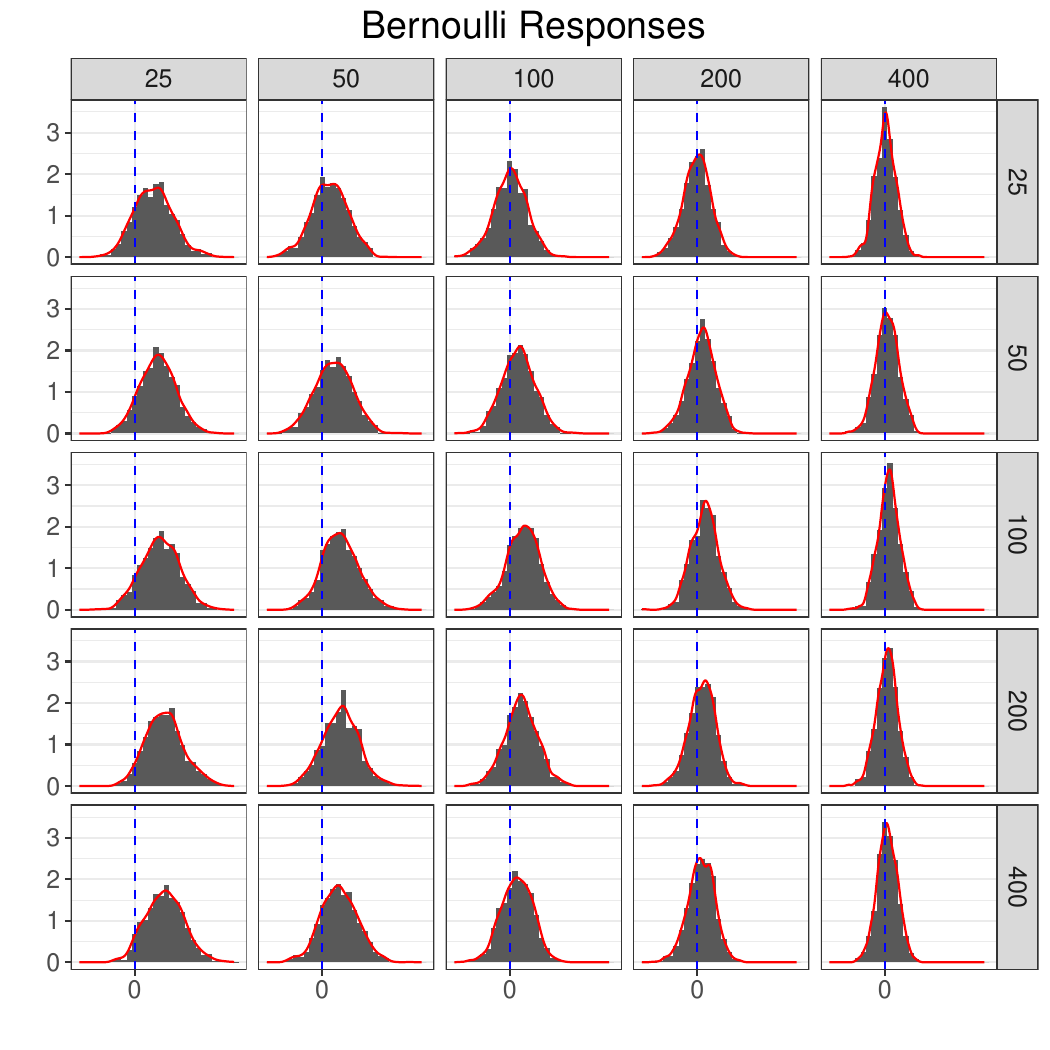}
\caption{Histograms for the third components of $\hat{\bmbeta} - \dot{\bmbeta}$ (left panels) and $\hat{\bm{b}}_1 - \dot{\bm{b}}_1$ (right panels), under the unconditional regime. Vertical facets represent the cluster sizes, while horizontal facets represent the number of clusters. The dotted blue line indicates zero, and the red curve is a kernel density smoother.} 
\end{figure}

\subsection{{\boldmath ${G}$} = 0.50  {\boldmath ${I}_2$}}

\begin{figure}[H]
\centering
\includegraphics[width=0.95\linewidth]{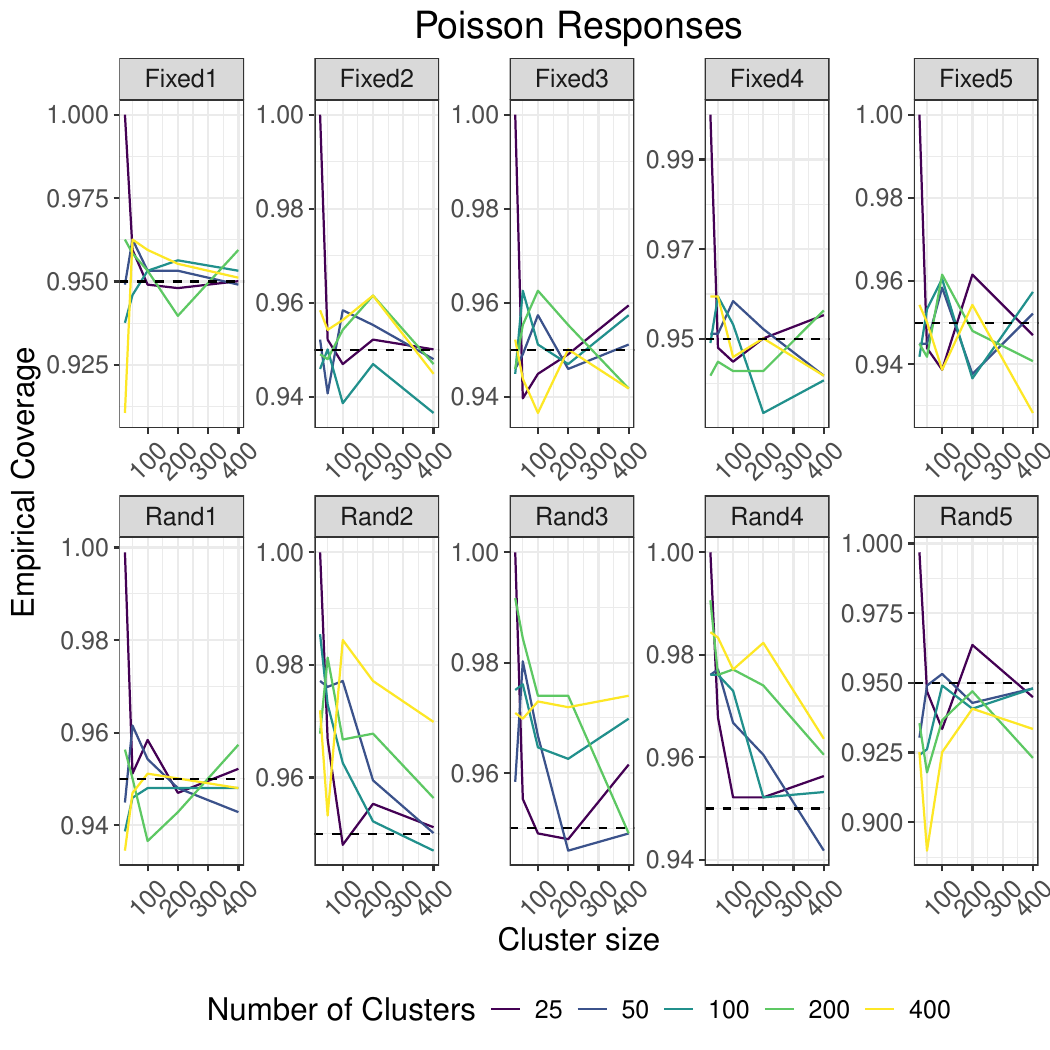}
\caption{Empirical coverage probability of 95\% coverage intervals for the five fixed and random effects estimates, obtained under the unconditional regime with Poisson responses.} 
\end{figure}

\begin{figure}[H]
\centering
\includegraphics[width=0.95\linewidth]{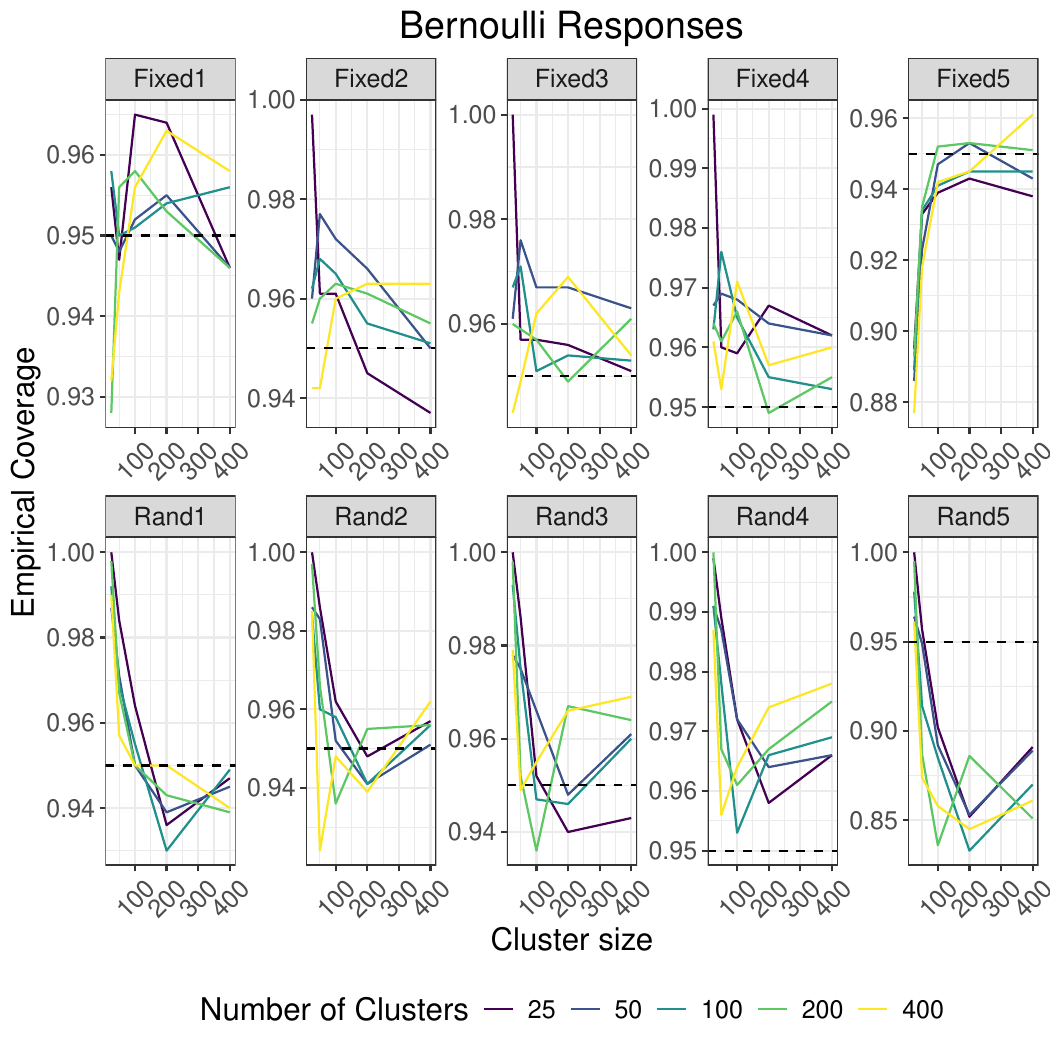}
\caption{Empirical coverage probability of 95\% coverage intervals for the five fixed and random effects estimates, obtained under the unconditional regime with Bernoulli responses.} 
\end{figure}

\begin{figure}[H]
\includegraphics[width=0.7\linewidth]{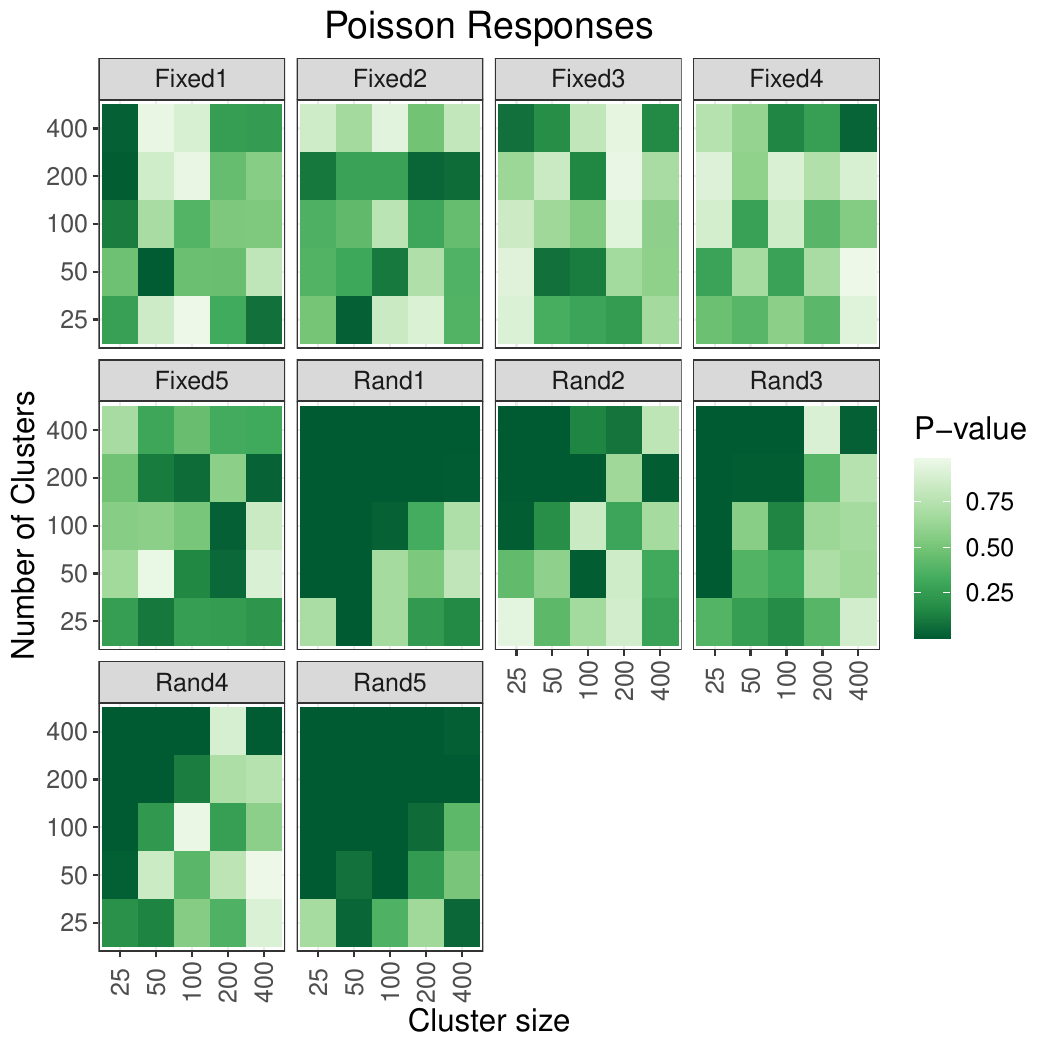}
\includegraphics[width=0.7\linewidth]{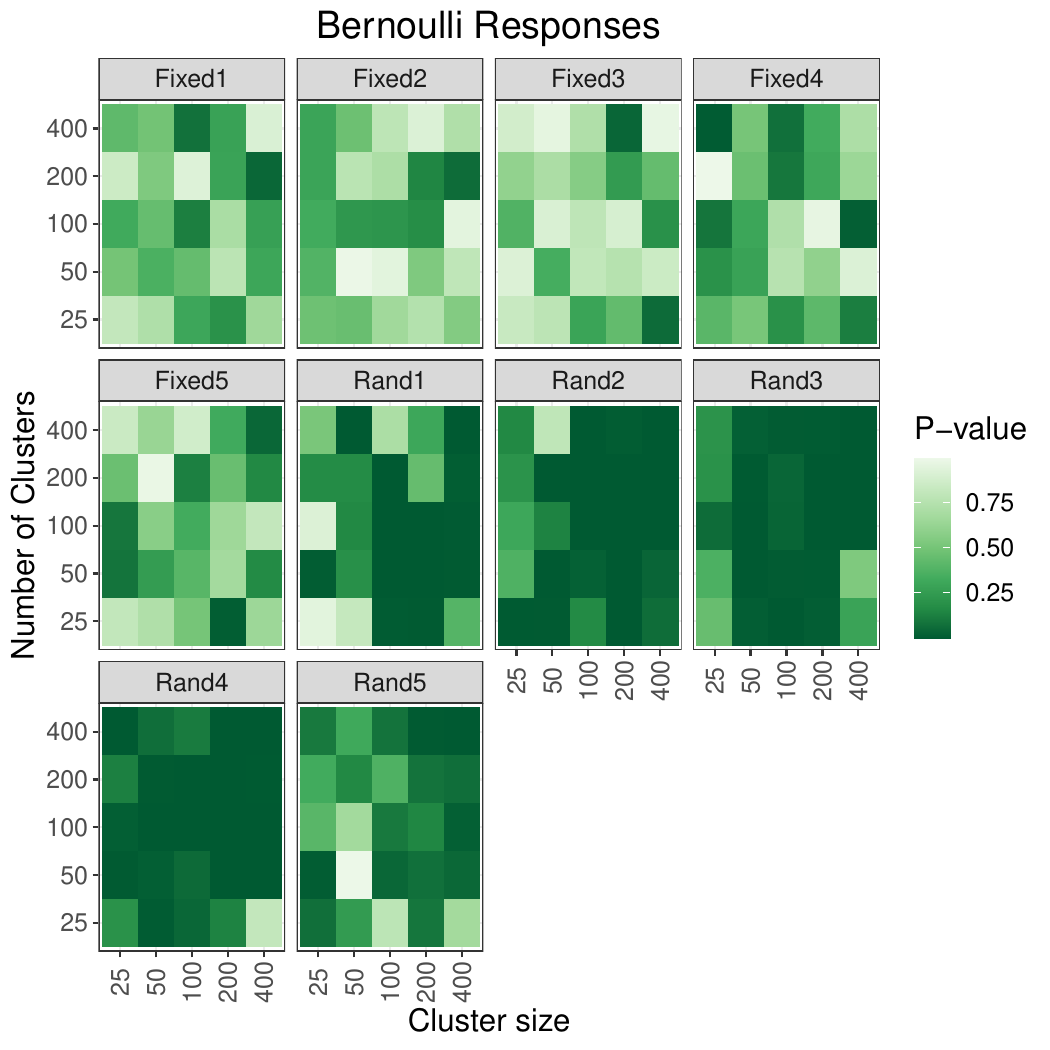}
\caption{$p$-values from Shapiro-Wilk tests applied to the fixed and random effects estimates obtained using maximum PQL estimation, under the unconditional regime.} 
\end{figure}

\begin{figure}[H]
\includegraphics[width=0.49\linewidth]{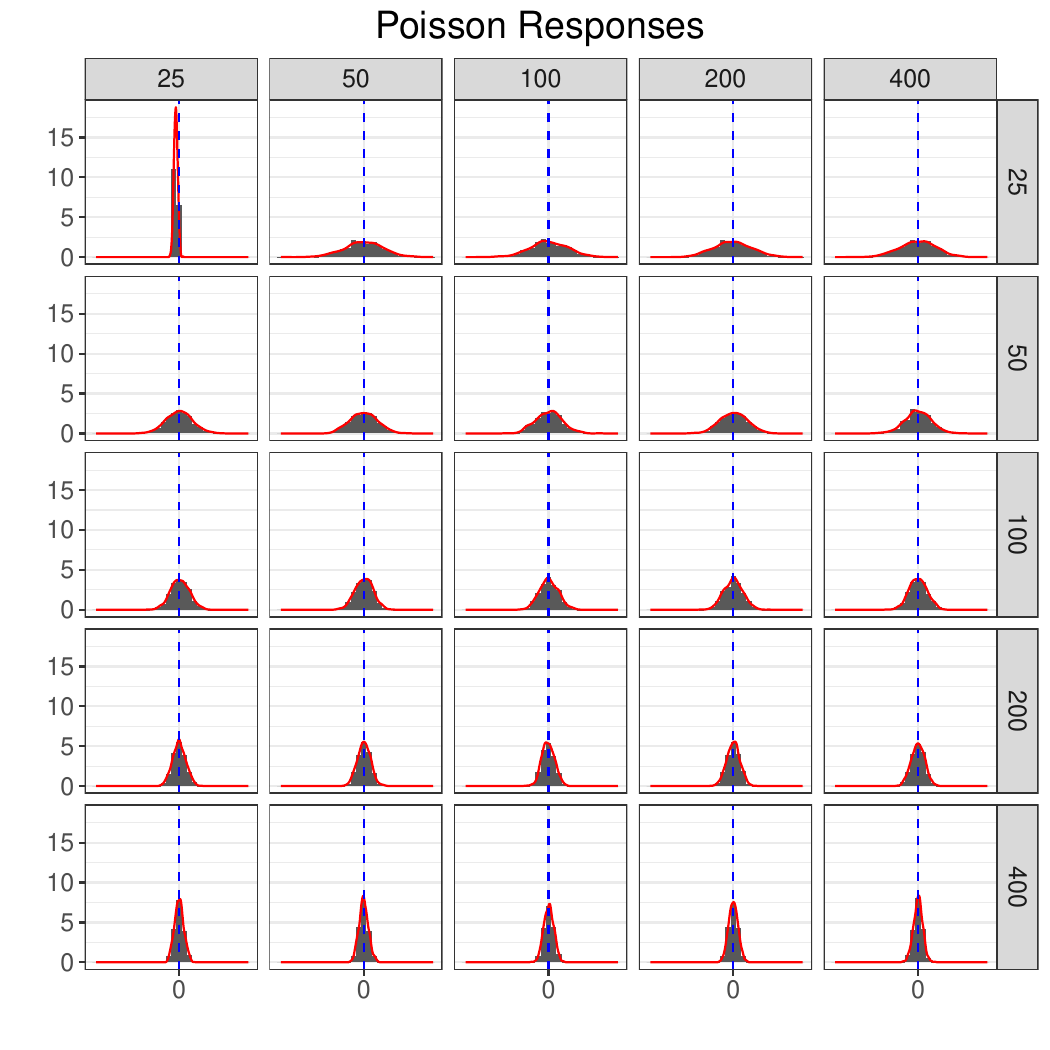}
\includegraphics[width=0.49\linewidth]{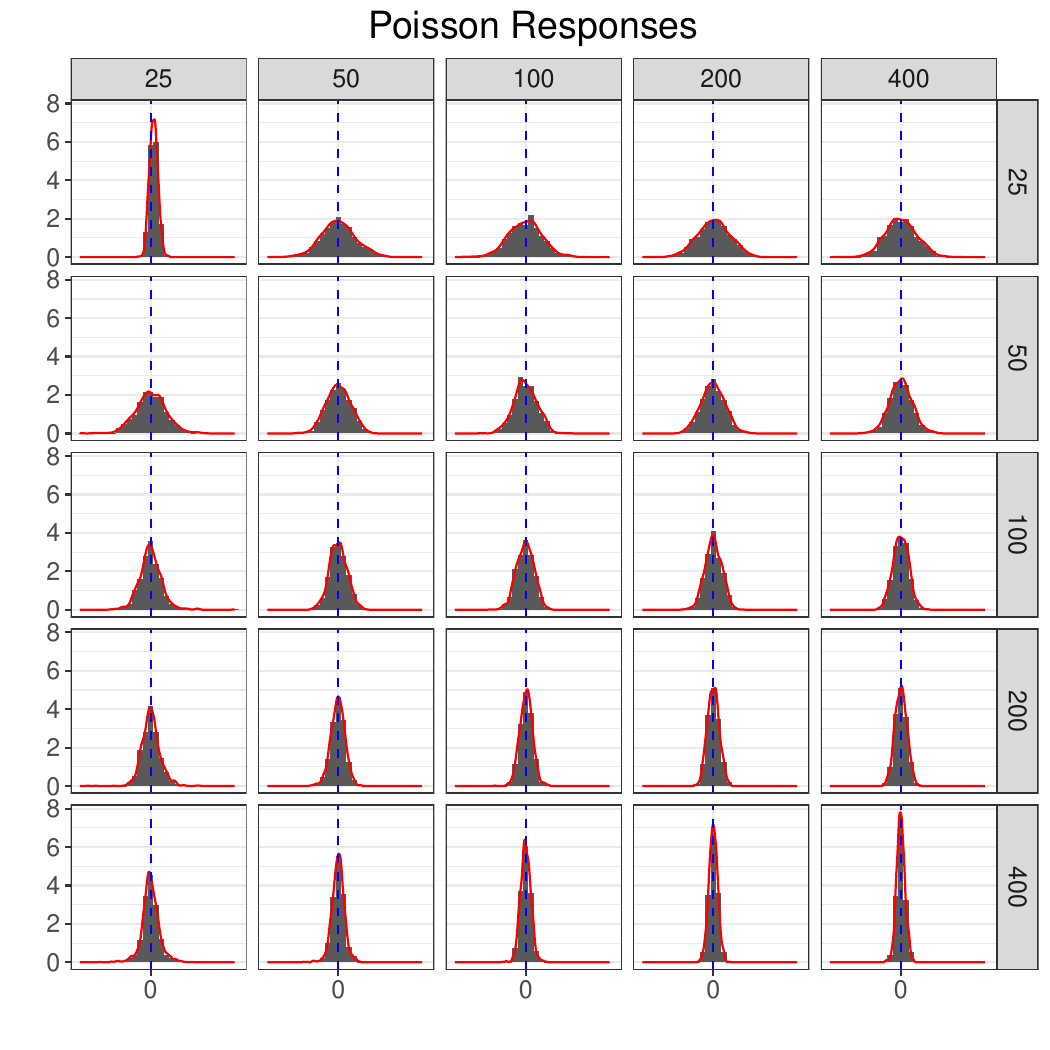}
\includegraphics[width=0.49\linewidth]{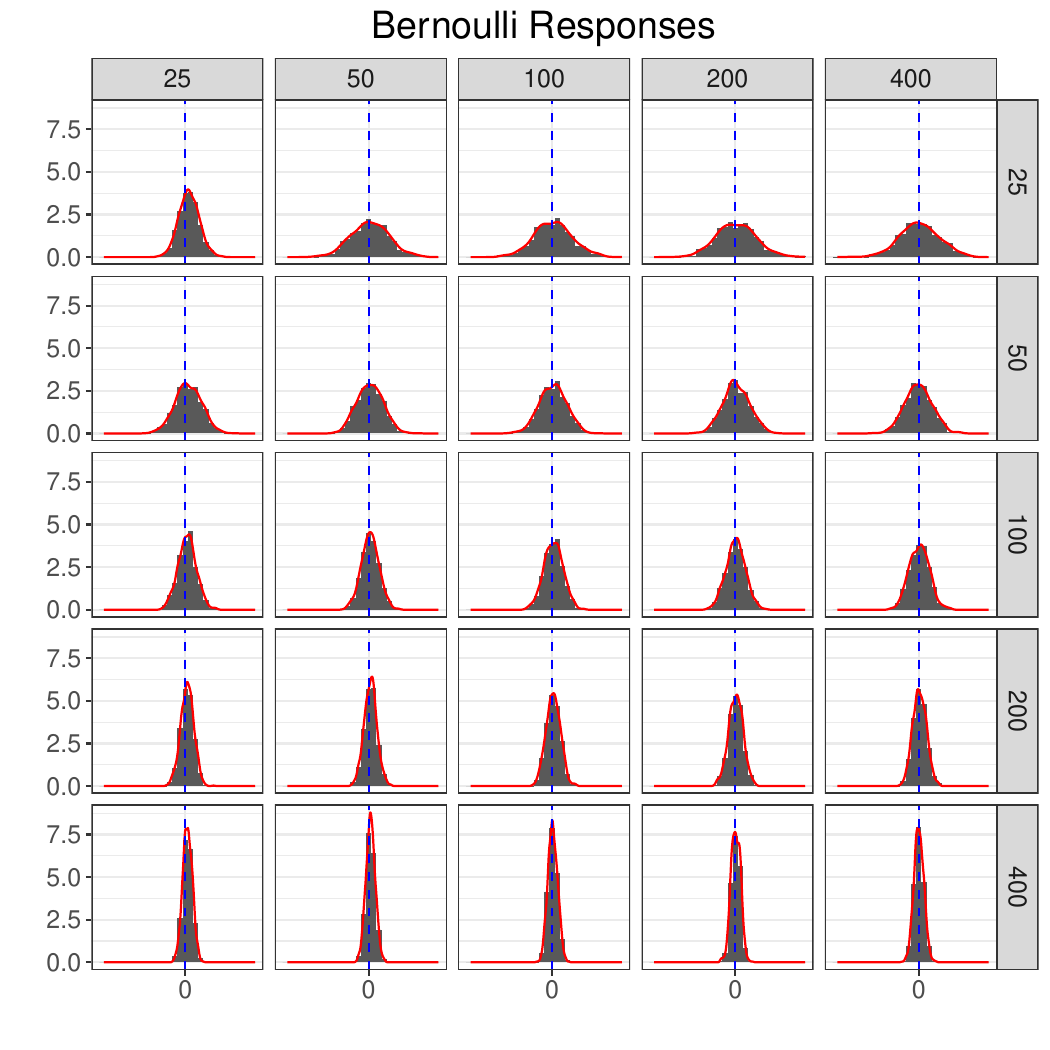}
\includegraphics[width=0.49\linewidth]{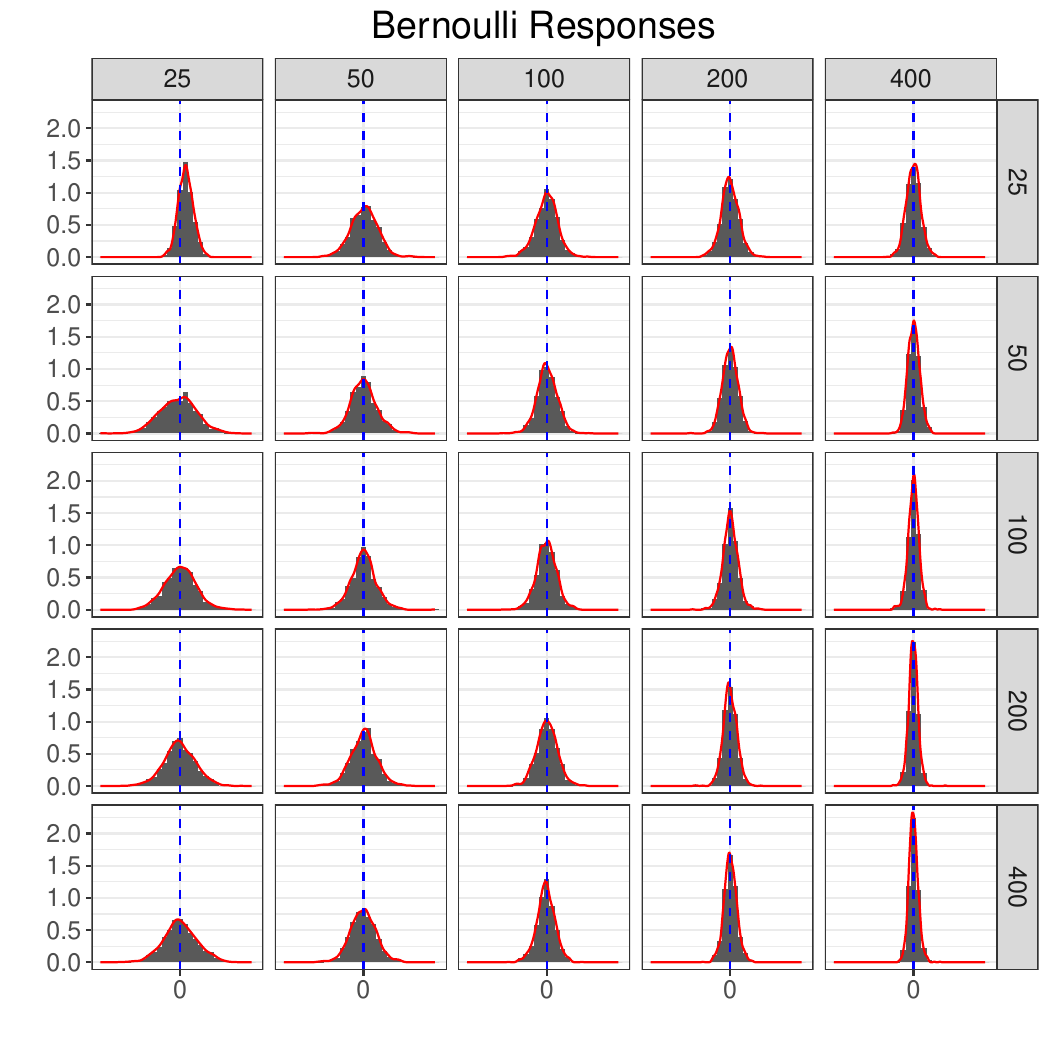}
\caption{Histograms for the third components of $\hat{\bmbeta} - \dot{\bmbeta}$ (left panels) and $\hat{\bm{b}}_1 - \dot{\bm{b}}_1$ (right panels), under the unconditional regime. Vertical facets represent the cluster sizes, while horizontal facets represent the number of clusters. The dotted blue line indicates zero, and the red curve is a kernel density smoother.} 
\end{figure}

\begin{figure}[H]
\centering
\includegraphics[width=0.95\linewidth]{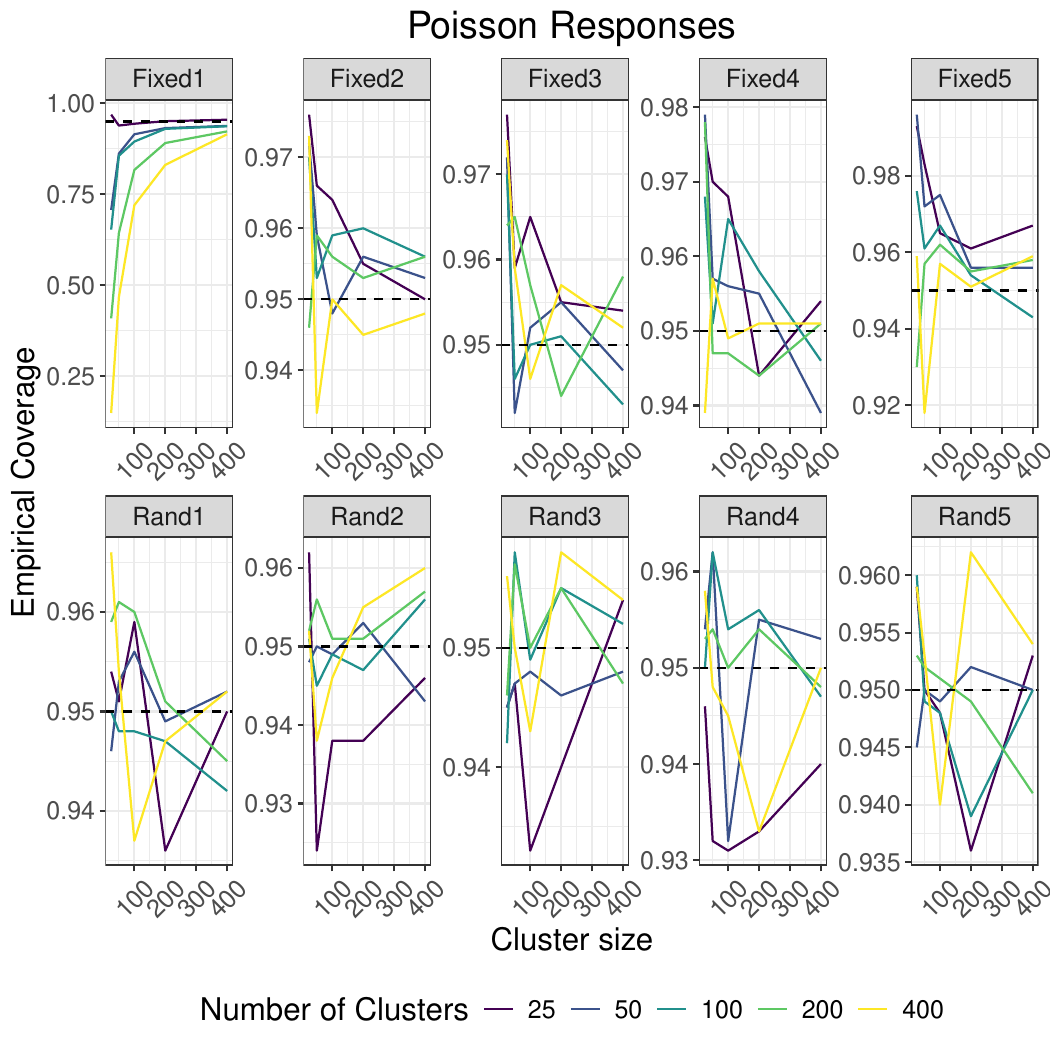}
\caption{Empirical coverage probability of 95\% coverage intervals for the five fixed and random effects estimates, obtained under the conditional regime with Poisson responses.} 
\end{figure}

\begin{figure}[H]
\centering
\includegraphics[width=0.95\linewidth]{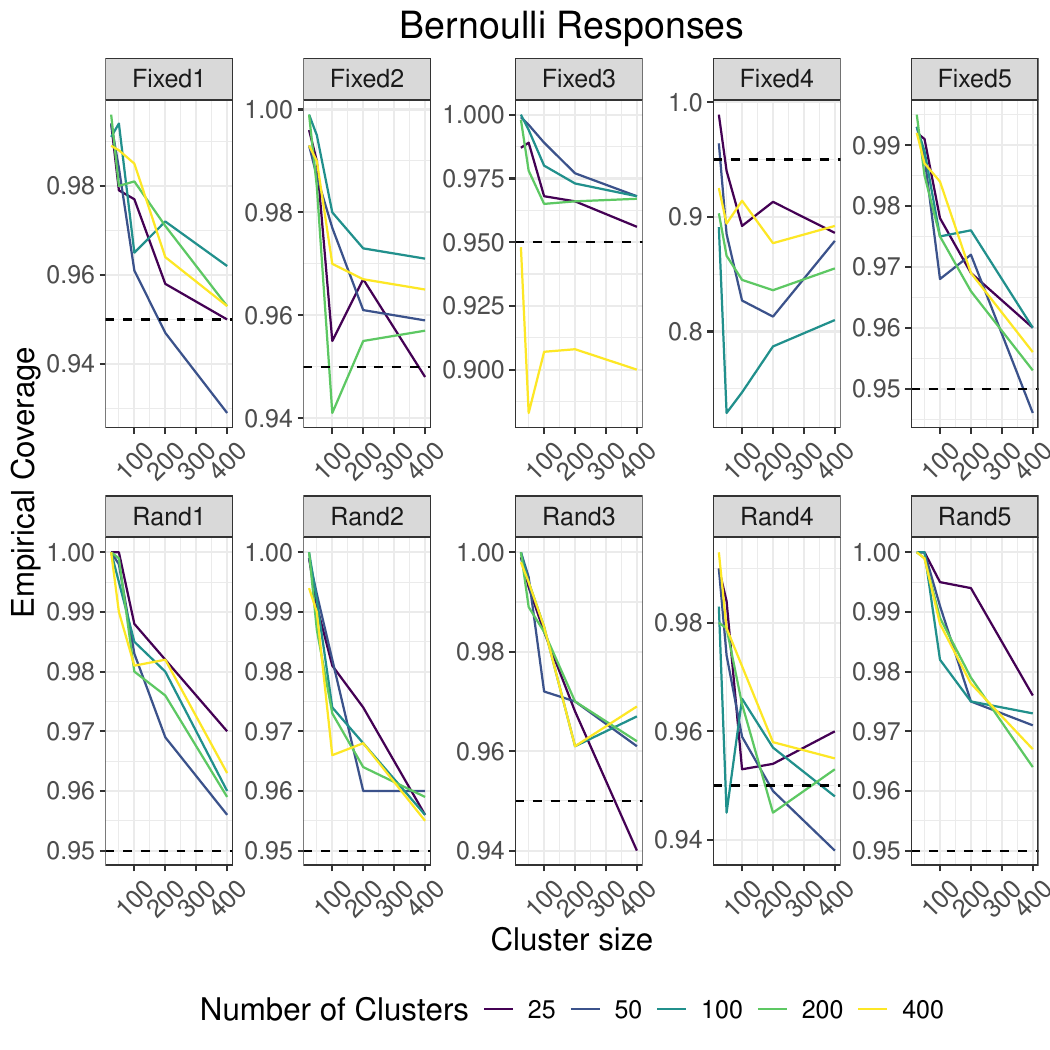}
\caption{Empirical coverage probability of 95\% coverage intervals for the five fixed and random effects estimates, obtained under the conditional regime with Bernoulli responses.} 
\end{figure}

\begin{figure}[H]
\includegraphics[width=0.7\linewidth]{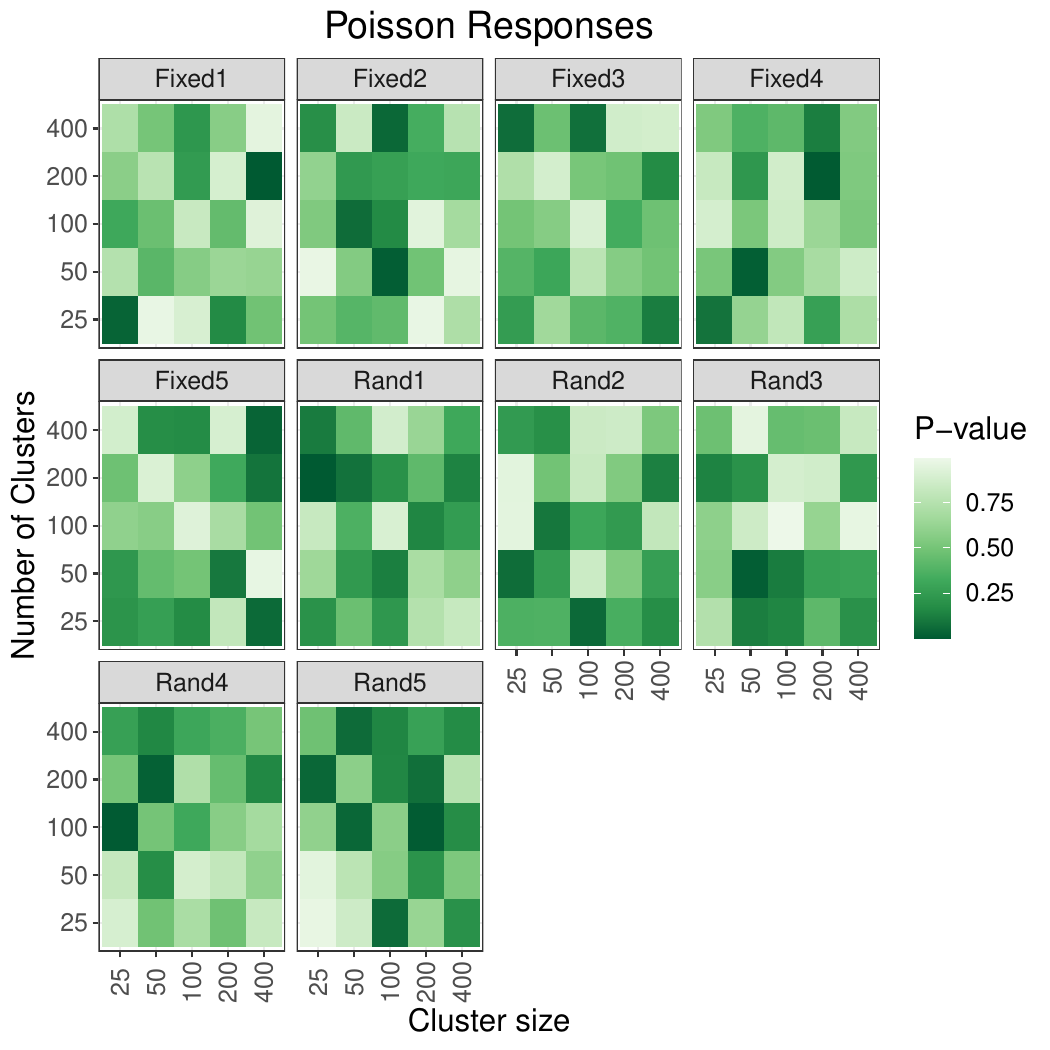}
\includegraphics[width=0.7\linewidth]{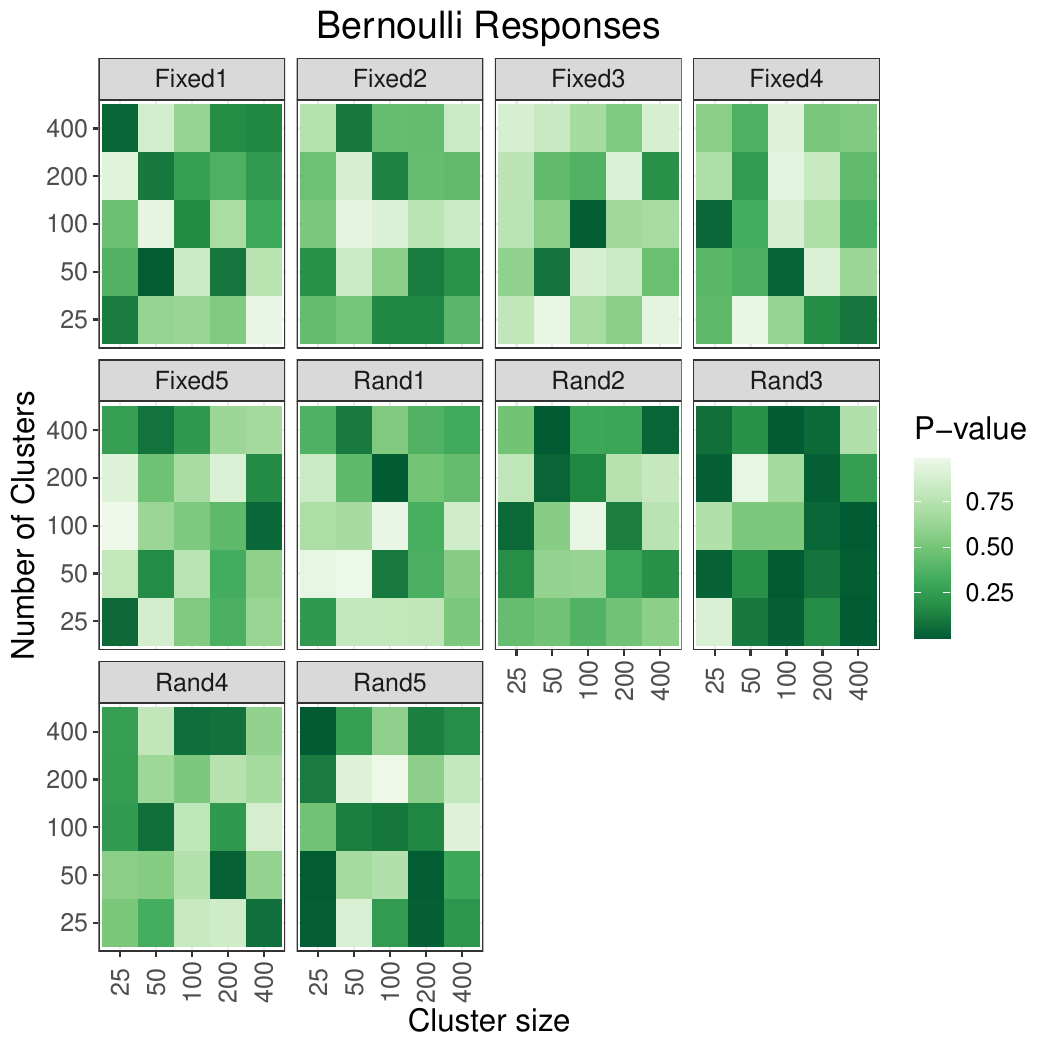}
\caption{$p$-values from Shapiro-Wilk tests applied to the fixed and random effects estimates obtained using maximum PQL estimation, under the conditional regime.} 
\end{figure}

\begin{figure}[H]
\includegraphics[width=0.49\linewidth]{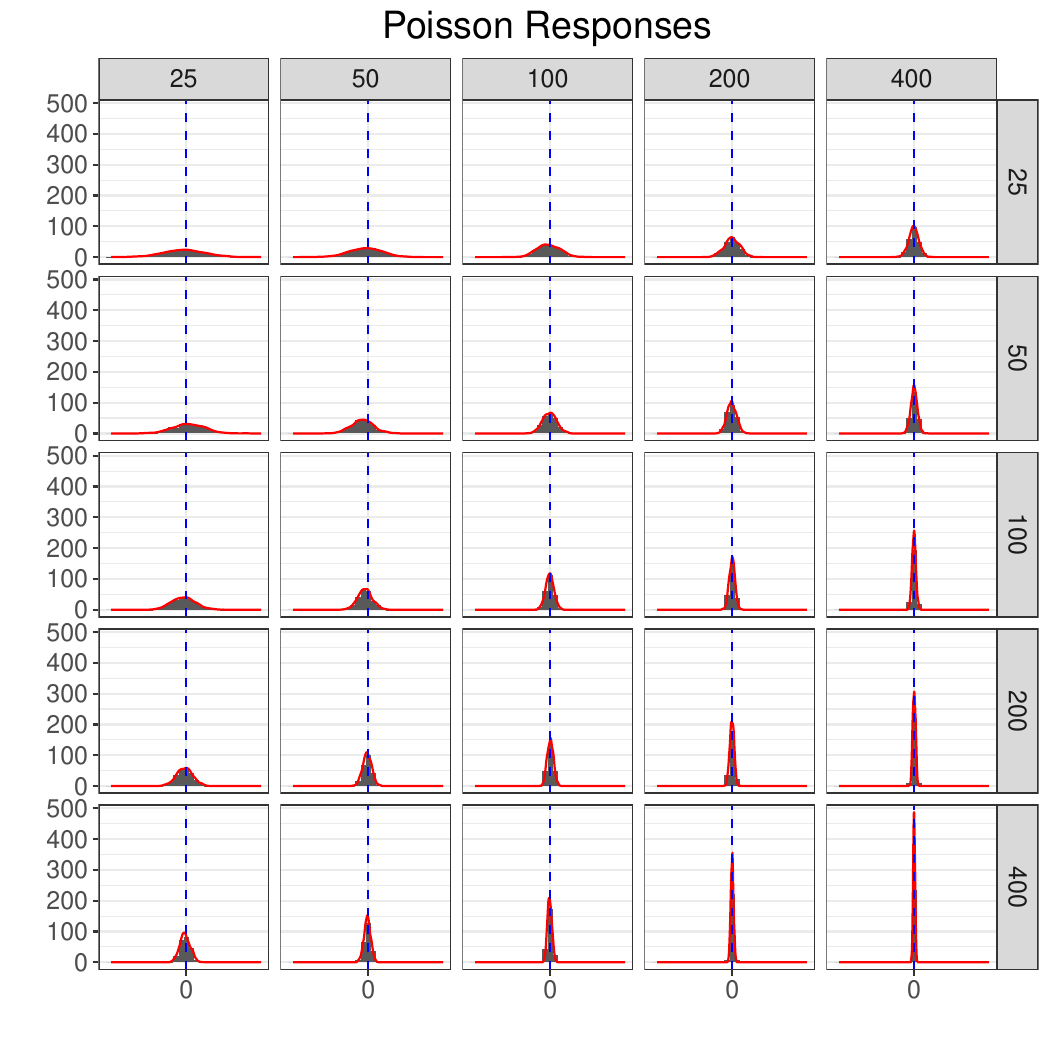}
\includegraphics[width=0.49\linewidth]{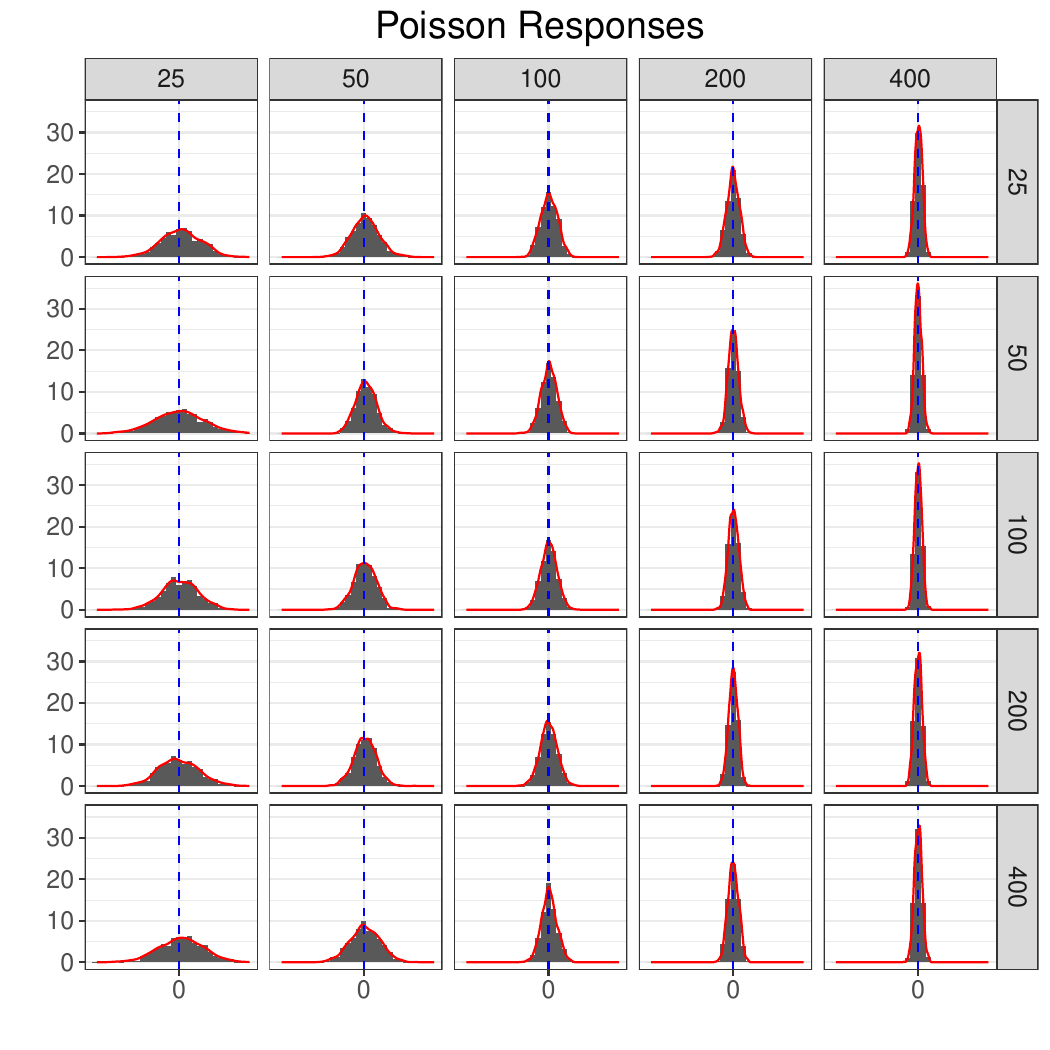}
\includegraphics[width=0.49\linewidth]{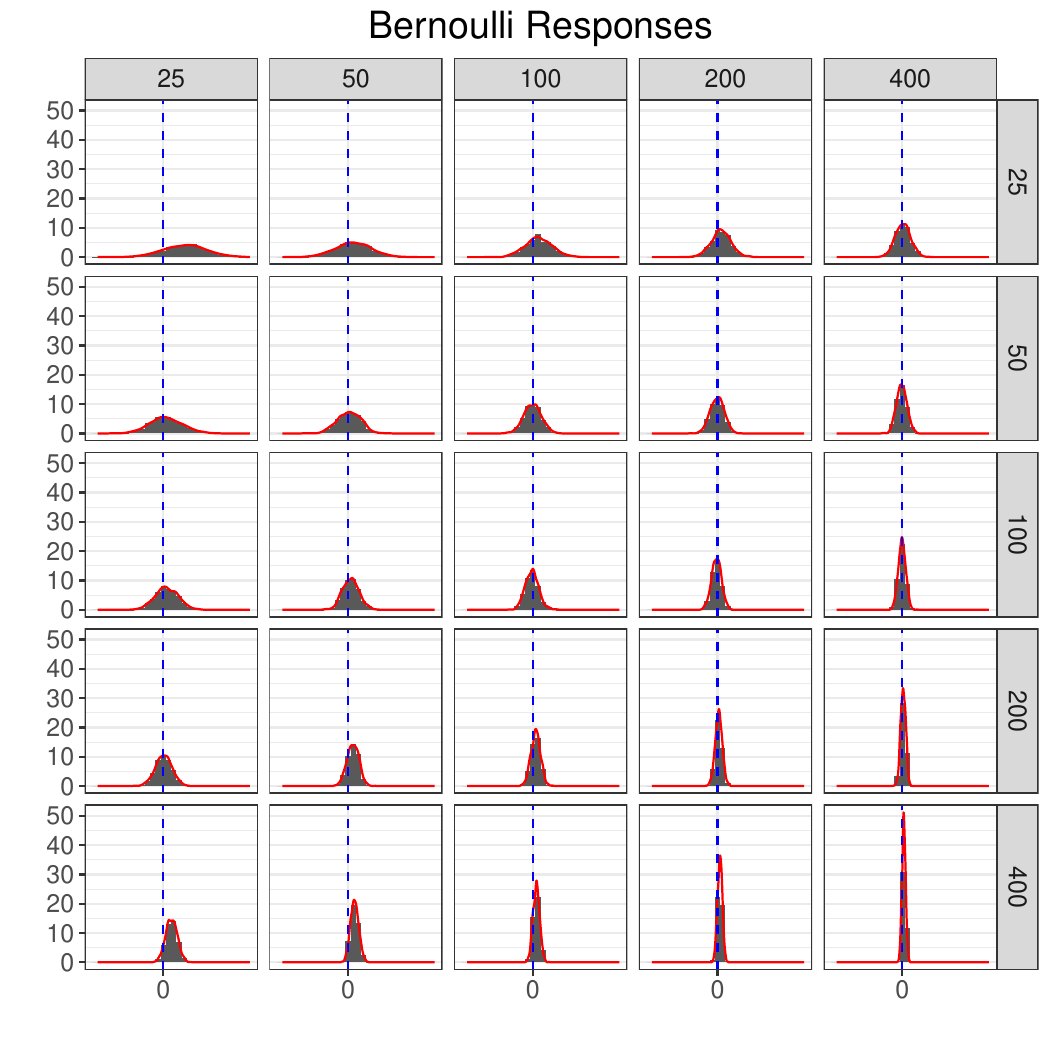}
\includegraphics[width=0.49\linewidth]{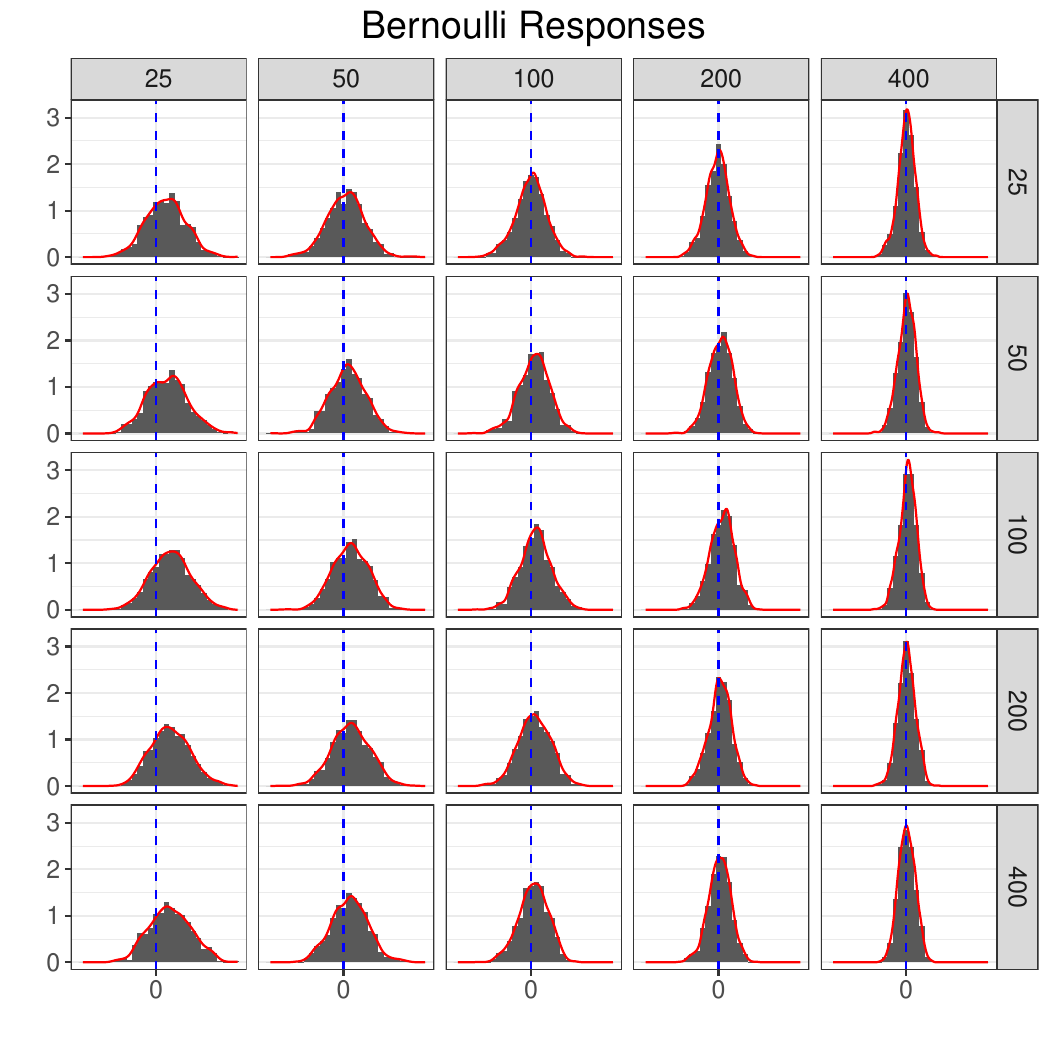}
\caption{Histograms for the third components of $\hat{\bmbeta} - \dot{\bmbeta}$ (left panels) and $\hat{\bm{b}}_1 - \dot{\bm{b}}_1$ (right panels), under the unconditional regime. Vertical facets represent the cluster sizes, while horizontal facets represent the number of clusters. The dotted blue line indicates zero, and the red curve is a kernel density smoother.} 
\end{figure}

\subsection{{\boldmath ${G}$} = {\boldmath ${I}_2$}}

\begin{figure}[H]
\centering
\includegraphics[width=0.95\linewidth]{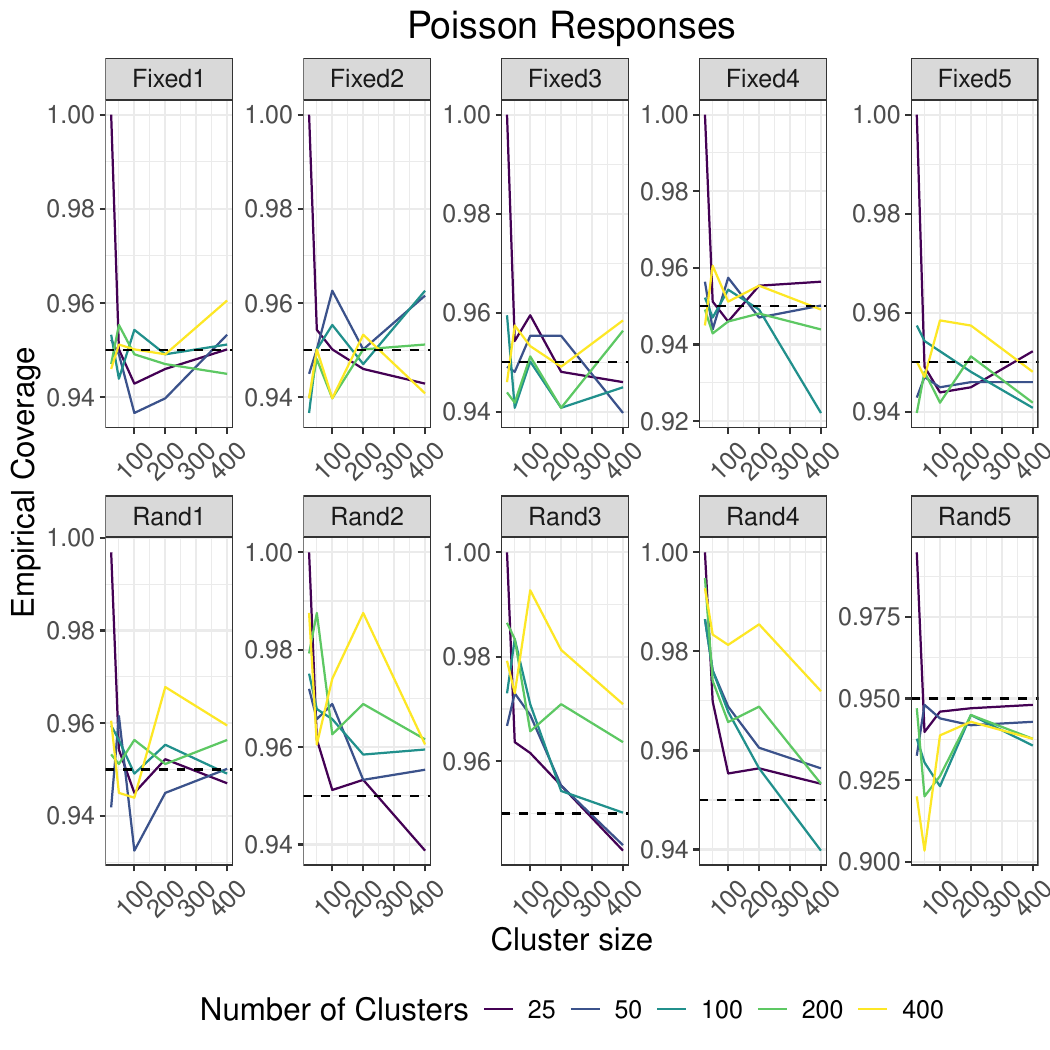}
\caption{Empirical coverage probability of 95\% coverage intervals for the five fixed and random effects estimates, obtained under the unconditional regime with Poisson responses.} 
\end{figure}

\begin{figure}[H]
\centering
\includegraphics[width=0.95\linewidth]{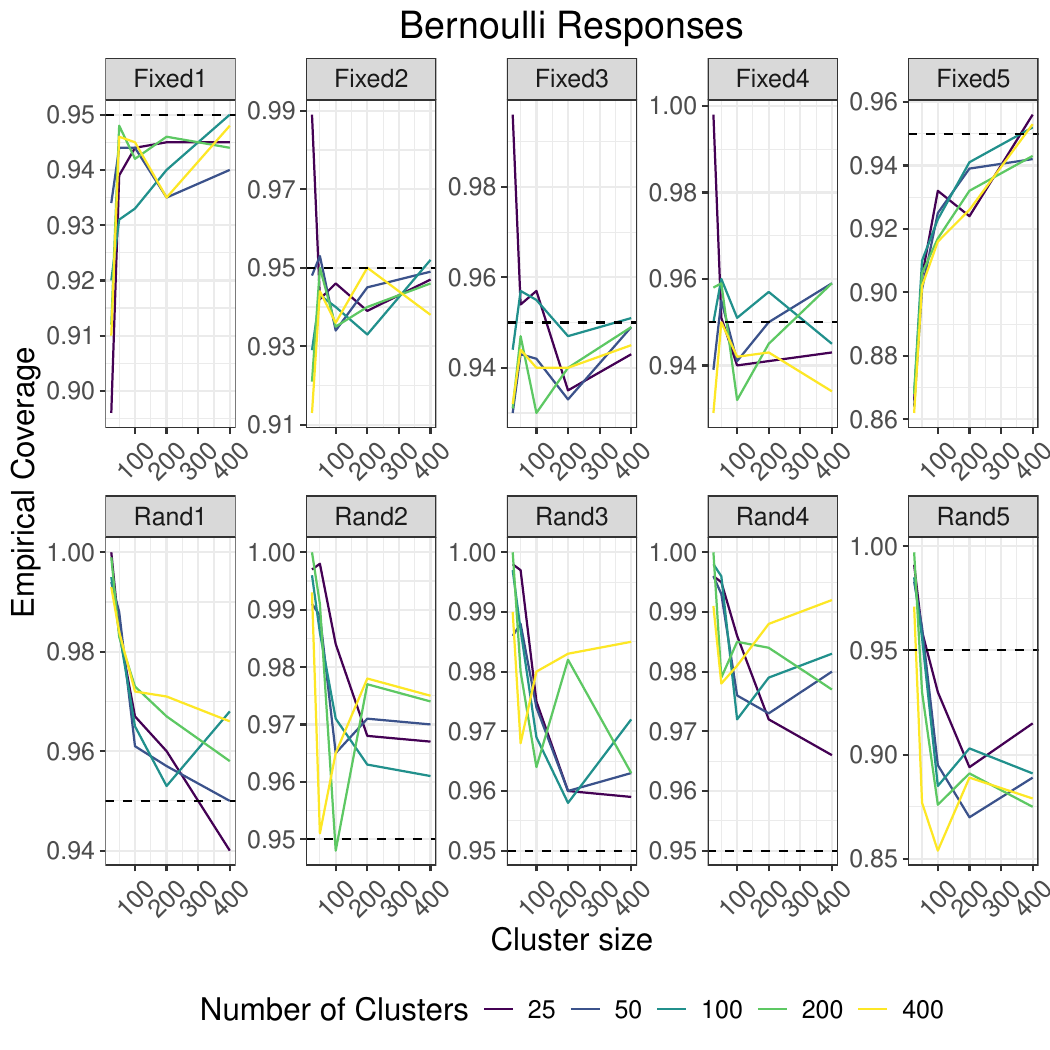}
\caption{Empirical coverage probability of 95\% coverage intervals for the five fixed and random effects estimates, obtained under the unconditional regime with Bernoulli responses.} 
\end{figure}

\begin{figure}[H]
\includegraphics[width=0.7\linewidth]{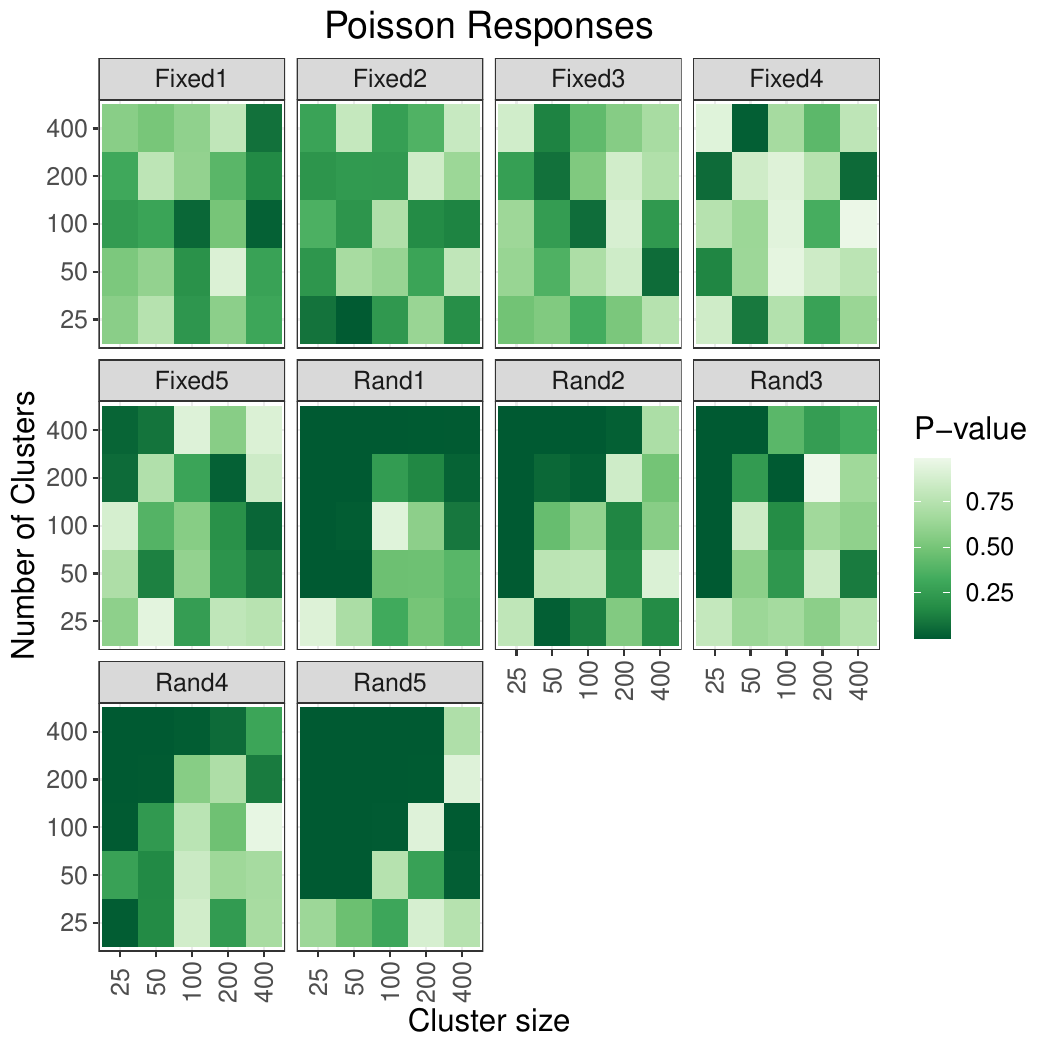}
\includegraphics[width=0.7\linewidth]{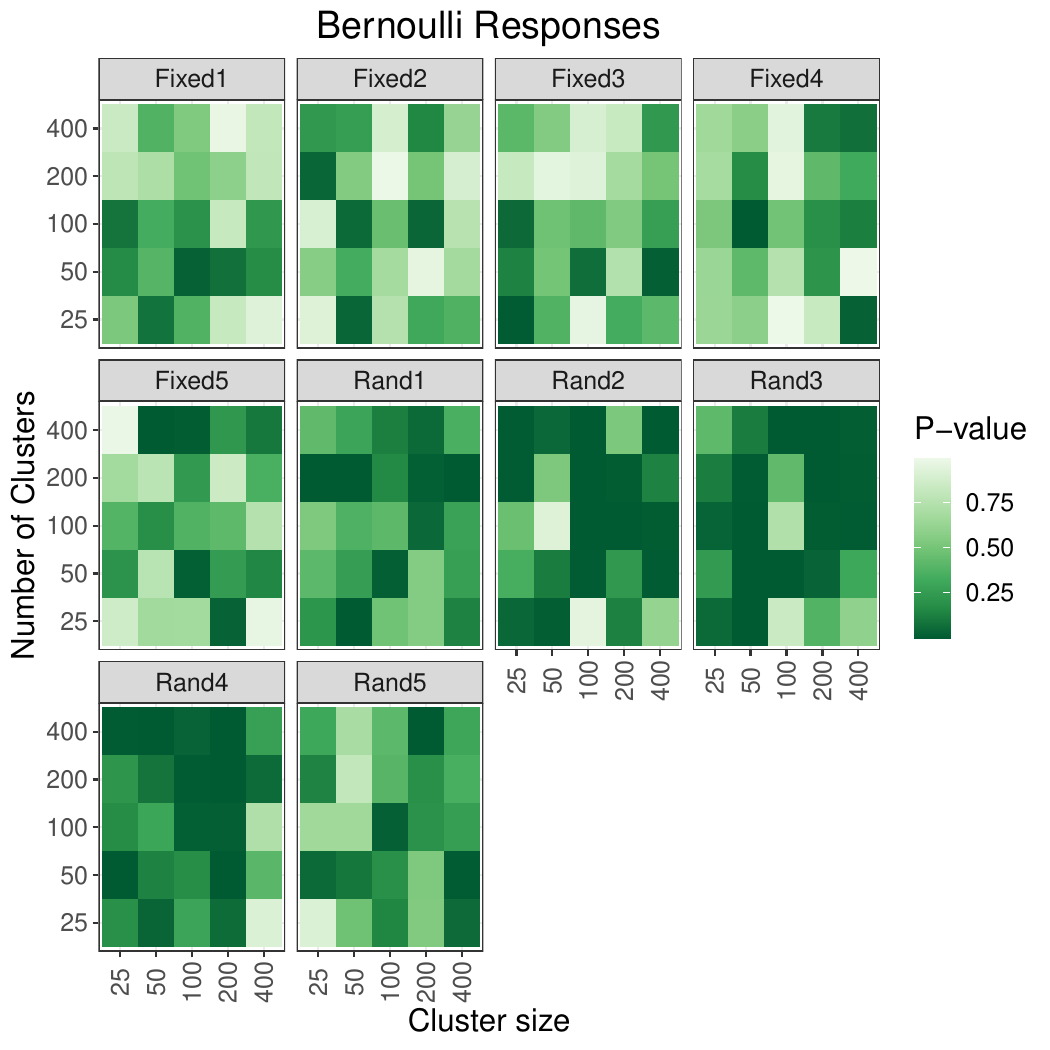}
\caption{$p$-values from Shapiro-Wilk tests applied to the fixed and random effects estimates obtained using maximum PQL estimation, under the unconditional regime.} 
\end{figure}

\begin{figure}[H]
\includegraphics[width=0.49\linewidth]{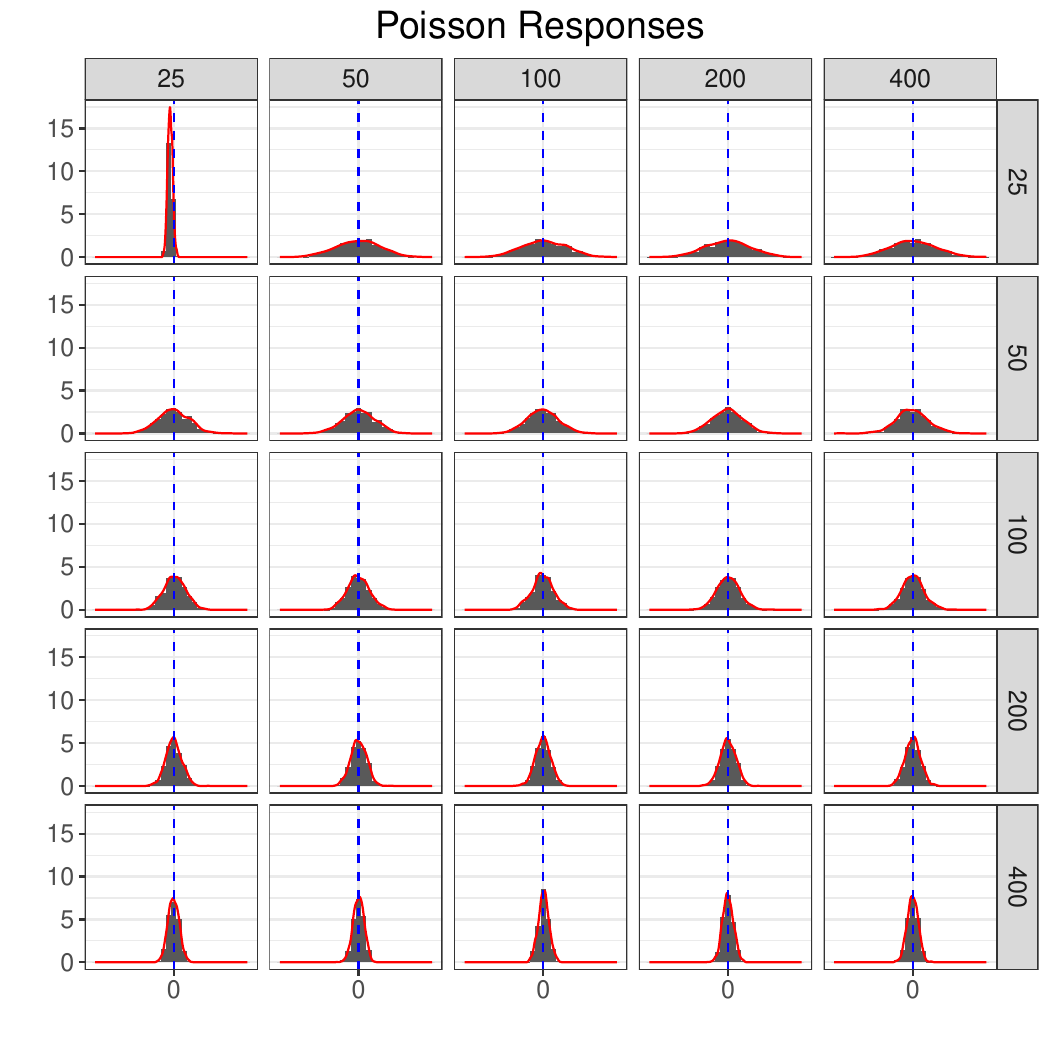}
\includegraphics[width=0.49\linewidth]{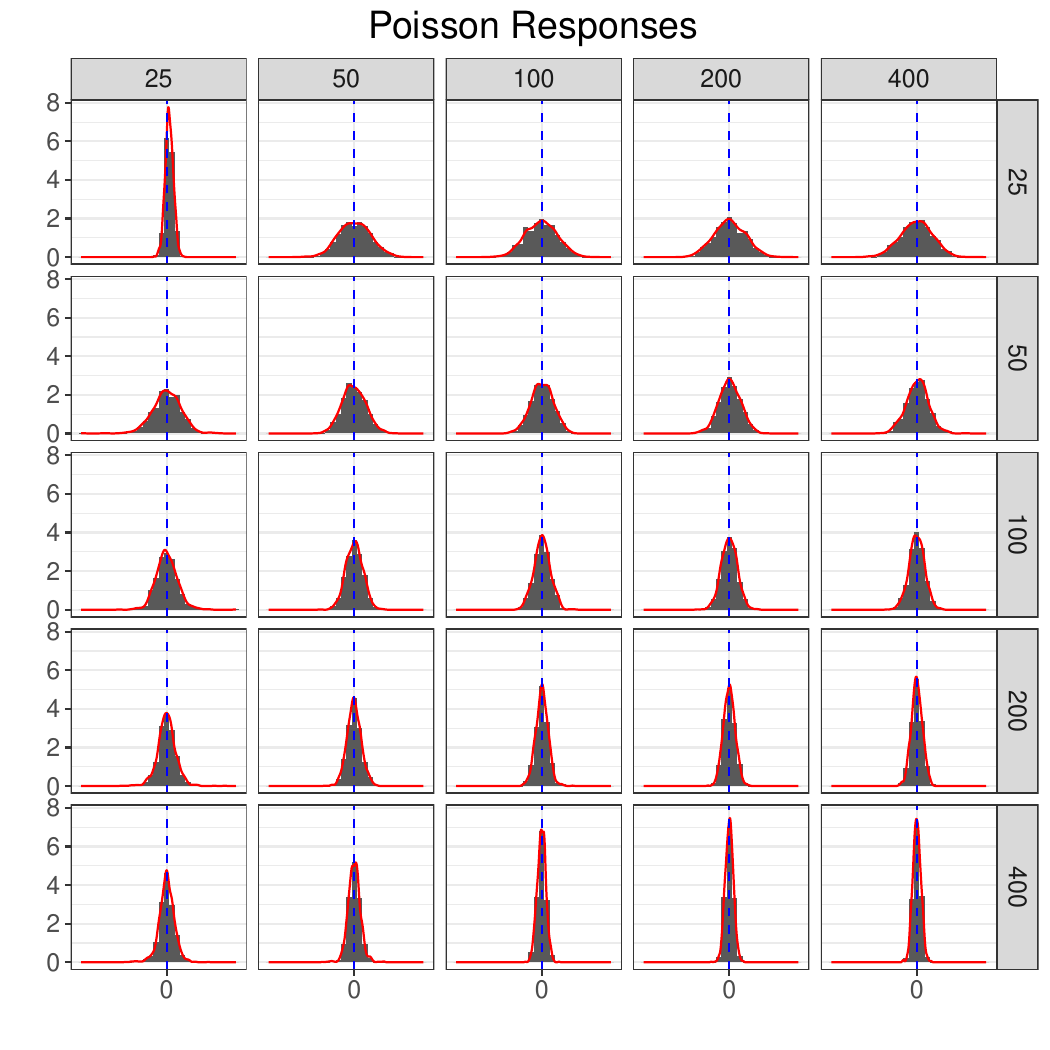}
\includegraphics[width=0.49\linewidth]{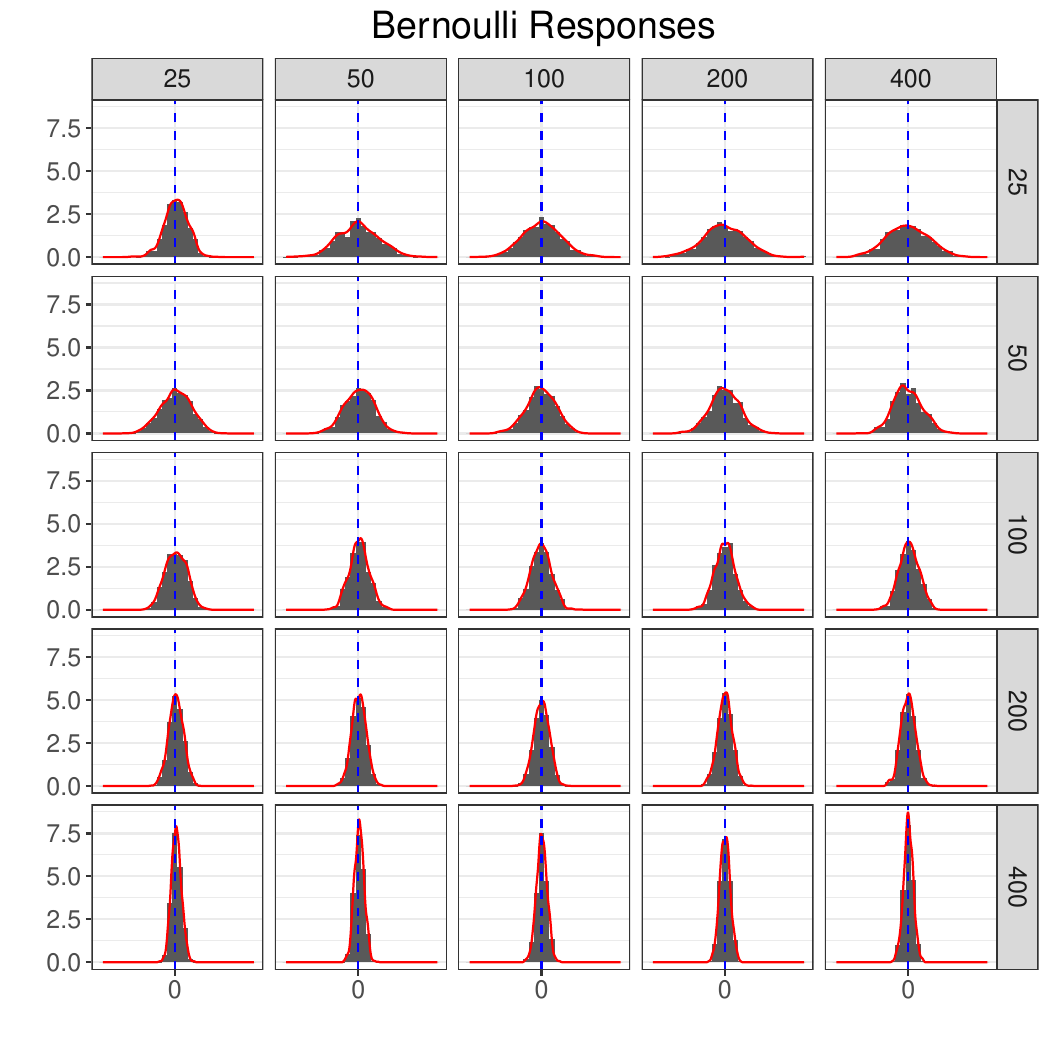}
\includegraphics[width=0.49\linewidth]{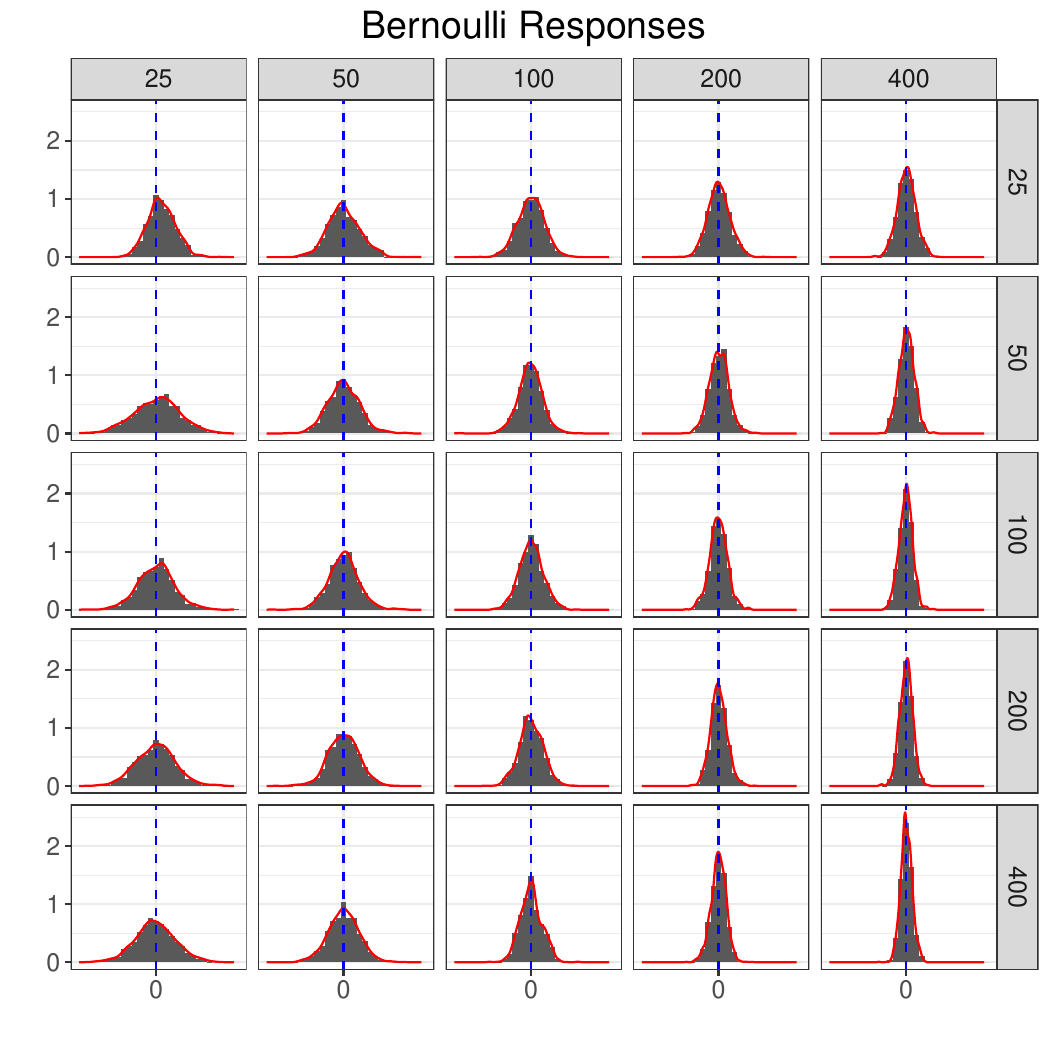}
\caption{Histograms for the third components of $\hat{\bmbeta} - \dot{\bmbeta}$ (left panels) and $\hat{\bm{b}}_1 - \dot{\bm{b}}_1$ (right panels), under the unconditional regime. Vertical facets represent the cluster sizes, while horizontal facets represent the number of clusters. The dotted blue line indicates zero, and the red curve is a kernel density smoother.} 
\end{figure}

\begin{figure}[H]
\centering
\includegraphics[width=0.95\linewidth]{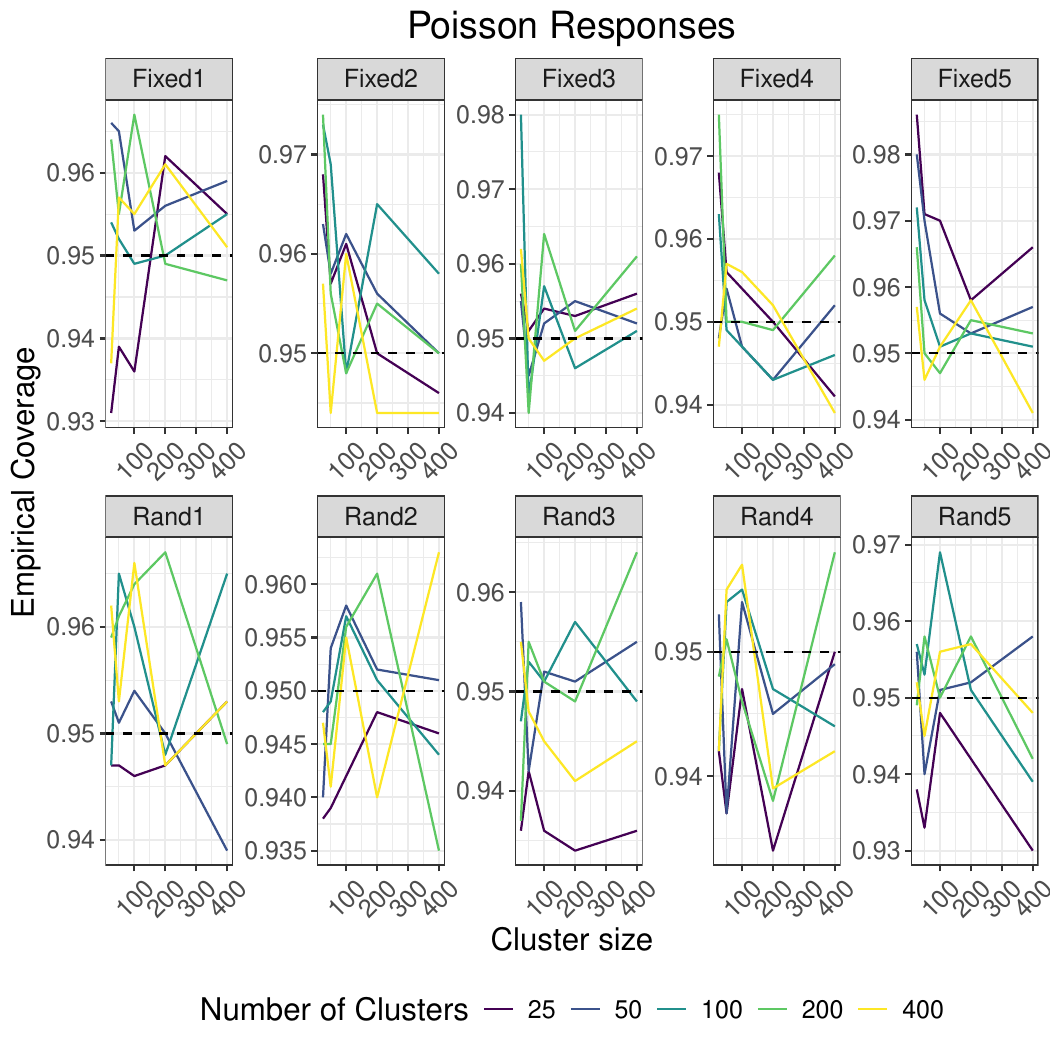}
\caption{Empirical coverage probability of 95\% coverage intervals for the five fixed and random effects estimates, obtained under the conditional regime with Poisson responses.} 
\end{figure}

\begin{figure}[H]
\centering
\includegraphics[width=0.95\linewidth]{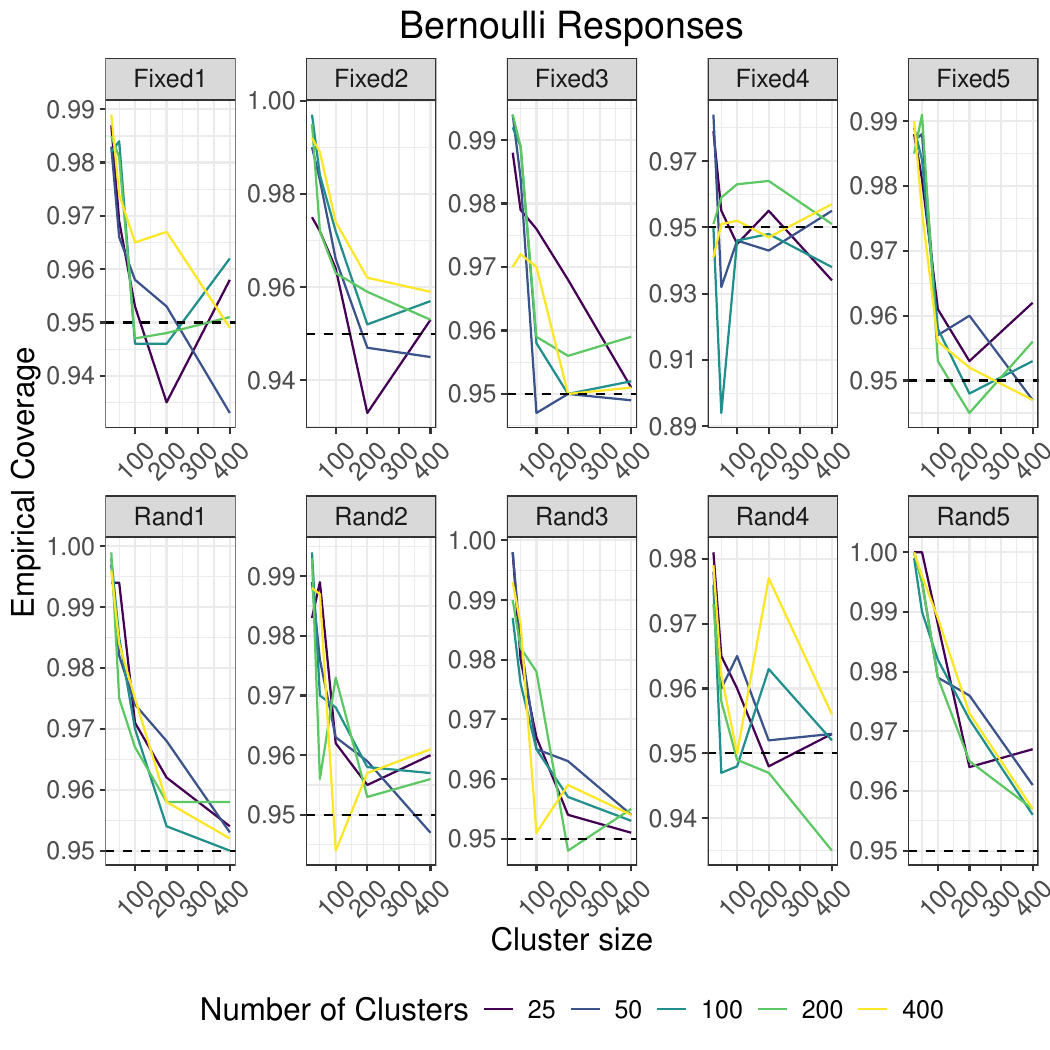}
\caption{Empirical coverage probability of 95\% coverage intervals for the five fixed and random effects estimates, obtained under the conditional regime with Bernoulli responses.} 
\end{figure}

\begin{figure}[H]
\includegraphics[width=0.7\linewidth]{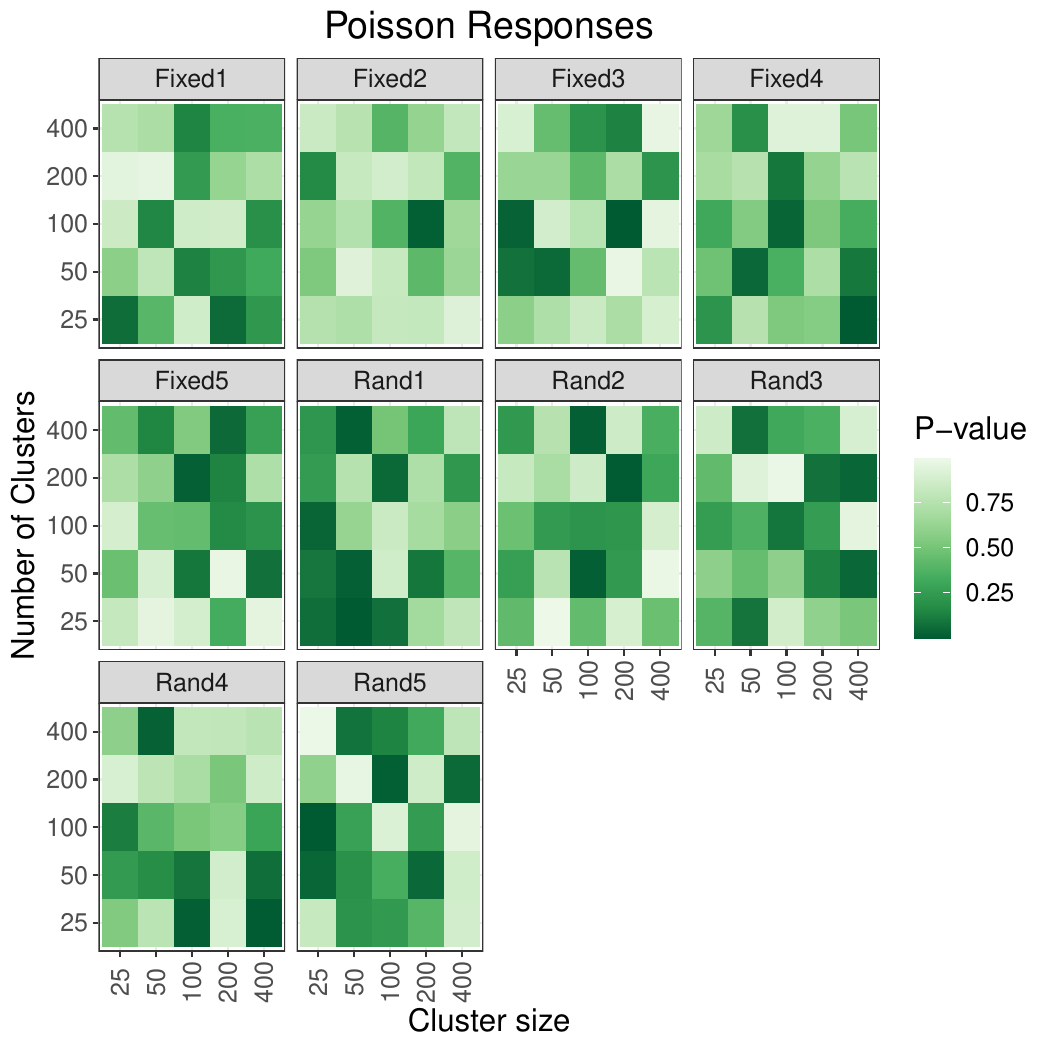}
\includegraphics[width=0.7\linewidth]{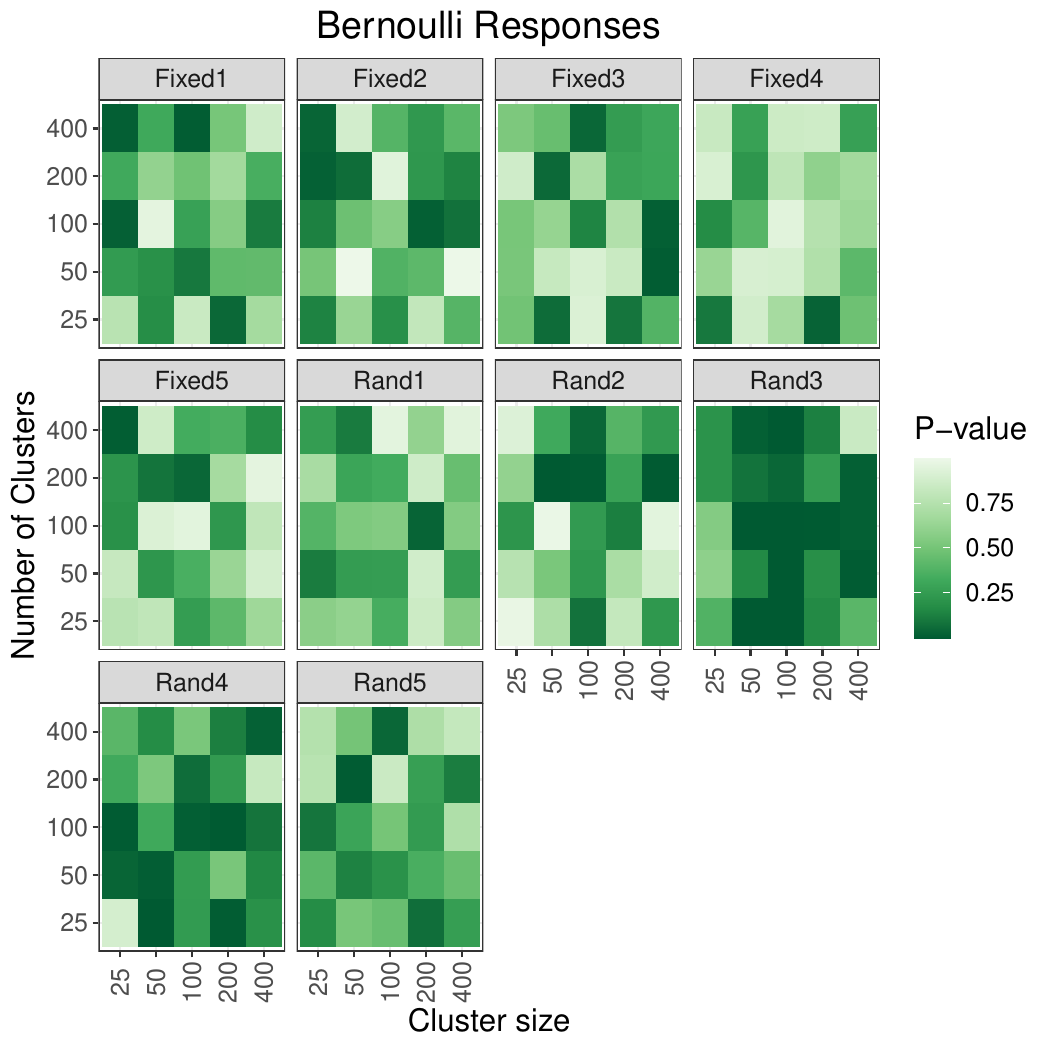}
\caption{$p$-values from Shapiro-Wilk tests applied to the fixed and random effects estimates obtained using maximum PQL estimation, under the conditional regime.} 
\end{figure}

\begin{figure}[H]
\includegraphics[width=0.49\linewidth]{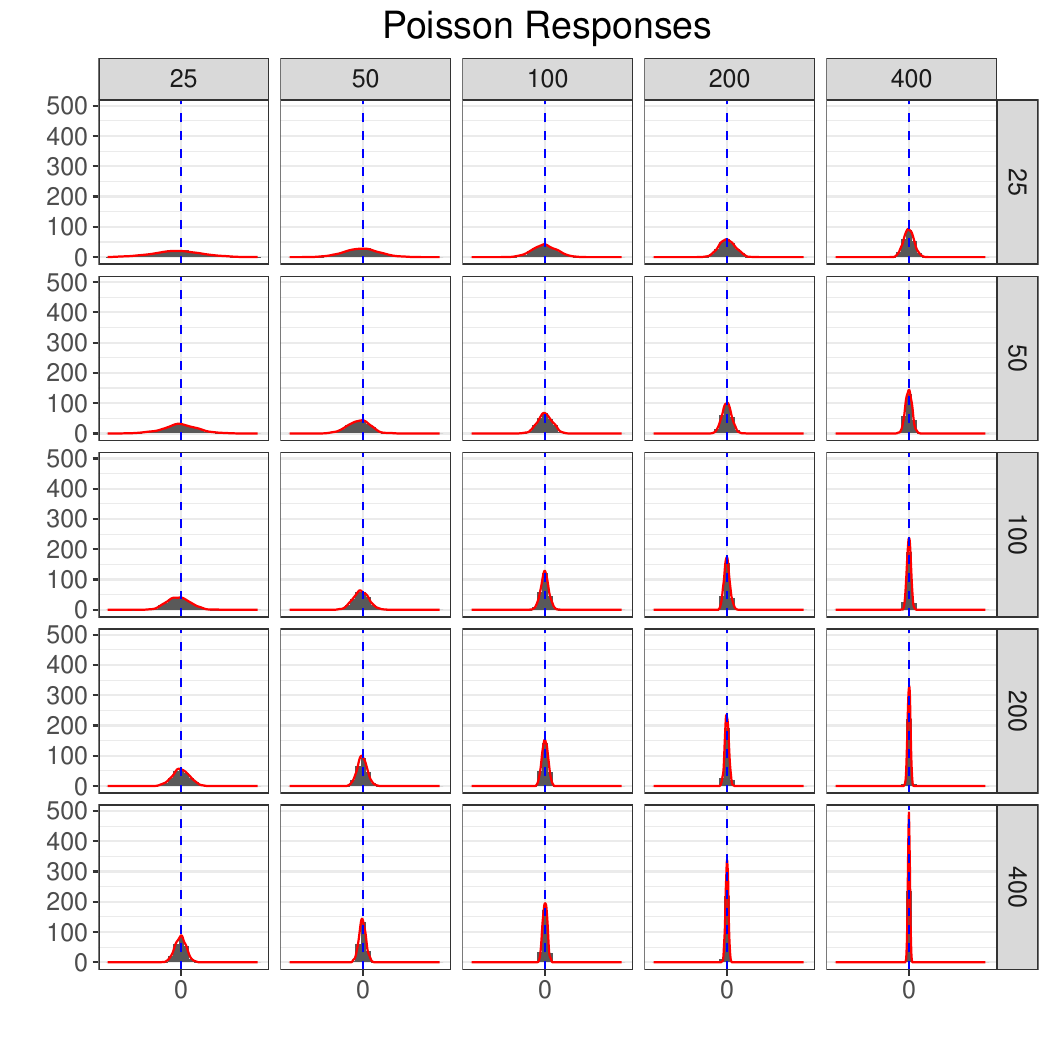}
\includegraphics[width=0.49\linewidth]{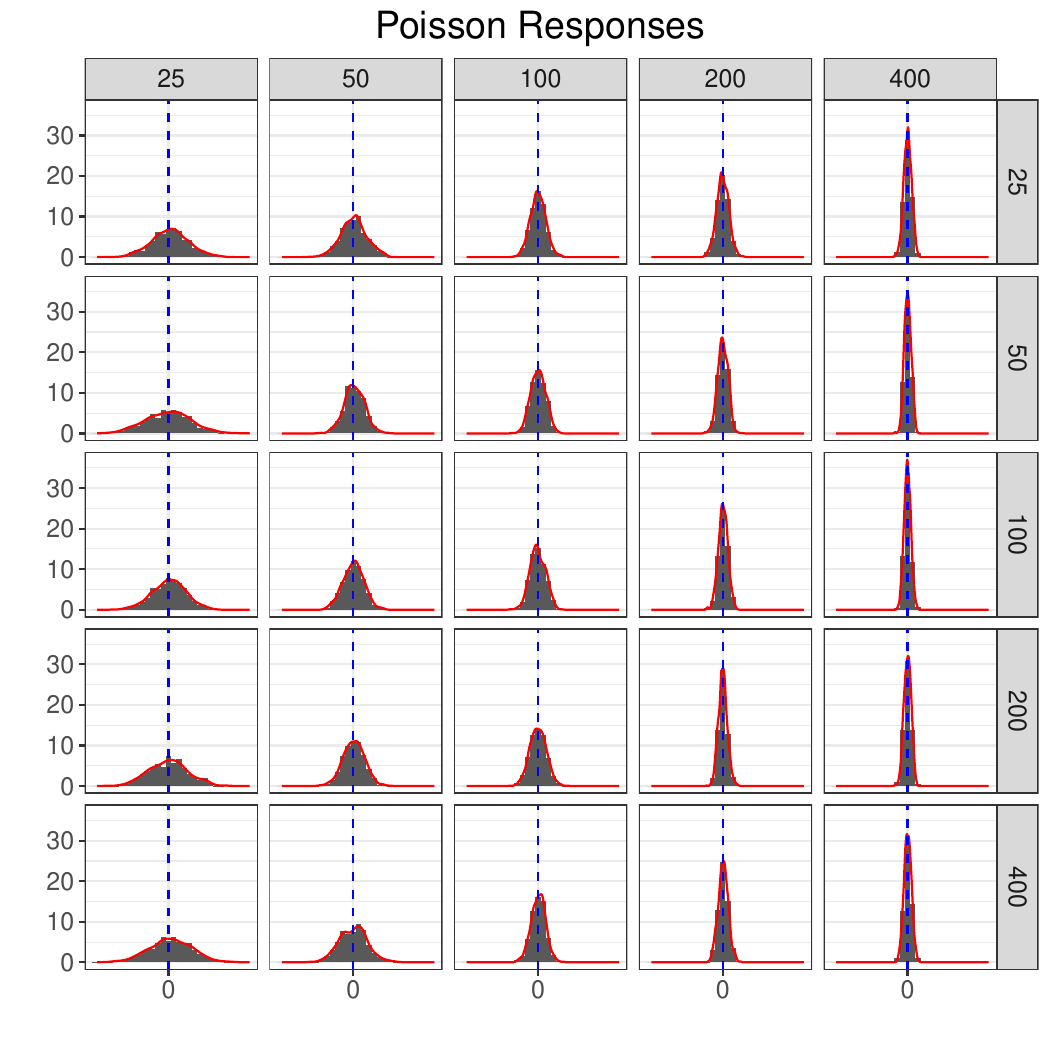}
\includegraphics[width=0.49\linewidth]{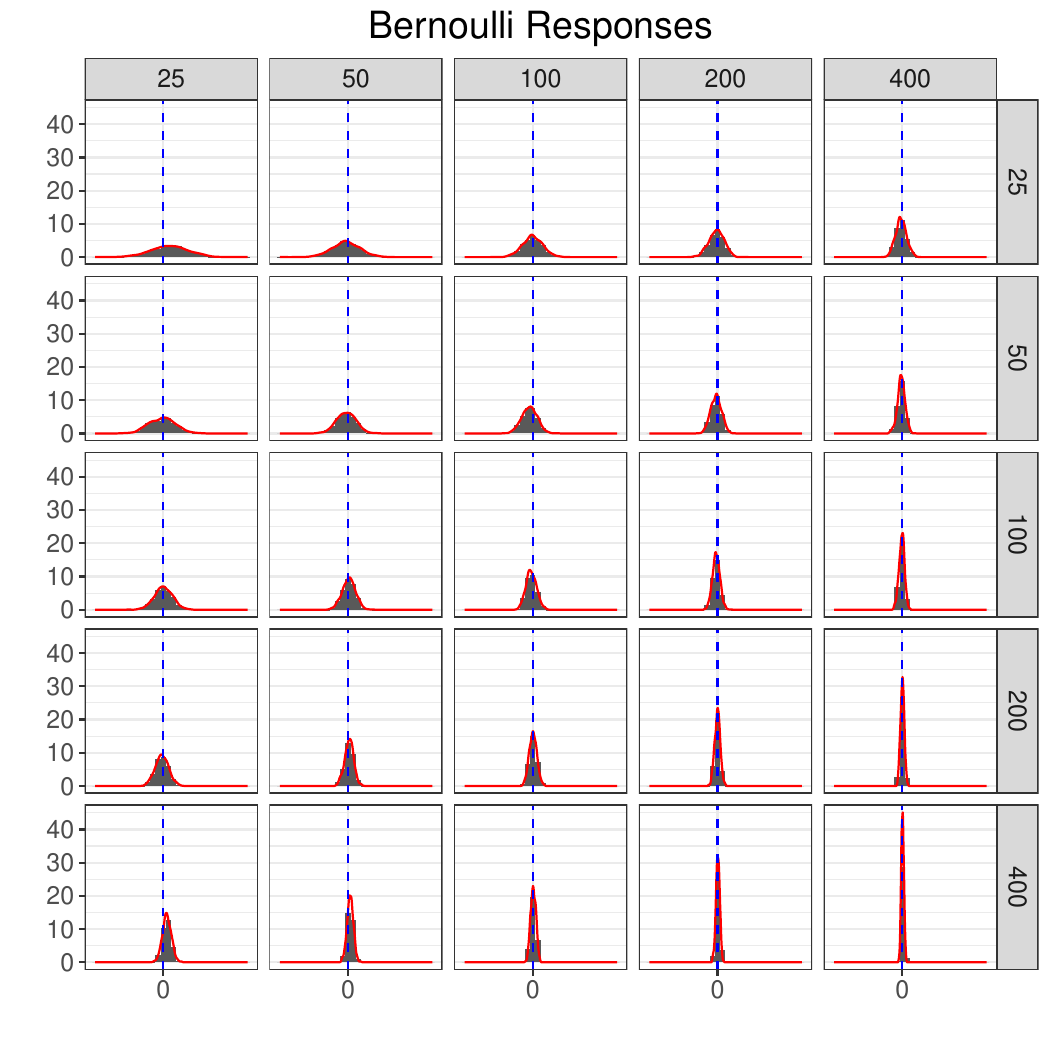}
\includegraphics[width=0.49\linewidth]{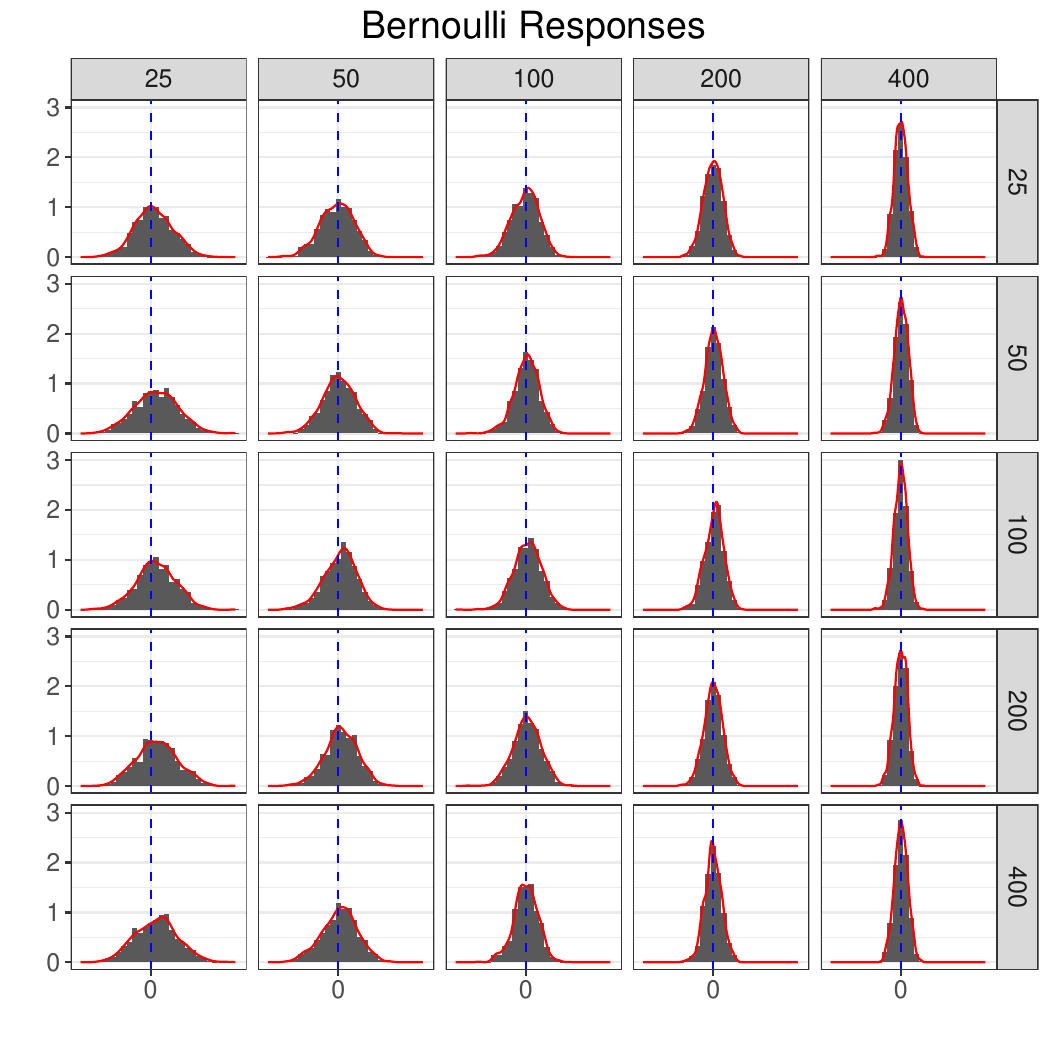}
\caption{Histograms for the third components of $\hat{\bmbeta} - \dot{\bmbeta}$ (left panels) and $\hat{\bm{b}}_1 - \dot{\bm{b}}_1$ (right panels), under the unconditional regime. Vertical facets represent the cluster sizes, while horizontal facets represent the number of clusters. The dotted blue line indicates zero, and the red curve is a kernel density smoother.} 
\end{figure}

\subsection{{\boldmath ${G}$} = 2  {\boldmath ${I}_2$}}

\begin{figure}[H]
\centering
\includegraphics[width=0.95\linewidth]{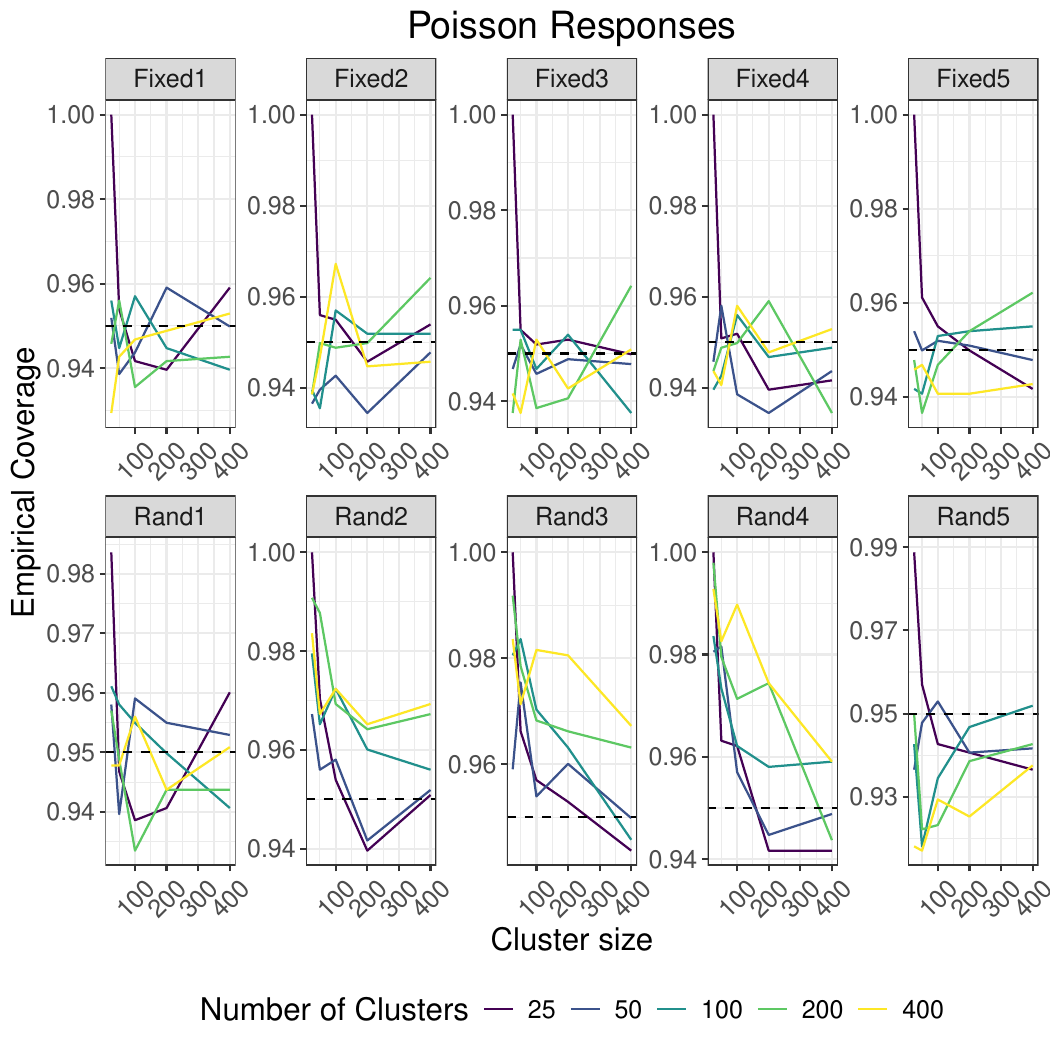}
\caption{Empirical coverage probability of 95\% coverage intervals for the five fixed and random effects estimates, obtained under the unconditional regime with Poisson responses.} 
\end{figure}

\begin{figure}[H]
\centering
\includegraphics[width=0.95\linewidth]{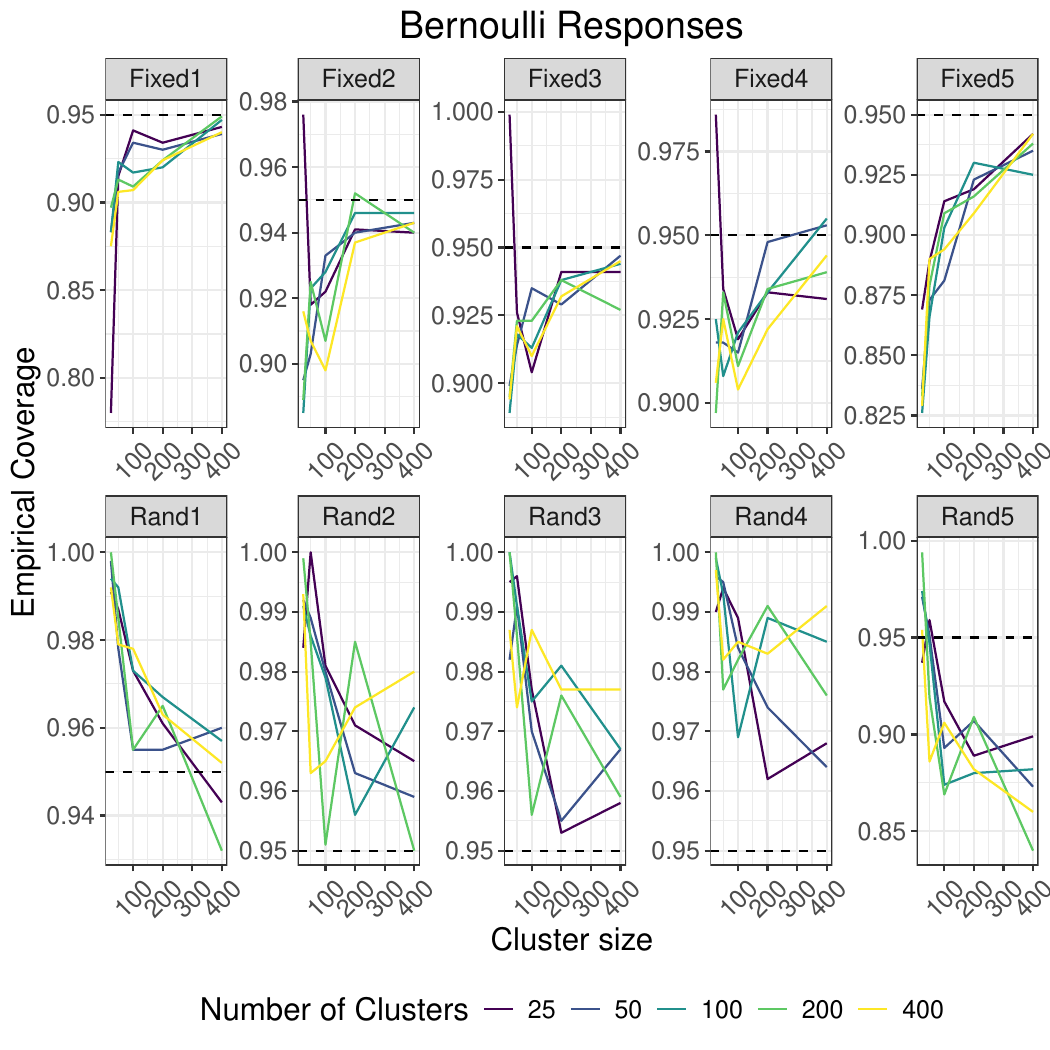}
\caption{Empirical coverage probability of 95\% coverage intervals for the five fixed and random effects estimates, obtained under the unconditional regime with Bernoulli responses.} 
\end{figure}

\begin{figure}[H]
\includegraphics[width=0.7\linewidth]{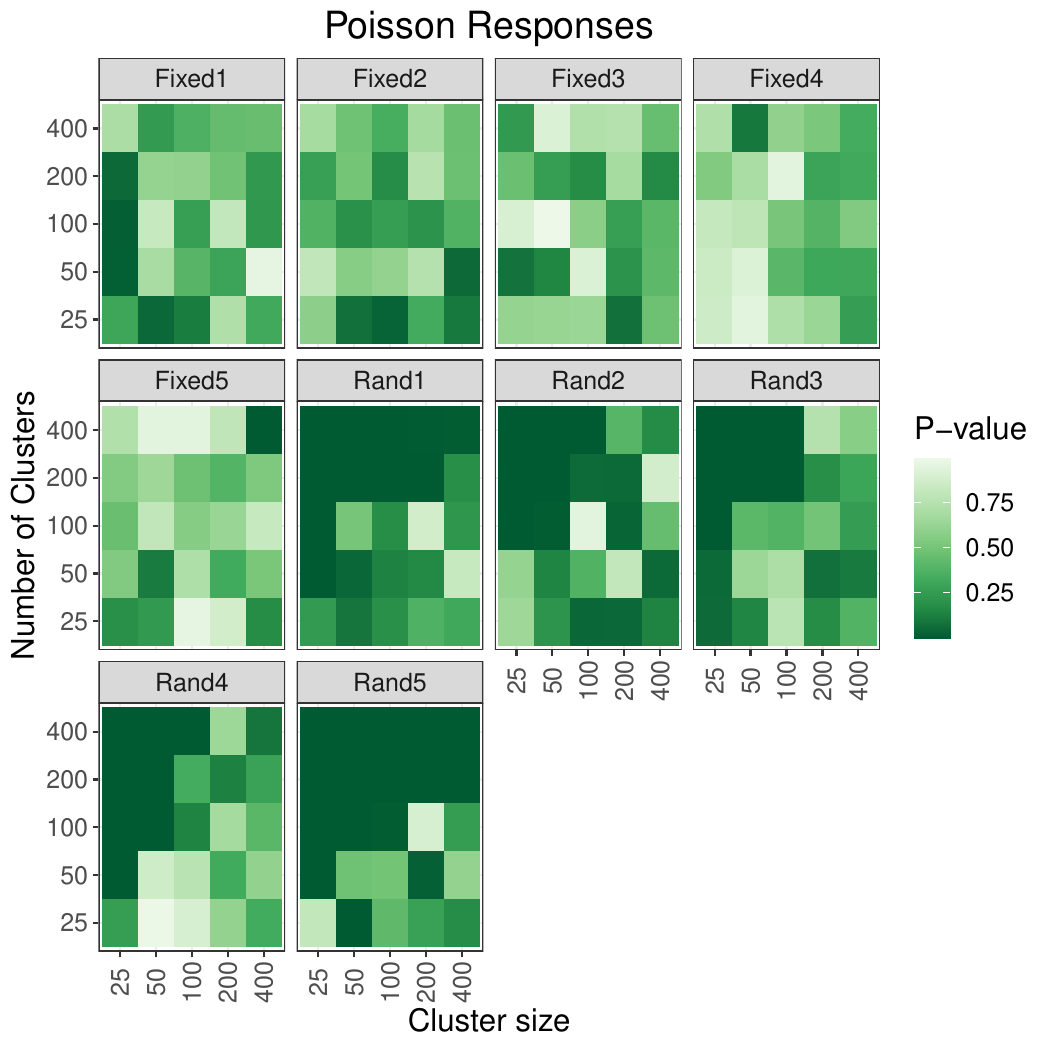}
\includegraphics[width=0.7\linewidth]{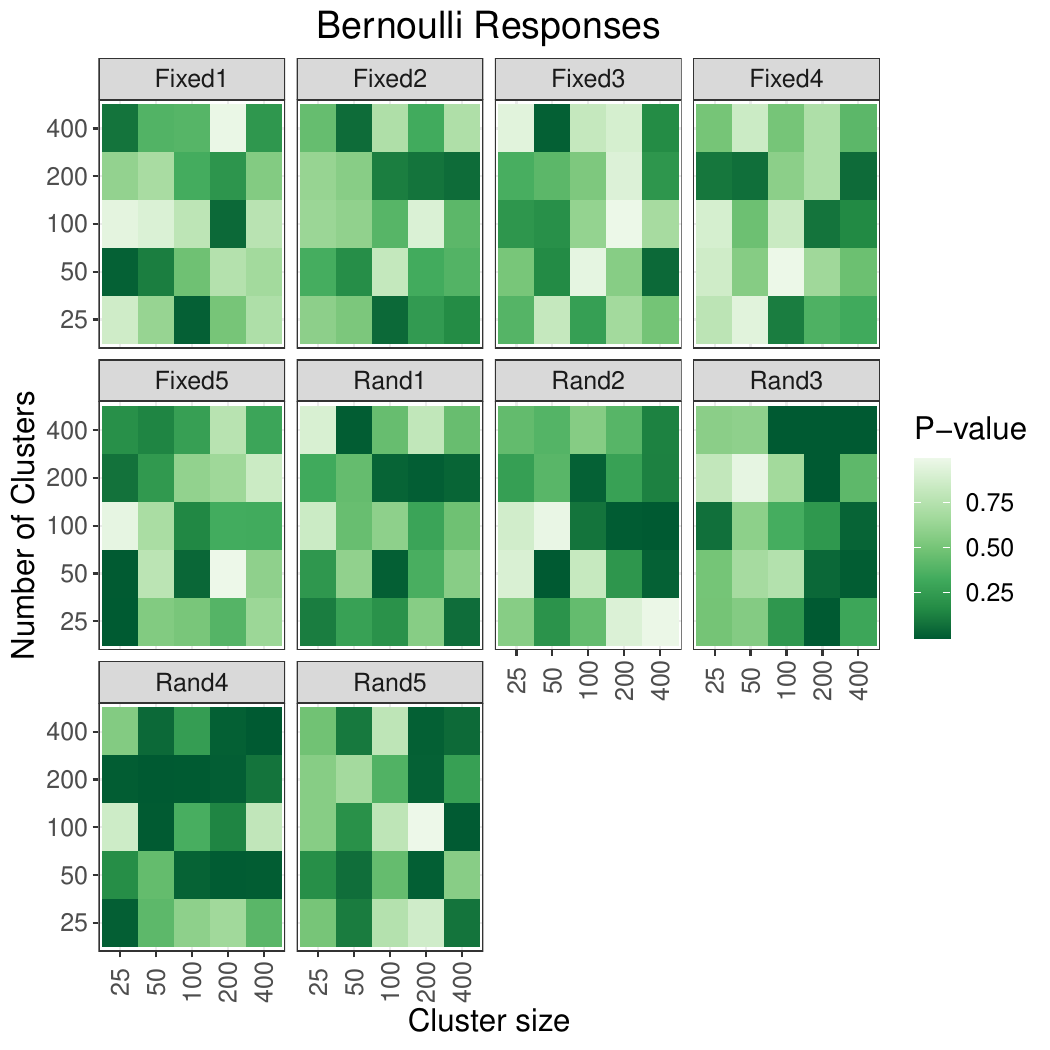}
\caption{$p$-values from Shapiro-Wilk tests applied to the fixed and random effects estimates obtained using maximum PQL estimation, under the unconditional regime.} 
\end{figure}

\begin{figure}[H]
\includegraphics[width=0.49\linewidth]{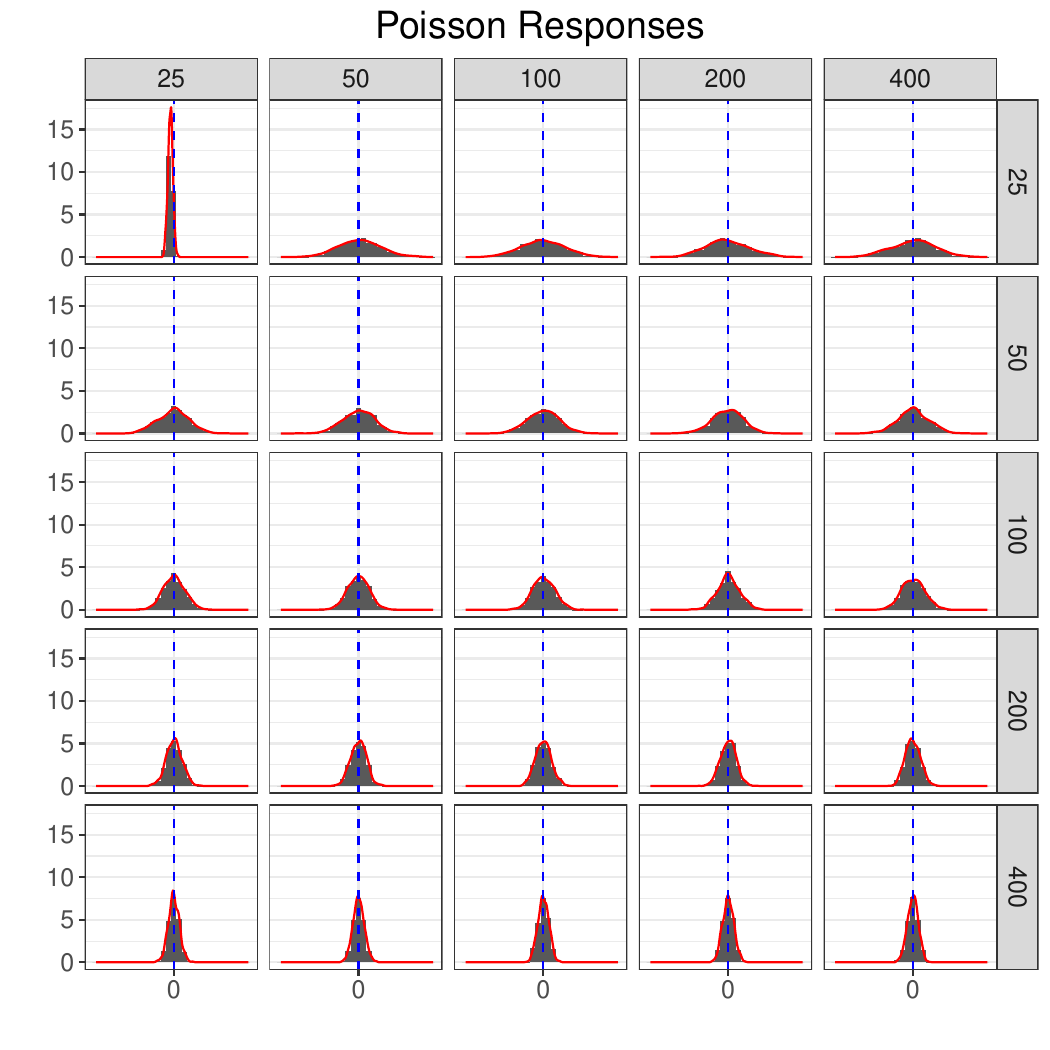}
\includegraphics[width=0.49\linewidth]{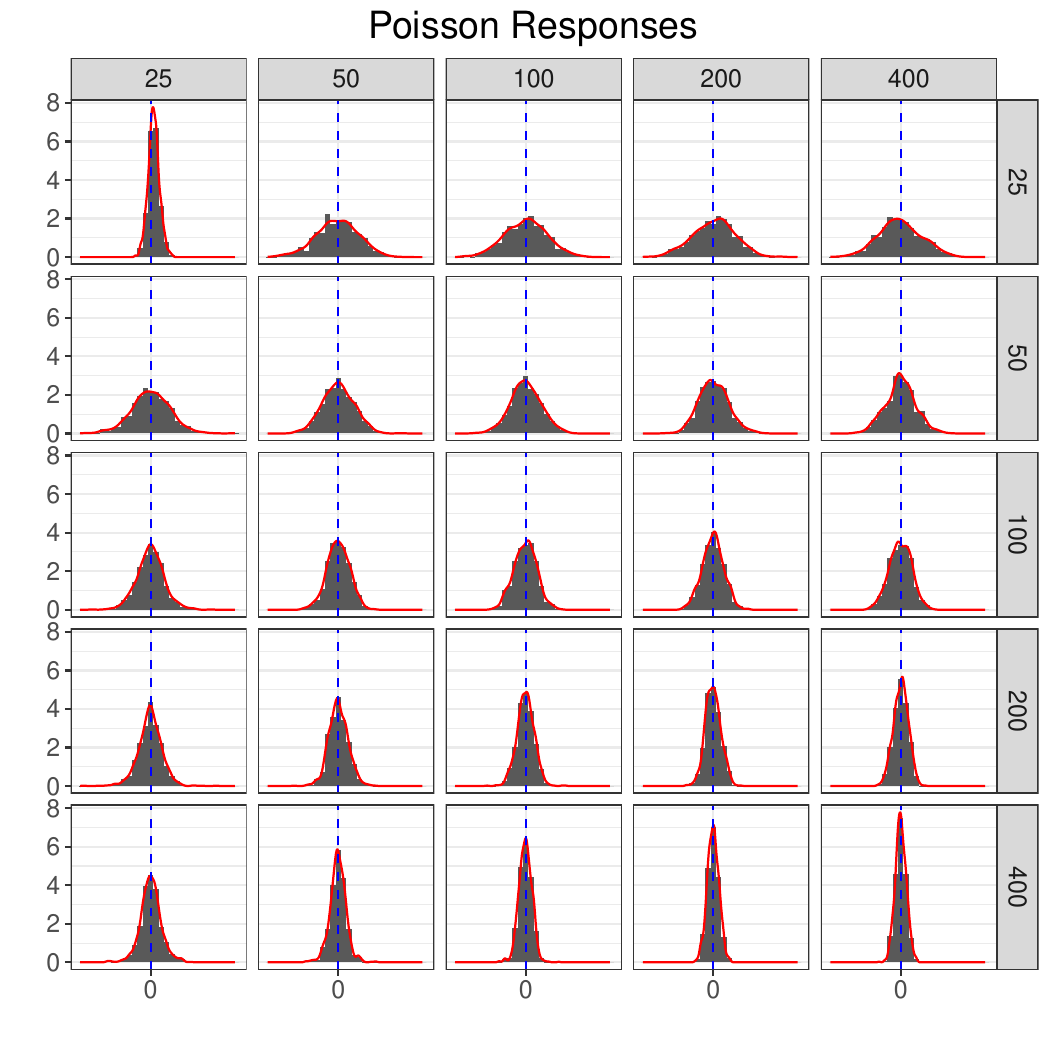}
\includegraphics[width=0.49\linewidth]{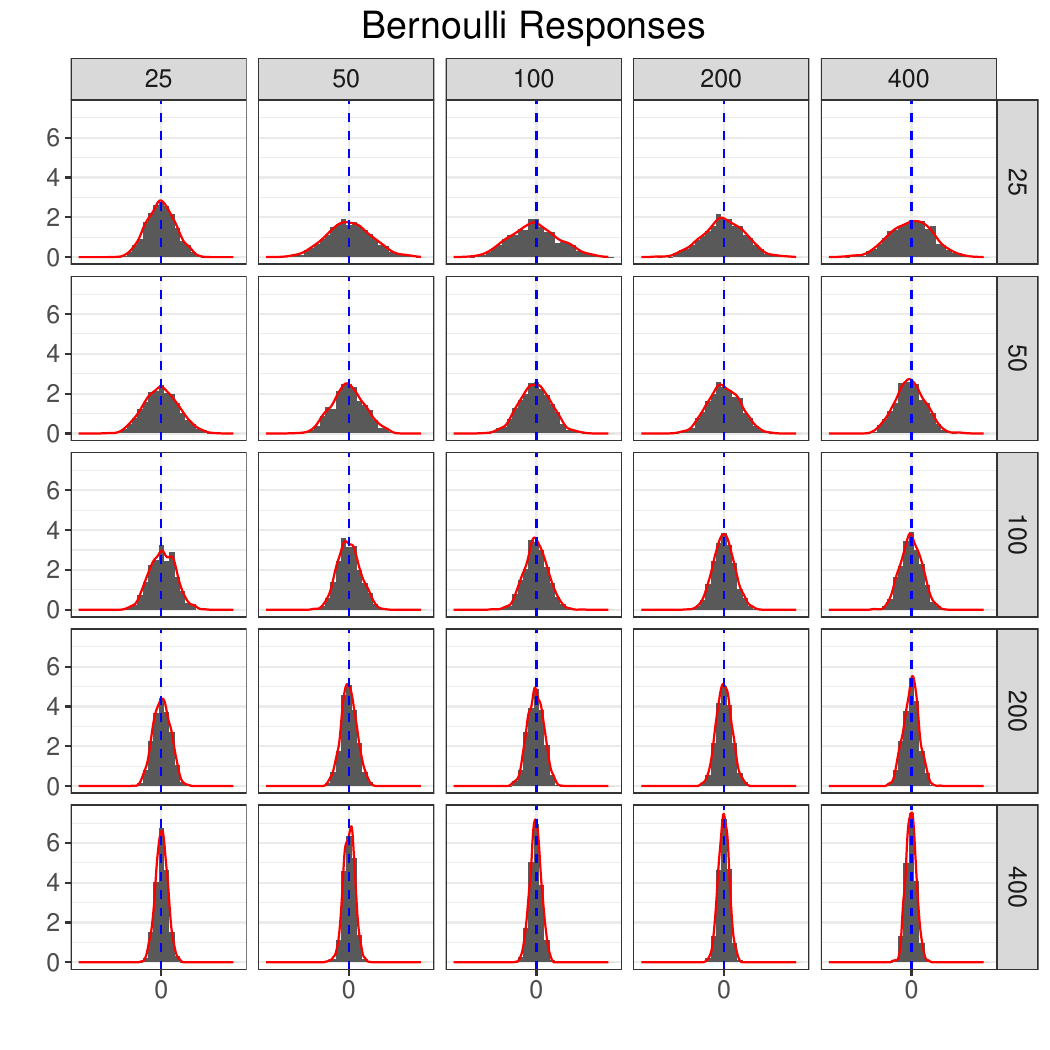}
\includegraphics[width=0.49\linewidth]{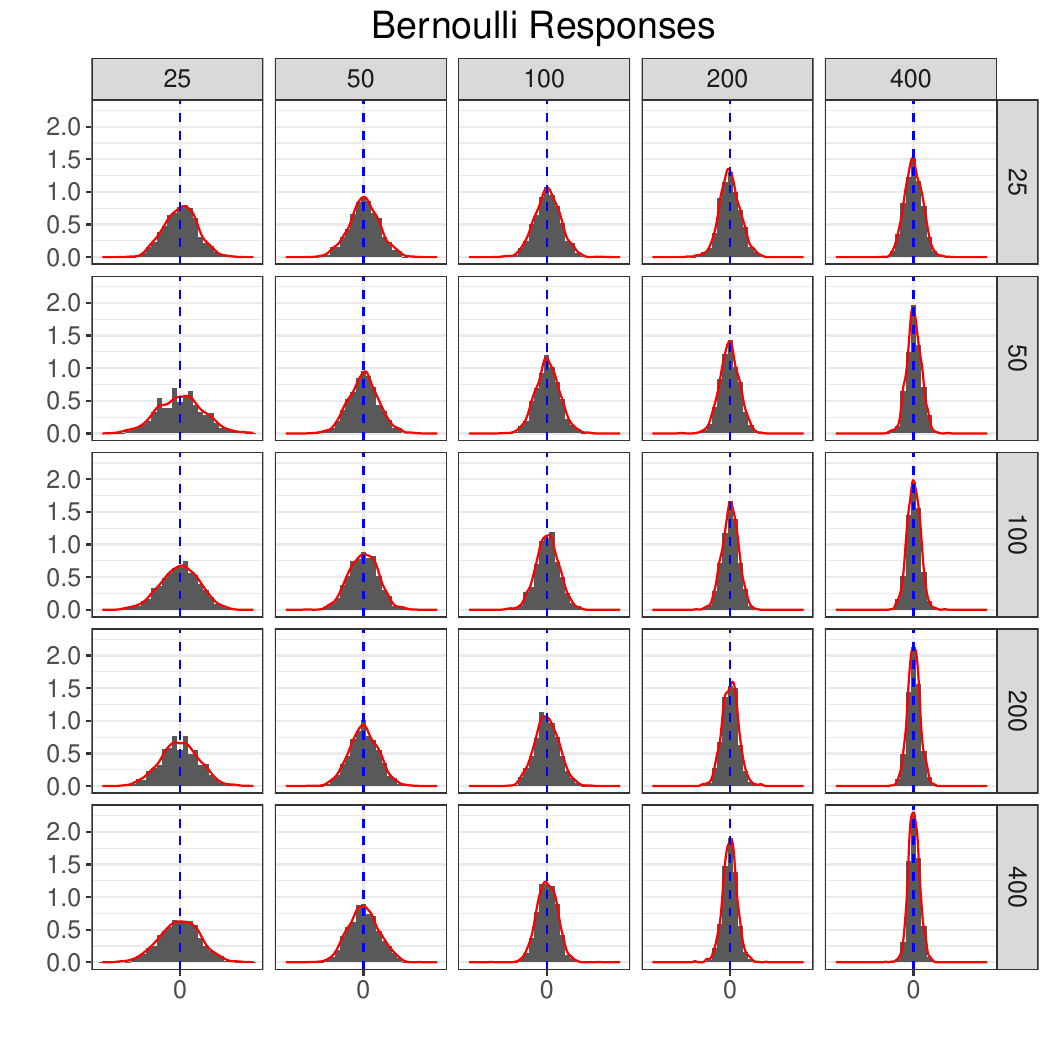}
\caption{Histograms for the third components of $\hat{\bmbeta} - \dot{\bmbeta}$ (left panels) and $\hat{\bm{b}}_1 - \dot{\bm{b}}_1$ (right panels), under the unconditional regime. Vertical facets represent the cluster sizes, while horizontal facets represent the number of clusters. The dotted blue line indicates zero, and the red curve is a kernel density smoother.} 
\end{figure}

\begin{figure}[H]
\centering
\includegraphics[width=0.95\linewidth]{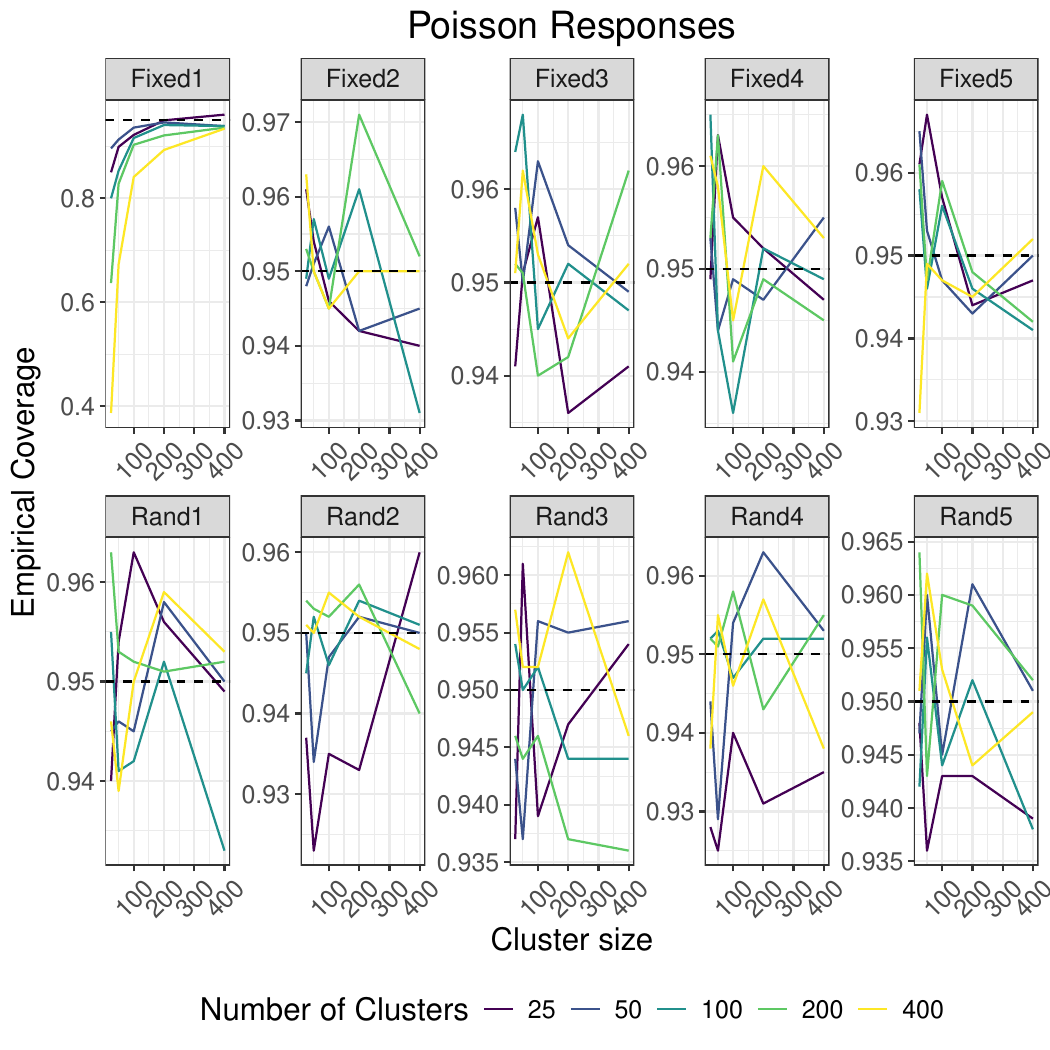}
\caption{Empirical coverage probability of 95\% coverage intervals for the five fixed and random effects estimates, obtained under the conditional regime with Poisson responses.} 
\end{figure}

\begin{figure}[H]
\centering
\includegraphics[width=0.95\linewidth]{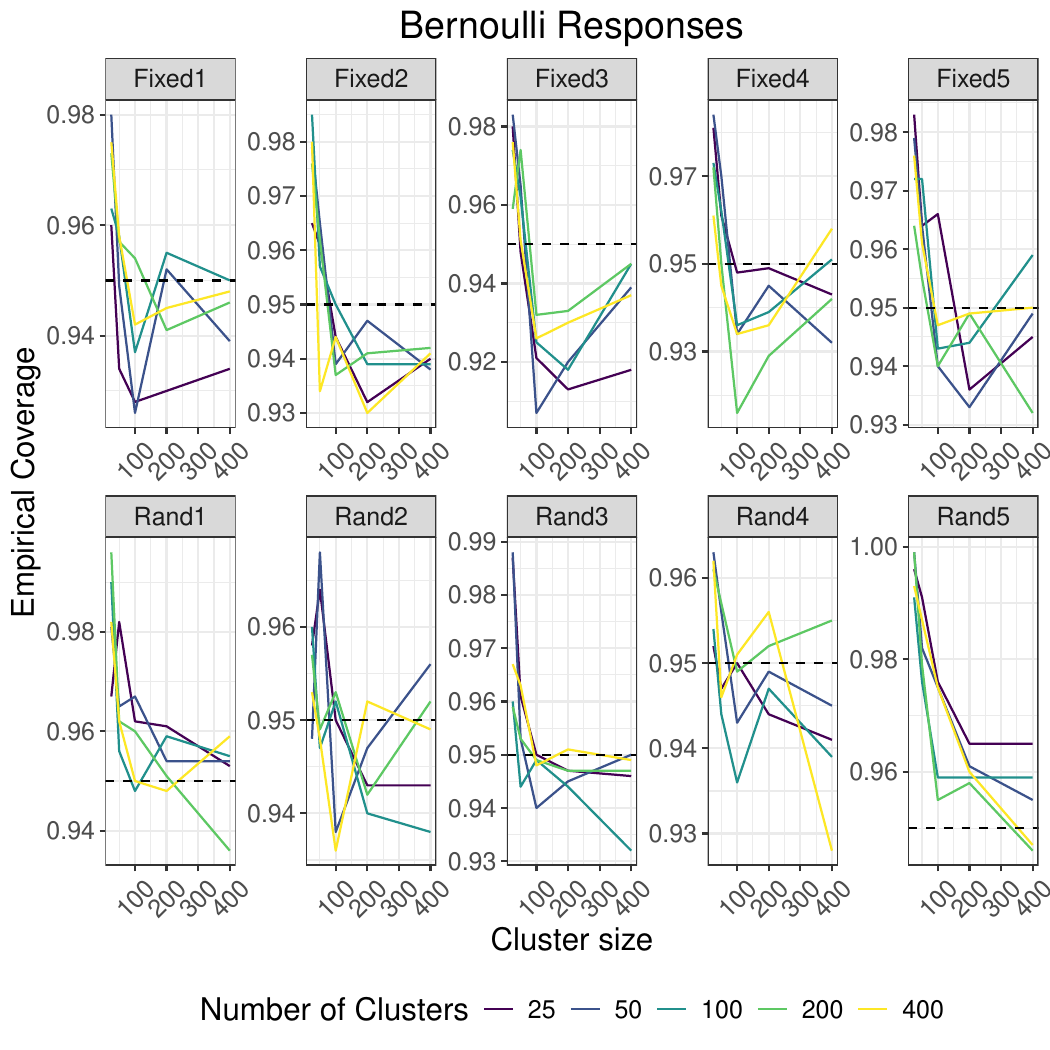}
\caption{Empirical coverage probability of 95\% coverage intervals for the five fixed and random effects estimates, obtained under the conditional regime with Bernoulli responses.} 
\end{figure}

\begin{figure}[H]
\includegraphics[width=0.7\linewidth]{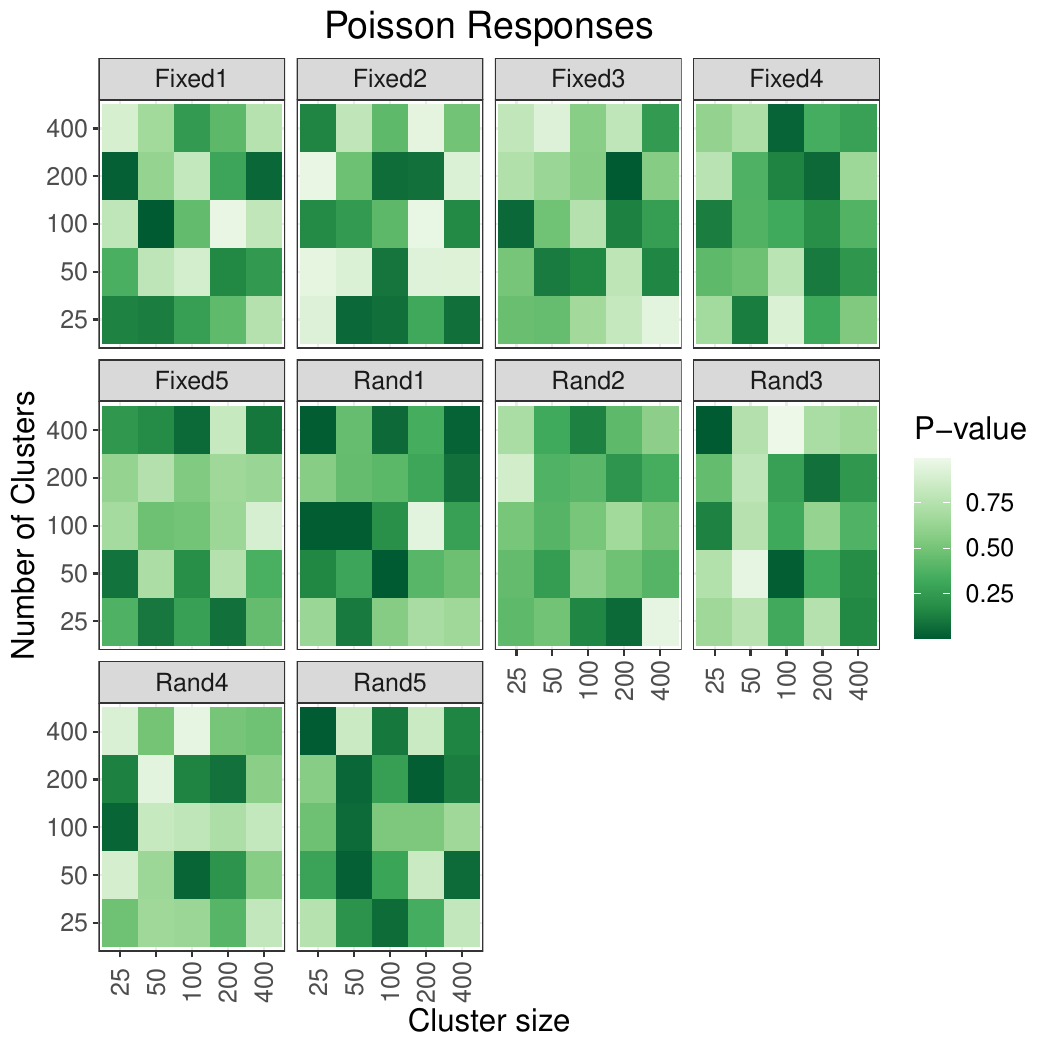}
\includegraphics[width=0.7\linewidth]{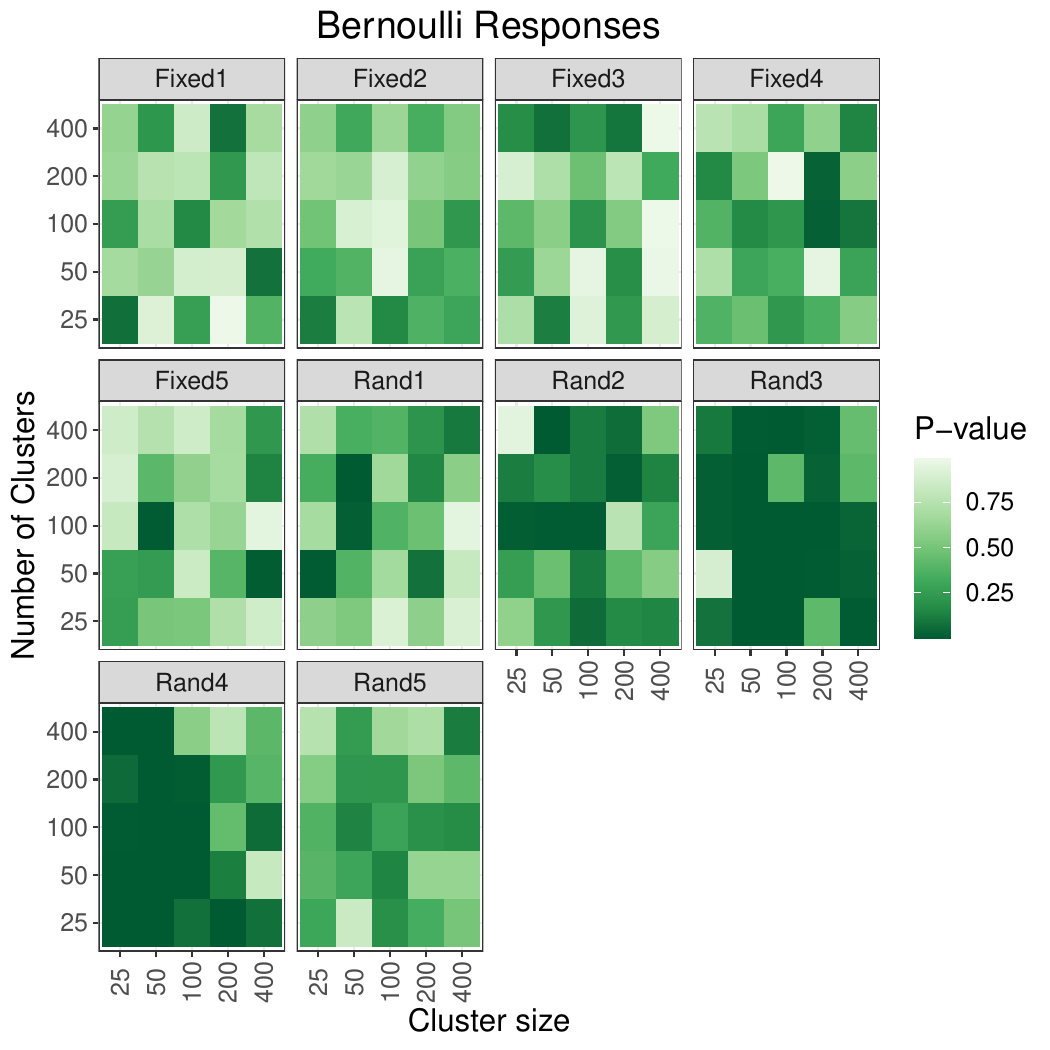}
\caption{$p$-values from Shapiro-Wilk tests applied to the fixed and random effects estimates obtained using maximum PQL estimation, under the conditional regime.} 
\end{figure}

\begin{figure}[H]
\includegraphics[width=0.49\linewidth]{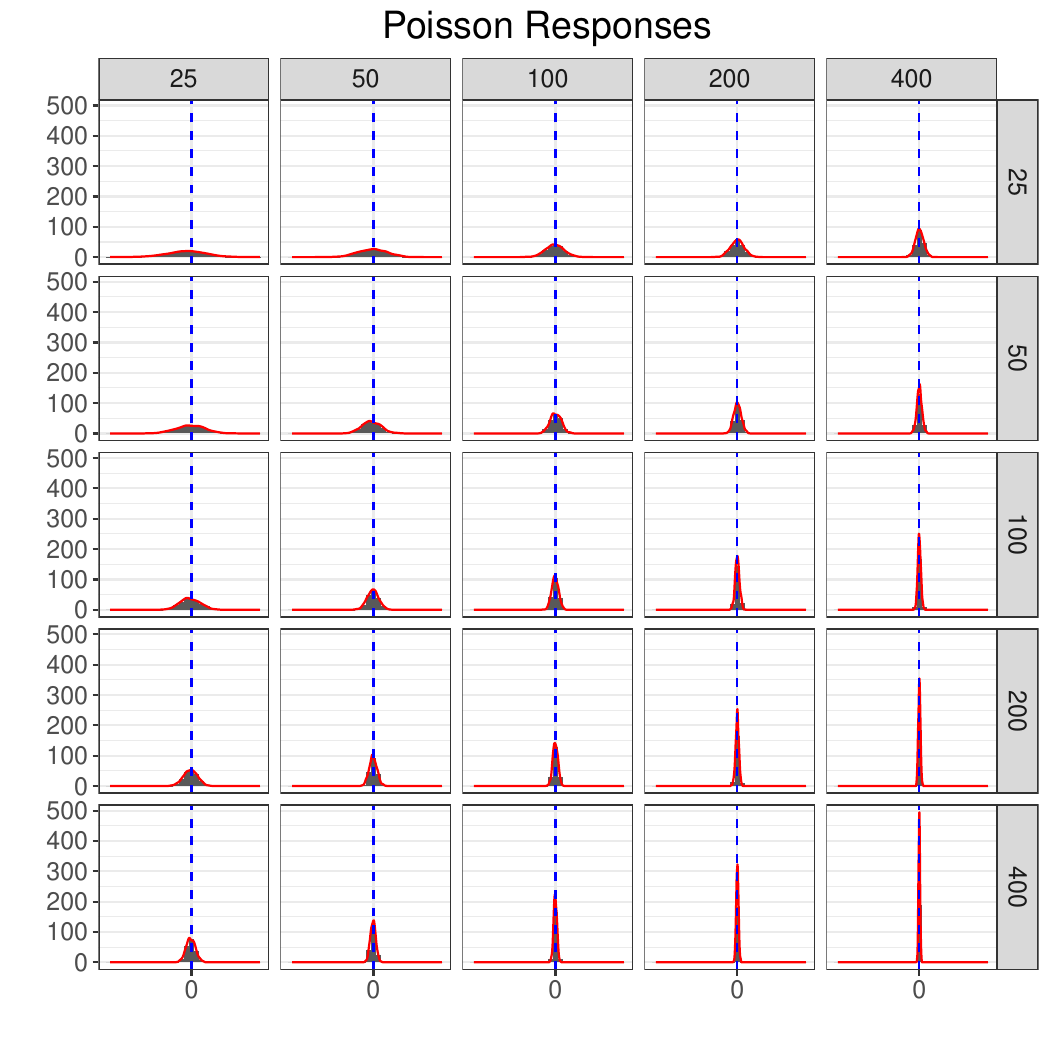}
\includegraphics[width=0.49\linewidth]{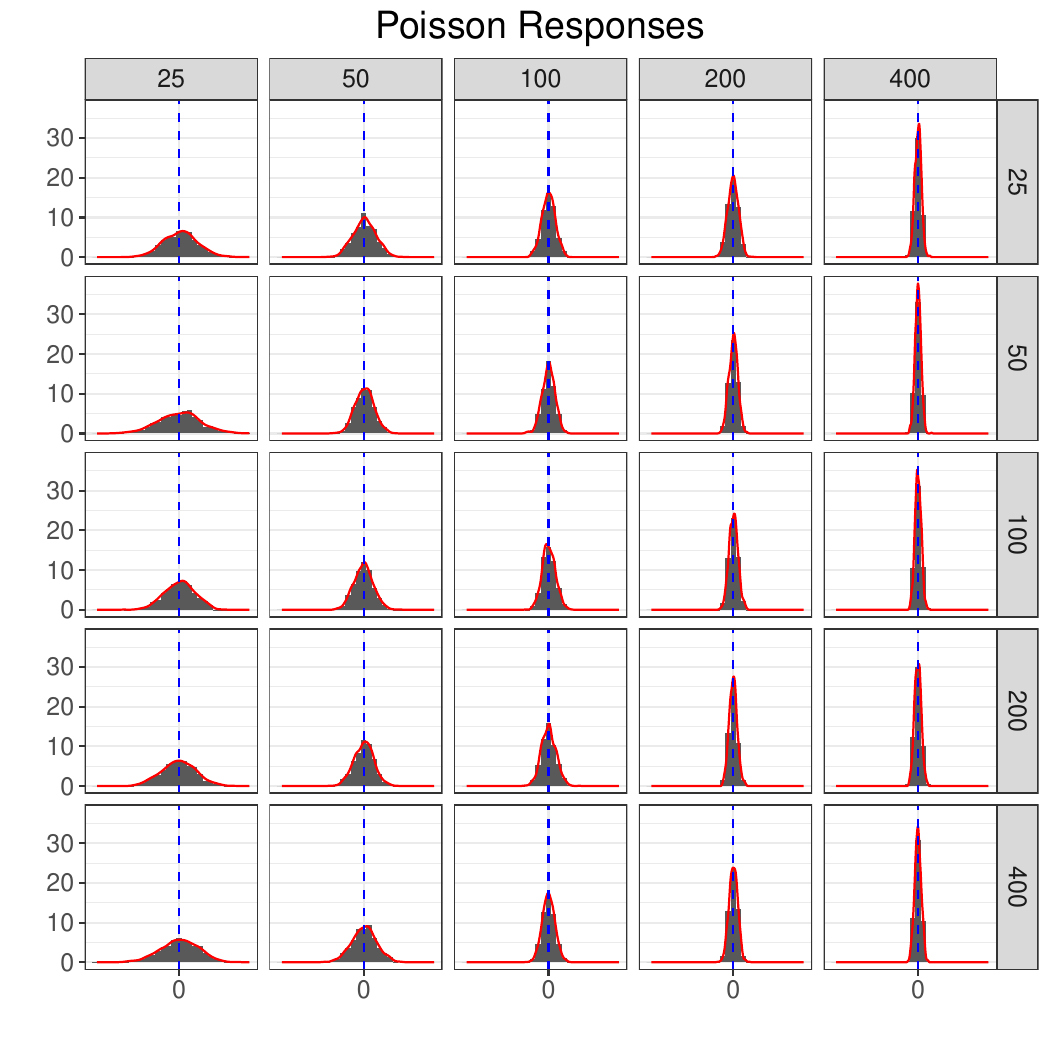}
\includegraphics[width=0.49\linewidth]{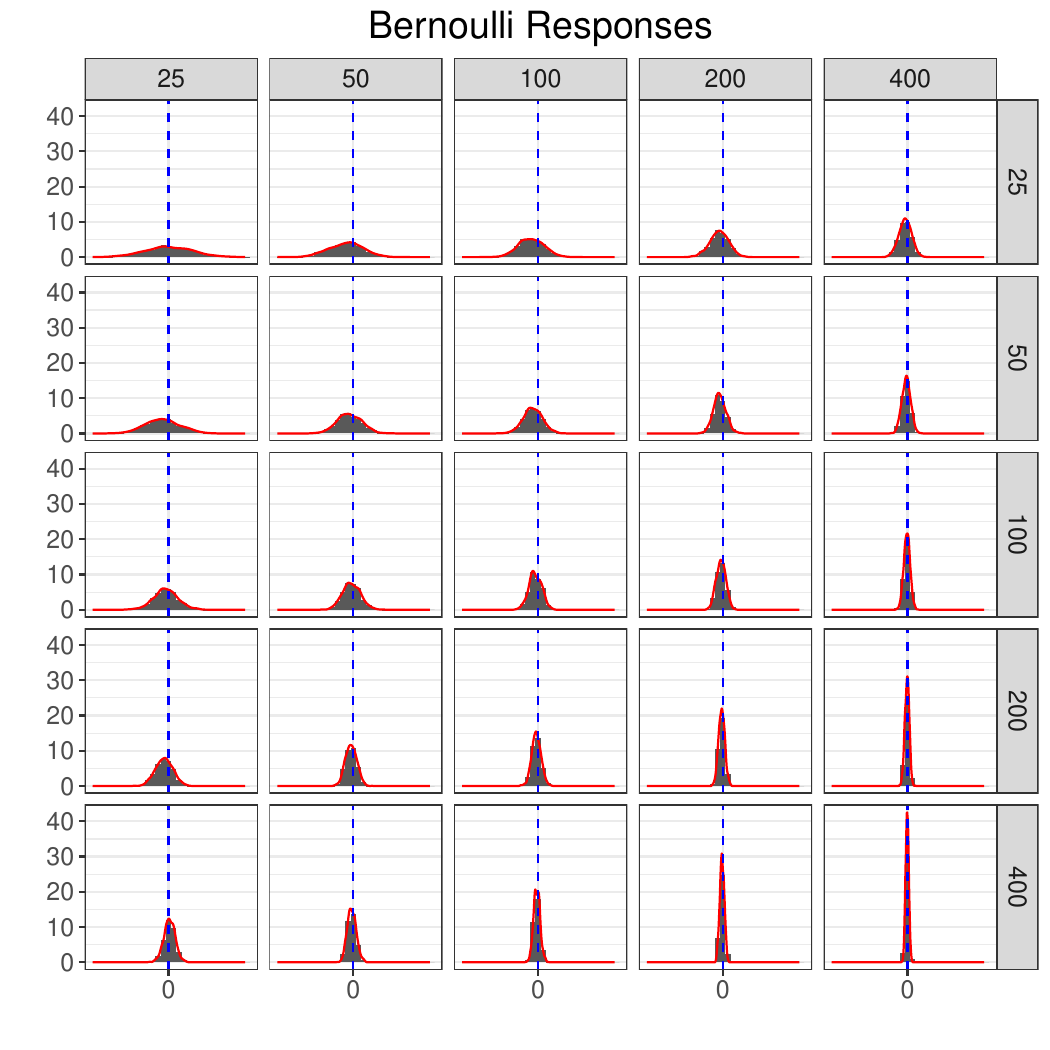}
\includegraphics[width=0.49\linewidth]{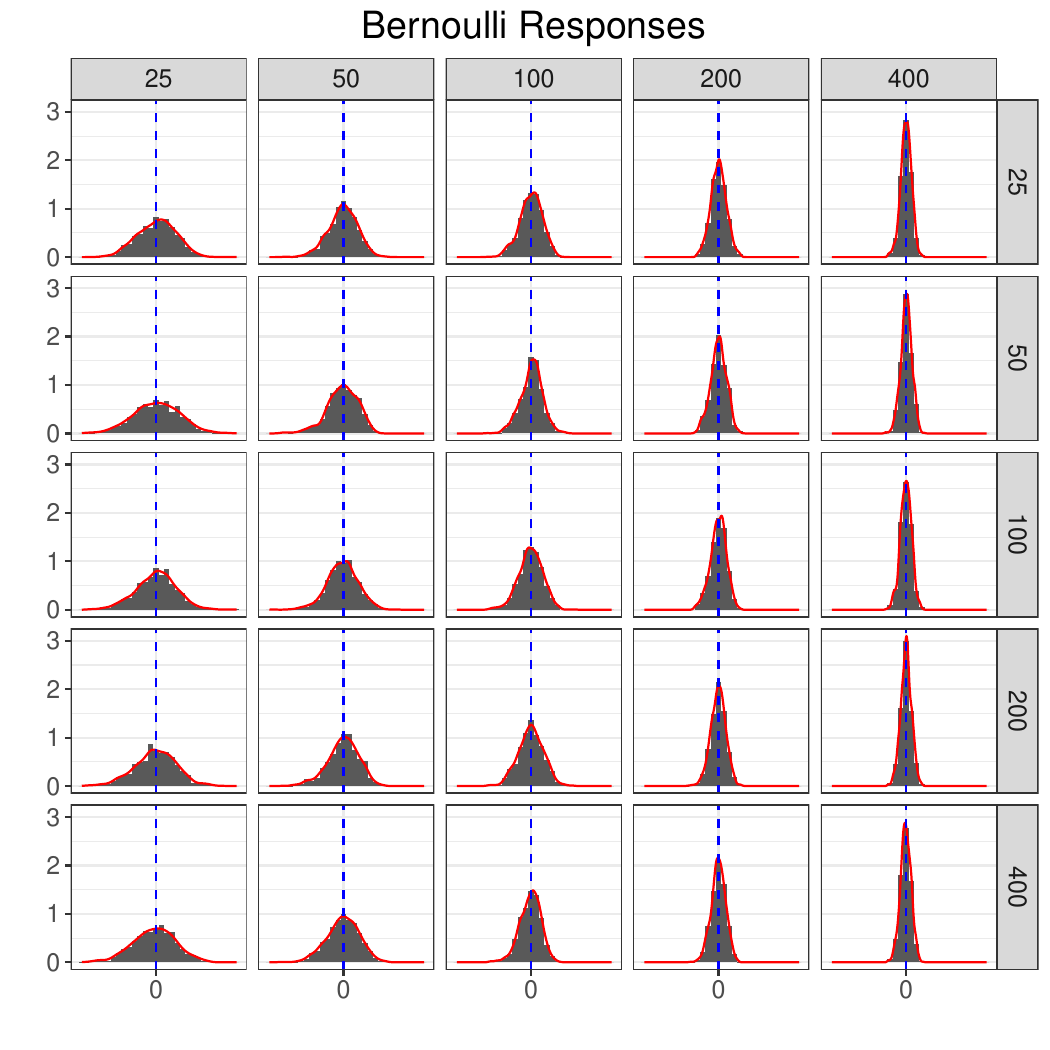}
\caption{Histograms for the third components of $\hat{\bmbeta} - \dot{\bmbeta}$ (left panels) and $\hat{\bm{b}}_1 - \dot{\bm{b}}_1$ (right panels), under the unconditional regime. Vertical facets represent the cluster sizes, while horizontal facets represent the number of clusters. The dotted blue line indicates zero, and the red curve is a kernel density smoother.} 
\end{figure}

\subsection{{\boldmath ${G}$} = 4  {\boldmath ${I}_2$}}

\begin{figure}[H]
\centering
\includegraphics[width=0.95\linewidth]{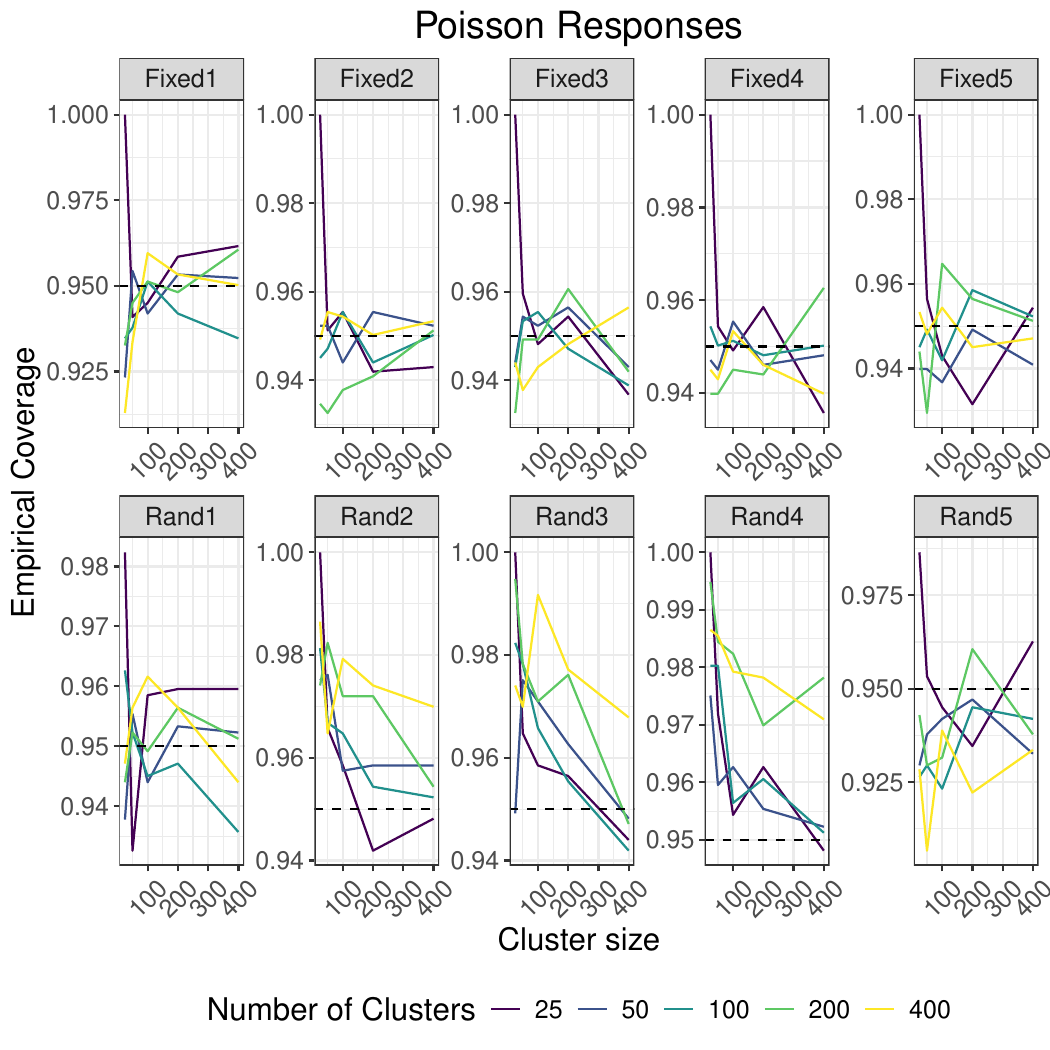}
\caption{Empirical coverage probability of 95\% coverage intervals for the five fixed and random effects estimates, obtained under the unconditional regime with Poisson responses.} 
\end{figure}

\begin{figure}[H]
\centering
\includegraphics[width=0.95\linewidth]{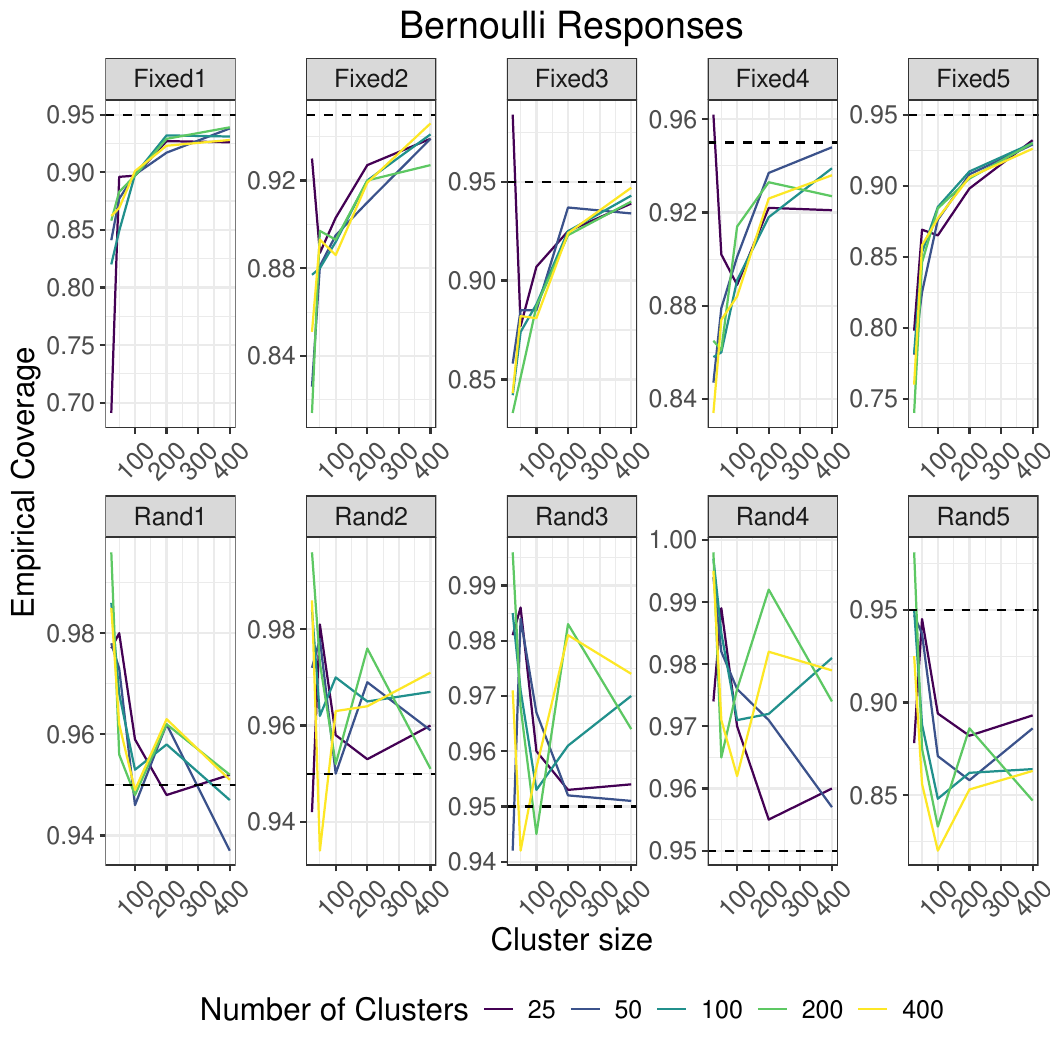}
\caption{Empirical coverage probability of 95\% coverage intervals for the five fixed and random effects estimates, obtained under the unconditional regime with Bernoulli responses.} 
\end{figure}

\begin{figure}[H]
\includegraphics[width=0.7\linewidth]{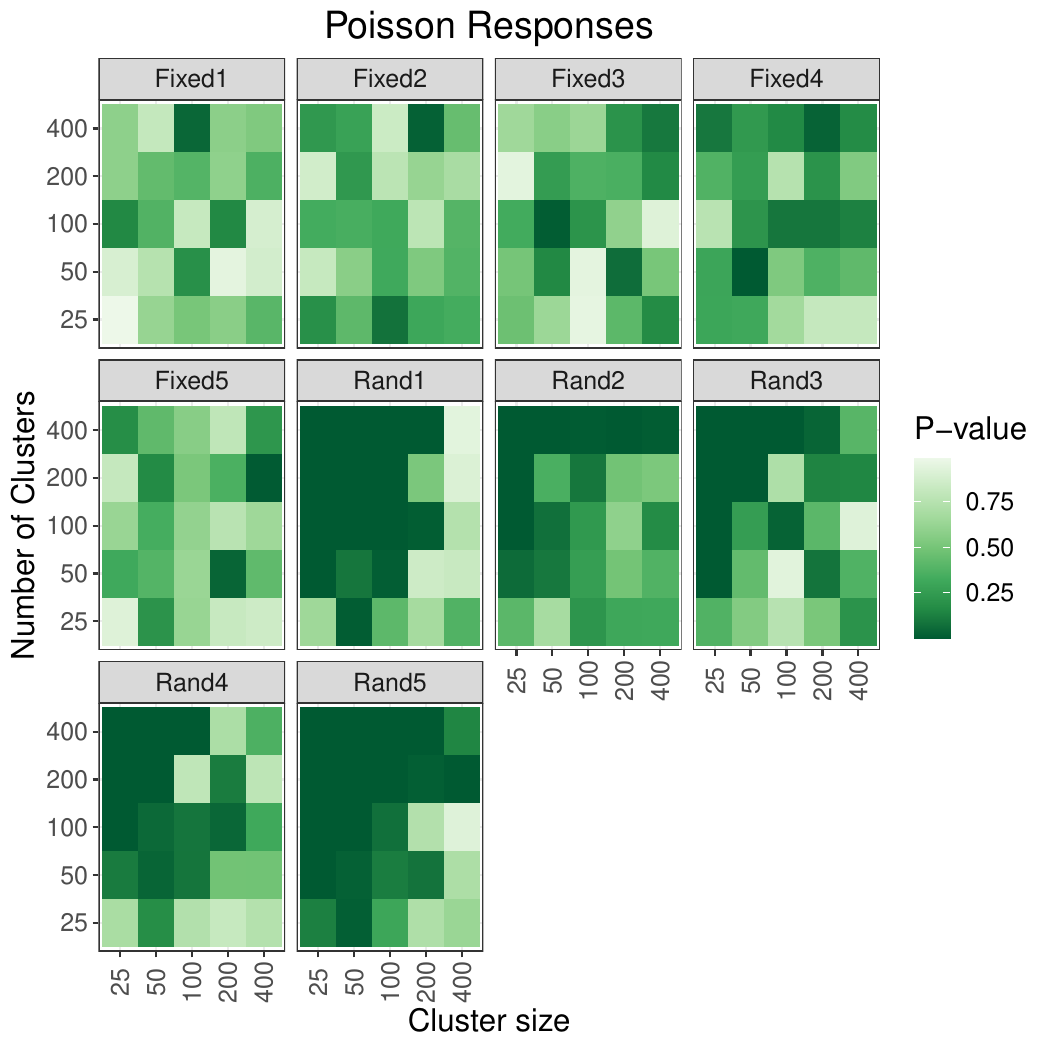}
\includegraphics[width=0.7\linewidth]{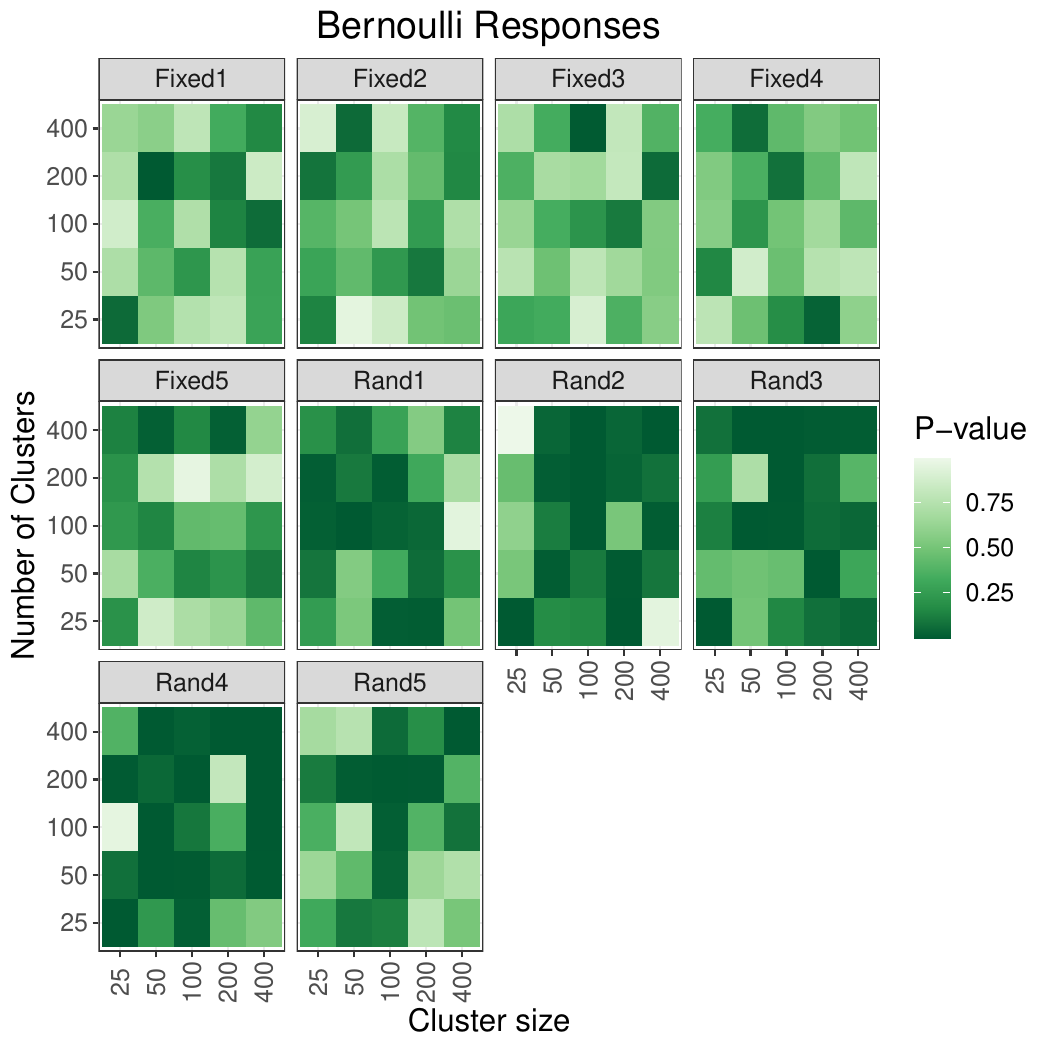}
\caption{$p$-values from Shapiro-Wilk tests applied to the fixed and random effects estimates obtained using maximum PQL estimation, under the unconditional regime.} 
\end{figure}

\begin{figure}[H]
\includegraphics[width=0.49\linewidth]{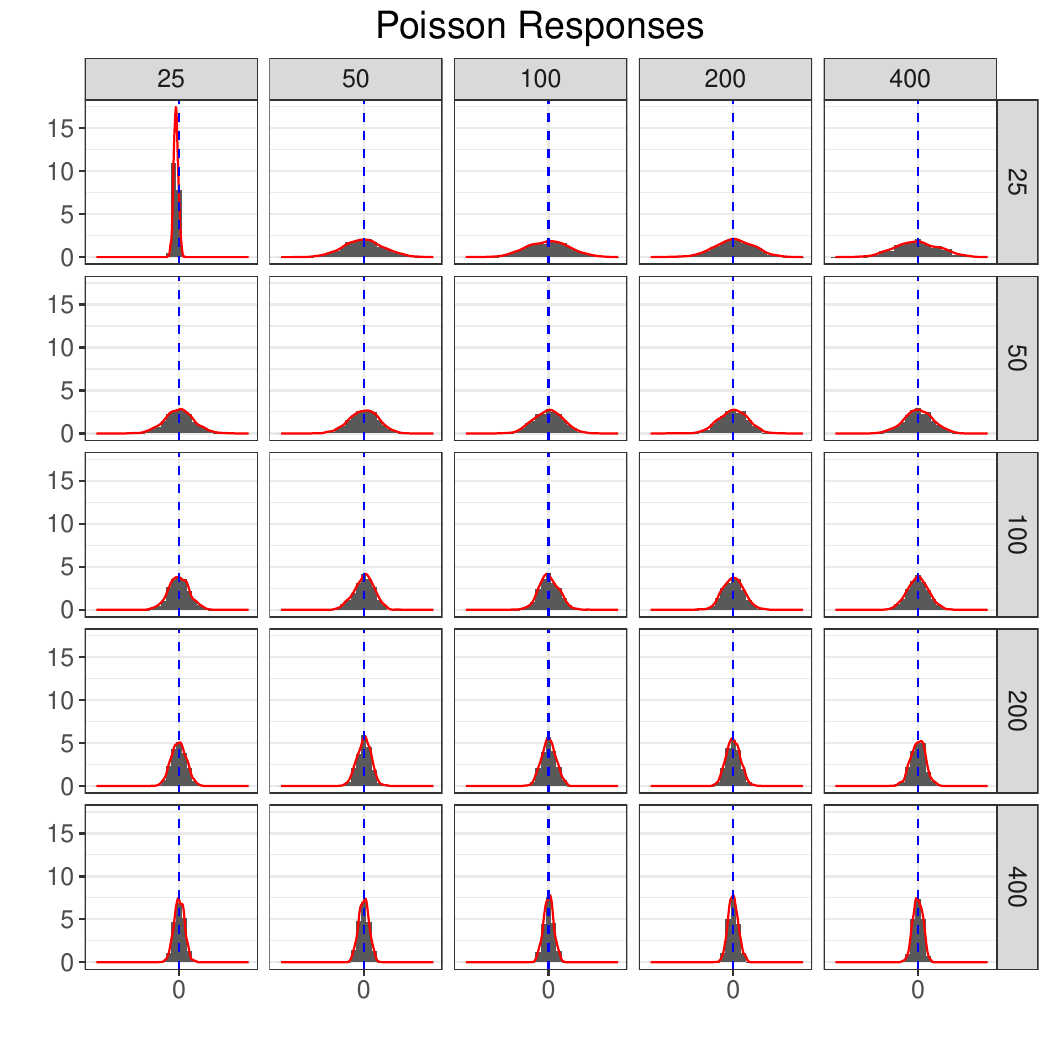}
\includegraphics[width=0.49\linewidth]{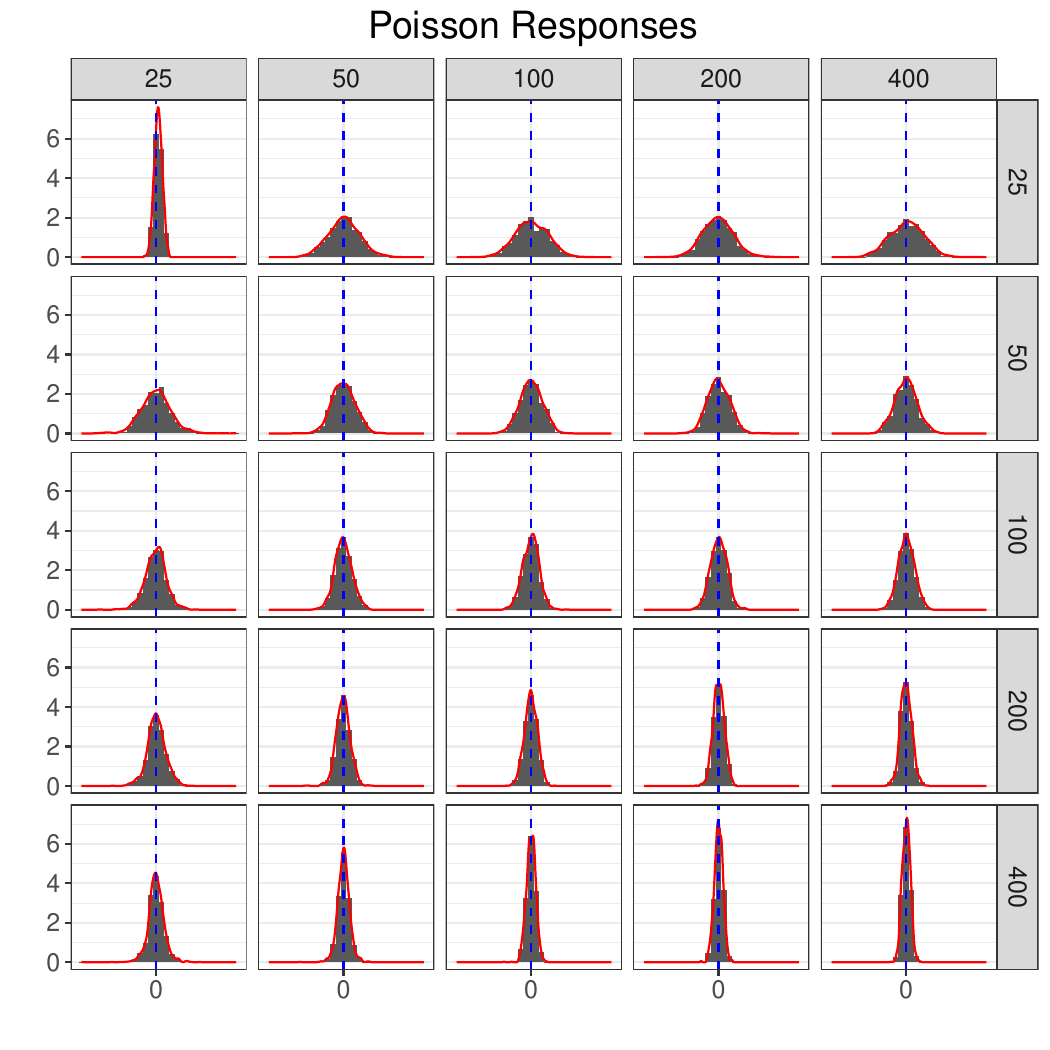}
\includegraphics[width=0.49\linewidth]{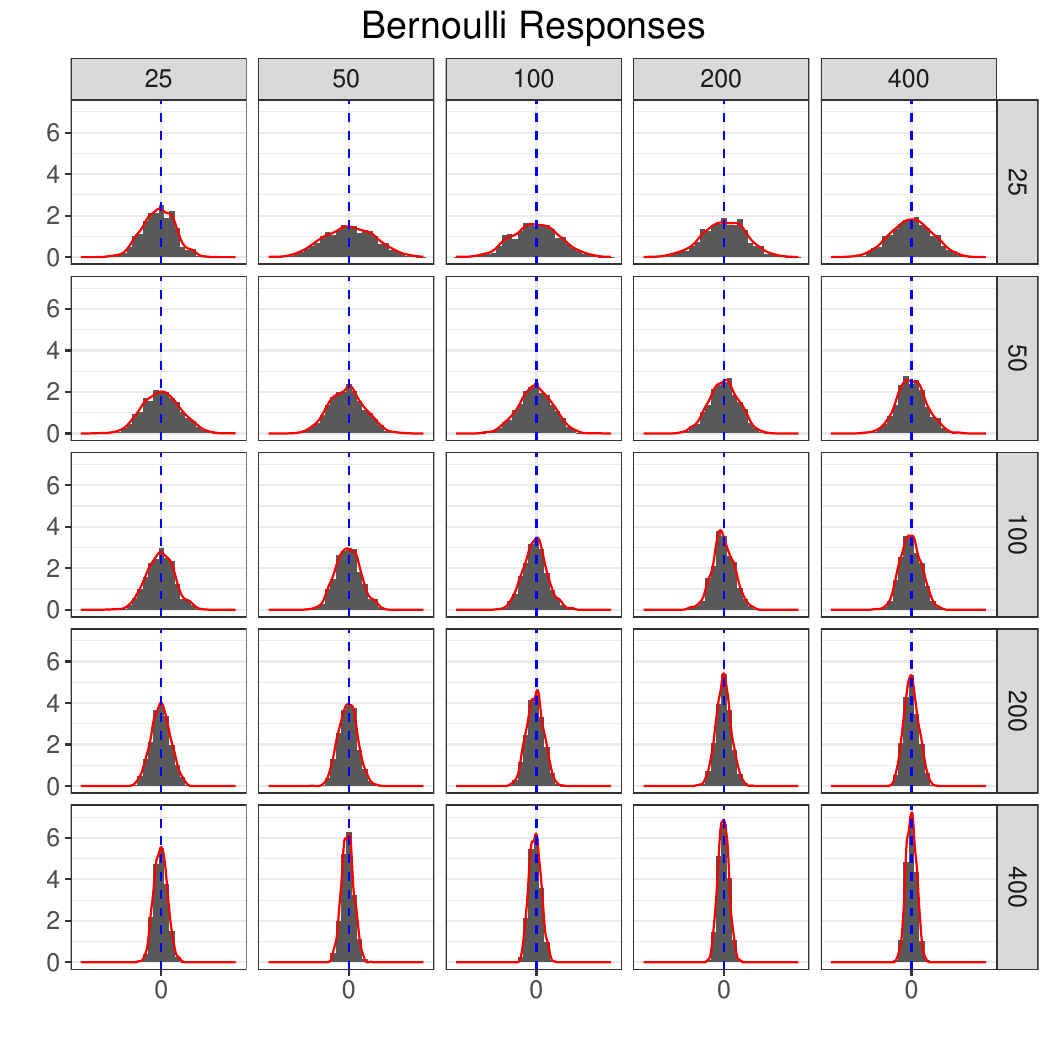}
\includegraphics[width=0.49\linewidth]{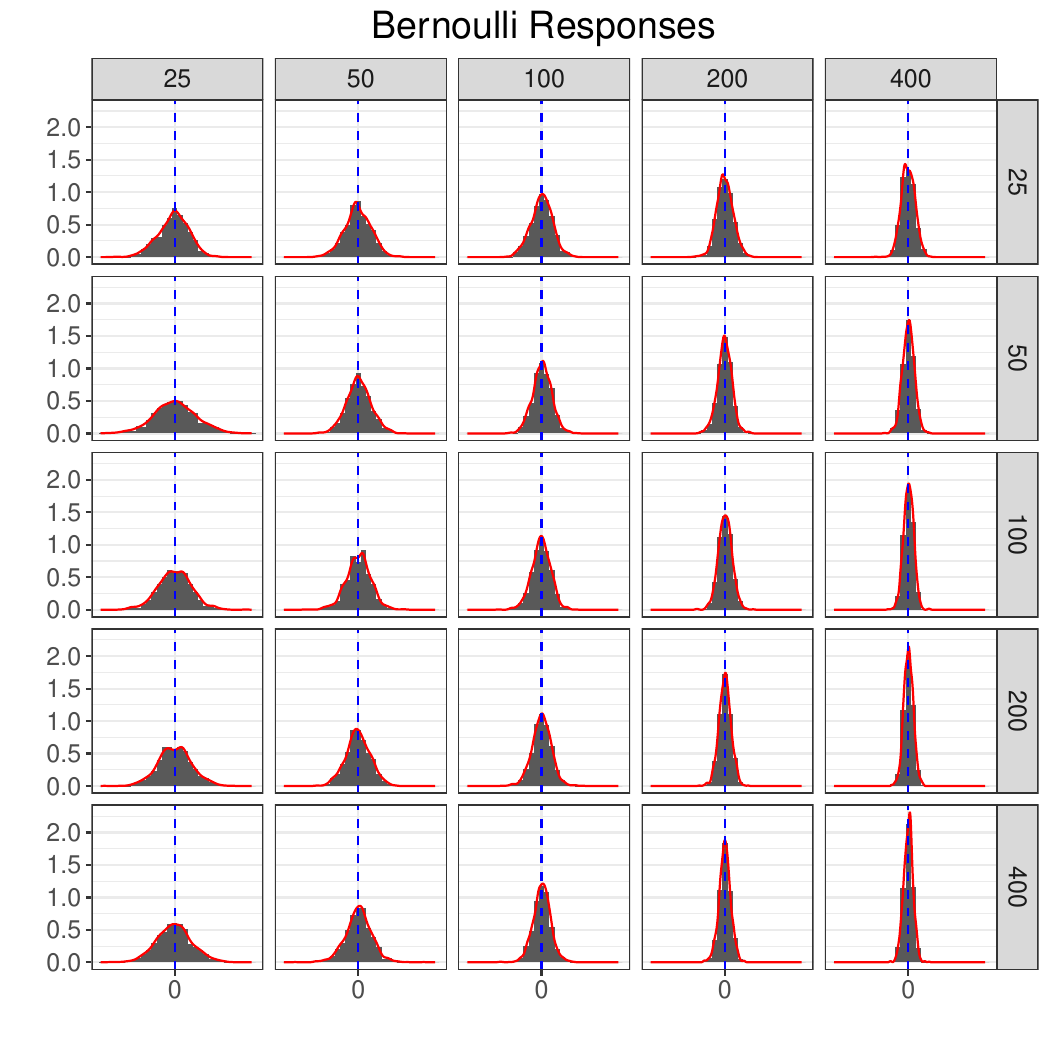}
\caption{Histograms for the third components of $\hat{\bmbeta} - \dot{\bmbeta}$ (left panels) and $\hat{\bm{b}}_1 - \dot{\bm{b}}_1$ (right panels), under the unconditional regime. Vertical facets represent the cluster sizes, while horizontal facets represent the number of clusters. The dotted blue line indicates zero, and the red curve is a kernel density smoother.} 
\end{figure}

\begin{figure}[H]
\centering
\includegraphics[width=0.95\linewidth]{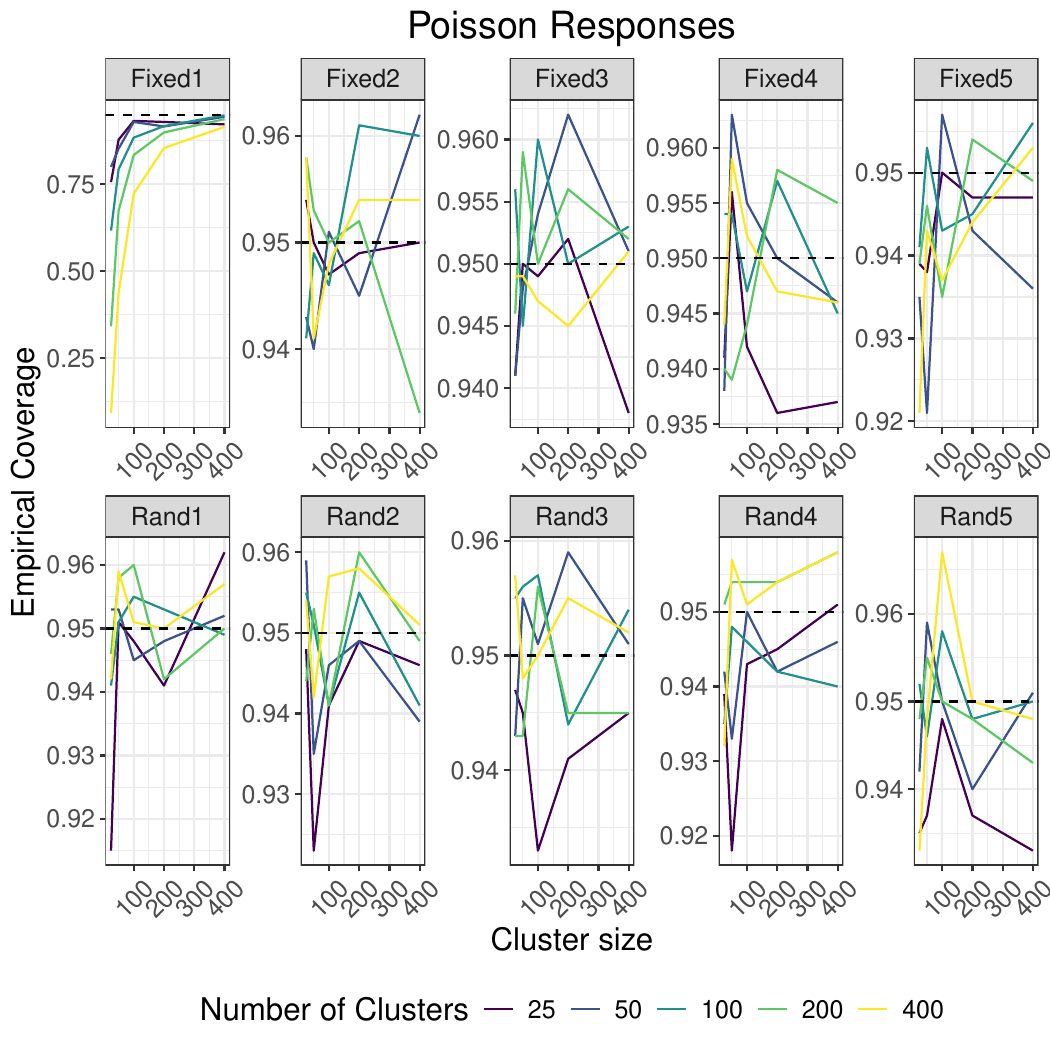}
\caption{Empirical coverage probability of 95\% coverage intervals for the five fixed and random effects estimates, obtained under the conditional regime with Poisson responses.} 
\end{figure}

\begin{figure}[H]
\centering
\includegraphics[width=0.95\linewidth]{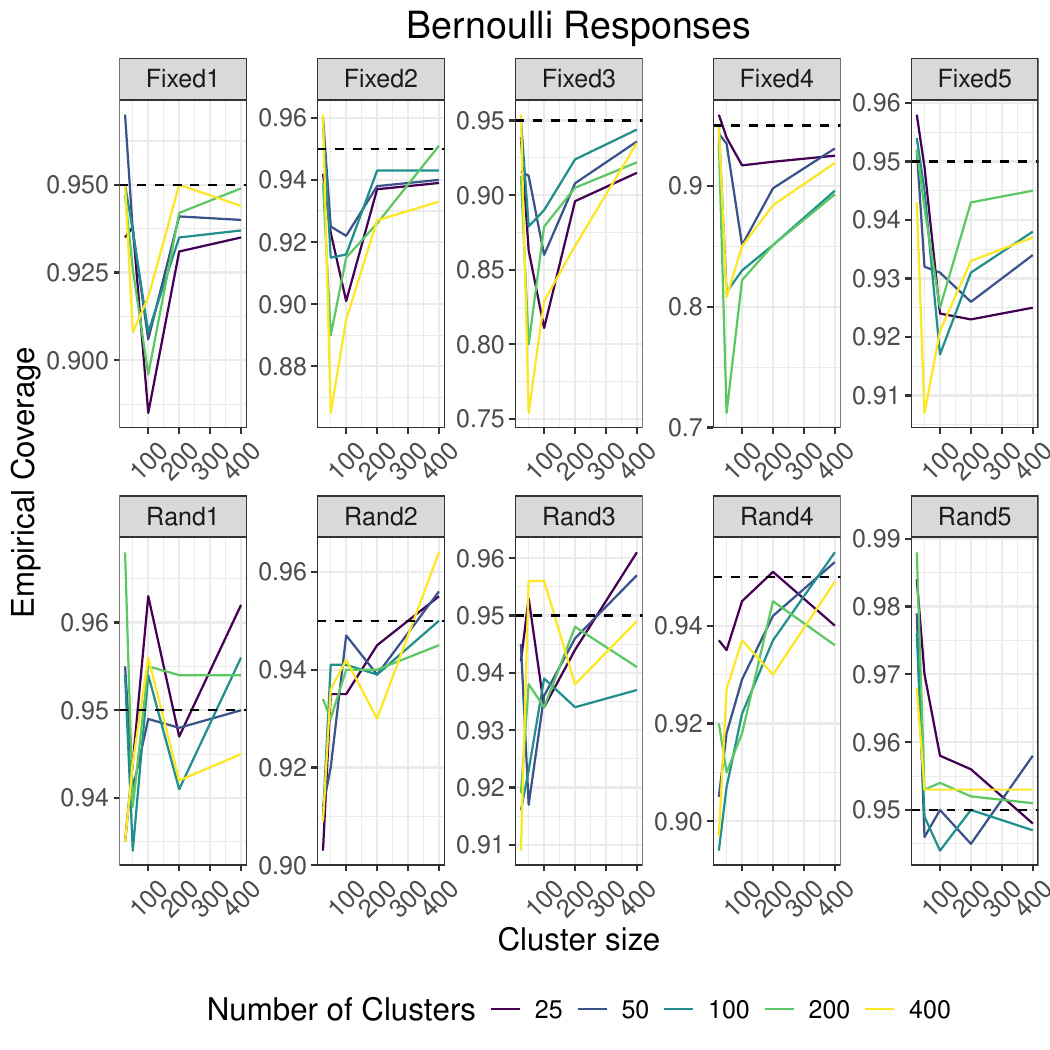}
\caption{Empirical coverage probability of 95\% coverage intervals for the five fixed and random effects estimates, obtained under the conditional regime with Bernoulli responses.} 
\end{figure}

\begin{figure}[H]
\includegraphics[width=0.7\linewidth]{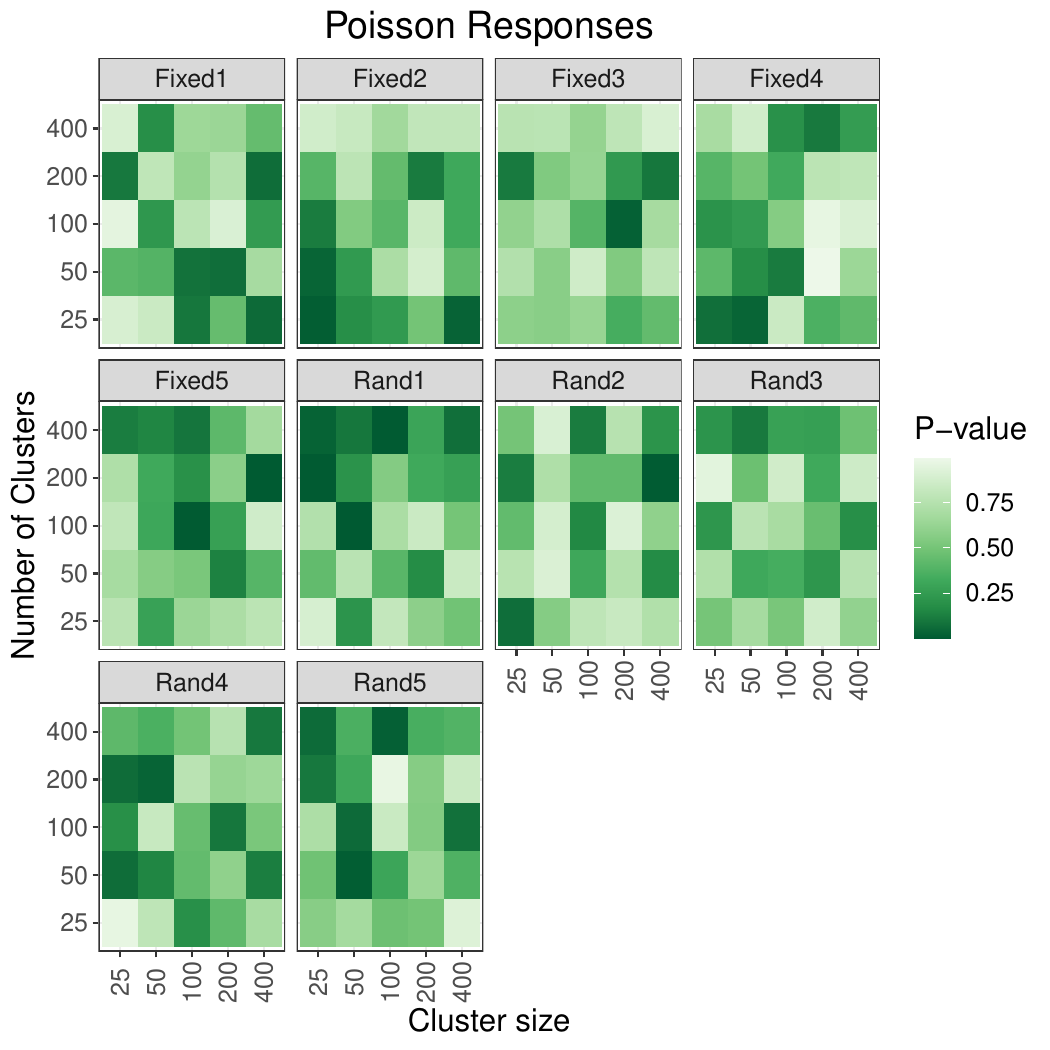}
\includegraphics[width=0.7\linewidth]{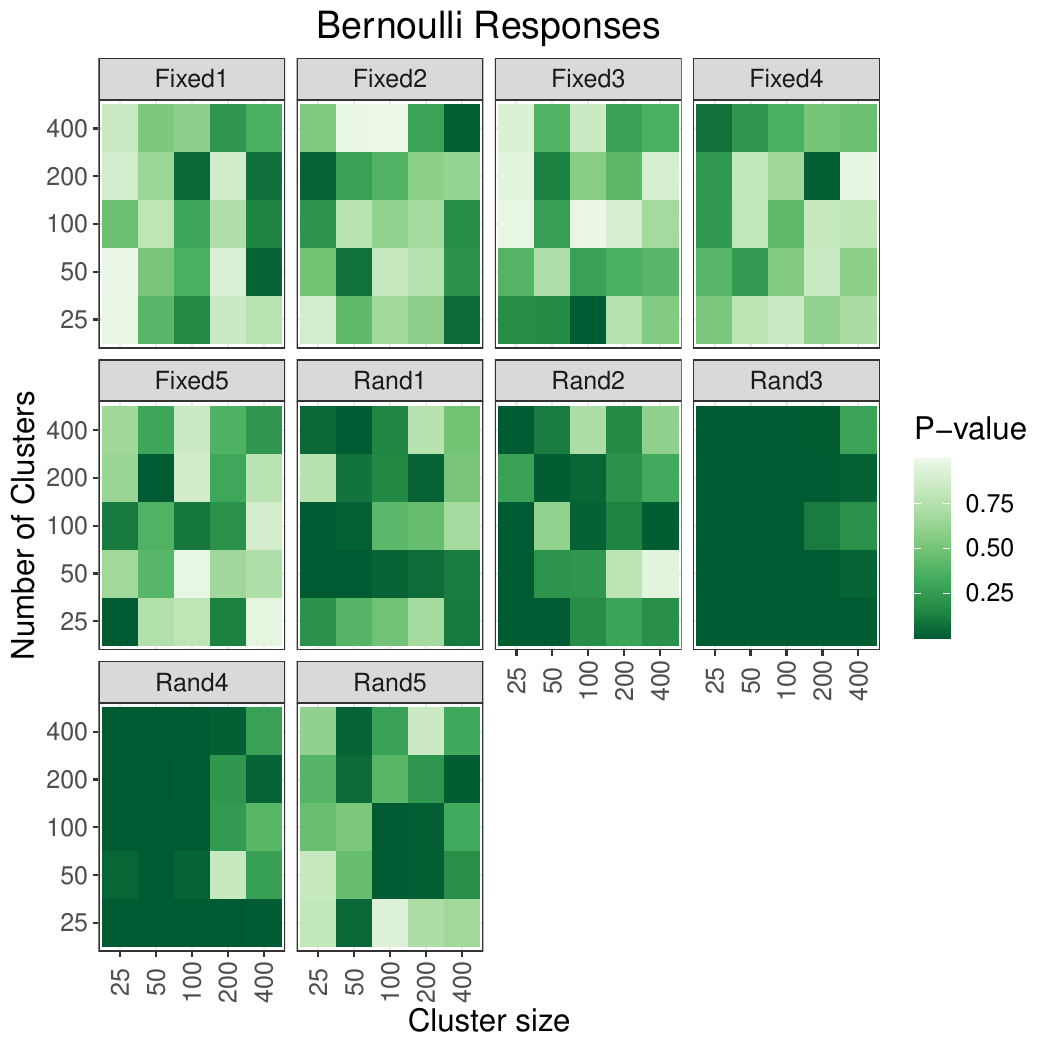}
\caption{$p$-values from Shapiro-Wilk tests applied to the fixed and random effects estimates obtained using maximum PQL estimation, under the conditional regime.} 
\end{figure}

\begin{figure}[H]
\includegraphics[width=0.49\linewidth]{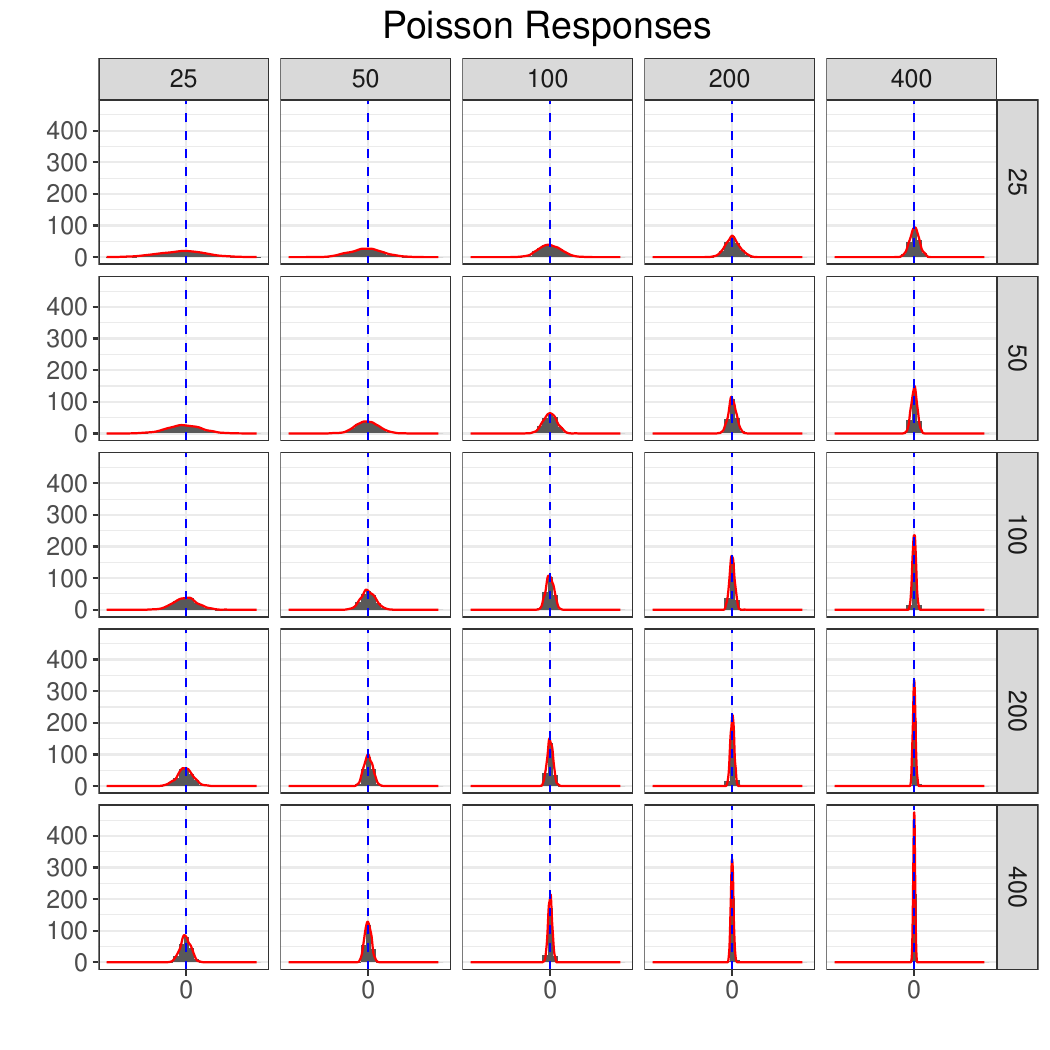}
\includegraphics[width=0.49\linewidth]{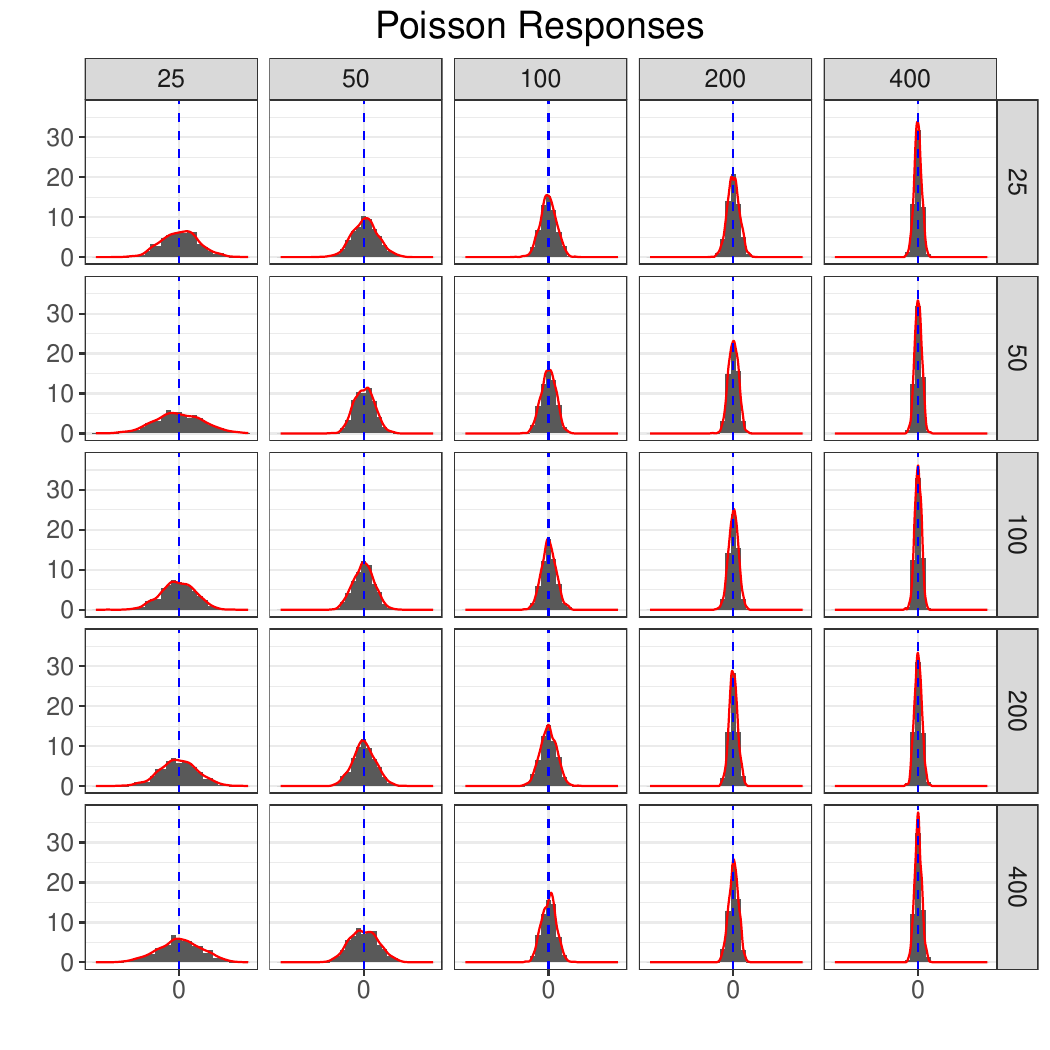}
\includegraphics[width=0.49\linewidth]{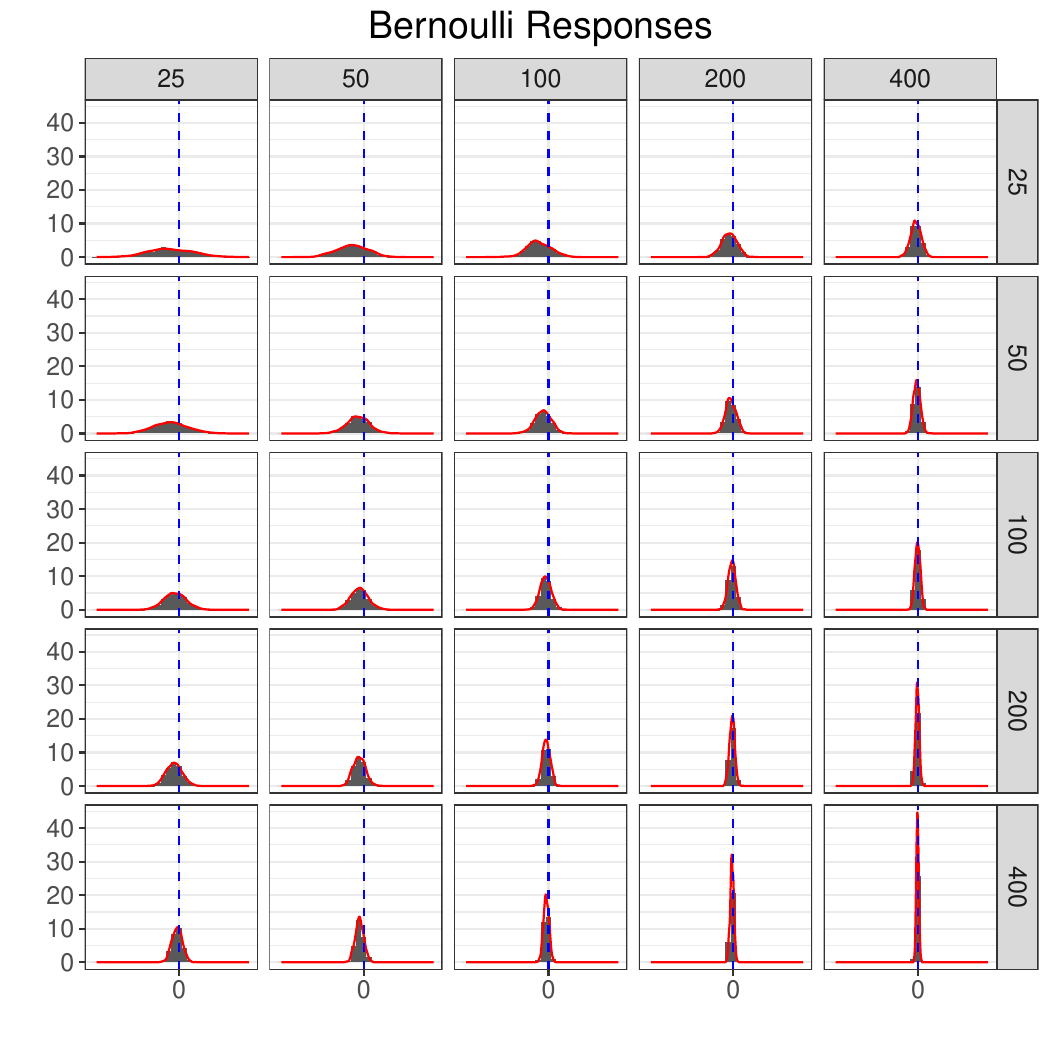}
\includegraphics[width=0.49\linewidth]{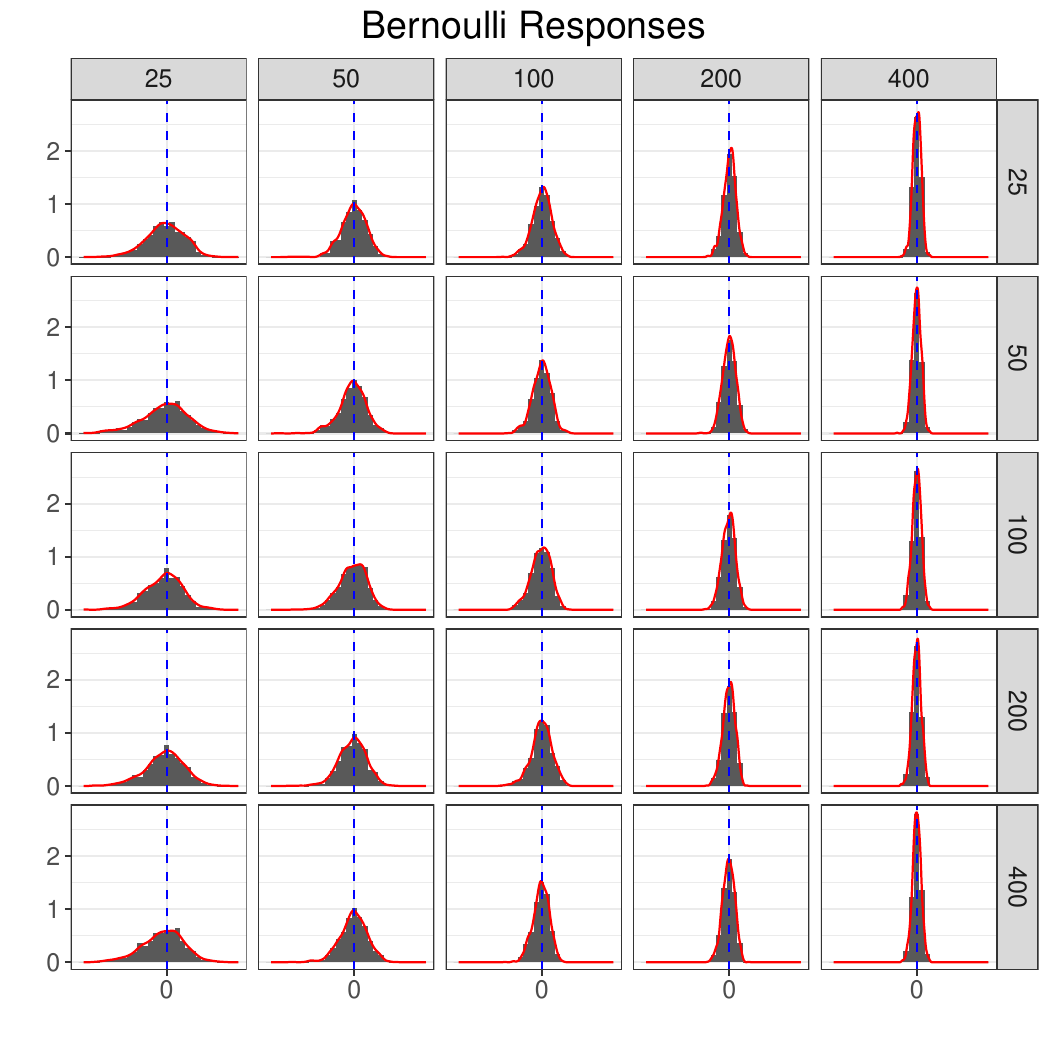}
\caption{Histograms for the third components of $\hat{\bmbeta} - \dot{\bmbeta}$ (left panels) and $\hat{\bm{b}}_1 - \dot{\bm{b}}_1$ (right panels), under the unconditional regime. Vertical facets represent the cluster sizes, while horizontal facets represent the number of clusters. The dotted blue line indicates zero, and the red curve is a kernel density smoother.} 
\end{figure}

\end{document}